\documentclass[12pt, reqno]{amsart}
\usepackage{amssymb}
\usepackage{graphicx}
\usepackage{xcolor} 
\usepackage{tensor}

\usepackage{enumitem}

\usepackage{fullpage} 

\usepackage[english]{babel}
\usepackage[utf8]{inputenc}
\usepackage[colorinlistoftodos]{todonotes}

\usepackage{amsfonts}
\usepackage{amssymb}
\usepackage{amsmath}
\usepackage{amsthm}

\usepackage{graphics}
\usepackage{pdfsync}

\usepackage{graphicx}
\usepackage{amssymb}
\usepackage{epstopdf}

\usepackage{color}
\usepackage{tikz}
\usepackage{scrextend}

\newtheorem{theorem}{Theorem}[section]
\newtheorem{corollary}[theorem]{Corollary}
\newtheorem{lemma}[theorem]{Lemma}
\newtheorem{proposition}[theorem]{Proposition}
\theoremstyle{definition}
\newtheorem{definition}[theorem]{Definition}

\theoremstyle{remark}
\newtheorem{remark}[theorem]{Remark}
\numberwithin{equation}{section}

\newtheorem{thm}{Theorem}[section] 

\theoremstyle{plain} 
\newcommand{\thistheoremname}{}
\newtheorem{genericthm}[thm]{\thistheoremname}
\newenvironment{namedthm}[1]
  {\renewcommand{\thistheoremname}{#1}%
   \begin{genericthm}}
  {\end{genericthm}}

\newtheorem*{genericthm*}{\thistheoremname}
\newenvironment{namedthm*}[1]
  {\renewcommand{\thistheoremname}{#1}%
   \begin{genericthm*}}
  {\end{genericthm*}}

\newcommand{\cH}{\mathcal H}
\newcommand{\R}{\mathbb{R}}

\newcommand{\wht}[1]{\widehat{#1}}

\newcommand{\0}{\emptyset}

\newcommand{\sgm}{\sigma}

\newcommand{\tht}{\theta}

\newcommand{\calT}{\mathcal T}

\newcommand{\mb}{\mathbb}
\newcommand{\be}{\begin{equation}}
\newcommand{\ee}{ \end{equation}}

\newcommand{\ep}{\varepsilon}

\newcommand{\sg}{\sigma}
\newcommand{\al}{\alpha}

\newcommand{\dd}{\,\mathrm{d}}

\newcommand{\vn}[1]{\|#1\|}
\newcommand{\vm}[1]{\left|#1\right|}

\newcommand{\lpr}{ \left( }
\newcommand{\rpr}{ \right) }

\newcommand{\lng}{\langle}
\newcommand{\rng}{\rangle}

\newcommand{\lmd}{\lambda}

\newcommand{\pt}{\partial}

\usepackage{mathtools}
\newcommand{\defeq}{\vcentcolon=}

\newcommand{\ls}{\lesssim}

\usepackage{marginnote}

\usepackage{xr-hyper}
\usepackage{hyperref}

\begin{document}

\title{Wave maps on (1+2)-dimensional curved spacetimes}
\author{Cristian Gavrus} 
\address{Department of Mathematics, Johns Hopkins University}
\email{cgavrus1@jhu.edu}

\author{Casey Jao} 
\address{Department of Mathematics, University of California, Berkeley}
\email{cjao@math.berkeley.edu}

\author{Daniel Tataru}
\address{Department of Mathematics, University of California, Berkeley}
\email{tataru@math.berkeley.edu}

\begin{abstract}
In this article we initiate the study of $1+2$ dimensional
wave maps on a curved spacetime in the low regularity setting.
Our main result asserts that in this context the wave maps equation is locally well-posed at almost critical regularity.

As a key part of the proof of this result, we generalize the classical optimal bilinear $ L^2$ estimates for the wave equation to variable coefficients, by means of wave packet decompositions and characteristic energy estimates. This allows us to iterate in a curved $ X^{s,b}$ space.
\end{abstract}

\maketitle

\setcounter{tocdepth}{1}
\tableofcontents

\section{Introduction} \

Let the $(1+2)$-dimensional spacetime $ \mb{R}_t \times \mb{R}_x^2 $ be endowed with a Lorentzian metric
\[
g = g_{\alpha \beta}(t, x) dx^\alpha dx^\beta,
\]
so that the time slices $t = const$ are space-like.
Here $x^0 = t$ and we  adopt the standard convention of referring to spacetime coordinates by Greek indices and purely spatial coordinates by Roman indices.

Given a smooth Riemannian manifold $ (M,h) $ with uniformly bounded geometry, 
a \emph{wave map} $u: (\R^{1+2}, g) \to (M, h)$ is formally a
critical point of the Lagrangian
\begin{align*}
\mathcal{L} (u) 
&= \frac{1}{2} \int_{\R^{1+2}} \lng  du, du \rng_{ T^{*}( \R \times 
\mb{R}^2 )
	\otimes u^{*} T M} \dd \text{vol}_g,
\end{align*}
where $ u^{*} T M = \bigcup_x \{ x \} \times T_{u(x)} M $ is the pullback of $ 
TM $ by $ u $ and $ u^{*} h $ is the pullback metric. 
In local coordinates on $ M $ this becomes
$$
\mathcal{L} (u) =  \frac{1}{2} \int  g^{\al \beta}(z) h_{ij} (u(z)) \pt_{\al} 
u^i(z) \pt_{\beta} u^j(z) \sqrt{\vm{g(z)} } \dd z
$$

One may think of wave maps as the hyperbolic counterpart of harmonic maps. The Euler-Lagrange equations take the following form, and we refer to \cite[Chapter 8]{jost2008riemannian} for related computations: 
\begin{align}
\label{e:EL}
\frac{1}{\sqrt{\vm{g } }} {\bf D}_{\al} \lpr  \sqrt{\vm{g } } g^{\al \beta} \partial_{\beta} u \rpr=0
\end{align}
Here $  {\bf D} $ denotes the pullback covariant derivative on $ u^{*} T M $ given by $ {\bf D}_X V=\nabla_{u_{*} X} V $ where $ \nabla $ is the Levi-Civita connection on $ T M $.

\

In coordinates, wave maps solve a coupled system of nonlinear wave equations. We review two useful settings for this problem:

\

\emph{Intrinsic formulation.} Suppose the image of $ u $ is supported in the domain of a local coordinate patch of $ M $. Then the wave maps equation~\eqref{e:EL} is written as
\be \label{WM:int}
\tilde{\Box}_g u^i=- \Gamma^i_{jk} (u) g^{\al \beta } \partial_{\al}u^j \partial_{\beta} u^k 
\ee
where $ \tilde{\Box}_g $ is the Laplace-Beltrami wave operator 
\be \label{LB:op} 
\tilde{\Box}_g u = \vm{g }^{-\frac{1}{2} } \pt_{\al} \Big(  \vm{g }^{\frac{1}{2}} g^{\al \beta} \partial_{\beta} u \Big)
\ee
denoted this way in order to distinguish it from its principal part $ \Box_g $ defined by
$$ 
 \Box_g u = g^{\al \beta} \partial_{\al} \partial_{\beta} u,
$$
and $ \Gamma^i_{jk} $ are the Christoffel symbols on $(M,h)$. 

\

\emph{Extrinsic formulation.} When the manifold $ M $ is isometrically embedded into an euclidean space $ \mb{R}^m $ wave maps can be equivalently defined extrinsically; see for instance \cite{shatah1998geometric}. If $ M $ is compact, such an embedding exists for $ m $ large enough by Nash's theorem. The Lagrangian becomes 
$$ \mathcal{L} (u) =  \frac{1}{2} \int  g^{\al \beta} \lng \partial_{\al} u, \partial_{\beta} u \rng \sqrt{\vm{g } } \dd z
$$
Formal critical points satisfy $ \tilde{\Box}_g u \perp T_u M $ and our equation takes the form
\be \label{WM:ext}
\tilde{\Box}_g u^i=- \mathcal{S}^i_{jk} (u) g^{\al \beta } \partial_{\al}u^j 
\partial_{\beta} u^k,
\ee
where $ \mathcal{S} $ stands for the second fundamental form on $M$. 

\

\emph{Initial data} for the Cauchy problem are chosen such that 
$$ u(0,x) \in M, \ \pt_t u(0,x) \in T_{u(0,x)} M .$$ The initial 
data space can be viewed as an infinite dimensional manifold, 
in which the wave map evolution takes place.

\

The formulations above exhibit the presence of the \emph{null form} 
\[ 
Q_g(u,v)= g^{\al \beta} \pt_{\al} u \pt_{\beta} v.
\] 
Its structure
is so that it eliminates quadratic resonant interactions in the wave map evolution, and has played a key role in the low regularity 
well-posedness of wave maps in the flat case (\cite{klainerman1993space}, 
\cite{klainerman1997remark}, \cite{klainerman2002bilinear}, \cite{Tat}, 
\cite{Tao2}) as well as in other problems, especially in low dimensions. A simple way  to see this cancellation is via the formula
\be \label{null:form:box}
2 Q_g(u,v)= 2 g^{\al \beta} \pt_{\al} u \pt_{\beta} v = \Box_g(uv) - u \Box_g 
v-v \Box_g u.	
\ee
However in \emph{Section} \ref{Sec:alg} we will obtain estimates for $ 
Q_g(u,v) $ directly. 

\

In both the intrinsic and the extrinsic formulations one can use the coordinates to define Sobolev spaces for the initial data sets.
To understand the relevant range of Sobolev indices we recall that 
in the Minkovski space the wave maps system is invariant under the dimensionless 
\emph{rescaling} 
\[
u \mapsto u(\lambda t, \lambda x), \qquad g \mapsto g(\lambda t, \lambda x) 
\]
which identifies the 
\emph{critical} (scale-invariant) initial data space as $\dot{H}^1 \times L^2 (\R^2)$. In addition, one may consider initial 
data spaces for more regular data, of the form 
\[
\cH^s = \{ (u_0,u_1): \R^2 \to TM;\ \nabla u_0, u_1  \in H^{s-1}(\R^2) \}, 
\qquad s > 1.
\]
It is this latter case which is considered in the 
present paper.

Several remarks are in order here. First of all,  the constant functions are always acceptable states in the context of manifold valued maps $u_0$; this is why we include them all in our state space $\cH^s$
above. Secondly, one may ask whether the definition of the above
spaces is context dependent. We separately discuss the two formulations
above.

The simplest set-up is in the case
of the extrinsic formulation with $(M,h)$ a compact manifold. 
There one can directly use the $H^{s-1}(\R^2)$ spaces for derivatives of $\R^m$
valued functions. For simplicity, this is the set-up we adopt here.

For the intrinsic formulation matters are not as straightforward. 
The fact that $s > 1$ heuristically insures that functions $u_0$ 
with $\nabla u_0 \in H^{s-1}$ are H\"older continuous. Thus locally
such functions are in the domain of a single chart, and the 
$H^{s-1}$ norms can be defined using local coordinates. Further,
by Moser estimates the local spaces are algebraically independent of
the choice of the local chart. Finally, the global spaces are obtained
from the local ones by using a suitable collection of local charts.
Of course, some care is required here in order to avoid a circular 
argument. Most generally such a construction would apply for manifolds $(M,g)$ with uniformly bounded geometry.

 Finally we remark that in the setting of low-regularity discontinuous solutions as in \cite{TataruWM}, defining the critical Sobolev 
 spaces $ \dot{H}^1_x(\mathbb{R}^2 \to M) $ is a delicate matter due to the need of having an appropriate topology on these spaces that is independent of the choice of isometric embeddings.

Throughout the paper we assume that
\begin{itemize}
    \item The metric coefficients $g_{\alpha \beta}$
and $g^{\alpha\beta}$ are uniformly bounded.
\item The surfaces $t = const$ are uniformly space-like,
\[
g_{ij} v^i v^j \gtrsim |v|^2.
\]
\end{itemize}
These conditions in turn imply the bound from below
\[
- g^{00} \gtrsim 1. 
\]

Now we are ready to formulate our main result, which establishes well-posedness in the \emph{energy-subcritical} regime 
$s > 1$, as a stepping stone toward the energy-critical problem on a curved 
background.  

\

\begin{theorem} \label{thm:local}
	Let $M \subset \mb{R}^m $ be an embedded manifold. Assume that
	$\partial^2_{t,x} g^{ij} \in L^2_t L^\infty_x$. Then the Cauchy problem for the wave maps equation \eqref{WM:ext} is locally 
	well-posed in $\cH^s$
	for $ 1< s \leq 2 $. 
\end{theorem}

\

Our result includes existence, uniqueness, locally Lipschitz dependence on the initial data and persistence of regularity as explained in Section~\ref{reduction}. 
 
 While our main result above applies to the large data problem,
the bulk of the paper is devoted to the small data problem, to which 
the large data case reduces after a suitable localization.
To state the small data result, we replace the qualitative property of the metric $\nabla^2 g \in L^2 L^\infty$ with a quantitative version
which applies in the unit time interval $t \in [0,1]$:
\begin{equation} \label{assumption:rescaled:metric}
\| \partial_{t,x} g^{\alpha\beta} \|_{L^\infty L^\infty} + \| \partial^2_{t,x} g^{\alpha\beta} \|_{L^2 L^\infty} \le \eta^2.
\end{equation}
Here $\eta \ll 1$ is a fixed small parameter.

For the next result below, we work in a local patch in the intrinsic setting. We assume that $0$ belongs to the range of local coordinates 
associated to this patch, and work with data $(u_0,u_1)$ which is small in $ H_x^s \times H_x^{s-1}$. Then $u_0$ is  continuous and uniformly
small. As long as this property persists, the solution will remain within the domain of the local patch. Hence its regularity can be measured in those coordinates. Our small data result is as follows:

\

\begin{theorem} \label{thm:small}
	Let $ 1<s \leq 2 $ and assume $ 
	g $ satisfies \eqref{assumption:rescaled:metric} in the time interval $[0,1]$. There exists $ \epsilon >0 $ such that for any 
	initial data set $ (u_0,u_1) $ satisfying  
	\be \label{small:id}
	\vn{(u_0,u_1)}_{H^s \times H^{s-1}} \leq \epsilon 
	\ee
	there exists a unique solution $ u $ to the wave maps problem 
	\eqref{WM:int} with this data in the space 
	$ C([0,1];H^s) \cap C^1([0,1];H^{s-1}) $, satisfying
	\be \label{sol:bd}
	\vn{(u,\partial_t u)}_{L^{\infty}(H^s \times H^{s-1})} \ls 
	\vn{(u_0,u_1)}_{H^s \times H^{s-1}}.
	\ee
	The solution has a 	Lipschitz continuous dependence on the initial data. 
\end{theorem}

\

Here the uniqueness is interpreted in the classical sense 
if $s = 2$. For $1< s < 2$, the $H^s \times H^{s-1}$ can be 
defined as the unique limits of $H^2 \times H^1$ solutions.
Alternatively, one may prove the uniqueness property in 
the $X^{s,\tht}$ spaces, see the discussion below.

The higher Sobolev regularity is limited to $H^2$ given the regularity of the metric $g$. However,
adding higher regularity to $g$ correspondingly adds regularity
to the class of regular solutions.

The objective of Section~\ref{reduction} will be to reduce 
our main result in Theorem~\ref{thm:local} to the small data result 
in Theorem~\ref{thm:small} by a standard scaling and finite speed of propagation argument.

A key role in our analysis is played by the spaces $X^{s,\theta}$ associated to the wave operator $\Box_g$. Indeed, our main well-posedness argument is phrased as a fixed point argument 
in $X^{s,\theta}$ where
\[
 1<s \leq 2, \qquad \frac{1}{2}<\tht \leq \min\{1, s -  \frac{1}{2}\}.
\]
In particular one can rephrase the uniqueness property 
in our main theorems as an unconditional uniqueness in $X^{s,\theta}$, and the Lipschitz dependence as a Lipschitz 
dependence in $X^{s,\theta}$.
These spaces are  defined in Section~\ref{sec:Xspaces}, and 
the study of their linear and bilinear properties occupies 
much of the paper.

\

\subsection{Previous works}
Here we provide a brief survey of previous results on Wave Maps well-posedness, with an emphasis on variable coefficients. 

The Cauchy problem on a flat background $(\R^{1+n}, -dt^2 + dx^2)$ is by now  well understood. In view 
of the scaling symmetry $u(t, x) \mapsto u(\lambda t, 
\lambda x)$, the critical Sobolev space is $\dot{H}^{\frac{n}{2}} \times 
\dot{H}^{\frac{n}{2} - 1} (\R^n)$. Local wellposedness in $H^s 
\times H^{s-1}$ was established for all subcritical regularities $s 
>\tfrac{n}{2}$  by Klainerman-Machedon for $n \ge 3$ and Klainerman-Selberg 
when $n = 2$ \cite{klainerman1993space}, \cite{klainerman1997remark}. The much more delicate 
critical problem $s = \tfrac{n}{2}$ was solved for small data in dimension $n = 
2$ by Tataru~\cite{Tat}, \cite{TataruWM}, Tao~\cite{Tao2} and 
Krieger~\cite{Krieger2004}, with further contributions in higher
dimension by Klainerman-Rodnianski~\cite{Klainerman2001}, and 
Shatah-Struwe~\cite{shatah1998geometric}, 
Nahmod-Stefanov-Uhlenbeck~\cite{Nahmod2003}. 
Further, when $n = 2$, 
the energy-critical problem in $H^1 \times L^2$ admits a global 
theory for large data as developed by Sterbenz-Tataru~\cite{Sterbenz2010}, 
\cite{Sterbenz2010a}, Krieger-Schlag~\cite{Krieger2012} and Tao \cite{Tao2008}.

For wave maps with variable coefficients, Geba~\cite{geba2009remark} 
established local  wellposedness in the subcritical regime $s > \tfrac{n}{2}$ when $3 \le n \le 
5$, building on previous work of Geba and Tataru~\cite{geba2008gradient}. More recently, Lawrie constructed global-in-time solutions on 
perturbations of $\R^{1+4}$ Minkowski space for small data in the critical 
space $H^2 \times H^1(\R^4)$, and Lawrie-Oh-Shahshashani obtained analogous 
small-data results on $\R \times \mathbb{H}^n$, $n \ge 
4$~\cite{lawrie2016cauchy}. See also the recent work of Li-Ma-Zhao on the 
stability of harmonic maps $\mathbb{H}^2 \to \mathbb{H}^2$ under the wave map 
flow~\cite{li2017asymptotic}.

A key component in the study of wave maps in the Minkowski case
at critical regularity is Tao's renormalization idea,  first introduced 
in \cite{Tao2}. The subcritical problem considered in this article  avoids the renormalization argument, which simplifies matters considerably. On the other hand, Strichartz estimates do not suffice to treat the full subcritical range  $s > \tfrac{n}{2}$ in low dimensions, in particular when $n = 2, 3$. Indeed the 
null structure $Q_g(u, u)$ distinguishes the wave maps system from equations with a generic quadratic derivative nonlinearity $(\nabla u)^2$; as observed by  Lindblad~\cite{Lindblad1996}, the latter can be illposed in $H^s \times H^{s-1}(\R^3)$ for $s = 2 > \tfrac{3}{2}$. The previous works~\cite{klainerman1993space,klainerman1997remark} 
rely on $X^{s,b}$ spaces to exploit the null structure in lower 
dimensions. We pursue an analogous strategy using the 
variable-coefficient $X^{s,b}$-type spaces first introduced by 
Tataru~\cite{Tataru1996} and further developed by by Geba and 
Tataru~\cite{geba2005disp}. However, the two-dimensional case involves additional subtleties as  we shall discuss shortly.

\

\subsection{Reduction of Theorem \ref{thm:local} to Theorem \ref{thm:small} } \label{reduction} \
Let 
\[
(\tilde{u}_0,\tilde{u}_1 ) :\mb{R}_x^2 \to M \times T_{\tilde{u}_0}M \subset \mb{R}^m \times \mb{R}^m 
\]
be an initial data such that 
\[
\vn{  (\nabla \tilde{u}_0,\tilde{u}_1 ) }_{ {H}^{s-1} } \leq R. 
\]
We have a solution $ \tilde{u} $ on a small time interval $ [0,T] $ if the rescaled function 
$$ u(t,x)=\tilde{u}(Tt,Tx) $$ 
is a solution on $ [0,1] $ using the rescaled metric $g(Tt,Tx)$.
This now obeys \eqref{assumption:rescaled:metric}, provided that $T$ is small enough. 

The data for the rescaled solution will satisfy the scale invariant 
bound
\[
 \vn{ (u_0,u_1)}_{ \dot{H}^1 \times L^2 } \leq R
\]
as well as the homogeneous bound
\[
 \vn{ (u_0,u_1)}_{ \dot{H}^s \times \dot{H}^{s-1} } \leq R T^{s-1} 
\]
We will choose $ T $ small enough so that
\[
RT^{s-1} \ll \epsilon.
\]
To obtain smallness of the full $ H^s \times H^{s-1} $ norms we truncate the initial data, then we apply Theorem \ref{thm:small} and we glue those solutions using finite speed of propagation. 

Let $ c $ be the largest speed of propagation and let $ (y_j)_{j} $ be the centers of a family of balls such that the truncated cones
$$ K_{j}=\{ (t,x) \ | \ ct+\vm{x-y_j} \leq c+1, \ t \in[0,1] \} 
$$
are finitely overlapping and cover $ [0,1] \times \mb{R}^2 $.

Let $ \chi_j $ be a smooth function which equals $ 1 $ on the ball $ B_j=B_{y_j}(c+1) $ and is supported on $ \tilde{B}_j=B_{y_j}(c+2) $. We denote by $ B_j^t=B_{y_j}(c+1-ct), \tilde{B}_j^{t} =B_{y_j}(c+2-ct)$ the corresponding balls at time $ t $. 

For every $ y_j $ we choose a local chart of $ M $ such that $ u_0(y_j) $ corresponds to the origin. 
We localize around $ y_j $ viewing (by slight abuse of notation) $ u_0,u_1 $ as having their image in the chart and defining
$ u_0^j =  \chi_j u_0,\  u_1^j=  \chi_j u_1 $. Since, $ u_0^j(y_j)=0 $, by homogeneous Sobolev embedding and Morrey's inequality we deduce that smallness is retained locally by the inhomogeneous Sobolev norms:
\begin{align}
\label{e:Hs-smallness}
\vn{ (u_0^j, u_1^j)}_{H^s \times H^{s-1}} \leq  \epsilon. 
\end{align} 
One uses Moser estimates to pass from the extrinsic Sobolev spaces to the Sobolev norms defined using the patch coordinates.

By Theorem \ref{thm:small} we obtain a solution $ u^j $ to \eqref{WM:int} which remains in the image of the chart on $ [0,1] $ by \eqref{sol:bd} and Sobolev embedding.  
Now viewing $ u^j $ as taking values in $ M \subset \mb{R}^m $, it solves \eqref{WM:ext} and we restrict it to the truncated cone $ K_j $. 

To obtain a solution $ u $ on $ [0,1] \times \mathbb{R}^2 $ defined by each $ u^j $ restricted to $ K_j $ we argue that any two of them must coincide on their common domain. For $H^2$ solutions this follows from the finite speed of propagation, which is then proved in a standard fashion. For rough solutions we use the well-posedness result from Theorem \ref{thm:small} to approximate them by $H^2$ solutions, and then we pass to the limit. 

Now we show that $ \nabla_{t,x} u(t) \in H^{s-1}(\mb{R}^2) $. First we note that, by \eqref{ball:localization}  we have
\begin{align}
\nonumber
\sum_j \vn{u_0^j}_{H^s}^2 & \simeq \sum_j \vn{\chi_j u_0}_{L^2( \tilde{B}_j)}^2 + \vn{\nabla_x (\chi_j u_0)}_{H^{s-1}}^2 \ls  \\
\label{sq:sum:interm}
& \ls  \sum_j \vn{ u_0}_{L^2( \tilde{B}_j)}^2 + \vn{ \nabla_x u_0}_{W^{s-1,2} ( \tilde{B}_j)}^2 \ls \sum_j \vn{ \nabla_x u_0}_{W^{s-1,2} ( \tilde{B}_j)}^2 
\end{align}
since by Morrey's inequality and Sobolev embedding we have
$$ \vn{  u_0}_{L^2 ( \tilde{B_j})} \ls \vn{\nabla_{x} u_0}_{L^{2+} (\tilde{B_j})} \ls  \vn{\nabla_{x} u_0}_{W^{s-1,2} (\tilde{B_j})}.
$$
Moreover, by \eqref{sq:sum:balls:easy} and \eqref{Sobolev:equiv} we have  
$$
\text{RHS }\eqref{sq:sum:interm}  \ls \sum_j \vn{  \nabla_x u_0}_{L^2( \tilde{B}_j)}^2 + \vm{  \nabla_x u_0}_{\dot{W}^{s-1,2} ( \tilde{B}_j)}^2 \ls \vn{u_0}_{\dot{H}^1}^2+\vn{\nabla_x u_0}_{\dot{H}^{s-1}}^2 \ls R^2.
$$
Similarly we obtain $ \sum_j \vn{u_1^j}_{H^{s-1}}^2 \ls R^2 $.
For the solution at time $ t $, 
using \eqref{sq:sum:balls} we get
$$
\vn{ \nabla_{t,x} u(t)}_{H^{s-1}}^2 \ls \sum_j \vn{\nabla_{t,x} u(t) }_{W^{s-1,2}(B_j^t)}^2 \ls \sum_j \vn{u^j[t]}_{H^s \times H^{s-1}}^2 \ls \sum_j \vn{u^j_0,u^j_1}_{H^s \times H^{s-1}}^2 \ls R^2.
$$
This proves that $ \tilde{u}(t) \in \mathcal{H}^s $ for $ t \in [0,T] $ and 
$$ \vn{\tilde{u}}_{L^{\infty} \mathcal{H}^s( [0,T] \times \mb{R}^2) } \leq C_R \vn{(\tilde{u}_0,\tilde{u}_1)}_{\mathcal{H}^s}
$$
with a constant $ C_R $ depending on $ R $. 

Now we address the locally Lipschitz dependence on the initial data. Let $ (\tilde{v}_0, \tilde{v}_1) :\mb{R}_x^2 \to M \times T_{\tilde{u}_0}M \subset \mb{R}^m \times \mb{R}^m $ be another initial data set such that 
$$ \vn{(\tilde{u}_0-\tilde{v}_0, \tilde{u}_1-\tilde{v}_1) }_{(\dot{H}^s \cap \dot{H}^1 \cap L^{\infty}) \times H^{s-1}}   $$ 
is small enough and let $ \tilde{v} $ be the solution on $ [0,T] $ with this data. Then using the argument above together with the Lipschitz dependence given by Theorem \ref{thm:small} of the local solutions we obtain 
$$
\vn{\tilde{u}-\tilde{v}}_{L^{\infty} \mathcal{H}^s( [0,T] \times \mb{R}^2) } \ls C_R \vn{(\tilde{u}_0-\tilde{v}_0,\tilde{u}_1-\tilde{v}_1)}_{\mathcal{H}^s}.
$$

Finally, we remark that assuming higher regularity for the metric $ g $ (such as $ \partial^k g \in L^{2} L^{\infty} $ and $ g \in L^{\infty} H^{k-1} $ for $ k \geq 3 $, see the discussion below in the proof of Theorem \ref{thm:small}) we have that $ (\dot{H}^n  \cap \dot{H}^1) \times H^{n-1} $ regularity of the initial data is maintained by the solution on $ [0,T]$ for $ n \leq k$.

\

\begin{remark}
If we assume the initial data $ (\tilde{u}_0,\tilde{u}_1 )$ to be only in $ \dot{H}^s \times \dot{H}^{s-1} $, then we can still construct a solution, but it will be only \emph{locally} in $ H^s \times H^{s-1} $ at future times.
\end{remark}

\

\subsection{Proof of Theorem \ref{thm:small}} \
Here we set up the fixed point argument which yields Theorem \ref{thm:small}, and show that the proof reduces to four estimates described below.

As a preliminary step, we replace the metric $g^{\alpha \beta}$ by 
$\tilde g^{\alpha \beta} = (g^{00})^{-1} g^{\alpha \beta}$ in order to insure
that $\tilde g^{00} = 1$; this is not crucial in the analysis, but yields 
some minor technical simplifications. The price to pay for this substitution
is that we get another term in the equations,
\begin{equation}\label{conformal}
\tilde{\Box}_{\tilde g} u^i=- \Gamma^i_{jk} (u) 
\tilde g^{\al \beta } \partial_{\al}u^j \partial_{\beta} u^k + \frac52
\partial^\alpha (\log |g^{00}|) \tilde g^{\alpha \beta} \partial_\beta u
\end{equation}
The extra term on the right will easily be perturbative in our setting.
Hence we will simply neglect it, drop the $\tilde g$ notation and simply assume that $g^{00} = -1$.

We now consider an initial data satisfying \eqref{small:id} and proceed to obtain the solution to \eqref{WM:int} on $ [0,1] $ by a fixed point argument, as $ u= \Phi(u) $ for the functional 
\be \label{functional}
\Phi^i(u) \defeq u_{lin}^i + \tilde{\Box}_g^{-1} \big(- \Gamma^i_{jk} (u) Q_g(u^j, u^k) \big) 
\ee
where $ u_{lin} $ is the solution of the linear equation with $ F=0 $
\be \label{inhom2:box:tilde}
 \tilde{\Box}_g u =F, \qquad u[0]=(u_0,u_1)
\ee 
while $ \tilde{\Box}_g^{-1} F $ is defined as the solution of \eqref{inhom2:box:tilde} with $ u[0]=(0,0) $. 

\

This argument relies on the choice of two Banach spaces $ X $ (for the components of our solutions) and $ N $ (for the perturbative nonlinearity) such that $ \Phi $ is a contraction on a small ball of $ X $. Specifically, 
$$ 
X=X^{s,\tht}, \qquad N=X^{s-1,\tht-1}
$$
which are defined in \emph{Section} \ref{sec:Xspaces} with $ \tht \in (1/2,1), \ \tht < s-1/2 $.

\

The \emph{linear mapping property} in Lemma \ref{lin:map:prop} states that for solutions of \eqref{inhom2:box:tilde} we have
\be \label{lin:mapp:prop}
\vn{u}_X \ls \vn{(u_0,u_1)}_{H^s \times H^{s-1}} + \vn{F}_N
\ee

\

Having \eqref{lin:mapp:prop} we also need to know that the mapping written schematically as $  \Gamma(u) Q_g(u,u) : X \to N $ holds, as well as 
$$
\vn{ \Gamma(u) Q_g(u,u) - \Gamma(v) Q_g(v,v) }_N \ls \epsilon \vn{u-v}_X 
$$
for $ u, v $ in a ball of radius $ C \epsilon $ in $ X $. These two properties are now easily reduced to the estimates \eqref{alg:est}-\eqref{nonlinear:diff} below; we obtain the $ \epsilon $ smallness for the difference since the nonlinearity is at least quadratic. We add that this contraction argument gives Lipschitz dependence on $ u_{lin} $, therefore on the initial data by \eqref{lin:mapp:prop}. 

\

The building blocks for the iteration are the following nonlinear estimates: 

\begin{flalign}
& \emph{Algebra property:}  &\vn{u \cdot v}_X &\ls \vn{u}_X \vn{v}_X &&  \label{alg:est} \\
& \emph{Product estimate:} &\vn{u \cdot F}_N & \ls \vn{u}_X \vn{F}_N \label{prod:est} && \\
& \emph{Null form estimate:} &\vn{ Q_g(u,v) }_N &\ls \vn{u}_X \vn{v}_X  \label{nf:est} &&\\
& \emph{Moser estimate:} &\vn{\Gamma(u) }_X &\ls \vn{u}_X (1+\vn{u}_X^M) \label{moser:est}. && 
\end{flalign}

\

The proof of the \emph{Algebra property} \eqref{alg:est} occupies the main part of this paper, being the object of \emph{Section \ref{Sec:alg}} which is based on the results of \emph{Sections~\ref{sec:Xspaces}-\ref{s:CE} }.
All the other properties rely fundamentally on \eqref{alg:est}: 

\

The  \emph{Product estimate} \eqref{prod:est} follows easily in \emph{Section \ref{Sec:Prod:est}} as a corollary of the estimates  established in the proof of \eqref{alg:est}, using a duality argument.

\
 
We obtain the \emph{Null form estimate} \eqref{nf:est} as a consequence of the 
identity \eqref{null:form:box} together with \eqref{alg:est}, \eqref{prod:est} and the linear bound 
$\vn{\Box_g u}_N \ls \vn{u}_X$ from Lemma \ref{lin:map:Box}.

\

The nonlinear \emph{Moser estimate} \eqref{moser:est} is proved in \emph{Section  \ref{Sec:Moser}}. For the purpose of the fixed point iteration argument based on \eqref{nf:est} we may subtract a constant and assume that $ \Gamma(0)=0 $. Using the \emph{Algebra property} \eqref{alg:est} we may subtract a polynomial from $ \Gamma $ and assume that $ \partial^{\al} \Gamma(0) =0 $ for $ \vm{\al} \leq C $. Modifying $ \Gamma $ outside a neighborhood of the origin, it suffices to prove \eqref{moser:est} under the assumption that $ \Gamma $ and its derivatives are uniformly bounded. 
Moreover, using the fundamental theorem of calculus, when $ \vn{u}_X, \vn{v}_X $ are bounded we obtain as a consequence 
\be \label{nonlinear:diff}
\vn{\Gamma(u)- \Gamma(v) }_X \ls \vn{u-v}_X.
\ee

\

\emph{Persistence of regularity.} Assuming we control $ k \geq 2 $ derivatives of the metric, and thus for the rescaled metric on [0,1] we assume $ \vn{ \partial^k g}_{L^2 L^{\infty}} \ll 1 $ and $ g \in L^{\infty} H^{k-1} $, we show that $ H^{\sg} \times H^{\sg-1} $ regularity of the initial data is maintained in time for any $ \sg \in [s, k ] $. Let $ M \geq 1 $ and 
$$
\vn{u[0]}_{H^{\sg} \times H^{\sg-1}} \leq M, \qquad \vn{u[0]}_{H^s \times H^{s-1}} \leq \epsilon.
$$
We obtain our solution as a fixed point for $ \Phi $ from \eqref{functional}, which is a contraction on a small ball of the space $ Z=X^{\sg,\tht} \cap X^{s,\tht} $ endowed with the norm
$$
\vn{u}_Z=\frac{\epsilon}{M} \vn{u}_{X^{\sg,\tht}} + \vn{u}_{X^{s,\tht}}
$$
provided $ \epsilon $ is sufficiently small, independently of $ M $. By Remark \ref{X:high:reg}, the mapping property \eqref{lin:mapp:prop} holds also for $ H^{\sg} \times H^{\sg-1}, X^{\sg,\tht} $ and $ X^{\sg-1,\tht-1} $. The nonlinear estimates \eqref{alg:est}-\eqref{moser:est} are replaced by 
\begin{align*}
 \vn{u \cdot v}_{X^{\sgm,\tht}} & \ls \vn{u}_{X^{\sgm,\tht}} \vn{v}_{X^{s,\tht}} + \vn{u}_{X^{s,\tht}} \vn{v}_{X^{\sgm,\tht}}   \\
 \vn{Q_g(u,v)}_{X^{\sg-1,\tht-1}} & \ls \vn{u}_{X^{s,\tht}} \vn{v}_{X^{\sg,\tht}} + \vn{u}_{X^{\sg,\tht}} \vn{v}_{X^{s,\tht}}  \\
 \vn{u \cdot F}_{X^{\sg-1,\tht-1}} & \ls \vn{u}_{X^{s,\tht}} \vn{F}_{X^{\sg-1,\tht-1}} + \vn{u}_{X^{\sg,\tht}} \vn{F}_{X^{s-1,\tht-1}}  \\
 \vn{\Gamma(u)}_{X^{\sg,\tht}} & \ls \vn{u}_{X^{\sg,\tht}}
\end{align*}
We refer to \eqref{alg:sgm}, \eqref{prod:est:higher} and Remark 
\ref{rk:cor:prodest:higher} for a discussion of these properties. We conclude 
that the unique solution obtained earlier in $ X=X^{s,\tht} $ is also in $ 
X^{\sg,\tht} $ and $ \vn{u}_{X^{\sg,\tht}} \ls M $. 
 
In both fixed point arguments we have continuous dependence on the initial data, therefore the obtained solution is the unique strong limit of smooth solutions.

\

\subsection{Main ideas} \
We now provide an outline of the main ideas of the paper which are the 
ingredients used to establish the building blocks 
\eqref{alg:est}-\eqref{moser:est} of our results.

\

\emph{Curved $ X^{s,b} $ spaces.} The classical $ X^{s,b} $ spaces are 
multiplier weighted $ L^2 $-spaces on the Minkowski space-time adapted to the 
symbol of the wave operator $ \Box $ similarly to how Sobolev spaces are 
associated to the Laplacian $ \Delta $, see 
\cite{Klainerman1996} and the references therein. These are used 
to prove local 
well-posedness for constant coefficients wave maps above scaling: 
\cite{klainerman1993space} ($ d \geq 3 $) and \cite{klainerman1997remark} ($ 
d=2 $).
For the history of classical $ X^{s,b} $ spaces and applications we refer to 
the survey article \cite{klainerman2002bilinear}.

A variable coefficient version of the $ X^{s,b} $ spaces, defined in physical space, 
was first introduced\footnote{Other
variable-coefficient $X^{s,b}$ constructions have been proposed, for instance via spectral theory on a smooth compact 
manifold~\cite{Burq2005}. } in \cite{Tataru1996}, and then further developed 
in \cite{geba2008gradient} in order to study semilinear wave 
equations on curved backgrounds with a
generic quadratic derivative nonlinearity. 
These spaces were later utilized by Geba~\cite{geba2009remark} in his treatment 
of energy-subcritical wave maps in dimensions $3 \le d \le 5$. We borrow this 
notion of curved $ X^{s,b} $ spaces in the present
paper and review the relevant definitions and lemmata in 
\emph{Section~\ref{sec:Xspaces}}. 

In two space dimensions, new techniques are required to establish
$X^{s,b}$ estimates, in particular the crucial algebra
property~\eqref{alg:est}.  
Using Strichartz estimates as for the case $ d \geq 3 $
would incur an 
unacceptable loss of derivatives when $d=2$. Instead we adapt 
some ideas of Tataru from energy-critical wave maps in Minkowski 
space; 
see in particular~\cite[Theorem~3]{Tat}. Our analysis involves 
bilinear angular decompositions, wave packets, and energy estimates 
along 
certain null hypersurfaces. 

\

\emph{Bilinear estimates} in $ X^{s,b} $ require optimal control of $ \vn{u_{\lmd} \cdot v_{\mu}}_{L^2([0,1]\times \mb{R}^2)} $ with bounds depending on the angular frequency localization of $ u_{\lmd} $ and $  v_{\mu} $. Decomposing the lower frequency function $ v_{\mu} $ into \emph{wave packets} concentrated on certain "tubes" $ T $, by orthogonality it will suffice to obtain bounds for $ \vn{u_{\lmd}}_{L^2(T)} $. In effect, this 
strategy is a variable-coefficient variant of the traveling wave decomposition 
employed in~\cite{Tat}. Foliating a tube $ T $ into null surfaces $ \Lambda $ (see \emph{Section \ref{sec:wp1}}) we reduce to controlling 
$$ \int_{ \Lambda} \vm{ u_{\lmd}}^2 \dd \sg. 
$$

We call these \emph{characteristic energy estimates}. While they can be proved by Fourier analysis in the Minkowski case \cite{Tat}, the present context requires a physical space proof which is considerably more delicate, as integration by parts based on the energy-momentum tensor only controls tangential derivatives such as $ \int_{ \Lambda} \vm{L u_{\lmd}}^2 \dd \sg $.

The problem of "inverting the $L $" in a manner consistent with the angular separation is addressed using microlocal analysis in \emph{Section \ref{s:CE}}.

\

\emph{Wave packets analysis.} A key technical device in this article consists in 
approximating 
solutions to the linear wave equations by 
square-summable
superpositions of \emph{wave packets}, which are localized in both space and 
frequency on the scale of the uncertainty principle, and propagate in spacetime 
along null hypersurfaces. Such representations of the wave group originate in 
the work 
of Cordoba-Fefferman~\cite{Cordoba1978}. For wave equations with $C^{1,1}$ 
coefficients, the first wave packet parametrix was given by 
Smith~\cite{smith1998parametrix}. An alternate construction, based on the use 
of phase space (FBI) transforms was provided by Tataru~\cite{Tat-nlw1,Tat-nlw2,Tat-nlw3}
for metrics satisfying $\partial^2 g \in L^1 L^\infty$.
Analogous constructions  have since been 
employed for even rougher coefficients, as in the the study of quasilinear wave equations in Smith-Tataru~\cite{smith2005sharp}. 

In this article we use Smith's parametrix, which easily extends 
to rougher metrics satisfying $\partial^2 g \in L^1 L^\infty$. There are two reasons for that:
(i) causality (the parametrix in \cite{Tat-nlw1,Tat-nlw2,Tat-nlw3} is 
developed in a microlocal space-time foliation, without reference to a time foliation)
and (ii) localization scales ( in slabs~\cite{smith1998parametrix} vs. tubes~\cite{Tat-nlw1,Tat-nlw2,Tat-nlw3}). Neither of these reasons is critical
but taken together they do make a difference from a technical standpoint.
\emph{Section~\ref{s:packets}} isolates the essential properties of 
Smith's parametrix that we shall need, while specific implementation details 
are reviewed in \emph{Appendix \ref{s:smithpackets}}. 

Another contribution of this paper is a \emph{wave packet 
characterization} of functions in $X^{s,b}$ with low 
modulation, see \emph{Proposition \ref{X:wp:dec}} and \emph{Corollary \ref{Cor:WP:dec}}.

\

\emph{Null structures} arise in equations from mathematical physics where they 
manifest through the vanishing of parallel interactions. The cancellations are 
realized in our setting expressing the null form $Q_g(u,v)$ in a suitable 
\emph{null 
frame} $ \{ L, \underline{L}, E \} $:
\[ 2 Q_{g} \big(u,v \big) = Lu \cdot  \underline{L}v +  \underline{L} u \cdot  Lv - 2 Eu 
\cdot Ev.
\] 
If ${L, E}$ are tangent to the null hypersurfaces along which $v$ propagates, 
then $Lv$ and $Ev$ are better than the tranverse derivative $\underline{L}v$. Hence 
there is a gain in $Q_g(u,v)$ over a generic quadratic term 
$\nabla u \cdot \nabla v$ if $u$ and $v$ propagate along nearby directions.

Estimates for 
null forms have played a key role in the low regularity 
well-posedness of wave maps in the flat case (\cite{klainerman1993space}, 
\cite{klainerman1997remark}, \cite{klainerman2002bilinear}, \cite{Tat}, 
\cite{Tao2}) and in other problems, especially in low dimensions. Several null 
form estimates for variable coefficients have been obtained before 
by Sogge \cite{sogge1993local}, Smith-Sogge \cite{smithsogge} and Tataru 
\cite{tataru2003null}, under various assumptions on the metric.

The strategy 
outlined above of obtaining $ L^2 $ estimates using wave packets and 
characteristic energy estimates applies as well for bounding $ \vn{Q_g(u_\theta,  v_\theta)}_{{L^2([0,1]\times \mb{R}^2)}} $
where the terms $ u_\theta,  v_\theta $ have an angular 
separation of $\simeq \alpha$. A Whitney-type decomposition is used to reduce to this assumption.

\

The \emph{Moser estimate}~\eqref{moser:est}, required for non-analytic target 
manifolds, 
is proved in \emph{Section~\ref{Sec:Moser}} following the method of iterated 
multilinear paradifferential expansions introduced in~\cite{TataruWM}.


\subsection{Notations and preliminaries}
\label{s:preliminaries}

\begin{itemize}
    \item We denote by $\tau$ and $\xi$ the time, respectively the spatial Fourier 
    variables. For $0 \ne \xi$, write $\wht{\xi} := \xi/|\xi|$ for its projection to
the (Euclidean) unit sphere.
    
 \item We denote the Cauchy data at time $s$ by $ v[s]=(v(s),\partial_t v(s)) $ for any time $ s $.
 
 \item Mixed Lebesgue norms shall be denoted by $\|u\|_{L^p_t L^q_x} := \bigl\| 
\|u(t)\|_{L^q_x} \bigr\|_{L^q_t}$; to unclutter the notation we often omit the 
subscripts. Also, when $p=q$ we write $L^p := L^p_t L^p_x$.

\item In the context of wave packet sums we will  denote by $ \ell^2_{T}  $ the following norm
$$ \vn{c_T }_{\ell^2_{T }}^2  =\sum_{T \in \calT_{\lmd} } \vm{c_T}^2 $$ 

\item To define Littlewood-Paley decompositions
we fix a partition of unity of the positive real line \[1 =
\sum_{\lambda \geq 1 \text{ dyadic}}
s_\lambda(r),\] where  $s_\lambda$ is supported on the interval $r \in
[\lambda/2, 2\lambda]$ for $ \lmd \geq 2 $ and in $ r \leq 2 $ for $ \lmd=1$. Setting $P_\lambda(\xi) := s_\lambda(|\xi|)$
and $P_\lambda(\tau, \xi) := s_\lambda(|(\tau, \xi)|)$, we define Littlewood-Paley
frequency decompositions on space and spacetime:
\begin{align*}
  1_{\R^2} = \sum_{\lambda} P_\lambda(D_x), \quad 1_{\R \times \R^2} =
  \sum_{\lambda} P_{\lambda}(D_t, D_x).
\end{align*}
The projections $P_{<\lambda}$, $P_{>\lambda}$ are defined in the
usual manner, and we also set  $P_{[\lambda_1,
  \lambda_2]} = \sum_{\lambda \in [\lambda_1, \lambda_2]} P_\lambda$.

For a multiplier $s_\lambda$, we write $\tilde{s}_\lambda$ for a
slightly wider version of $s_\lambda$ so that $s_\lambda  =
\tilde{s}_\lambda s_\lambda$.

In this article, Littlewood-Paley decompositions are defined
respect to the $x$ variable with one main exception: the (dual) metric
$g$ shall always be mollified in spacetime, and for a frequency $\mu$
we write $g_{<\mu} := P_{<\mu} (D_t, D_x) g.  $

\item Pseudo-differential operators in this paper are defined using Kohn-Nirenberg quantization. 
When there is no danger of confusion, we sometimes denote both a symbol 
$\phi(x, \xi)$ and its corresponding operator $\phi(X, D)$ by the abbreviation 
$\phi$.

\item The global Sobolev spaces $ \dot{H}^s  $ and $ H^s  $ are defined for $ s \in \mb{R} $ using the Fourier transform by
\be \label{Sobolev:global}
\vn{u}_{\dot{H}^s} =\vn{ \vm{\xi}^s \hat{u}}_{L^2},  \qquad \qquad  \vn{u}_{H^s}=\vn{ (1+\vm{\xi}^2)^{s/2} \hat{u}}_{L^2}
\ee

\item We define the local Sobolev spaces $ W^{s-1,2}(B) $, for a ball $ B $ of radius $ \simeq  1 $ or for $ B=\mathbb{R}^2 $. When $ s-1 \in \{0,1\} $ we use the classical definition, while when $ s-1 \in (0,1) $ we have the norm
\be 
\vn{v}_{ W^{s-1,2}(B)}^2 =\vn{v}_{L^2(B)}^2+ \vm{v}_{ \dot{W}^{s-1,2}(B)}^2
\ee
where $  \vm{\cdot}_{ \dot{W}^{s-1,2}(B)} $ denotes the Gagliardo seminorm
\be \label{Gagliardo:seminorm}
 \vm{v}_{ \dot{W}^{s-1,2}(B)}^2 = \int_B \int_B \frac{\vm{v(x)-v(y)}^{2}}{\vm{x-y}^{2s}} \dd x \dd y. \ee

\item When $ B = \mathbb{R}^2 $ one has 
\be \label{Sobolev:equiv}
 \vm{v}_{ \dot{W}^{s-1,2}( \mathbb{R}^2)} \simeq \vn{v}_{\dot{H}^{s-1}}, \qquad \vn{v}_{ W^{s-1,2}(\mathbb{R}^2)} \simeq \vn{v}_{H^{s-1}}.
\ee

When $ \psi \in C^{0,1}(\tilde{B}) $ is supported in a slightly smaller ball one has
\be \label{ball:localization}
\vn{\psi v}_{H^{s-1}} \ls \vn{\psi v}_{W^{s-1,2} (\tilde{B})} \ls \vn{v}_{W^{s-1,2} (\tilde{B})}.
\ee

We refer to \cite{di2012hitchhiker} for these properties. 
Clearly, when the balls $ (B_j)_j $ are finitely overlapping one has
\be \label{sq:sum:balls:easy}
\sum_j \vm{v}_{ \dot{W}^{s-1,2}(B_j)}^2 \ls \vm{v}_{ \dot{W}^{s-1,2}(\mb{R}^2)}^2
\ee

Conversely, assuming in addition that $ (B_j)_j $ cover $ \mb{R}^2 $, by splitting the $ \mb{R}^2 \times \mb{R}^2 $ integral in \eqref{Gagliardo:seminorm} into regions where there exists $ j $ such that both $ x,y \in B_j $ and regions where $ \vm{x-y}> \delta $ (where we bound $ \vm{v(x)-v(y)} \leq \vm{v(x)}+\vm{v(y)} $ and use the $ L^2(B_j) $ norms), we obtain 
\be \label{sq:sum:balls}
\vm{v}_{ \dot{W}^{s-1,2}(\mb{R}^2)}^2 \ls \sum_j \vn{v}_{ W^{s-1,2}(B_j)}^2 
\ee

\end{itemize}

\bigskip

{\bf Bounds on the metric.}
Observe  that we have for 
any $\lambda$ the
pointwise estimate
\begin{align*}
\begin{split}
&|  \partial^2 g_{<\lambda}| \lesssim M ( \partial^2 g) \le M( \|
\partial^2 g\|_{L^\infty_x}),\\
&|\partial^k \partial^2 g_{<\lambda}| \lesssim \lambda^k M ( \|
\partial^2 g \|_{L^\infty_x}),
\end{split}
\end{align*}
where $M$ is the Hardy-Littlewood maximal function, and by Bernstein
\begin{align}
\label{e:metric_est1}
\begin{split}
\| \partial^k \partial^2 g_{<\lambda} \|_{L^\infty} \lesssim
\lambda^k \lambda^{\frac{1}{2}}  \| \partial^2  g\|_{L^\infty_x L^2_t}
\le \lambda^{\frac{1}{2} + k} \| \partial^2 g\|_{L^2_t L^\infty_x}.
\end{split}
\end{align}
Similarly, from the bound
\begin{align*}
|  g_\lambda| = |P_{\lambda} D^{-2} D^2 g| \lesssim \lambda^{-2}
M(\| \partial^2 g\|_{L^\infty_x})
\end{align*}
we have
\begin{align}
\label{e:metric_est2}
\|  g_\lambda \|_{L^2_t L^\infty_x} \lesssim \lambda^{-2} \| \partial^2
g\|_{L^2_t L^\infty_x}.
\end{align}

As a consequence of the above bounds on $g$, we recall the commutator estimate
\be \label{com:est}
\vn{[ \Box_{g_{<\sqrt{\lmd}} },P_{\lmd} ] v}_{L^2_x} \lesssim \lmd \vn{v}_{ L^2_x}+ \vn{\pt_t v}_{L^2_x}
\ee
which follows from
$$
[ \Box_{g_{<\sqrt{\lmd}} },P_{\lmd} ]=[g_{<\sqrt{\lmd}},P_{\lmd}] \pt_{t,x} \pt_x, \qquad \vn{[g_{<\sqrt{\lmd}},P_{\lmd}]}_{L^2 \to L^2} \ls \lmd^{-1} \vn{\nabla g}_{L^{\infty}}.
$$

\

\subsection*{Acknowledgments} The first author
was partially supported by the Simons Foundation. The second  author was
partially supported by an NSF Postdoctoral Fellowship. The third 
author was partially supported by the NSF grant DMS-1800294 as well as by a Simons Investigator grant from the Simons Foundation.

\

\section{Curved \texorpdfstring{$X^{s,b}$}{Xsb} spaces} \ \label{sec:Xspaces}

In this section we present and expand on the theory of $X^{s,b}$ spaces associated to wave operators with variable coefficients from \cite{geba2008gradient} (see also  \cite{Tataru1996}). An alternative definition of $X^{s,b}$ spaces in the context of time-independent coefficients on compact manifolds was proposed in \cite{Burq2005} using spectral theory.  

\subsection{Definitions} We begin with our main building blocks, which are norms
associated to a single frequency $\lmd$ and modulation $d$:

\begin{definition} \label{def:X:1}
Let $ s \in \mathbb{R} $, $ \tht \in (0,1) $ and let $ I $ be a time interval. 
\begin{enumerate}
\item
For dyadic $ \lmd \geq d \geq 1 $, the norm of $ X^{s,\tht}_{\lmd,d}[I] $ is defined by
$$ \vn{u}_{ X^{s,\tht}_{\lmd,d}[I]}^2 = \lmd ^{2s} d^{2\tht} \vn{u}_{L^2(I\times \mathbb{R}^n)}^2 +  \lmd ^{2s-2} d^{2\tht-2} \vn{\Box_{g_{<\sqrt{\lmd}} } u }_{L^2(I\times \mathbb{R}^n)}^2. $$ 
When $ I=[0,1] $ we drop the $ I $ and denote simply $
X^{s,\tht}_{\lmd,d} $. \\
\item 
The norms of $ X^{s,\tht}_{\lmd, \leq h}, X^{s,\tht}_{\lmd, \leq h, \infty} $ are defined by 
$$  \vn{u}_{ X^{s,\tht}_{\lmd, \leq h} }^2=\inf \ \Big\{ \  \sum_{d=1}^{h} \vn{u_{\lmd,d}}_{X^{s,\tht}_{\lmd,d}}^2  \quad ; \quad u=  \sum_{d=1}^{h}  u_{\lmd,d} \ \Big\},
$$ 
$$  \vn{u}_{ X^{s,\tht}_{\lmd, \leq h, \infty} }^2=\inf \ \Big\{ \  \sup_{d \in \overline{1,h}} \vn{u_{\lmd,d}}_{X^{s,\tht}_{\lmd,d}}^2  \quad ; \quad u=  \sum_{d=1}^{h}  u_{\lmd,d} \ \Big\},
$$
for $ h \leq \lmd $. When $ h=\lmd $ we simply write $ X^{s,\tht}_{\lmd} $ for $ X^{s,\tht}_{\lmd, \leq \lmd} $. We also use a similar definition for $ X^{s,\tht}_{\lmd, [d_1,d_2]}  $ for a restricted range of modulations $ d_1 \leq d \leq d_2 $. \\
\item
For $ \delta>0 $ such that $ D=\delta^{-1} $ is a dyadic integer and $ \vm{I} \simeq \delta $, we analogously define $ X^{s,\tht}_{\lmd, \leq h}[I] $, $ X^{s,\tht}_{\lmd, \leq h,\infty}[I] $, $ X^{s,\tht}_{\lmd, [d_1,d_2]}[I]  $ for $ D \leq h \leq \lmd $,  with  summation, respectively supremum, taken over $ d \in \overline{D,\lmd} $ or $ d \in \overline{d_1,d_2} $. When $ D $ is bounded by a universal constant, we may equivalently take the summation over $ d \in \overline{1,\lmd} $.
\end{enumerate}
\end{definition}

\ 
\begin{remark}
	There is some flexibility in how to mollify the 
	coefficients 
	of 
	$\Box_g$. The 
	threshold $\sqrt{\lambda}$ is motivated by the hypothesis on $ \partial^2 g 
	$ 
	which gives
	$ \vn{g^{\al \beta}_{\geq\sqrt{\lmd}}}_{L^2_t L^{\infty}_x} \ls \lmd^{-1} $,
	making $ \Box_{g \geq\sqrt{\lmd}} P_{\lmd} $ effectively a first-order
	operator. Allowing higher frequencies would yield equivalent norms, 
	and indeed a paradifferential-type cutoff $<\lambda$ would be more natural 
	when considering merely Lipschitz coefficients.
	
\end{remark}

The full iteration spaces are defined as follows.

\begin{definition} \label{def:X:2}
Let $ s \in \mathbb{R} $ and $ \tht \in (0,1) $. 
\begin{enumerate}
\item A function $ u \in L^2([0,1],H^s(\mathbb{R}^n)) $ is said to be in $ X^{s,\tht} $ if it has finite norm defined by 
$$ \vn{u}_{ X^{s,\tht} }^2=\inf \ \Big\{ \ \sum_{\lmd = 1}^{\infty} \sum_{d=1}^{\lmd} \vn{u_{\lmd,d}}_{X^{s,\tht}_{\lmd,d}}^2  \quad ; \quad u= \sum_{\lmd = 1}^{\infty} \sum_{d=1}^{\lmd} P_{\lmd} u_{\lmd,d} \ \Big\}$$ \\
\item A function $ f \in L^2([0,1],H^{s+\tht-2} (\mathbb{R}^n)) $ is said to be in $ X^{s-1,\tht-1} $ if it has finite norm defined by
$$ \vn{f}_{ X^{s-1,\tht-1} }^2=\inf  \Big\{ \ \vn{f_0}_{L^2 H^{s-1}}^2+ \sum_{\lmd = 1}^{\infty} \sum_{d=1}^{\lmd} \vn{f_{\lmd,d}}_{X^{s,\tht}_{\lmd,d}}^2  ; \ f=f_0+ \sum_{\lmd = 1}^{\infty} \sum_{d=1}^{\lmd} \Box_{g<\sqrt{\lmd}}  P_{\lmd} f_{\lmd,d}  \Big\}$$ 
\end{enumerate}
\end{definition}

\begin{remark} Roughly speaking,
$X^{s,\theta}_{\lambda, d}$ will hold the portion of $u$ at frequency
$\lambda$ and ``modulation'' $d$. 
If  $g$ is the Minkowski metric and we modify the above definition of
  $X^{s,\theta}$ by replacing $L^2([-1,1]\times \mathbb{R}^n)$ with $L^2(
  \mathbb{R} \times \mathbb{R}^n)$, then for $u_{\lambda, d}$ localized to
  frequency $|\xi| \simeq \lambda$ and modulation
  $\bigl| |\tau| - |\xi| \bigr| \simeq d$ with $d \le \lambda$ we have
  \begin{align*}
    \|u_{\lambda, d}\|_{X^{s,\theta}_\lambda} \simeq \lambda^{s}
    d^\theta \| u_{\lambda, d} \|_{L^2} \simeq \| u_{\lambda, d}
    \|_{X^{s, \theta}_{\lambda, d}}.
  \end{align*}
  On the other hand, if $d \gg \lambda$,
  \begin{align*}
    \| u_{\lambda, d} \|_{X^{s, \theta}_{\lambda}} \simeq
    \lambda^{s+\theta-2} \| \Box u_{\lambda, d} \|_{L^2} \simeq \| \Box
    u_{\lambda, d} \|_{L^2 H^{s+\theta-2}}.
  \end{align*}
In the variable-coefficient context, modulation cannot be precisely
interpreted in terms of localization in Fourier space. Indeed, a
typical solution to $\Box_{g<\sqrt{\lambda}} u = 0$ will have
uncertainty at least $\sqrt{\lambda}$ in both temporal and spatial
frequency. Nonetheless, the intuition from the constant-coefficient
case still provides useful heuristics for proving estimates.
\end{remark} 

\
\begin{remark} \label{rmk:tildeX:loc}
By \cite[Corollary 2.5]{geba2008gradient}, in the definition of the $ X^{s,\tht} $ and $ X^{s-1,\tht-1} $ spaces  one can replace the $ X^{s,\tht}_{\lmd,d} $ norm by the norm of $ \bar{X}^{s,\tht}_{\lmd,d} $ defined by
$$ \vn{u}_{ \bar{X}^{s,\tht}_{\lmd,d}}^2 = \lmd ^{2s-2} d^{2\tht} \vn{\nabla_{t,x} u}_{L^2}^2 +  \lmd ^{2s-2} d^{2\tht-2} \vn{\Box_{g<\sqrt{\lmd} } u }_{L^2}^2. $$
This is based on the estimate in \cite[Lemma 2.4]{geba2008gradient}:
\be \label{tildeX:loc}
 \lmd^{s-1} d^{\tht} \vn{\nabla_{t,x} P_{\lmd} u}_{L^2}+ \lmd ^{s-1} d^{\tht-1} \vn{\Box_{g<\sqrt{\lmd} } P_{\lmd} u }_{L^2} \ls \vn{u_{\lmd,d}}_{X^{s,\tht}_{\lmd,d}}.
\ee 
This bound shows that on frequency localized functions, the $ X^{s,\tht}_{\lmd,d} $  and $ \bar{X}^{s,\theta}_{\lmd,d} $ norms are comparable 
and also that in Definition \ref{def:X:2} one can assume that $ u_{\lmd.d} $ and $ f_{\lmd,d} $ are localized at frequency $ \lmd $. Moreover, based on \eqref{tildeX:loc} we have the straightforward embedding
\be
X^{s-1,\tht-1} \subset L^2 H^{s+\tht-2}. 
\ee
\end{remark}

\subsection{Basic properties} \ Here we show how some properties known for the classical $ X^{s,b} $ spaces generalize to variable coefficients. We begin with some properties for our frequency localized building blocks:

\begin{proposition} Let $ u $ be a $ \lmd$-frequency-localized function on $ [0,1] \times \mb{R}^n $. Then:

\begin{enumerate}[label=(\arabic*),ref=(\arabic*)]
\item \label{P:En:est} (Energy estimates) For $ d \gtrsim \vm{I}^{-1} $ and any $ v $ one has:
\begin{align} \label{energy:X}
\lmd^{s-1} d^{\tht-\frac{1}{2}} \vn{\nabla_{t,x} u}_{L^{\infty} L^2[I]} & \ls \vn{u}_{X^{s,\tht}_{\lmd,d}[I]}
\\
\label{energy:est:X}
\vn{v}_{C H^{s}\cap C^1 H^{s-1}} & \ls \vn{v}_{ X^{s,\tht}} \quad \text{if} \quad \tht > \frac{1}{2}.
\end{align}

\item \label{P:timeloc} (Time localization) Let $ 1 \leq d_1 \leq d_2 \leq \lmd $ and let $ \chi_{d_2}(t) $ be a bump function in time localized on the $ d_2^{-1} $ scale. Then 
\be  \label{X:modulations:intervals:eq} \chi_{d_2} : \tilde{S}_{\lmd} X_{\lmd,d_1}^{1,\frac{1}{2}} \to X_{\lmd,d_2}^{1,\frac{1}{2}} \qquad
 \text{i.e.}  \qquad \vn{ \chi_{d_2} u}_{X_{\lmd,d_2}^{1,\frac{1}{2} }} \ls \vn{u}_{X_{\lmd,d_1}^{1,\frac{1}{2}}}. \ee
 
\item \label{P:ext} (Global extension) There exists a frequency localized extension of $ u $ from $ [0,1] $ to $ \tilde{u} $ supported in $(d^{-1}, 1+d^{-1} ) $ such that
 \begin{align*}
      d^{\tht} \| \nabla_{t,x}  \tilde{u} \|_{L^2(\R \times \R^n)} + d^{\tht-1} \|
      \Box_{g_{\sqrt{\lambda}}}  \tilde{u} \|_{L^2(\R \times \R^n)} \lesssim \vn{u}_{X^{1,\tht}_{\lmd,d}}.
    \end{align*}
    
\item \label{P:dec} (Time orthogonality) Let $ 1 \leq d \leq d' \leq d'' \leq \lmd $. For smooth partitions of unity with respect to time intervals of length $ d^{-1} $: $ 1=\sum_j \chi_d^j(t) $, one has
\begin{align}
\label{time:ort:1}
\vn{u}_{ X_{\lmd,d'}^{s,\tht}}^2 & \simeq \sum_j \vn{ \chi_d^j(t) u}_{X_{\lmd,d'}^{s,\tht}}^2 \\
\label{time:ort:2}
\vn{u}_{ X_{\lmd,[d',d'']}^{s,\tht}}^2 & \simeq \sum_j \vn{ \chi_d^j(t) u}_{X_{\lmd,[d',d'']}^{s,\tht}}^2 
\end{align}

\item \label{P:scaling} (Scaling)  Set $ u^{\delta}(t,x)=u(\delta t, \delta x) $ defined on $ [0,1] $ where $ \delta d \geq 1 $. Then:
\be \label{eq:scaling}
\vn{u^{\delta}}_{X_{\delta \lmd,\delta d}^{s,\tht}} \simeq \delta^{s+\tht-\frac{n+1}{2}} \vn{u}_{X_{\lmd,d}^{s,\tht}[0,\delta]}
\ee
where $ X_{\delta \lmd,\delta d}^{s,\tht} $ is defined using the metric $ g^{\delta}(t,x)=g(\delta t, \delta x) $.

\end{enumerate}
\end{proposition} 

\begin{proof}
{\bf \ref{P:En:est}} By energy estimates for the wave equation we obtain
$$ \vn{\nabla_{t,x} u}_{L^{\infty} L^2[I]}^2 \ls \vm{I}^{-1} \vn{\nabla_{t,x} u}_{L^2[I]}^2 + \vn{\nabla_{t,x} u}_{L^2[I]} \vn{ \Box_{g_{< \sqrt{\lmd}}} u}_{L^2[I]}
$$
which implies \eqref{energy:X} by \eqref{tildeX:loc}.
From this, for $ \tht > \frac{1}{2} $ and $ I=[0,1] $ we sum over $ d $ and use the definition of $ X^{s,\tht} $ to obtain the $ L^{\infty}H^{s} \times L^{\infty}H^{s-1} $ bound in \eqref{energy:est:X}.

Now we prove that the map 
$$ t \in [0,1] \to H^s \times H^{s-1} \ni v[t]
$$
is continuous when $ v \in X^{s,\tht} $. By the Fundamental theorem of calculus and Cauchy-Schwarz in $ t $, by summing modulations we have
$$ \vn{v_{\lmd}(t+h)-v_{\lmd}(t)}_{H^s} \leq C_\lmd \vm{h}^{\frac{1}{2}} \vn{v_{\lmd} }_{X_{\lmd}^{s,\tht}} $$
for any $ t,t+h \in [0,1] $. Let $ \epsilon >0 $ and define $ \lmd_{\epsilon} $ such that 
$$ \sum_{\lmd > \lmd_{\epsilon}} \vn{v_{\lmd}}_{L^{\infty} H^s}^2 \leq \epsilon^2. $$
Then
\begin{align}
\vn{v(t+h)-v(t)}_{H^s}^2 & \ls \sum_{\lmd \leq \lmd_{\epsilon}}  \vn{v_{\lmd}(t+h)-v_{\lmd}(t)}_{H^s}^2 +\epsilon^2 \leq \\
& \leq \vm{h}^{\frac{1}{2}} C_{\epsilon}  \vn{v }_{X_{\lmd}^{s,\tht}}  + \epsilon^2 \ls \ep^2.  
\end{align}
for $ h $ small enough. The same argument applies for $ \pt_t v $.

{\bf \ref{P:timeloc}}
By H\"older's inequality and \eqref{energy:X} we have
$$ \lmd d_2^{\frac{1}{2}} \vn{\chi_{d_2} u}_{L^2} \ls  \lmd d_2^{\frac{1}{2}}  \vn{\chi_{d_2}}_{L^2_t} \vn{u}_{L^{\infty} L^2} \ls  \vn{u}_{X_{\lmd,d_1}^{1,\frac{1}{2}}}.
$$
For the term $ \Box_{g<\sqrt{\lmd}} ( \chi_{d_2} u) $ we consider 
\begin{align*}  
& d_2^{-\frac{1}{2}} \vn{ \chi_{d_2}  \Box_{g<\sqrt{\lmd}} u}_{L^2}   \ls d_1^{-\frac{1}{2}}  
 \vn{ \Box_{g<\sqrt{\lmd}}u}_{L^2} \ls \vn{u}_{X_{\lmd,d_1}^{1,\frac{1}{2}}}  \\
& d_2^{-\frac{1}{2}}  \vn{ \partial_t^2 \chi_{d_2}  u}_{L^2}  \ls d_2^{-\frac{1}{2}} \vn{\partial_t^2 \chi_{d_2} }_{L^2_t} \vn{u}_{L^{\infty} L^2} \ls \vn{u}_{X_{\lmd,d_1}^{1,\frac{1}{2}}} \\
& d_2^{-\frac{1}{2}}  \vn{ \partial_t \chi_{d_2}  \partial_{t,x}  u}_{L^2}  \ls d_2^{-\frac{1}{2}} \vn{\partial_t \chi_{d_2} }_{L^2_t} \vn{\partial_{t,x} u}_{L^{\infty} L^2} \ls \vn{u}_{X_{\lmd,d_1}^{1,\frac{1}{2}}}
\end{align*}

{\bf \ref{P:ext}} First we extend $ u $ by solving $ \Box_{g_{\sqrt{\lambda}}} u =0$ outside $ [0,1] $ and then we define $ \tilde{u} = \chi \tilde{P}_{\lmd} u $ where $\chi(t)$ is smooth time cutoff supported in $(-d^{-1}, 1+d^{-1})$ equal to $ 1 $ on $ [0,1] $. For $t \in [1, 1+d^{-1}]$
    \begin{align*}
      \| \nabla_{t, x} u(t)\|_{L^2_x} &\lesssim \| \nabla_{t,x}
      u(t - d^{-1})\|_{L^2} + \| \Box_{g_{\sqrt{\lambda}}}
                                              u \|_{L^1 L^2 ( [1-d^{-1}, 1+d^{-1}] \times \R^n)}\\
      &\lesssim \| \nabla_{t,x} u (t-d^{-1}) \|_{L^2_x} +
        d^{-1/2} \| \Box_{g_{\sqrt{\lambda}}} u \|_{L^2}.
    \end{align*}
    Therefore
    \begin{align*}
      \| \nabla_{t,x} \tilde{u} \|_{L^2( [1, 1+d^{-1}], \R^n) }
      \lesssim \| \nabla_{t,x} u\|_{L^2[0,1]} + d^{-1} \|
      \Box_{g_{\sqrt{\lambda}}} u \|_{L^2[0,1]}.
    \end{align*}
We repeat this analysis on the interval $[-d^{-1}, 0]$. For the term $  \Box_{g_{\sqrt{\lambda}}} \tilde{u} $ outside of $ [0,1] $ we write
$$  \Box_{g_{\sqrt{\lambda}}} \chi \tilde{P}_{\lmd} = \chi''(t) \tilde{P}_{\lmd} + 2 \chi'(t) \partial_{t,x}  \tilde{P}_{\lmd} + \chi [ \Box_{g_{\sqrt{\lambda}}}, \tilde{P}_{\lmd}] + \chi \tilde{P}_{\lmd} \Box_{g_{\sqrt{\lambda}}}.
$$  
Then we use the already established $ L^2 $ bounds together with \eqref{com:est}.

{\bf \ref{P:dec}} See \cite[(53)]{geba2008gradient}, which is an easy commutation argument (see also Remark \ref{rmk:tildeX:loc}).

{\bf \ref{P:scaling}} The equivalence follows from a change of variables. The only issue is that the scaling does not commute with taking the $ P_{<\sqrt{\lmd}} $ localization for $ g $. Instead we have $ (g_{<\sqrt{\lmd \delta^{-1}}})^{\delta}=(g^{\delta})_{ <\sqrt{\delta \lmd} } $. It remains to estimate 
$$  \lmd^{-1} d^{-1} \vn{\Box_{g_{[\sqrt{\lmd},\sqrt{\lmd \delta^{-1}}] }} u }_{L^2 } \ls  
 \lmd^{-2} d^{-1} \vn{ \partial^2 g_{[\sqrt{\lmd},\sqrt{\lmd \delta^{-1}}] }}_{L^2 L^{\infty}} \vn{\partial^2 u }_{L^{\infty} L^2 } 
 \ls \vn{u }_{X_{\lmd,d}^{0,0}[0,\delta]}. 
$$ 
\end{proof} 

We next drop the frequency localization. Then one may ask whether
we could also discard the frequency localization in the 
coefficients. This is indeed the case for a restricted range 
of Sobolev indices. Precisely, we have

\begin{proposition}\label{p:Xst-global}
 Let $0 \leq s \leq 2$ and $0 < \theta < 1$. Then $u \in X^{s,\theta}$
 iff it admits a decomposition 
 \[
 u = \sum_{d \geq 1} u_d
 \]
 so that 
 \[
 \sum_d d^{2\theta} \| \nabla u_d\|_{L^2 H^{s-1}}^2 + d^{2\theta -2}\|\Box_g  u_d\|_{L^2 H^{s-1}}^2 < \infty
 \]
 with an equivalent norm given by
\begin{equation}
\| u \|_{X^{s,\theta}}^2 = \inf_{u = \sum u_d}
\sum_d d^{2\theta} \| \nabla u_d\|_{L^2 H^{s-1}}^2 + d^{2\theta -2}\|\Box_g  u_d\|_{L^2 H^{s-1}}^2
\end{equation}
\end{proposition}
We note that the counterpart of this result for $X^{s-1,\theta-1}$
is also valid.
\begin{proof}
a) In one direction, given $u_{\lambda,d}$ we define 
 \[
 u_d = \sum_{\lambda > d} u_{\lambda,d} 
 \]
 and prove the appropriate bounds for $u_d$.
 
 b) In the opposite direction, we define 
 \[
 u_{\lambda,d} = P_\lambda u_d, \qquad d < \lambda 
 \]
 and
 \[
 u_{\lambda, \lambda} = \sum_{d > \lambda} P_\lambda u_d
 \]
 and again prove the appropriate bounds.
 \end{proof}

Continuing our global description of the $X^{s,\theta}$ spaces,
for $s \in [0,2]$ we define the endpoints $X^{s,0}$ and $X^{s,1}$
with norms
\[
\| u\|_{X^{s,0}}^2 = \| \nabla u\|_{L^2 H^{s-1}}^2
\]
respectively
\[
\| u\|_{X^{s,1}}^2 = \| \nabla u\|_{L^2 H^{s-1}}^2 + \| \Box_g u\|_{L^2 H^{s-1}}^2
\]
Then we can also describe the full family of $X^{s,\theta}$
spaces as follows:

\begin{proposition}
For $0 \leq s \leq 2$ and $0 \leq \theta < 1$ we have:

\begin{enumerate}
\item \label{P:interpolate} (Interpolation) The space 
$X^{s,\theta}$ can be described by interpolation as
\be \label{interpolate}
X^{s,\theta} = [X^{s,0},X^{s,1}]_\theta
\ee
\item \label{P:duality} (Duality) For $ \frac{1}{2} < \tht < 1 $ we have 
\be \label{duality}
X^{s-1,\tht-1}=(X^{1-s,1-\tht}+L^2 H^{2-s-\tht} )' .
\ee
\end{enumerate}
\end{proposition}

We remark that the first part of the proposition in particular shows that the spaces $X^{s,\theta}$ defined in \cite{Tataru1996} and in 
\cite{geba2005disp} coincide.

\begin{proof}
{\bf (\ref{P:interpolate})}
Examining the equivalent definition of the $X^{s,\theta}$ spaces
in Proposition~\ref{p:Xst-global}, one immediately sees that 
it is nothing but the real $(\theta,2)$ interpolation space 
between $X^{s,0}$ and $X^{s,1}$ constructed using the $J$-method.
Since these are Hilbert spaces, the outcome coincides with the one
provided by complex interpolation.

{\bf (\ref{P:duality})} This is proved in \cite[Lemma 2.13]{geba2008gradient}.
\end{proof}

It will be technically convenient at some junctures to view functions
in $X^{s, \theta}[I]$ as restrictions of globally defined
functions. For this purpose we assume that the metric $g$ is extended to a 
global Lorentzian metric in $\R^{1+n}$. This can be taken constant outside a 
compact time interval.

  \begin{corollary}
    In the definition of $X^{s, \theta}_{\lambda}$, we may assume in
    the decomposition $u_\lambda = \sum u_{\lambda, d}$  that
    $\operatorname{supp} u_{\lambda, d} \subset (-d^{-1}, 1+d^{-1}) \times
    \R^n$ and take all spacetime norms over $\R^{1+n}$. The same holds for other intervals.
  \end{corollary}

We also have 

\begin{lemma} \label{v00:en:est}
Let $ I=[0,\delta] $ and $ D=\delta^{-1} $. If $ v[0]=(0,0) $, then
\be \label{X:est:idzero}  \vn{\tilde{P}_{\lmd} v }_{X_{\lmd,D}^{0,\frac{1}{2}}[I] } \ls \lmd^{-1} D^{-\frac{1}{2}} \vn{ \Box_{g_{<\sqrt{\lmd}} } \tilde{P}_{\lmd} v }_{ L^2[I] }.
\ee
\end{lemma}
\begin{proof}
Since $ v[0]=(0,0) $, the estimate follows from energy estimates.
\end{proof}

\

\subsection{Linear mapping properties} 
Consider the linear problem 
\be \label{inhom:box}
\Box_g v =f, \qquad v[0]=(v_0,v_1). 
\ee

\begin{lemma}
Let  $ 0<s <3 $ and $ 0 < \tht \leq 1 $. Then the  linear equation \eqref{inhom:box} is well-posed in $ H^s \times H^{s-1} $ and 
\be \label{ref:box:2}
\vn{v}_{X^{s,\tht}} \ls \vn{(v_0,v_1)}_{H^s \times H^{s-1}} + \vn{f}_{L^2 H^{s-1}}.
\ee
\end{lemma}
\begin{proof}
See \cite[Lemma 2.11]{geba2008gradient}. 
\end{proof}

We want to extend this bound for the Laplace-Beltrami operator $ \tilde{\Box}_g  $ given by \eqref{LB:op}.

\begin{lemma} \label{lin:map:prop}
Let $ 1 \leq s \leq 2 $ and $ \frac{1}{2} < \tht < 1 $. The solution of the linear equation 
\be \label{inhom:box:tilde}
 \tilde{\Box}_g u =F, \qquad u[0]=(u_0,u_1)
\ee 
satisfies
\be \label{lin:map:prop:eq}
\vn{u}_{X^{s,\tht}} \ls \vn{(u_0,u_1)}_{H^s \times H^{s-1}} + \vn{F}_{X^{s-1,\tht-1}}.
\ee
\end{lemma}
\begin{proof}
We first recall from \cite[Lemma 2.12]{geba2008gradient} that 
\be \label{ref:box:1}
\Box_g^{-1}  : X^{s-1,\tht-1} \to X^{s,\tht},
\ee
where $ \Box_g^{-1} f $ is the solution $ v $ of the inhomogeneous equation \eqref{inhom:box} with $ v[0]=(0,0) $.
Denote by $ S(v_0,v_1) $ the solution of \eqref{inhom:box} with $ f=0 $.

We write
$$  \tilde{\Box}_g - \Box_g = h^{\al} \pt_{\al} $$
where 
\be \label{small:h:Linfty}
 \vn{h^{\al}}_{L^2 L^{\infty}} + \vn{\nabla_x h^{\al}}_{L^2 L^{\infty}} \ll 1 \ee
and we will treat $ h^{\al} \pt_{\al} u $ as a perturbation.  A function $ u $ solves \eqref{inhom:box:tilde} if $ u=\Phi(u) $ where $ \Phi(v) \defeq S(u_0,u_1) + \Box_g^{-1} (F - h^{\al} \pt_{\al} v) $. 
To show that $  \Phi : X^{s,\tht} \to X^{s,\tht} $ and that  $ \Phi $ is a contraction on $ X^{s,\tht} $, considering \eqref{ref:box:1}, \eqref{ref:box:2} and \eqref{energy:est:X} it remains to check that
\be \label{contraction}
\vn{h^{\al} \pt_{\al} v}_{L^2 H^{s-1}} \ll \vn{\nabla_{t,x} v}_{L^{\infty} H^{s-1}} \ls \vn{ v}_{X^{s,\tht} }. 
\ee
Using \eqref{small:h:Linfty} we obtain this bound first for $ s=1 $ and $ s=2 $ and by interpolation also for $ 1 \leq s \leq 2 $. 
\end{proof}

Finally, we recall
\begin{lemma} \label{lin:map:Box}
For $ 0<s<3 $ and $0 \leq \theta \leq 1$, the operator $ \Box_g $ maps
$$
 \Box_g : X^{s,\tht} \to X^{s-1,\tht-1}.
$$
\end{lemma}

\begin{proof}
See \cite[Proposition 3.1]{geba2009remark}.
\end{proof}

\

\begin{remark}[Higher regularity] \label{X:high:reg} Let $ k \geq 3 $. Assuming control of $ k $ derivatives of the metric $g $, we obtain the previous properties for a wider range of $ s $. Thus, \cite{geba2008gradient}[Lemma 2.9] will hold for $ -k+2<s<k+1 $, which implies that \eqref{ref:box:2}, \eqref{ref:box:1} and Lemma \ref{lin:map:Box} extend to this range of $ s$.  

Assuming $ \vn{\partial^k g}_{L^2 L^{\infty}}\ll 1 $ (since $ g $ is a rescaled metric) we extend \eqref{lin:map:prop:eq} to $ s \in [1,k] $.
We also refer to \cite{smith1998parametrix}[Theorem 4.7] for the fact that the parametrix representation with $ H^s \times H^{s-1} $ bounds extends to $ s \in [1,k] $ under the assumption $ g \in L^{\infty} H^{k-1} $. 

\end{remark}

\

\subsection{$ \tilde{X}^{s,\tht} $ spaces} \ In the proof of the Moser estimate \eqref{moser:est} it will be useful to have the following modification of the $ X_{\lmd}^{s,\tht} $ norms.

\begin{definition} \label{tld:X}
Let $ s \in \mathbb{R} $, $ \tht \in (0,1) $ and dyadic $ \lmd \geq 1 $.
\begin{enumerate}
\item The norm of $ \tilde{X}_{\lmd,\lmd}^{s,\tht} $ is defined for $ \lmd $-frequency localized functions $ u $ by 
$$
\vn{u}_{\tilde{X}_{\lmd,\lmd}^{s,\tht}}= \lmd^{s+\tht} \vn{ \nabla_{t,x} u}_{L^2}+  \lmd^{s+\tht-\frac{1}{2}} \vn{\nabla_{t,x} u }_{L^{\infty} L^2}. 
$$
For $ 1 \leq d < \lmd $ the norm of $ \tilde{X}_{\lmd,d}^{s,\tht} $ is defined by $ \vn{u}_{\tilde{X}_{\lmd,d}^{s,\tht}} = \vn{u}_{X_{\lmd,d}^{s,\tht}} $.
\item The norm of $ \tilde{X}^{s,\tht}_{\lmd} $ is defined for $ \lmd $-frequency localized functions $ u $ by 
$$  \vn{u}_{ \tilde{X}^{s,\tht}_{\lmd} }^2=\inf \ \Big\{ \  \sum_{d=1}^{\lmd/2} \vn{u_{\lmd,d}}_{X^{s,\tht}_{\lmd,d}}^2 +\vn{u_{\lmd,\lmd}}_{\tilde{X}^{s,\tht}_{\lmd,\lmd}}^2  \quad ; \quad u=  \sum_{d=1}^{\lmd/2}  u_{\lmd,d}+u_{\lmd,\lmd} \ \Big\},
$$ 
\item A function $ u \in L^2([0,1],H^s) $ is said to be in $ \tilde{X}^{s,\tht} $ if it has finite norm defined by 
$$ \vn{u}_{ \tilde{X}^{s,\tht} }^2=\inf \ \Big\{ \ \sum_{\lmd = 1}^{\infty} \vn{u_{\lmd}}_{ \tilde{X}^{s,\tht}_{\lmd} }^2  \quad ; \quad u= \sum_{\lmd = 1}^{\infty} P_{\lmd} u_{\lmd} \ \Big\}$$ 
where the $ u_{\lmd} $ can be assumed wlog to be $ \lmd $-frequency localized.
\end{enumerate}
\end{definition}

Clearly  $ X^{s,\tht} \subset \tilde{X}^{s,\tht} $. The only difference between the two spaces occurs at high modulations, where we discarded the terms $ \lmd^{s+\tht-2} \vn{\Box_{g_{<\sqrt{\lmd}} } u_{\lmd,\lmd}}_{L^2} $, making the norm $  \tilde{X}^{s,\tht} $ smaller. This will be useful in the proof of the Moser estimate \eqref{moser:est}, see Remark \ref{Rk:tildeX:Moser}. We can recover the $ X^{s,\tht} $ bound if we control high modulations through $ \Box_g $:

\begin{lemma} \label{lemma:XXtildeBox}
If $ f \in \tilde{X}^{s,\tht} $ and $ \Box_g f \in L^2 H^{s+\tht-2} $ then $ f \in X^{s,\tht} $ and
$$ \vn{f}_{ X^{s,\tht}} \ls \vn{f}_{  \tilde{X}^{s,\tht}} + \vn{  \Box_g f }_{L^2 H^{s+\tht-2}}. $$
\end{lemma}

\begin{proof}
The proof is straightforward using definitions \ref{def:X:1}, \ref{def:X:2}, \ref{tld:X} and the properties from Section \ref{sec:Xspaces}. 
\end{proof}

\

\subsection{Half-waves norms}  For microlocal analysis purposes, we would like an equivalent
  definition of the $X^{s, \theta}_{\lambda, d}$ norms in terms of
  half-waves.
    We factor the symbol of $\Box_{g<\sqrt{\lambda}}$ as
  \[\tau^2 - 2 g_{<\sqrt{\lambda}}^{0j} \tau \xi_j -
    g^{ab}_{<\sqrt{\lambda}} \xi_a \xi_b =  (\tau +
  a^+)(\tau + a^-),\] where
  \[a^\pm =   - g_{<\sqrt{\lambda}}^{0j}\xi_j \mp \sqrt{ (g_{<\sqrt{\lambda}}^{0j} \xi_j)^2
      + g^{ab}_{<\sqrt{\lambda}} \xi_a  \xi_b}= - g^{0j}_{\sqrt{\lambda}} \xi_j \mp a  .\] Write
  \[A^\pm = - g_{<\sqrt{\lambda}}^{0j}(t,x) D_j \mp a(t,x, D).\]
    Observe that while $a$ is not exactly localized
    at frequencies $< \sqrt{\lambda}$, one does have
  \begin{align*}
    P_{\ge\lambda} (D_x) a \in \lambda^{-N} S^1
    \text{ for any } N,
  \end{align*}
  where $S^1$ denotes the classical symbol class. This decay will be
  more than adequate for our purposes.

  \begin{lemma}
  	\label{l:Xsb-halfwave}
    In the definition of $X^{s, \theta}_{\lambda, d}$, $\|\Box_{g_{<\sqrt{\lambda}}}
    P_\lambda u\|_{L^2}$ may be replaced by $\| (D_t+A^-)(D_t + A^+)
    P_\lambda u\|_{L^2}$ or $\| (D_t + A^+)(D_t+A^-) P_\lambda u\|_{L^2}$.
  \end{lemma}

    \begin{proof}
    We have
    \begin{align*}
      &(D_t+A^-)(D_t +A^+)P_\lambda - (D_t+A^-)P_{>\lambda/8} (D_t + A^+)
      P_{>\lambda/8} P_\lambda \\
      &= (D_t+A^-)P_{<\lambda/8} (D_t +A^+)
        P_\lambda = (D_t+A^-) P_{<\lambda/8} A^+ P_\lambda\\
      &= P_{<\lambda/8} (D_tA^+) P_\lambda + P_{<\lambda/8} A^+ P_\lambda
        D_t + A^- P_{<\lambda/8} A^+ P_\lambda,
    \end{align*}
    In view of the off-diagonal estimates
    \begin{align*}
      &\|(P_{< \lambda/8} + P_{>8\lambda}) A P_{\lambda} \|_{L^2 \to
        L^2} = \| (P_{< \lambda /
      8} + P_{>8\lambda})A_{> \lambda/8} P_{\lambda} \|_{L^2 \to L^2} 
      = O(\lambda^{-\infty}),
    \end{align*}
    where $A_{>\lambda/8} = (P_{>\lambda/8} (D_x) a) (t, x, D)$, and the pseudo-differential calculus
    on $\R^{1+n}$, we have
    \begin{align*}
      &\| (\Box_{g<\sqrt{\lambda}} - (D_t+A^-)(D_t + A^+) ) P_\lambda u\|_{L^2} \\&\lesssim
        \| (\Box_{g<\sqrt{\lambda}} - (D_t+A^-)P_{>\lambda/8} (D_t + A^+)
        P_{>\lambda/8}) P_\lambda u\|_{L^2} + O(\lambda^{-\infty})\|\nabla_{t,x} u\|_{L^2}\\
&\lesssim \| \nabla_{t,x} u\|_{L^2}.
    \end{align*}
  \end{proof}

Next we show how to reduce the study of small modulation spaces $ X^{s,\tht}_{\lmd,1} $ to half-wave norms.

\begin{proposition}
  \label{p:half-wave}
  Suppose $u: \R \times \R^{1+n} \to \R$ satisfies
  $u = P_\lambda(D_x) u$, and write
  \begin{align*}
    u = P_{>-\lambda/64}(D_t) u_\lambda + P_{<-\lambda/64} (D_t)
    u_\lambda := u^{+} + u^{-}.
  \end{align*}
  Then
  \begin{align*}
   \|\nabla_{t, x} u^{\pm} \|_{L^2} +  \lambda \| (D_t \pm A^{\pm}) u^{\pm}
    \|_{L^2}  + \|\Box_{g<\sqrt{\lambda}} u^\pm\|_{L^2} \lesssim \| \nabla_{t,x} u\|_{L^2} + \| \Box_{g<\sqrt{\lambda}} u\|_{L^2}.
  \end{align*}
\end{proposition}

\begin{proof}
  We write $S^+(D_t) = P_{>-\lambda/64}(D_t)$ and
  $S^{-}(D_t) = P_{<-\lambda/64}(D_t)$.  Let
  \[
  T_\lambda = \tilde{S}_{>-\lambda/16}(D_t)
  \widetilde{P}_{\lambda}(D_x).
  \]
  Then
  \begin{align*}
    \| (1-T_\lambda) (D_t + A^+) S^+(D_t) P_\lambda (D_x)
    \|_{L^2 \to L^2} = \| (1-T_\lambda) A^+_{>\lambda/64}
    S^+(D_t) P_{\lambda}(D_x) \|_{L^2 \to L^2} = O(\lambda^{-\infty}),
  \end{align*}
  where $A^+_\lambda$ is mollified in $(t,x)$. Now
  on the support of $T_\lambda$, the symbol $\tau + a^-$ is elliptic and
  belongs to $ S^1_{1, \frac{3}{4}}(\R^{1+n})$, hence there is a parametrix $Q
  \in OPS^{-1}_{1, \frac{3}{4}}(\R^{1+n})$ such that $Q (D_t+A^-) + R= T_\lambda$
  with $R \in S^{-\infty}(\R^{1+n})$.

  Write
  \begin{align*}
    (D_t+A^+) u^+ &= Q (D_t +A^-) T_\lambda (D_t + A^+)u^+ + R T_\lambda (D_t
                  + A^+) u^+ + (1-T_\lambda) (D_t + A^+) u^+\\
                  &= Q \tilde{T}_\lambda (D_t + A^-) T_\lambda (D_t + A^+) u^+ \\
    &+ Q
      (1-\tilde{T}_\lambda) (D_t + A^-) T_\lambda (D_t + A^+) u^+ +
      RT_\lambda (D_t + A^+) u^+\\
    &+ (1-T_\lambda) (D_t + A^+) u^+.
  \end{align*}
  Therefore
  \begin{align*}
    \| (D_t + A^+) u^+\|_{L^2} &\lesssim \lambda^{-1} \| (D_t+A^-)
    T_\lambda (D_t + A^+) u^+ \|_{L^2} + \lambda^{-N} \| (D_t +
    A)T_\lambda (D_t +A^+) u^+\|_{L^2} \\ &+ \lambda^{-N} \|
                                          \nabla_{t,x} u \|_{L^2}.
  \end{align*}
  The main term is
  \begin{align}
  \label{e:half-wave1}
  \begin{split}
    (D_t+A^-) T_\lambda (D_t + A^+)  u^+ &=  T_\lambda (D_t +
    A^-)(D_t+A^+) T'_\lambda u
    \\
    &+ [D_t + A^-, T_\lambda] T'_\lambda (D_t + A^+)  u + [D_t + A^-,
                                       T_\lambda] [(D_t + A^+), T'_\lambda]  u
    \end{split}
  \end{align}
  where
  $T_\lambda' = P_{>-\lambda/64}(D_t)\tilde{S}_{\lambda}(D_x) u$. 
  The second and third terms are bounded by $ \|\nabla_{t, x}u\|_{L^2}$
  provided that we verify the commutator estimate
  \begin{align}
    \nonumber
    [A, T_\lambda] : L^2 \to L^2, \quad A = A^\pm.
  \end{align}
  To see this, we use a spherical harmonics expansion
  \begin{align*}
    a = \sum_{l} b_l(t, x) h_l (\wht{\xi})|\xi|, \quad
    \wht{\xi} = \xi/|\xi|,
  \end{align*}
  where $b_l(t, x) = \langle a, h_l \rangle_{L^2(S^{n-1})}$ satisfies
  the same estimates as $a$ with an additional factor $c_N\langle l
  \rangle^{-N}$, for instance
  \begin{align*}
    P_{\ge\lambda}(D_t, D_x) b_l = \langle P_{\ge\lambda}(D_t, D_x) a, h_\ell
    \rangle_{L^2( S^{n-1}) } \le c_N \lambda^{-N} \langle l\rangle^{-N} .
  \end{align*}
  Thus we get
  \begin{align*}
    [A, T_\lambda] = \sum_l [B_l, P_{>-\lambda/16}(D_t)
    \tilde{S}_\lambda(D_x) ] h_l (D_x) |D_x| P_{<2\lambda}(D_x).
  \end{align*}
  By a standard kernel estimate of the form
  \begin{align*}
    |[b_l, P_{>-\lambda/8}(D_t) \tilde{S}_\lambda (D_x)] (x, y)|
    \lesssim \langle l \rangle^{-N}  \| \nabla g \|_{L^\infty} \lambda^{-1} \lambda^{n+1} \langle ( \lambda(t-s),
    \lambda(x-y) ) \rangle^{-N},
  \end{align*}
  Schur's test yields
  \begin{align*}
    \| [b_l, P_{>-\lambda/8}(D_t) \tilde{S}_\lambda(D_x) ] \|_{L^2 \to
    L^2} \lesssim \lambda^{-1} \langle l \rangle^{-N},
  \end{align*}
  and the claim follows.
  
  For the first term on the right side of~\eqref{e:half-wave1}, we use the 
  proof of the 
  the previous lemma and similar commutator estimates as above to obtain
  \begin{align*}
    &\| (D_t+A^-) (D_t + A^+) T'_\lambda u\|_{L^2} \\
    &\le \| \Box_{g<\sqrt{\lambda}}
    T'_\lambda u\|_{L^2} + \| [\Box_{g<\sqrt{\lambda}} - (D_t+A^-)(D_t+A^+)]T'_\lambda
    u\|_{L^2}\\
    &\lesssim \| \Box_{g<\sqrt{\lambda}} u\|_{L^2} + \| [g^{0j}_{<\sqrt{\lambda}},
      T'_\lambda] \partial_t \partial_j u\|_{L^2} + \|
      [g^{ab}_{<\sqrt{\lambda}}, T'_\lambda] \partial_a \partial_b
      u\|_{L^2} + \| \nabla_{t,x} u \|_{L^2}\\
    &\lesssim \| \nabla_{t,x} u\|_{L^2} + \| \Box_{g<\sqrt{\lambda}} u\|_{L^2}.
  \end{align*}
  This last estimate also yields the desired bound for $\|\Box_{g<\sqrt{\lambda}}
  u^+\|_{L^2}$, and altogether we find that
  \begin{align*}
    \| (D_t+A^+) u^+\|_{L^2} \lesssim
    \lambda^{-1}\|\nabla_{t,x}u\|_{L^2} + \lambda^{-1} \| \Box_{g<\sqrt{\lambda}} u\|_{L^2},
  \end{align*}
  as needed.
\end{proof}

If $u$ is compactly supported in time, we may smoothly truncate the half-waves in time to obtain:
\begin{corollary}
  \label{c:half-wave}
  Suppose $u \in P_\lambda(D_x) u$ has Fourier transform supported in
  $[\lambda/2, 2\lambda]$ and is supported in $[-c, c] \times \R^n$. Then there exists a decomposition
  \begin{align*}
    u = u^+ + u^-
  \end{align*}
  where $u^\pm$ are supported in $[-2c, 2c] \times \R^n$, and satisfy
  the estimates of the previous proposition.
\end{corollary}

\section{Null frames} \ 
\label{s:nullframes}

\subsection{Null foliations}
   \label{s:nullfoliations}
In this section we construct the null ``hyperplanes" along which wave packets 
propagate. Factor the principal symbol of $\Box_g$ as $(\tau+a^+)(\tau+a^-)$.
For each direction $\theta \in S^{1}$ and sign $\pm$, we construct optical 
functions
$\Phi^{\pm}_\theta$ as solutions to the eikonal equation
\begin{align*}
  \partial_t \Phi_\theta^\pm + a^{\pm} (t, x, \partial_x \Phi_\theta^\pm
  ) = 0, \quad \Phi^\pm_\theta(0, x) = \langle x, \theta \rangle.
\end{align*}
By the standard theory of Hamilton-Jacobi equations, for small $\eta$ these admit classical solutions on the spacetime
slab $[-10, 10] \times \R^2$. 

Recall that $\Phi$ is constructed via the Hamilton flow for
$a^\pm$ defined by
\begin{align}
\label{e:HamiltonODE}
\dot{x} = a^\pm_\xi(t,x,\xi), \quad \dot{\xi} = -a^\pm_x(t,x,\xi).
\end{align} 
Solutions to this systems are called bicharacteristics. Write $t \mapsto 
(x^\pm_t, \xi^\pm_t)$ for the solution
initialized at $(x_0, \xi_0)$. The map $(x_0, \xi_0) \mapsto (x^{\pm}_t, \xi^{\pm}_t)$ is 1-homogeneous in the second variable. Moreover, 
a routine linearization argument reveals that
\begin{lemma}
\begin{align}
\nonumber
\frac{\partial (x^\pm_t, \xi^\pm_t)}{\partial (x_0, \xi_0) } =
\left(\begin{array}{cc} I + O(\eta) & B(t)\\ O(\eta) & I +
O(\eta) \end{array}\right),
\end{align}
where the matrix $B(t)$ has norm $O(t)$. 
\end{lemma}
\begin{proof}
	The
	linearized system is
	\begin{align*}
	\dot{y} &= a_{\xi x} y + a_{\xi \xi} \zeta,\\
	\dot{\zeta} &= -a_{xx}y  - a_{x \xi} \zeta.
	\end{align*}
	The half-wave symbols $a$ inherit the derivative bounds on the metric:
	\begin{align*}
	    \| \partial_{t,x} \partial^k_\xi a (t, x, \xi)\|_{L_{t,x}^\infty} + \| \partial^2_{t,x} \partial^k a_\xi (t,x, \xi) \|_{L^2_t L^\infty_x } \lesssim \eta \quad \text{for bounded } \xi,
	\end{align*}
	hence Gronwall yields the preliminary
	estimate $| ( y(t), \zeta(t)) | \lesssim |( y(0), \zeta(0) )|$.
	
	Consider initial data $y(0) = I$, $\zeta(0) = 0$. Then
	\begin{align*}
	|\zeta(t)| \lesssim  \| \partial_{x}^2 a\|_{L^2 L^\infty}\sqrt{t} +
	\int_0^t |\zeta(s)| \, ds,
	\end{align*}
	so $|\zeta(t)| \lesssim \eta \sqrt{t}$. Substituting this into the
	equation for $y$, we obtain
	\begin{align*}
	|y(t) - I| \lesssim \eta.
	\end{align*}
	Now consider initial data $y(0) = 0, \eta(0) = I$. Then $|y(t)|
	\lesssim t$, and
	\begin{align*}
	|\zeta (t)-I| \lesssim \int_0^t |a_{xx}| s \, ds + \int_0^t
	|a_{x\xi} (s) \zeta(s)| \, d\xi \lesssim \eta.
	\end{align*}
\end{proof}

Hence we may parametrize the graph of
the $\pm$ flow map at time $t$ by $(x^\pm_t, \xi_0) \mapsto (x_0, \xi^\pm_t)$, 
via the
diffeomorphisms
\begin{align*}
(x^\pm_t, \xi_0) \mapsto (x_0, \xi_0) \mapsto (x^\pm_t, \xi^\pm_t) \mapsto
(x_0, \xi^\pm_t).
\end{align*}
A short computation then yields
\begin{align}
\label{e:jacobian_alt_param}
\frac{\partial (x_0, \xi^\pm_t)}{\partial (x^\pm_t, \xi_0) } =
\frac{\partial (x_0, \xi^\pm_t)}{ \partial (x_t, \xi^\pm_t)} \cdot 
\frac{\partial
	(x^\pm_t, \xi^\pm_t)}{\partial (x_0, \xi_0)} \cdot \frac{\partial (x_0, 
	\xi_0)}{
	\partial (x^\pm_t, \xi_0)} =\left(\begin{array}{cc} I + O(\eta) & B(t)\\ 
	O(\eta) & I +
O(\eta) \end{array}\right).
\end{align}
We define $\xi^\pm_\theta(t, x)$ by the relation
\begin{equation}
\label{e:fourier-variable}
\xi^\pm_{\xi_0}(t, x^\pm_t) := \xi^\pm_t (x_0, \xi_0),
\end{equation}
and recall that the
method of characteristics construction gives
\begin{equation}
\label{e:optical-construction}
\partial_x \Phi_\theta^\pm (t, x) = \xi^\pm_\theta (t, x).
\end{equation}

As
$ g^{\alpha \beta} \partial_\alpha \Phi_\theta^{\pm}\,
\partial_\beta \Phi_\theta^\pm = 0$, we obtain a foliation 
$\Lambda_\theta^\pm$ of
$[-10,10]\times \R^2$ for each $\theta$ by the null ``hyperplanes''
\begin{equation} 
\label{e:foliation}
\Lambda_{\theta, h}^\pm := \{ \Phi^\pm_{\theta} = h\}.
\end{equation}

The regularity of these null surfaces is easy to compute:

\begin{lemma}
  \label{l:optical_bounds} The functions $\Phi_\theta^\pm$ have regularity
  $\partial^2_{t,x} \Phi_\theta^\pm = O(\eta)$.
\end{lemma}

\begin{proof}
To simplify notation, we fix $\theta$ and focus on the
$+$ case for the remainder of this section. We redenote $a = a^+$,  $\Phi = 
\Phi^+_\theta$.
  In our new notation, $\Phi$ solves
  \begin{align*}
    \partial_t \Phi + a(t, x, \partial_x \Phi) = 0, \quad \Phi(0, x) =
    \langle x, \theta \rangle,
  \end{align*}
Differentiating the identity~\eqref{e:optical-construction} gives
$\partial_x^2 \Phi = \partial_x \xi_\theta = O(\eta)$.  The
estimates involving time derivatives now follow by differentiating
the equation:
\begin{align*}
&\partial_t\partial_x \Phi + a_x + a_\xi \partial^2_x 
\Phi = 0, 
\
\partial_t \partial_x \Phi(0, x) = 0
\Rightarrow \partial_t \partial_x \Phi = O(\eta),\\
&\partial_t^2 \Phi + a_t + a_\xi \partial_t \partial_x 
\Phi = 
0, \
\partial^2_t \Phi(0, x) = 0
\Rightarrow \partial^2_t\Phi = O(\eta).
\end{align*}
  \end{proof}

\begin{lemma}
  [Separation between null planes]
  $\operatorname{dist}(\Lambda_{h_1, \theta}, \Lambda_{h_2, \theta}) \sim |h_1
  - h_2|$. More precisely, there exists constants $c_1, c_2 >0$ such
  that for each $(t, x) \in \Lambda_{h_1}$, we have
  \begin{align*}
  c_1|h_1-h_2| \le  d ((t,x), \Lambda_{h_2}) \le c_2|h_1 - h_2|,
  \end{align*}
  where $d$ denotes Euclidean distance measured in the time slice $\{t\} \times \R^2$.
\end{lemma}

\begin{proof}
  Without loss of generality assume $h_2 > h_1$. By the bounds on $\partial^2 \Phi$,
  we have that $|\partial_x \Phi| = 1 + O(\eta)$. The Euclidean
  gradient flow $\dot{\gamma} = \tfrac{ \nabla_x \Phi }{ | \nabla_x
    \Phi|}$ satisfies
  \begin{align*}
\Phi(t, \gamma(s) ) - \Phi (t, \gamma(0) ) = \int_0^s
    |\nabla_x \Phi (t, \gamma(\tau) )| \, d \tau \in [C_1 s,  C_2s]
  \end{align*}
  for absolute constants $C_1, C_2$. Thus the Euclidean distance of
  $\Lambda_{h_2}$ is at
  most $|h_1-h_2|/C_1$. If $\eta$ is any other unit speed curve with
  $\eta(0) = (t, x)$, then
  \begin{align*}
    |\Phi(t, \eta(s)) - \Phi(t, \eta(0))| \le \int_0^s |\nabla_x \Phi
    (t, \theta(\tau) ) | \, d\tau \le C_2 s,
  \end{align*}
  so the distance is at least $|h_1-h_2|/C_2$.
\end{proof}

The next two lemmas compare the bicharacteristics and null foliations for
mollified and unmollified metrics.

\begin{lemma}{\cite[Prop. 4.3]{geba2005disp}}
  \label{l:freq-loc-flow}
  If $\R^n \times S^{n-1} \ni (x, \xi) \mapsto (x_{t}, \xi_t), (x_{t, \lambda}, 
  \xi_{t,
    \lambda})$ are the $+$ bicharacteristics for the metrics $g$ and
  $g_{<\sqrt{\lambda}}$ with the same initial data, then
  \begin{align*}
    |x_{t, \lambda} - x_t| \lesssim \lambda^{-1/2}, \quad
    |\xi_{t, \lambda} - \xi_t|  \lesssim \lambda^{-1/2}.
  \end{align*}
\end{lemma}

\begin{lemma}
  \label{l:freq-loc-foliation}
  (Foliations for frequency truncated metrics)
 Let $\Lambda^{\lambda}_{h}$ be the foliation defined as before
  but replacing $g$ with $g_{< \lambda^{\frac{1}{2}}}$. Then
  $\operatorname{dist}(\Lambda_{h, \theta}, \Lambda^\lambda_{h, \theta}) \lesssim \lambda^{-1}$.
\end{lemma}

\begin{proof}

  The optical function $\Phi^\lambda$ for $\Lambda^\lambda$ satisfies the
  corresponding eikonal equation
  \begin{align*}
    \partial_t \Phi^\lambda + a_{<\lambda^{\frac{1}{2}}} (t, x, \partial_x \Phi^\lambda) =
    0, \quad \Phi^\lambda (0, x) = \langle x, \theta \rangle,
  \end{align*}
  where $a_{<\lambda^{\frac{1}{2}}} = a_{<\lambda^{\frac{1}{2}}}^+$ is the $+$ half-wave symbol for the
  mollified metric.  The difference $\psi := \Phi - \Phi^\lambda$
  solves the transport equation
  \begin{align*}
    &\partial_t \psi + \langle v(t,x), \partial_x \psi \rangle =
    a_{<\lambda^{1/2}} (t, x, \partial_x\Phi^\lambda) - a (t, x, \partial_x
      \Phi^\lambda ) = O(\lambda^{-1}), \quad \psi(0, x) = 0,\\
    &v = \int_0^1 a_{\xi} (t, x, \partial_x \Phi^\lambda +
    s \partial_x (\Phi -\Phi^\lambda) ) \, ds,
  \end{align*}
  which may be integrated along characteristics to yield $|\psi(t,
  x)| \lesssim |t \lambda^{-1}|$.
\end{proof}

Certain estimates will be expressed in spacetime
coordinates adapted to the foliation $\Lambda_\theta$. The following
construction is modeled on~\cite{smith2005sharp}.

Let $\Phi = \Phi_\theta$ be the optical function for
$\Lambda_\theta$. We rotate the coordinates in $x$ by setting \[x_\theta = \langle x,
\theta \rangle, \quad x_{\theta}' = \langle x, \theta^\perp \rangle,\] where
$\theta^\perp$ is clockwise rotation of $\theta$ by angle
$\pi/2$. Then in the coordinates $(t, x_{\theta}', x_\theta)$, we have
$\partial_{x_\theta} \Phi (0, \cdot) = 1$ and $\partial^2 \Phi =
O(\eta)$. Hence $\partial_{x_\theta} \Phi = 1 + O(\eta)$ for all
$(t, x) \in [-10, 10]\times \R^2$.

Provided $\eta$ is sufficiently small, the global implicit
function theorem lets us write
\begin{align*}
  \Lambda_{h,\theta} = \{ (t, x_{\theta}', \psi_\theta(t, x_{\theta}', h)\}
\end{align*}
for some $C^2$ function $\psi_\theta$. Then $(t, x'_\theta, h)$ also
define coordinates on $[-10, 10 ] \times \R^2$ via
\begin{align*}
  (t, x'_\theta, h) \mapsto (t, x'_\theta, \psi_\theta (x'_\theta, h)
  ) \in \Lambda_{h,\theta}.
\end{align*}
Straightforward computations show that
\begin{align}
  \label{e:foliation_coord_deriv}
  \begin{split}
  \frac{ \partial (x_\theta', x_\theta ) }{ \partial (x'_\theta, h)} &=
                                                                       I + O(\eta),\\
  \frac{ \partial^2 (x_\theta', x_\theta ) }{ \partial (x'_\theta,
    h)^2} &= O(\eta).
  \end{split}
\end{align}
The variable $h$ parametrizes the leaves of the foliation and is constant along 
geodesics.


\subsection{The null frame}
\label{s:nullframe}
The future-pointing geodesic generator of $\Lambda_\theta$ is $L = -\nabla \Phi$. We complete
this to a null frame by the following standard construction.
Let $E = \langle e(t, x), \partial_x \rangle$ be a vector field tangent to the fixed-time slices of
$\Lambda_\theta$, defined concretely in terms of the rotated
coordinates $(x_\theta', x_\theta)$ as
\begin{align*}
 E = \frac{\tilde{E} } {\langle \tilde{E}, \tilde{E} \rangle^{1/2} }, \quad  \tilde{E}(t, x_\theta', x_\theta) = \partial_{x_\theta'} +  (\partial_{x_\theta'} \psi)\partial_{x_\theta}.
\end{align*}

Then $\langle L, E \rangle = -d_{t,x}\Phi (E) = 0$; in fact one also has
\begin{align}
  \label{e:E-symbol}
  0 = d_x \Phi(E) = \langle E, \xi_\theta(t, x) \rangle = \langle e(t,
  x), \xi_\theta(t,x) \rangle,
\end{align}
where $\xi_\theta(t, x)$ is the Fourier variable defined 
by~\eqref{e:fourier-variable}.
Finally, let
$\underline{L}$ be a null vector field transversal to $\Lambda_\theta$ and
satisfying $\langle \underline{L}, E \rangle = 0, \langle L,
\underline{L} \rangle = -1$. The vector fields $\{L, \underline{L}, E
\}$ form a null frame adapted to the foliation $\Lambda_\theta$.

\begin{lemma}
  \label{l:half-fullwave-bichar}
We have $L(t,x) = \sigma(t,x) [\partial_t + \langle a_\xi (t, x, 
\wht{\xi_\theta}(t, x) ), \partial_x
  \rangle ]$ for some bounded function $\sigma$.
\end{lemma}

\begin{proof}
  The multiplicative factor reflects the difference between the null
  bicharacteristics of the half wave symbol and full wave
  symbol. Write $z = (t, x)$, $\zeta = (\tau, \xi)$, and let $p = g^{\alpha \beta} \zeta_\alpha \zeta_\beta = p^+ p^- =
  (\tau + a^+)(\tau + a^-)$. Let $\gamma(s) = (z(s), \zeta(s))$ be a null
  bicharacteristic for $p^+$. Along this curve one has
  \begin{align*}
    \dot{z} = p^+_\zeta = (p^-)^{-1} p_\zeta , \quad \dot{\zeta} = - p^+_z = - (p^-)^{-1} p_z,
  \end{align*}
  Therefore
  \begin{align*}
    \dot{z}^\alpha g_{\alpha \beta} = 2(p^-)^{-1} g^{\alpha \mu}
    \zeta_\mu g_{\alpha \beta} = 2(p^-)^{-1} \zeta_\beta = 2(p^-)^{-1}
    \partial_\beta \Phi (z(s)),
  \end{align*}
  so
  \begin{align*}
   -2(p^-)^{-1} L =  2(p^-)^{-1} \nabla \Phi = \dot{z} =  \partial_t + \langle
    a_\xi , \partial_x \rangle.
  \end{align*}
  Finally, observe that $p^- = p^+ + p^- - p^+ = -2 \sqrt{ (g^{0j}
    \xi_j)^2 + g^{ab} \xi_a \xi_b}$ is bounded above and below along
  $\gamma$.
\end{proof}

\

\section{Wave packets analysis} \
\label{s:packets}

In this section we collect and generalize the salient features of Smith's 
wave packet parametrix \cite{smith1998parametrix}. More 
specific implementation details are recalled in the appendix.

\subsection{Packets, tubes, and null surfaces} \label{sec:wp1} 

We begin by clarifying our ``wave packet" terminology.
For $\omega \in S^{n-1}$,
let $\gamma_\omega^\pm(t) = (x^\pm(t), \xi^\pm(t) )$ be a
bicharacteristic for half wave symbol $\tau + a^{\pm}$ with
$\xi^\pm(0) = \omega$. Let $\omega(t) := \xi(t) / |\xi(t)|$ be the
projection of $\xi$ on the unit sphere.
\begin{definition}
	\label{d:packet}
	A smooth function $u$ at frequency $\lambda$ is a
	normalized wave packet for
	the bicharacteristic $\gamma_\omega^{\pm}$ if:
	
	\begin{itemize}
		\item $u$ is localized in phase space along $\gamma^\pm$:  there 
		exist constants
		$C_{N}$ such that
		\begin{align}
		\label{e:packet_decay}
		|u(t, x)| \le C_{N} \lambda^{\frac{3}{4}} ( \lambda |\langle 
		x - x^\pm (t),
		\omega^\pm (t) \rangle| + \lambda^{1/2} |(x - x^{\pm}(t)) \wedge
		\omega^\pm(t) | \rangle | )^{-N},
		\end{align}
		and similar estimates hold for $\bigl(\lambda \langle \omega^\pm (t),
		\partial_x \rangle\bigr)^a$ and $[\lambda^{1/2} (\omega^{\pm}(t) \wedge
		\partial_x)]^b$ applied to $u^\pm_\gamma$.
		\item $Lu$ satisfies the same estimates with constant $C_N(t)$, where 
		$|C_N(t)| \lesssim_N \| \partial^2 g(t)\|_{L^\infty}$,
		and $L$ is the null generator for the null foliation 
		$\Lambda_\omega^\pm$ defined in Section~\ref{s:nullfoliations}.
	\end{itemize}
\end{definition}

For each frequency $\lambda > 1$, let $\omega$ vary over a maximal collection 
$\Omega_{\lambda^{-1/2}}$
of unit vectors separated by at least $\lambda^{-\frac{1}{2}}$. To each such 
$\omega$ we associate a lattice $ \Xi_{\lambda}^{\omega} $ in 
the physical space $ \mb{R}_x^{n} $ on the dual scale, i.e. spaced $ 
\lambda^{-1} 
$ in the $ \omega $ direction and spaced $ \lambda^{-\frac{1}{2}} $ in 
directions 
in $ \omega^{\perp} $. Let \[\mathcal{T}_\lambda = \{(x, \omega) : \omega \in 
\Omega_{\lambda^{-1/2}}, \ x \in \Xi_\lambda^\omega\}.\]

To each bicharacteristic $\gamma^\pm = (x^\pm(t), \xi^\pm(t))$ with $(x^\pm(0), 
\xi^\pm(0)) \in \mathcal{T}_\lambda$, we 
associate a spacetime
``$\lambda$-tube''
\begin{align}
\label{e:tube-def}
T := \{ (t,x): \lambda |\langle x - x^\pm (t),
\omega^\pm (t) \rangle| + \lambda^{1/2} |(x - x^{\pm}(t)) \wedge
\omega^\pm(t) | \rangle | \le 10\},
\end{align}
on which packets for $\gamma^\pm$ concentrate. Each tube, say with initial 
data $(x_0,
\omega)$, is a $\lambda^{-1/2}$
neighborhood of a ray, intersected with a $\lambda^{-1}$
neighborhood of the null surface $\Lambda_{h, \omega}$ containing the
ray. 

It is suggestive to identify $T$ with $\gamma^\pm$ and denote normalized 
packets 
by $u_T$. Let $\mathcal{T}_{\lambda}^{\pm}$ denote $\lambda$-tubes 
associated to the
bicharacteristics $\gamma^\pm$ for the metric $g$, 
initialized in the lattice $\mathcal{T}_\lambda$. 

By Lemmas~\ref{l:freq-loc-flow} and \ref{l:freq-loc-foliation}, using the 
mollified metric $g_{<\sqrt{\lambda}}$
instead of $g$ would yield an essentially
equivalent family of tubes in the sense that each tube from one family
intersects boundedly many tubes from the other family, which are in
turn contained in the dilate of the first tube by a fixed
factor.

The tubes $\mathcal{T}_{\lambda, \omega}^{\pm}$ associated to a given initial
direction $\omega$ are finitely overlapping and admit a natural notion of distance which is
convenient for expressing the decay of wave packets. Following 
Smith~\cite[Section 2]{smith1998parametrix}, if $T_1, T_2 \in
\mathcal{T}_{\lambda, \omega}^+$ are
tubes with initial data $(x_j,
\omega), \ j= 1, 2,$ we define
\begin{align*}
  d(T_1, T_2) := \lambda |\langle x_1 - x_2, \omega \rangle| +
  \lambda^{\frac{1}{2}} | (x_1 - x_2)  \wedge \omega|.
\end{align*}
Let $(t, x_{1,\omega}'(t), h_1)$ and $(t, x_{2, \omega}'(t), h_2)$
denote the corresponding rays in the foliation-adapted
coordinates. It
follows from~\eqref{e:foliation_coord_deriv} that each tube takes the
form
\begin{align*}
  T_j = \{ |x_\omega' - x_{j,\omega}'(t)| \lesssim
  \lambda^{-\frac{1}{2}}, \ |h - h_j| \lesssim \lambda^{-1}\},
\end{align*}
and that
$|x_{1, \omega}'(t) - x_{2,\omega}'(t)| + |h_1 - h_2| \sim |x_{1,
  \omega}'(0) - x_{2, \omega}'(0)| + |h_1 - h_2|$. Hence for $\lambda
\gg 1$ we can also write
\begin{align*}
d(T_1, T_2) &\sim \lambda|h_1 - h_2| +
  \lambda^{\frac{1}{2}}|x_{1,\omega'}(t) - x_{2,\omega'}(t)|.
\end{align*}

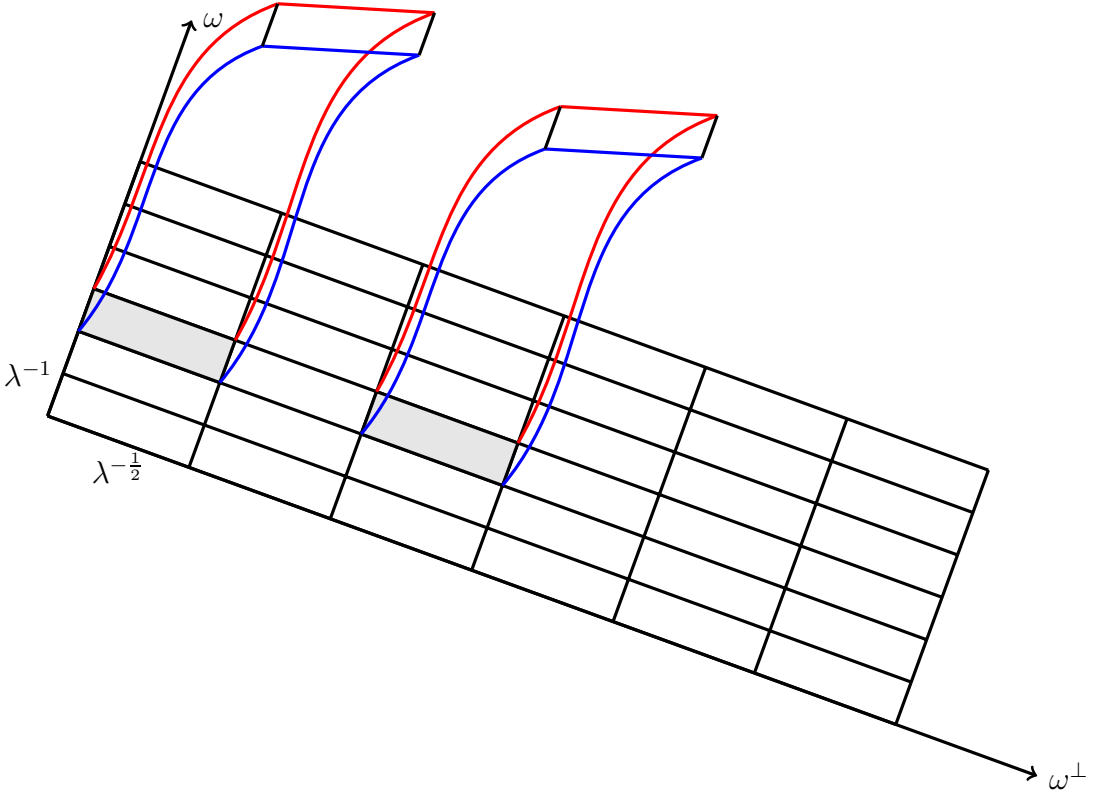
\begin{figure}
\caption{Tubes corresponding to the same direction $ \omega $ are finitely overlapping.}
\

\
\begin{tikzpicture}
\draw[very thick,xstep=2cm,ystep=0.6cm,rotate around={160:(0,0)}]
(0,0) grid (12,3.6);

\draw[->, very thick, rotate around={160:(0,0)}] (12,3.6) -- (12,-2) node[right] {$\omega$}; 
\draw[->, very thick, rotate around={160:(0,0)}] (12,3.6) -- (-2,3.6) node[right] {$\omega^{\perp}$}; 

\draw[rotate around={160:(0,0)}] (12,3.6) -- (11,3.6) node[below]  {$\lmd^{-\frac{1}{2}}$};
\draw[rotate around={160:(0,0)}] (12,3.6) -- (12,3) node[left]  {$\lmd^{-1}$};

\filldraw[fill=gray!20!white,draw=black,very thick,rotate around={160:(0,0)}] (8,2.4) rectangle (6,1.8);
\draw[blue,very thick,rotate around={160:(0,0)}](8,2.4) to [out=-110,in=40] (7,-2);
\draw[blue,very thick,rotate around={160:(0,0)}](6,2.4) to [out=-110,in=40] (5,-2.6);
\draw[red,very thick,rotate around={160:(0,0)}](8,1.8) to [out=-100,in=40] (7,-2.6);
\draw[red,very thick,rotate around={160:(0,0)}](6,1.8) to [out=-100,in=40](5,-3.2);

\draw[blue,very thick,rotate around={160:(0,0)}](7,-2) to  (5,-2.6);
\draw[red,very thick,rotate around={160:(0,0)}](7,-2.6) to (5,-3.2);
\draw[very thick,rotate around={160:(0,0)}](7,-2) to  (7,-2.6);
\draw[very thick,rotate around={160:(0,0)}](5,-2.6) to (5,-3.2);

\filldraw[fill=gray!20!white,draw=black,very thick,rotate around={160:(0,0)}] (12,2.4) rectangle (10,1.8);

\draw[blue,very thick,rotate around={160:(0,0)}](12,2.4) to [out=-110,in=40] (11,-2);
\draw[blue,very thick,rotate around={160:(0,0)}](10,2.4) to [out=-110,in=40] (9,-2.6);
\draw[red,very thick,rotate around={160:(0,0)}](12,1.8) to [out=-100,in=40] (11,-2.6);
\draw[red,very thick,rotate around={160:(0,0)}](10,1.8) to [out=-100,in=40](9,-3.2);

\draw[blue,very thick,rotate around={160:(0,0)}](11,-2) to  (9,-2.6);
\draw[red,very thick,rotate around={160:(0,0)}](11,-2.6) to (9,-3.2);
\draw[very thick,rotate around={160:(0,0)}](11,-2) to  (11,-2.6);
\draw[very thick,rotate around={160:(0,0)}](9,-2.6) to (9,-3.2);

\end{tikzpicture}
\end{figure}

If $u_T$ is a packet with initial direction $\omega$ and $T'$ is any
other tube with initial direction $\omega$, at time $0$ the decay
condition~\eqref{e:packet_decay} translates to
\begin{align*}
  |u_T(0)|_{T'} \lesssim_N \lambda^{\frac{3}{4}} \langle d(T,T')\rangle^{-N}.
\end{align*}

For $t \ne 0$, a similar bound holds but requires justification since
the condition~\eqref{e:packet_decay} is expressed in an
orthogonal coordinate system whereas the null surfaces are curved.
\begin{lemma}
  \label{e:packet_decay_foliation}
  Let $u = u_T$ be a frequency $\lambda$ wave packet for the $+$
  bicharacteristic $(x_t^+, \xi_t^+)$ initialized at $(0,
  \omega)$. Let $\Lambda_{\omega}^\lambda$ be the $+$ null foliation
  with direction $\omega$ for $g_{<\sqrt{\lambda} }$, and denote by
  $\{L, \underline{L}, E\}$ the associated null frame.  Then if
  $T' \in T_{\omega,\lambda}^+$ is any other tube, one has for all $N$
   \begin{align*}
     |u_T|_{T'} &\lesssim_N \lambda^{\frac{3}{4}} \langle d(T,
                  T')\rangle^{-N}\\
     |Eu_T|_{T'} &\lesssim_N \lambda^{\frac{3}{4} + \frac{1}{2}} \langle d(T,
                  T')\rangle^{-N}.    
   \end{align*}
   More precisely, if $(t, x_\omega'
   (t), x_\omega(t) )$ represents $x_t^+$ in the rotated coordinates
   and $\tilde{u}(t, x_\omega', h)
   = u_T(t, x_\omega', \psi(t, x_\omega', h) ), \, \widetilde{Eu}
   (t, x_\omega', h) = (Eu_T)(t, x_\omega', \psi(t, x_\omega',
   h)) $ represent $u$ and $Eu$ in the foliation adapted coordinates, then
   \begin{align*}
     &|\tilde{u}| \lesssim_N \lambda^{\frac{3}{4}} \langle \lambda^{\frac{1}{2}} |x_\omega' -
       x_\omega'(t) | + |\lambda h| \rangle^{-N} \text{ for all } N.\\
     &|\widetilde{E u}| \lesssim_N \lambda^{\frac{3}{4} + \frac{1}{2}}  \langle \lambda^{\frac{1}{2}} |x_\omega' -
       x_\omega'(t) | + |\lambda h| \rangle^{-N} \text{ for all } N.
   \end{align*}
\end{lemma}

\begin{proof}
Modulo multiplicative factors of size
  $1 + O(\eta)$, we have
  \begin{align*}
    &\omega(t) =
  -\partial_{x_\omega'} \psi ( x'(t), 0) \partial_{x_\omega'}+ \partial_{x_\omega},\\
&E(t, x_\omega', x_\omega(t, x_\omega', h)) = \partial_{x_\omega'} + \partial_{x_\omega'}
    \psi (t, x_\omega', h) \partial_{x_\omega}.
  \end{align*}
  By a slight abuse of notation we also write $E = E(t, x_\omega',
  h)$.
  
We express the wave packet decay condition~\eqref{e:packet_decay}
in terms of the coordinates $(x'_\omega, h)$. Let
\begin{align*}
  y_\omega &=  \langle (x_\omega' - x'_\omega(t), \psi(t, x_\omega') - \psi
       (t, x_\omega'(t), 0)), \omega(t) \rangle,\\
  y'_\omega &= \langle (x_\omega' - x'_\omega(t), \psi(t, x_\omega') - \psi
       (t, x'_\omega(t), 0)), E(t) \rangle.
\end{align*}
Then
\begin{align*}
  y'_\omega &= x'_{\omega}-x'_{\omega}(t)  + [ \psi(t, x'_{\omega}, h) - \psi(t, x'_{\omega}, 0) + \psi(t,
              x'_{\omega}, 0) - \psi(t, x'_{\omega}(t), 0)] \partial_{x} \psi(t, x'_{\omega}(t),
              0),\\
  y_\omega &=  \psi(t, x'_{\omega}, h) - \psi(t, x'_{\omega}, 0) + \psi(t, x'_{\omega}, 0) -
             \psi(t, x'_{\omega}(t), 0) - (x-x'_{\omega}(t)) \partial_{x'_{\omega}} \psi(t,
             x'_{\omega}(t), 0)
\end{align*}
So
\begin{align*}
  |y'_\omega|  \ge |x'_{\omega} - x'_{\omega}(t)| - c (\eta t)^2 |x'_{\omega}-x'_{\omega}(t)| - c\eta
  t h \ge \frac{1}{2}|x'_{\omega}-x'_{\omega}(t)|
\end{align*}
whenever $\dfrac{C|x'_{\omega} - x'_{\omega}(t)|}{ \eta t} \ge h$. Similarly,
\begin{align*}
  |y_\omega| \ge (1 + O(\eta)) h - c \eta t |x'_{\omega}-x'_{\omega}(t)| \ge \frac{1}{2}ch
\end{align*}
if $|x'_{\omega}-x'_{\omega}(t)| \le \tfrac{Ch}{\eta t}$. Thus
\begin{align*}
  |\tilde{u}(t, x'_{\omega}, h)| \lesssim_N \left\{\begin{array}{ll}
\lambda^{\frac{3}{4}} \langle \lambda h\rangle^{-N}, &          |x'_{\omega}-x'_{\omega}(t)|\le  \frac{\eta t h}{C},\\
 \lambda^{\frac{3}{4}}\langle   \lambda^{\frac{1}{2}}
 |x'_{\omega}-x'_{\omega}(t)| + |\lambda h|^{-N}
 \rangle^{-N}, & \frac{\eta t h}{C} \le
|x'_{\omega}-x'_{\omega}(t)| \le
\frac{Ch}{\eta t},
\\
\lambda^{\frac{3}{4}} \langle \lambda^{\frac{1}{2}} |x'_{\omega} - x'_{\omega}(t)| \rangle^{-N}, 
& |x'_{\omega}- x'_{\omega}(t)|     \ge \frac{Ch}{\eta t},\end{array}\right.
\end{align*}
but one may of course insert $\lambda^{\frac{1}{2}} |x'_{\omega} -
x'_{\omega}(t)|$ and $\lambda h$ in the first and third regimes
respectively. 

To verify the bound on $\widetilde{Eu}$, write
\begin{align*}
  E(t, x_\omega', h ) u = E(t, x_\omega'(t), 0) u + [E(t,
  x_\omega', h)  - E(t, x_\omega'(t), 0)] u.
\end{align*}
For the first term we use the hypothesis that $\lambda^{-1/2}(\omega(t) \wedge
\partial_x)u$ satisfies the packet bounds~\eqref{e:packet_decay}. The
second term can be written as
\begin{align*}
  \bigl( \partial_{x_\omega'} \psi (t, x_\omega', h) -
  \partial_{x_\omega'} \psi (t, x_\omega'(t), 0)
  \bigr)\partial_{x_\omega} u,
\end{align*}
which is $O( \lambda |x_\omega - x_\omega'(t)| + \lambda |h|)$ times a
normalized packet.
\end{proof}

\begin{corollary}
  \label{c:packet_decay_other_foliation}
  Let $u$ be a frequency $\lambda$ wave packet for the bicharacteristic
  $(x^+_t, \xi^+_t)$ of $g_{<\lambda^{\frac{1}{2}}}$ initialized at $(0,
  \omega)$. Let $\Lambda_\omega$ be the corresponding null foliation
  for the \emph{untruncated} metric $g$. Then the previous estimates
  hold with $\Lambda^\lambda_\omega$ replaced by the corresponding
  foliation $\Lambda_\omega$ for the untruncated metric $g$.
\end{corollary}

\begin{proof}
  The above lemmas imply that the foliations for $g$ and
  $g_{< \sqrt{\lambda}}$ are interchangeable as far as the bound with
  respect to $h$ is concerned. Similarly, as shown in~\cite[Prop. 4.3]{geba2005disp} the bicharacteristics for $g$ and
  $g_{<\sqrt{\lambda}}$ differ by $O(\lambda^{-1/2})$.
\end{proof}

To express the decay of packets more compactly, we introduce the weight
\begin{align*}
m_T(t, x) := \bigl(1 + \lambda |\langle x - x^\pm 
(t),
\omega^\pm (t) \rangle| + \lambda^{1/2} |(x - x^{\pm}(t)) \wedge
\omega^\pm(t) | \bigr),
\end{align*}
write
\begin{align}
\label{e:wpnorm}
\|u_T\|_{WP_T^N} := \| \lambda^{-\frac{3}{4}} m_T^N u_T\|_{L^\infty},
\end{align}
and write $\|\cdot\|_{WP}$ to denote a generic $WP_T^N$ norm with the 
understanding that the constants in bounds depend on $N$.

\subsection{Superpositions of wave packets}
The following statement summarizes how the different wave packets fit together 
in Smith's parametrix. Note that in the rest of the paper we shall only use the properties below, and not the specifics of Smith's construction. 
\begin{namedthm}{Parametrix property} \label{param:property}
	Let $ I=[0,\delta] $. For large enough dyadic $ \lmd \geq 1$, the following 
	properties hold: 
	
	\begin{enumerate}
		\item
		For any initial data $ (u_1,u_2) \in L^2 \times H^{-1} $ localized at 
		frequency $ \simeq \lambda $, there exists a $ \lambda $-wave packet 
		superposition which depends linearly on $ (u_1,u_2)  $:
		\be \label{wp:dec1}
		u=u^{+}+u^{-}, \qquad \qquad u^{\pm}=\sum_{T \in \calT_{\lmd}^{\pm}} 
		c_T u_T 
		\ee 
		such that $ (u(0),\partial_t u(0))=(u_1,u_2) $ and 
		\be \label{wp:dec2}
		\sum_{\pm,T \in \calT_{\lmd}^{\pm}} \vm{c_T}^2  \approx 
		\vn{(u_1,u_2)}_{L^2 \times H^{-1}}^2  \ee
		For any such decomposition one has
		\be
		\label{param:small}  \vn{\Box_{g<\sqrt{\lmd} } u(t)}_{L^2_x}  
		\lesssim  \lmd  \| \partial^2 g(t)\|_{L^\infty_x} 
		\vn{c_{T}}_{\ell^2_T}  \qquad  \forall t \in I.  
		\ee
		\item Let $ D=\delta^{-1} $. For any solution of the homogeneous 
		problem 
		\be  \label{homogeneous:sol}
		\Box_{g_{<\sqrt{\lmd}} } v=0, \qquad  v[0]=(v_1,v_2). 
		\ee
		there exists $ u $ as above satisfying 
		\begin{align} 
		\vn{\tilde{P}_{\lmd} (v-u) }_{X_{\lmd,D}^{0,\frac{1}{2}}[I] } & 
		\lesssim \delta \lmd^{-1} \vn{(v_1,v_2)}_{ H^{1} \times L^2 } \\
		\vn{c_{T}}_{\ell^2_T} & \ls  \lmd^{-1} \vn{(v_1,v_2)}_{ H^{1} \times 
			L^2 }
		\end{align}
		\item 
		Let $ T,T' \in \calT_{\lmd,\omega}^{\pm} $ be two tubes associated to 
		the same direction $ \omega $. 
		For $ \nu \geq \lmd $, let the vector fields $ L,\underline{L},E $ associated 
		to $ \pm g_{\sqrt{\nu}} $ and to $ \omega $ form a null frame as in 
		section~\ref{s:nullframe}.
		Then one has:
		\begin{align}
		\label{wp:Linfty:T}
		\vn{u_T }_{L^{\infty}(T')} \ls & \ \ \ \frac{\lmd^{\frac{3}{4}}}{\lng  
			d(T,T') \rng^N} \\
		\label{wp:L:T}
		\vn{L u_T }_{L^2 L^{\infty}(T')} \ls & \ \ \ 
		\frac{\lmd^{\frac{3}{4}}}{\lng  d(T,T') \rng^N} \\
		\label{wp:E:T} 
		\vn{E u_T }_{L^{\infty}(T')} \ls & \lmd^{\frac{1}{2}}  
		\frac{\lmd^{\frac{3}{4}}}{\lng  d(T,T') \rng^N}  \\
		\label{wp:barL:T}
		\vn{\underline{L} u_T }_{L^{\infty}(T')} \ls & \ \lmd \  
		\frac{\lmd^{\frac{3}{4}}}{\lng  d(T,T') \rng^N} 
		\end{align}
		Moreover, the same inequalities hold with $ u_T $ replaced by $ 
		P_{\lmd} u_T $.
		\item For any $ t \in I $ and sign $ \pm $, for $ (c_T)_T \in  \ell^2_T 
		$ one has
		
		\begin{equation} \label{wp:sql2}
		\vn{\sum_{T \in \calT_{\lmd}^{\pm}} c_T u_T(t) }_{L^2_x}^2 \ls   
		\sum_{T \in \calT_{\lmd}^{\pm}} \vm{c_T}^2, \quad \vn{\sum_{T \in 
				\calT_{\lmd}^{\pm}} c_T u_T^{\prime}(t) }_{L^2_x}^2 \ls \lmd  
				\sum_{T 
			\in \calT_{\lmd}^{\pm}} \vm{c_T}^2
		\end{equation}
		
	\end{enumerate}
\end{namedthm}

\

We note that property (2) follows from property (1) together \eqref{com:est}, 
\eqref{param:small} and Lemma \ref{v00:en:est}. 
In Section~\ref{s:smithpackets} 
we discuss these properties in the context of Hart Smith's wave packet 
parametrix from \cite{smith1998parametrix}.

\begin{remark}
The decay properties~\eqref{wp:Linfty:T} through \eqref{wp:barL:T}  reduce 
certain $L^2$ bilinear estimates 
to the characteristic energy estimates in section~\ref{s:CE}. Here is a typical 
computation. Suppose $v = \sum_{T \in \mathcal{T}_{\lambda, \omega}^+}  a_T 
v_T$ is a superposition of frequency $\lambda$ wave packets for a 
given initial 
direction $\omega$. Then by Schur's test we deduce
\begin{align*}
	\| u v\|_{L^2}^2 &\lesssim \sum_T \| uv \|_{L^2(T)}^2 = \sum_T \sum_{T_1, 
	T_2} a_{T_1} \overline{a_{T_2}} \langle u v_{T_1}, u v_{T_2} 
	\rangle_{L^2(T)}\\
	&\lesssim \sum_{T} \sum_{T_1, T_2} \lambda^{\frac{3}{2}} \| u\|_{L^2(T)}^2 
	d(T, T_1)^{-N} d(T, T_2)^{-N} |a_{T_1}a_{T_2}|\\
	&\lesssim \bigl(\sup_T \| u\|_{L^2(T)}^2 \bigr)\ \lambda^{\frac{3}{2}} 
	\sum_T 
	|a_T|^2,
\end{align*}
and the sup term is essentially an estimate for $u$ over the null surfaces 
associated to the direction $\omega$.
\end{remark}

\subsection{Preliminaries to the general decomposition} 
The next goal is to obtain a more general wave packet decomposition similar to \eqref{wp:dec1} for functions in $ X^{s,\tht} $ which are close to being solutions of \eqref{homogeneous:sol} in the sense of having low modulation. To allow for the extra flexibility of having inhomogeneities $ \Box_{g_{<\sqrt{\lmd}}} v $, the resulting decomposition (Prop. \ref{X:wp:dec} and Cor. \ref{Cor:WP:dec}) will have coefficients $ c_T(\cdot) $ that depend on time, which arise from Duhamel's formula
$$  v=v_0+\int_I 1_{t \geq s} v_s \dd s
$$
We first express the functions $ v_s $ in the following way:

\begin{lemma} \label{parametrix:lemma:repr}
For $ s \in I=[0,\delta] $, let $ v_s $ be the solution of the equation
 \begin{equation*}
 \begin{cases}
              \Box_{g_{<\sqrt{\lmd}}} v_s=0 \ \qquad \quad \text{on  }  I \times \mathbb{R}^2 \\
            v_s[s]=(f_s,g_s).
       \end{cases}
\end{equation*}
where $ f_s,g_s $ are assumed to be localized at frequency $ \simeq \lambda $. Then, there exists a wave packet superposition (initialized at $ t=0 $)
$$   u_s =  \sum_{\pm,T \in \calT_{\lmd}^{\pm}} c_{T,s} u_T	
$$
and a function $ w_s $ such that 
$$ \tilde{P}_{\lmd} v_s= \tilde{P}_{\lmd} (u_s+w_s)
$$
where
\begin{align}
& w_s[s]=(0,0) \\
& \vn{\tilde{\tilde{P}}_{\lmd} w_s }_{X_{\lmd,D}^{0,\frac{1}{2}}[I] } \ls \delta \vn{(f_s,g_s) }_{L^2 \times H^{-1}} \label{X:sm:delta} \\
& \vn{c_{T,s}}_{\ell^2_T} \ls \vn{(f_s,g_s) }_{L^2 \times H^{-1}}.
\end{align}
\end{lemma}

\begin{proof}
Denote $ M=\vn{(f_s,g_s) }_{L^2 \times H^{-1}} $ and for brevity we will denote $ \vn{\cdot}_{X_{\lmd,D}^{0,\frac{1}{2}}[I] } $ by $ \vn{\cdot}_X $. Note that $ \vn{v_s[0]}_{H^1 \times L^2} \ls \lmd M $. We apply the Parametrix property \ref{param:property} with $ (v_1,v_2)=v_s[0] $ and we obtain a representation $ \tilde{P}_{\lmd} v_s= \tilde{P}_{\lmd} u^1+ \tilde{P}_{\lmd} w^1 $ where $ u^1 $ is a wave packet superposition. For $ s=0 $ this is sufficient. For $ s \neq 0 $,
even though $ w^1[s] \neq (0,0) $, we have $ \vn{\tilde{\tilde{P}}_{\lmd} w^1 }_X \ls \delta M $. 

Now we iterate this construction. For $ i \geq 1 $ we write  
$$ \tilde{P}_{\lmd} w^i= \tilde{P}_{\lmd} (\tilde{\tilde{P}}_{\lmd} w^i-v^{i+1})+\tilde{P}_{\lmd} v^{i+1}  $$ 
where $ v^{i+1} $ solves the homogeneous equation
$$    \Box_{g<\sqrt{\lmd} } v^{i+1}=0, \qquad  v^{i+1}[s]=\tilde{\tilde{P}}_{\lmd} w^i[s]. $$
Assuming $ \vn{\tilde{\tilde{P}}_{\lmd} w^i }_X \ls \delta^i M $ we have $ \vn{v^{i+1}}_{L^{\infty}(H^1 \times L^2) } \ls \lmd \delta^i M  $. 
As before we use the Parametrix property \ref{param:property} to write $ \tilde{P}_{\lmd} v^{i+1}= \tilde{P}_{\lmd} (u^{i+1}+w^{i+1}) $ with $ \vn{\tilde{\tilde{P}}_{\lmd} w^{i+1} }_X \ls \delta^{i+1} M $.

From the above we obtain $ \tilde{P}_{\lmd} v_s= \tilde{P}_{\lmd} (u_s+w_s) $ by defining 
$$  u_s=\sum_{i \geq 1} u^i, \qquad w_s=\sum_{i \geq 1} \tilde{\tilde{P}}_{\lmd} w^i-v^{i+1}
$$
Note that $ w_s[s]=(0,0) $. Both series converge geometrically due to the powers of $ \delta $.  
\end{proof}

\begin{corollary} \label{parametrix:cor:estw}
Let $ v \in X_{\lmd,D}^{0,\frac{1}{2}}[I] $ for $ I=[0,\delta] $  and let $ w $ be defined by
$$ w=w_0 + \int_I 1_{t \geq s} w_s \dd s $$
where $ w_0, w_s $ are obtained from Lemma \ref{parametrix:lemma:repr} applied with $ (f_0,g_0)=\tilde{P}_{\lmd} v[0] $, respectively $ (f_s,g_s)=(0,\Box_{g_{<\sqrt{\lmd}}} \tilde{P}_{\lmd} v(s)) $.  Then,
$$  \vn{\tilde{\tilde{P}}_{\lmd} w }_{X_{\lmd,D}^{0,\frac{1}{2}}[I] } \ls \delta \vn{\tilde{\tilde{P}}_{\lmd} v }_{X_{\lmd,D}^{0,\frac{1}{2}}[I] }
$$ 
\end{corollary}
\begin{proof} 
To ease notation we denote $ \tilde{\tilde{P}}_{\lmd} $ by $ P $. The inequality for $ w_0 $ follows immediately due to \eqref{energy:X}. Denoting by $ \tilde{w} $ the integral term, we have $ \tilde{w}[0]=(0,0) $ and, since $ w_t[t]=(0,0) $, we have
\be \label{box:integrall} \Box_{g_{<\sqrt{\lmd}}} \tilde{w}(t) = \int_I 1_{t \geq s} \Box_{g_{<\sqrt{\lmd}}} w_s \dd s.
\ee
Note that by \eqref{X:sm:delta}, \eqref{energy:X} and H\" older's inequality we have 
\be \label{Swlvx} \vn{P \tilde{w} (t)}_{L^2 \times H^{-1}} \ls \delta \vn{P v}_{X_{\lmd,D}^{0,\frac{1}{2}}[I]}   
\ee

We write
$$ \Box_{g_{<\sqrt{\lmd}}} P \tilde{w}= P \Box_{g_{<\sqrt{\lmd}}}  \tilde{w} + [ \Box_{g_{<\sqrt{\lmd}}}, P ]  \tilde{w}
$$
and apply \eqref{X:est:idzero}:
$$ \vn{ \tilde{\tilde{P}}_{\lmd} w }_{X_{\lmd,D}^{0,\frac{1}{2}}[I] }  \ls 
\lmd^{-1}  D^{-\frac{1}{2} } ( \vn{P \Box_{g_{<\sqrt{\lmd}}}  \tilde{w}}_{L^2[I]}+ \vn{ [ \Box_{g_{<\sqrt{\lmd}}}, P ]  \tilde{w}  ) }_{L^2[I]} ) $$
The second term is estimated by \eqref{com:est} and \eqref{Swlvx}. For the first term on the RHS we apply Minkowski's inequality to \eqref{box:integrall}
\begin{align*}
&  \lmd^{-1}   D^{-\frac{1}{2} }  \vn{P \int_I 1_{t \geq s} \Box_{g_{<\sqrt{\lmd}}}  w_s \dd s }_{L^2[I]} \ls \lmd^{-1}  D^{-\frac{1}{2} } \int_I  \vn{P \Box_{g_{<\sqrt{\lmd}}}  w_s }_{L^2[I]}  \dd s  
\\
& \ls \lmd^{-1}  D^{-\frac{1}{2} }  \int_I \vn{ \Box_{g_{<\sqrt{\lmd}}} P w_s }_{L^2[I]}  + \vn{[P, \Box_{g_{<\sqrt{\lmd}}}]  w_s  }_{L^2[I]}    \dd s \ls \int_I \vn{P w_s}_{X_{\lmd,D}^{0,\frac{1}{2}}[I] }\dd s \\
&  \ls \delta \int_I \lmd^{-1} \vn{\Box_{g_{<\sqrt{\lmd}}} \tilde{P}_{\lmd} v(s) }_{L^2_x} \dd s  \ls \delta \vn{\tilde{\tilde{P}}_{\lmd} v }_{X_{\lmd,D}^{0,\frac{1}{2}}[I] }
\end{align*}
where we have used \eqref{com:est}, \eqref{X:sm:delta} and H\" older's inequality in $ s $.
\end{proof}
\ 

\subsection{A wave packet characterization of the $ X^{s,\tht} $ spaces.} \
With the previous preliminaries we are now ready to state our general decomposition (see also \cite[Sec. 4]{tataru2003null}).

\begin{proposition} \label{X:wp:dec}
Let $ I=[0,\delta] $, $ \mu \geq D=\delta^{-1} $ and let $ v \in X_{\mu,D}^{0,\frac{1}{2}}[I] $ be localized at frequency $ \simeq \mu $. We denote $ M \defeq \vn{\tilde{\tilde{P}}_{\mu} v }_{X_{\mu,D}^{0,\frac{1}{2}}[I] } \lesssim \vn{v}_{X_{\mu,D}^{0,\frac{1}{2}}[I]} $. 
\begin{enumerate} [leftmargin=*] 
\item Then, there exists a wave packet decomposition 
\be  P_{\mu} v(t) = P_{\mu} \sum_{\pm,T \in \calT_{\mu}^{\pm}} a_T(t) u_T(t)
\ee 
where the time-dependent coefficients satisfy, for all $ t \in I $, 
\begin{align} 
\vn{a_T}_{\ell^2_{T } L^{\infty}_{t} }  & \ls  M \label{coeff:1} \\ 
\vn{a_T^{\prime} (t)}_{\ell^2_{T }}  & \ls  \mu^{-1} \vn{\Box_{g_{<\sqrt{\mu}} } \tilde{P}_{\mu} v(t)}_{L^2_x}  \nonumber \\
\vn{a_T^{\prime} (t)}_{L^2_t \ell^2_{T }} & \ls D^{\frac{1}{2}} M \label{coeff:2}  \\
\sum_{\pm,T \in \calT_{\mu}^{\pm}} a_T^{\prime}(t) u_T(t)  &=0 \label{coeff:3}. 
\end{align} \\

\item Conversely, if \eqref{coeff:1}, \eqref{coeff:2}, \eqref{coeff:3} hold for some $ M $ and 
$$ w= \sum_{\pm,T \in \calT_{\mu}^{\pm}} a_T(t) u_T(t), $$
then  $ \vn{w}_{X_{\mu,D}^{0,\frac{1}{2}}[I] } \ls M $. 
\end{enumerate}
\end{proposition}

\begin{proof} 
For brevity we will denote $ \vn{\cdot}_{X_{\mu,D}^{0,\frac{1}{2}}[I] } $ by $ 
\vn{\cdot}_X $ and $ P_{\mu},  \tilde{P}_{\mu}, \tilde{\tilde{P}}_{\mu} $ by $ 
P, \tilde{P}, \tilde{\tilde{P}} $.

{\bf (1)} \ Note that it suffices to prove that there exists a decomposition
$$ P v(t) =P w(t) + P \sum_{\pm,T \in \calT_{\mu}^{\pm}} a_T(t) u_T(t)
$$
for some function $ w $ with bound $ \vn{\tilde{\tilde{P}} w  }_X \ls \delta \vn{\tilde{\tilde{P}} v }_X $ and $ (a_T)_T $ satisfying the requirements above, since then the proposition follows by an iterative argument. 

By Duhamel's formula we represent 
$$  \tilde{P} v=v_0+\int_I 1_{t \geq s} v_s \dd s
$$ 
where $ v_0 $ and each $ v_s $ are solutions to the problems
\begin{equation*}
 \begin{cases}
             \Box_{g_{<\sqrt{\mu}}} v_0=0   \\
             v_0[0]= \tilde{P} v[0]
       \end{cases} \qquad \qquad \qquad
 \begin{cases}
              \Box_{g_{<\sqrt{\mu}}} v_s=0  \\
            v_s[s]=(0,\Box_{g_{<\sqrt{\mu}}} \tilde{P} v(s)).
       \end{cases}
\end{equation*}
We apply Lemma \ref{parametrix:lemma:repr} and obtain:
\begin{align*}
& P v_0 =P u_0 + P w_0,  \quad & P v_s =P u_s + P w_s \\
&u_0 =  \sum_{\pm,T \in \calT_{\mu}^{\pm}} c_T u_T , \quad & u_s =  \sum_{\pm,T \in \calT_{\mu}^{\pm}} c_{T,s} u_T	
\end{align*}
with the bounds
\begin{align}
&  \vn{c_T}_{\ell^2_T} \ls \vn{ \tilde{P} v[0] }_{L^2 \times H^{-1} }\ls M , \quad & \vn{c_{T,s}}_{\ell^2_T} \ls \mu^{-1} \vn{\Box_{g_{<\sqrt{\mu}}} \tilde{P} v(s) }_{L^2_x}  \label{cT:bounds}
\end{align}
and $ u_s(s)=v_s(s)=0 $. We write $ P v=P \tilde{P} v=P u+P w $ where
$$ u=u_0+ \int_I 1_{t \geq s} u_s \dd s, \qquad w=w_0+ \int_I 1_{t \geq s} w_s \dd s
$$
By Corollary \ref{parametrix:cor:estw} we have $ \vn{\tilde{\tilde{P}} w  }_X \ls \delta \vn{\tilde{\tilde{P}} v }_X $.

We obtain the representation
$$ u(t)= \sum_{\pm,T \in \calT_{\mu}^{\pm}} a_T(t) u_T(t)
$$
where, for any sign $ \pm $ and any $ T \in \calT_{\mu}^{\pm} $:  
$$ a_T(t)=c_T+ \int_I 1_{t \geq s} c_{T,s}  \dd s \qquad \vm{a_T}_{L^{\infty}_t} \ls \vm{c_T}+ \int_I  \vm{c_{T,s}}  \dd s
$$
and $ a_T^{  \prime} (t)=c_{T,t} $ for all $ t \in I $. We have
\begin{align*}
\vn{a_T}_{\ell^2_{T } L^{\infty}_{t} }  & \ls \vn{c_T }_{\ell^2_{T }}+\vn{ \int_I  \vm{c_{T,s}}  \dd s }_{\ell^2_{T }} \ls M+  \int_I  \vn{c_{T,s}}_{\ell^2_{T }}  \dd s \ls \\
& \ls M+ \int_I \mu^{-1} \vn{\Box_{g_{<\sqrt{\mu}}} \tilde{P} v(s) }_{L^2_x}   \dd s \ls M+ \delta^{\frac{1}{2}} \mu^{-1}  \vn{\Box_{g_{<\sqrt{\mu}}} \tilde{P} v}_{L^2_{t,x}} \ls M.
\end{align*}
This verifies \eqref{coeff:1}. The next condition holds due to \eqref{cT:bounds}:  
$$  \vn{a_T^{  \prime} (t)}_{\ell^2_{T }} =\vn{c_{T,t} }_{\ell^2_{T }} \ls \mu^{-1} \vn{\Box_{g_{<\sqrt{\mu}}} \tilde{P} v(s) }_{L^2_x}
$$
which also gives \eqref{coeff:2} by integration. 

Since $ u_s(s)=v_s(s)=0 $ we have $ \sum a_T^{  \prime} (s) u_T(s)=0 $ for any $ s \in I $, obtaining \eqref{coeff:3}, which completes the proof of the first part.

\

{\bf (2)} \ By H\" older's inequality in time and \eqref{wp:sql2} we have
$$ D^{\frac{1}{2}} \vn{w}_{L^2(I \times \mathbb{R}^2)} \ls \sum_{\pm} \vn{ \sum_{T \in \calT_{\mu}^{\pm}} a_T(t) u_T(t)}_{L^{\infty}_t L^2_x}  \ls \sum_{\pm} \vn{a_T}_{L^{\infty}_{t} \ell^2_{T } }  \ls \sum_{\pm} \vn{a_T}_{\ell^2_{T } L^{\infty}_{t} }  \ls  M
$$
Now we consider the term $ \Box_{g_{<\sqrt{\mu}}} w $ which we write as  
$$
 \Box_{g_{<\sqrt{\mu}}} w=\sum_{\pm,T \in \calT_{\mu}^{\pm}} \left[  a_T(t) \Box_{g_{<\sqrt{\mu}}} u_T(t) + a_T^{ \prime}(t) g \pt_{t,x} u_T(t) + \partial_t ( a_T^{ \prime}(t) u_T(t) ) 
 \right]
$$ 
For the last term we use \eqref{coeff:3}. For the first term we use \eqref{param:small}:
$$ D^{-\frac{1}{2}}  \vn{\sum a_T(t) \Box_{g_{<\sqrt{\mu}}} u_T(t)  }_{L^2(I 
\times \mathbb{R}^2)} \ls \lmd D^{-1} \| \partial^2 g\|_{L^2_t L_x^\infty} 
\sum_{\pm} \vn{a_T}_{L^{\infty}_{t} \ell^2_{T } } \ls \lmd M,
$$ 
while for the second term we use \eqref{wp:sql2} and \eqref{coeff:2}:
$$
 \vn{\sum a_T^{\prime}(t) \pt_{t,x} u_T(t)  }_{L^2(I \times \mathbb{R}^2)} \ls 
 \mu \vn{a_T^{  \prime} (t)}_{L^2_t \ell^2_{T }}  \ls \mu D^{\frac{1}{2}} M ,
$$
which completes the proof.
\end{proof}

The previous proposition provides the main part of the wave packet decomposition. However, it does not provide control of the second derivatives (in time) of the coefficients. To remedy this we have the following corollary. 

\begin{corollary} \label{Cor:WP:dec}
Under the assumptions and notations of Prop. \ref{X:wp:dec}:
\begin{enumerate} [leftmargin=*] 
\item There exists a decomposition 
$$  P_{\mu} v=v^{+}+v^{-}+v_R
$$
where 
\be \label{vpm:wp:dec}
v^{\pm} = P_{\mu} \sum_{T \in \calT_{\mu}^{\pm}} c_T(t) u_T(t)
\ee
such that 
\be \label{vRpm:est}
 \vn{v^{\pm}}_{X_{\mu,D}^{0,\frac{1}{2}}[I]} \ls M, \qquad \vn{v_R}_{X_{\mu,D}^{0,\frac{1}{2}}[I]} \ls M, \ee
\be \label{vRpm:est2}  \vn{v_R}_{L^2[I]} \ls  \mu^{-1} D^{\frac{1}{2}} M
\ee
and
\begin{align} 
\vn{c_T}_{\ell^2_{T } L^{\infty}_{t} }  & \ls  M \label{reg:coeff:1} \\ 
\vn{c_T^{  \prime} (t)}_{L^2_t \ell^2_{T }} & \ls D^{\frac{1}{2}} M \label{reg:coeff:2}  \\
\vn{c_T^{  \prime \prime} (t)}_{L^2_t \ell^2_{T }} & \ls \mu D^{\frac{1}{2}} M   \label{reg:coeff:3}. 
\end{align} 

\item Let $ T,T' \in \calT_{\mu,\omega}^{\pm} $ be two tubes associated to the same direction $ \omega $. Then
\be \label{wp:barL:cT}
\vn{\underline{L} P_{\mu} \big( c_T u_T \big) }_{L^{\infty}(T')} \ls  \ \mu \  
\frac{\mu^{\frac{3}{4}}}{\lng  d(T,T') \rng^N} \vm{c_T}_{L^{\infty}_t} 
\ee

\end{enumerate}
\end{corollary}

\begin{corollary} \label{Cor:WP:dec:rem}
Moreover, for any function $ u_{\lmd} $ localized at frequency $ \simeq \lmd \gg \mu $ we have:
\be \label{bil:est:remainder1}
\vn{u_{\lmd} \cdot v_R}_{X_{\lmd,\mu}^{0,\frac{1}{4}}[I] } \ls \mu^{\frac{3}{4}} \vn{ u_{\lmd}}_{X_{\lmd,D}^{0,\frac{1}{2}}[I]} M
\ee
while for $ \lmd \simeq \mu \gtrsim \eta $ we have
\be \label{bil:est:remainder2}
\vn{P_{\eta} (u_{\lmd} \cdot v_R)}_{X_{\eta,\eta}^{0,\frac{1}{4}}[I] } \ls 
\frac{\lmd}{\eta^{\frac{1}{4}}} \vn{ u_{\lmd}}_{X_{\lmd,D}^{0,\frac{1}{2}}[I]} M
\ee
\end{corollary}

\begin{proof}[Proof of Corollary \ref{Cor:WP:dec}] {\bf(1)} \ 
The previous proposition provides a decomposition into $ + $ and $ - $ components and only condition \eqref{reg:coeff:3} is missing on the coefficients. To gain it, we regularize the coefficients $ a_T(t) $ in time on the $ \mu^{-1} $ scale at the expense of introducing the remainder $ v_R $ which obeys the favorable $ L^2 $ estimate \eqref{vRpm:est2}. We write and define
$$  a_T(t) = a_T^{<\mu}(t) +a_T^{>\mu}(t), \qquad  c_T(t) \defeq a_T^{<\mu}(t) $$ 
The conditions \eqref{reg:coeff:1}, \eqref{reg:coeff:2} are maintained from \eqref{coeff:1}, \eqref{coeff:2}, while \eqref{reg:coeff:3}  follows from \eqref{reg:coeff:2}. 

The $ v^{\pm} $ are defined by \eqref{vpm:wp:dec}. We prove that $  \vn{v^{\pm}}_{X_{\mu,D}^{0,\frac{1}{2}}[I]} \ls M $ and as a consequence we also obtain $ \vn{v_R}_{X_{\mu,D}^{0,\frac{1}{2}}[I]} \ls M $. We have to place $ v^{\pm} $ and $  \Box_{g_{<\sqrt{\mu}}} v^{\pm} $ in $ L^2 $. The fact that $ D^{\frac{1}{2}} \vn{v^{\pm}}_{L^2[I]} \ls M $ and 
$$  D^{-\frac{1}{2}} \mu^{-1} \vn{\sum c_T(t) \Box_{g_{<\sqrt{\mu}}} u_T(t)  }_{L^2(I \times \mathbb{R}^2)} \ls M
$$ 
$$
D^{-\frac{1}{2}} \mu^{-1} \vn{\sum c_T^{ \prime}(t) \pt u_T(t)  }_{L^2(I \times \mathbb{R}^2)}  \ls M 
$$ 
follow like in the proof of part (2) of Prop. \ref{X:wp:dec}. What remains is 
$$ D^{-\frac{1}{2}} \mu^{-1} \vn{\sum c_T^{ \prime \prime}(t) u_T(t)  }_{L^2(I \times \mathbb{R}^2)}  \ls D^{-\frac{1}{2}} \mu^{-1} \vn{c_T^{ \prime \prime} (t)}_{L^2_t \ell^2_{T }}     \ls M 
$$

For $ v_R $ we have 
$$
v_R = P_{\mu} \sum_{\pm,T \in \calT_{\mu}^{\pm}} a_T^{>\mu}(t) u_T(t)
$$
Since 
$$
\mu \vn{a_T^{>\mu}(t) }_{L^2_t \ell^2_{T }} \ls  \vn{a_T^{>\mu, \prime}(t) }_{L^2_t \ell^2_{T }} \ls D^{\frac{1}{2}} M
$$
we obtain \eqref{vRpm:est2}.
\\

{\bf(2)} \ To prove \eqref{wp:barL:cT}, for any $ t \in I $ we write 
$$
\underline{L} P_{\mu}  \big( c_T(t) u_T (t)\big)=  c_T(t) \underline{L} P_{\mu} u_T(t) + 
c_T'(t) P_{\mu} u_T(t)
$$ 
For the first term we use \eqref{wp:barL:T}, while for the second we use \eqref{wp:Linfty:T} together with 
\begin{align}
\label{reg:coeff:4}
\vm{c_T'}_{L^{\infty}_t } \ls \mu \vm{c_T}_{L^{\infty}_t} 
\end{align} 
which holds due to the time regularization done in Step (1). 
\end{proof}

\begin{proof}[Proof of Corollary \ref{Cor:WP:dec:rem}]
Note that by Bernstein's inequality and \eqref{vRpm:est2}, \eqref{vRpm:est}, \eqref{energy:X} we have
$$  \vn{v_R}_{L^2 L^{\infty}} \ls D^{\frac{1}{2}} M, \qquad \vn{\Box_{g_{<\sqrt{\mu}}} v_R}_{L^2 L^{\infty}} \ls \mu^2 D^{\frac{1}{2}} M, \qquad \vn{v_R}_{L^{\infty}_{t,x}} \ls \mu M.
$$
For the $ L^2 $ part of \eqref{bil:est:remainder1} we have
$$ \mu^{\frac{1}{4}} \vn{u_{\lmd} \cdot v_R}_{L^2} \ls \mu^{\frac{1}{4}} \vn{u_{\lmd}}_{L^{\infty} L^2} \vn{v_R}_{L^2 L^{\infty}} \ls \mu^{\frac{3}{4}} (D/\mu)^{\frac{1}{2}}  \vn{ u_{\lmd}}_{X_{\mu,D}^{0,\frac{1}{2}}[I]} M
$$ 
and recall that $ D \leq \mu $. For $ \Box_{g_{<\sqrt{\lmd}}}( u_{\lmd} \cdot v_R ) $ we consider
\begin{align*}
& \vn{ \Box_{g_{<\sqrt{\lmd}}} u_{\lmd} \cdot v_R}_{L^2} \ls \vn{\Box_{g_{<\sqrt{\lmd}}} u_{\lmd} }_{L^2} \vn{v_R}_{L^{\infty}} \ls \lmd D^{\frac{1}{2}} \mu \vn{ u_{\lmd}}_{X_{\mu,D}^{0,\frac{1}{2}}[I]} M \\
& \vn{u_{\lmd} \cdot  \Box_{g_{<\sqrt{\mu}}} v_R  }_{L^2} \ls \vn{u_{\lmd}}_{L^{\infty} L^2} \vn{ \Box_{g_{<\sqrt{\mu}}} v_R}_{L^2 L^{\infty}} \ls \mu^2 D^{\frac{1}{2}} \vn{ u_{\lmd}}_{X_{\mu,D}^{0,\frac{1}{2}}[I]}   M \\
& \vn{u_{\lmd} \cdot  (\Box_{g_{<\sqrt{\lmd}}}-\Box_{g_{<\sqrt{\mu}}})  v_R  }_{L^2} \ls \vn{u_{\lmd}}_{L^2} \vn{\mu^{-1} \partial^2  v_R  }_{L^{\infty}} \ls \mu^2 D^{\frac{1}{2}} \vn{ u_{\lmd}}_{X_{\mu,D}^{0,\frac{1}{2}}[I]}   M \\
& \vn{\partial u_{\lmd} \cdot  \partial v_R  }_{L^2} \ls \vn{\partial u_{\lmd}}_{L^{\infty} L^2} \vn{ \partial v_R}_{L^2 L^{\infty}} \ls \lmd \mu D^{\frac{1}{2}} \vn{ u_{\lmd}}_{X_{\mu,D}^{0,\frac{1}{2}}[I]}   M
\end{align*}
These estimates combine to complete the proof of \eqref{bil:est:remainder1} and 
we turn to \eqref{bil:est:remainder2}. Here we use Bernstein's inequality in 
the form $ P_{\eta}:L^2 L^1 \to \eta L^2 $. We have
$$ \eta^{\frac{1}{4}} \vn{P_{\eta} (u_{\lmd} \cdot v_R) }_{L^2} \ls 
\eta^{\frac{1}{4}} \eta \vn{u_{\lmd} \cdot v_R }_{L^2 L^1} \ls 
\eta^{\frac{1}{4}} \eta \vn{u_{\lmd}}_{L^{\infty}L^2} \vn{v_R}_{L^2}
$$ 
which clearly suffices using \eqref{vRpm:est2}. For $ \Box_{g_{<\sqrt{\eta}}}( u_{\lmd} \cdot v_R ) $ we similarly consider
\begin{align*}
& \vn{(\Box_{g_{<\sqrt{\lmd}}}-\Box_{g_{<\sqrt{\eta}}}) P_{\eta}( u_{\lmd} 
\cdot v_R ) }_{L^2} \ls \eta^2 \vn{u_{\lmd} \cdot v_R}_{L^2 L^1} \ls \\
& \qquad \qquad \qquad \qquad \qquad \qquad \qquad \ \ls \eta^2 \vn{u_{\lmd} }_{L^2} \vn{v_R}_{L^{\infty}L^2} \ls \eta^2 \vn{u_{\lmd}}_{X_{\mu,D}^{0,\frac{1}{2}}[I]} M  \\
& \eta \vn{\partial u_{\lmd} \cdot  \partial v_R  }_{L^2 L^1} \ls \eta \vn{\partial u_{\lmd}}_{L^{\infty} L^2} \vn{ \partial v_R}_{L^2} \ls \eta \lmd D^{\frac{1}{2}} \vn{ u_{\lmd}}_{X_{\mu,D}^{0,\frac{1}{2}}[I]}   M\\
& \eta \vn{ \Box_{g_{<\sqrt{\lmd}}} u_{\lmd} \cdot v_R}_{L^2 L^1} \ls \eta \vn{\Box_{g_{<\sqrt{\lmd}}} u_{\lmd} }_{L^2} \vn{v_R}_{L^{\infty}L^2} \ls \eta \lmd D^{\frac{1}{2}} \vn{ u_{\lmd}}_{X_{\mu,D}^{0,\frac{1}{2}}[I]} M \\
\end{align*}
and similarly for $  u_{\lmd} \cdot \Box_{g_{<\sqrt{\lmd}}} v_R$. Each of the term above times $ \eta^{-\frac{7}{4}} $ is $ \ls \lmd / \eta^{\frac{1}{4}} $.
\end{proof}

\

Finally, we have the following time-dependent wave packets sums analogue of \eqref{wp:Linfty:T}-\eqref{wp:barL:T}.

\begin{lemma}
For a fixed $ \omega $ and a dyadic frequency $ \eta $ let 
$$
v^{\omega}(t)=\sum_{ T\in \calT_{\eta}^{\pm}, \omega_T=\omega } c_T(t) u_T(t),  \qquad \qquad  t \in I. $$
For $ \nu \geq \eta $, let the vector fields $ L,\underline{L},E $ associated to $ 
\pm g_{\sqrt{\nu}} $ and to $ \omega $ form a null frame as in section 
\ref{s:nullframes}. 
Then one has:
\begin{align}
\label{wp:Linfty:T:sum}
\vn{v^{\omega}}_{L^{\infty}} \quad & \ls \eta^{\frac{3}{4}} \big( \sum_{ T, \omega_T=\omega } \vm{c_T}_{L^{\infty}_t}^2  \big)^{\frac{1}{2}} \\
\label{wp:E:T:sum}
\vn{E P_{\eta}  v^{\omega}} _{L^{\infty}} \ & \ls \eta^{\frac{5}{4}} \big( 
\sum_{ T, \omega_T=\omega } \vm{c_T}_{L^{\infty}_t}^2  \big)^{\frac{1}{2}} \\
\label{wp:L:T:sum}
\vn{ L P_{\eta}  v^{\omega}} _{L^2 L^{\infty}} & \ls \eta^{\frac{3}{4}} \big( 
\sum_{ T, \omega_T=\omega } \vm{c_T}_{L^{\infty}_t}^2  +  
\vm{c_T'}_{L^{2}_t}^2  \big)^{\frac{1}{2}} \\
\label{wp:barL:T:sum}
\vn{ \underline{L} P_{\eta}  v^{\omega}} _{L^\infty} & \ls 
\eta^{\frac{7}{4}} 
\big( \sum_{ T, \omega_T=\omega } \vm{c_T}_{L^{\infty}_t}^2  \big)^{\frac{1}{2}}.
\end{align}
\end{lemma}

\begin{proof}
These follow easily from \eqref{wp:Linfty:T}-\eqref{wp:barL:T}.
In the case of \eqref{wp:L:T:sum} and \eqref{wp:barL:T:sum} we have to consider the case in which the $ \partial_t $ derivative from $ L $ or $ \underline{L} $ falls on the coefficients $ c_T(t) $. We write
$$
L (P_{\eta}  v^{\omega})(t)=\sum_{ T, \omega_T=\omega }  c_T(t) L  P_{\eta} 
u_T(t) + \sum_{ T, \omega_T=\omega } c_T'(t) P_{\eta} u_T(t)
$$ 
For the first sum we use \eqref{wp:L:T}, while for the second we use \eqref{wp:Linfty:T}. The bound \eqref{wp:barL:T:sum}
follows from \eqref{wp:barL:cT}.
\end{proof}

\

\section{Microlocalized characteristic energy estimates}
\label{s:CE}

\

The proof of the algebra property $X^{s,\theta} \cdot X^{s, \theta}
\subset X^{s,\theta}$ relies on estimates for
directionally localized half-waves along certain characteristic surfaces.

\subsection{Microlocalization}

For each dyadic number $\alpha \le 1$, 
let $\tau + a_{<\alpha^{-1}} ^\pm$ be the half-wave symbols for the
operator $\Box_{g_{<\alpha^{-1}}}$ (note that $a_{<\alpha^{-1}}^{\pm}$ is not 
quite
frequency-localized due to the square root) and let $\Phi^{\alpha, \pm}_t (x,
\xi) = (x_t^{\alpha, \pm}, \xi_t^{\alpha, \pm})$ denote their
Hamiltonian flows.

On one hand, a routine linearization argument as in the proof of 
Lemma~\ref{l:optical_bounds} shows that the flow map satisfies
\begin{align}
\label{e:flow-deriv}
\frac{\partial^k(x^{\alpha, \pm}_t, \xi^{\alpha, \pm}_t)}{\partial (x, \xi)^k } 
=
O(\alpha^{1-|k|}), \quad |k| \ge 1.
\end{align}
On the other hand, in view of homogeneity the flow map is smoother in some 
directions than others. To capture this directional information, we consider a 
class of phase space metrics
adapted to the wave equation developed 
in~\cite{geba2007phase,geba2008gradient}. For each
$0 < \alpha < 1$,
define $g := g_{\alpha}$ by
\begin{align}
\label{e:g}
g_{\alpha, (x, \xi)}(y, \eta) = \frac{ | \langle y, \xi \rangle |^2 }{
	\alpha^4 |\xi|^2} + \frac{ |y \wedge \xi|^2 }{ \alpha^2
	|\xi|^2} + \frac{ |\langle \eta, \xi \rangle |^2 }{ |\xi|^4 } +
\frac{ |\eta \wedge \xi|^2}{ \alpha^2 |\xi|^4}.
\end{align}

We recall
\begin{lemma}[{\cite[Lemma 4.2]{geba2008gradient}}]
	The flows $\Phi^{\alpha, \pm}_{t}$ are bi-Lipschitz and $g_\alpha$-smooth.
\end{lemma}

These metrics shall define the symbol classes $S(m, g)$ in which we 
work. At a point $(x, \xi)$, the unit ball with respect to $g_{(x,\xi)}$
consists of a $\alpha^2 \times (\alpha)^{n-1}$ rectangle in the
spatial variable with long side orthogonal to $\xi$, and a
$|\xi| \times (\alpha |\xi|)^{n-1}$ rectangle in the frequency
variable with long side parallel to $\xi$. One easily verifies that
perturbing the basepoint $(x, \xi)$ within this unit ball yields
comparable metrics.
Appendix~\ref{s:microlocal} collects the relevant facts and notation concerning 
pseudo-differential calculus at this generality.

Throughout this article, for any frequency $\mu \ge 1$ we write
\[\alpha_\mu := \mu^{-1/2}.\] The parameter $\mu$ will eventually be the
smaller of the two frequencies in products of the form
$u_\lambda v_\mu \in X^{s, \theta}_\lambda \cdot X^{s,\theta}_\mu$,
and in this context, $\alpha_\mu$ represents the smallest angular
scale in the bilinear decomposition of the product.

If $\alpha \in [\alpha_\mu, 1]$, write $\Omega_{\alpha}$ for the
collection of half-open dyadic intervals on $S^1 = \mathbb{R} / \mathbb{Z}$
of width $\alpha$.

For $\theta \in \Omega_\alpha$, let $s_\theta^\alpha(\xi)$ be a
$0$-homogeneous function supported in the sector defined by
$\wht{\xi} := \xi/|\xi| \in C \theta$, where $C\theta$ denotes the
dilate of the interval $\theta$ about its center by some (fixed)
factor $C > 1$, and define time-dependent symbols
$\phi_{\theta}^{\alpha, \pm}$ by transporting $s_\theta^\alpha$ along
the flows $\Phi^{\alpha_\mu, \pm}_t$:
\begin{align} \label{pullback:flow}
  \phi_{\theta}^{\alpha, \pm}(t, x, \xi) := s_\theta^\alpha \circ
  \Phi_{-t}^{\alpha_\mu, \pm}.
\end{align}

By the previous Lemma, we observe
\begin{lemma}
  The symbols $\phi^{\alpha, \pm}_{\theta}$ satisfy
  $\partial_{x,\xi} \phi^{\alpha, \pm}_{\theta} \in S(
  (\alpha |\xi|)^{-1}, g_{\alpha_\mu})$. 
\end{lemma}
The notation $ S(m,g) $ refers to the symbol classes in Definition \ref{def:symbol:class}.
\begin{remark}
	A very similar construction was proposed by Geba-Tataru \cite[Section 
	4]{geba2008gradient}. 
	However, here the time dependent symbols are defined using the same flows 
	$\Phi_t^{\alpha_\mu, \pm}$ for all angular widths $\alpha \ge \alpha_\mu$, 
	whereas in that article the initial symbols $s_{\theta}^\alpha$ are 
	transported along the flows $\Phi_{t}^{\alpha, \pm}$. Consequently their 
	symbols satisfy the better bounds $\partial \phi_{\theta}^{\alpha, \pm} \in 
	S( \alpha|\xi|^{-1}, g_\alpha)$. 
\end{remark}

For each $\lambda \ge \mu$, define the symbols at frequency
$\lambda$ by
\begin{align}
  \label{e:pd-loc}
  \phi_{\theta, \lambda}^{\alpha, \pm}(t, x, \xi) :=
   P_{<\lambda/8} (D_x)\phi_{\theta}^{\alpha, \pm} (t, x, \xi) s_\lambda(\xi),
\end{align}
where $s_{\lambda}(\xi)$ is a smooth cutoff supported in the annulus
$|\xi| \in [\lambda/4, 4\lambda]$ and equal to $1$ for $|\xi| \in [\lambda/2, 
2\lambda]$.

\begin{lemma}
  \label{l:poisson-bracket}
 The symbols $\phi_{\theta, \lambda}^{\alpha, \pm}$ satisfy $\partial
 \phi_{\theta, \lambda}^{\alpha, \pm} \in 
S ( (\alpha \lambda)^{-1}, g_{\alpha_\mu})$, and also
\begin{align*}
 \{ \tau + a_{<\alpha_\mu^{-1} }^{\pm}, \phi_{\theta, \lambda}^{\pm, 
\alpha} \} \in S(1, g_{\alpha_\mu}).
\end{align*}

\end{lemma}
\begin{proof}
  The fact that $\phi_{\theta, \lambda}^{\alpha, \pm} \in S(1,
  \alpha_\mu)$ is straightforward. Since $\phi_\theta^{\alpha, \pm}$
  is transported along the Hamilton flow for $a_{<\alpha_\mu^{-1}} (t,
  x, \xi)$, the Poisson bracket takes the form
  \begin{align*}
    -(\partial_x a_{<\alpha_\mu^{-1}} \partial_\xi \tilde{s}_\lambda(\xi))
    \phi_{\theta, \lambda}^{\alpha, \pm}(t,x, \xi) +
    s_{\lambda}(\xi)[H_{a_{<\alpha_\mu^{-1}}^{\pm}}, P_{<\lambda/8} (D_x)]
    \phi_\theta^{\alpha, \pm} (t, x, \xi),
  \end{align*}
  where $\tilde{s}_\lambda$ is a fattened version of
  $s_{\lambda}$ and $H_a = \langle a_x, \partial_x\rangle - \langle
  a_\xi, \partial_\xi \rangle$ denotes the Hamiltonian vector field
  for a symbol $a$. The first term belongs to $S( \lambda^{-1},
  g_{\alpha_\mu})$.
  Now for any $k_1, k_2$, one has
    \begin{align*}
      \partial^{k_1}_x \partial^{k_2}_\xi [\partial_\xi
      a_{<\alpha_\mu^{-1}}^\pm, P_{<\lambda/8}] \partial_x
      \phi_\theta^{\alpha, \pm} = \sum_{j_1=0}^{k_1}
      \sum_{j_2=0}^{k_2} [ \partial^{j_1}_x \partial^{j_2}_\xi
      \partial_\xi a_{<\alpha_\mu^{-1}}^\pm, P_{<\lambda/8}]
      \partial^{k_1 -j_1}_x \partial^{k_2-j_2}_\xi (\partial_x
      \phi_\theta^{\alpha, \pm})
    \end{align*}
    As $\partial_x \partial_\xi a_{<\alpha_\mu^{-1}}^{\pm}$ is
    $g_{\alpha_\mu}$-smooth, for any
    $y_1, \dots, y_{k_1}, \eta_1, \dots, \eta_{k_2}$ with
    $g(y_j, 0) = g(0, \eta_j) = 1$,
    \begin{align*}
      \|[\langle y_1, \partial_x \rangle \cdots \langle y_{j_1},
      \partial_x \rangle \langle \eta_1, \partial_\xi \rangle \dots
      \langle \eta_{j_2}, \partial_\xi\rangle
      \partial_\xi a_{<\alpha_\mu^{-1}}^\pm, P_{<\lambda/8}]\|_{L^p \to L^p}
      \lesssim \lambda^{-1}.
    \end{align*}
    Consequently
    \begin{align*}
      \Bigl| \prod_{j_1=1}^{k_1} \prod_{j_2=1}^{k_2} \langle y_{j_1},
      \partial_x \rangle \langle \eta_{j_2}, \partial_\xi \rangle
     [ \partial_{\xi} a_{<\alpha_\mu^{-1}}^{\pm}, P_{<\lambda/8}]
      \partial_x \phi_\theta^{\alpha, \pm}(t,x, \xi)\Bigr| \lesssim \lambda^{-1}
      (\alpha |\xi|)^{-1},
    \end{align*}
    therefore
    \begin{align*}
      [\partial_\xi a_{<\alpha_\mu^{-1}}^{\pm}, P_{<\lambda/8} ]
      (\partial_x \phi_{\theta}^{\alpha, \pm}) \in S( (\alpha \lambda
      |\xi|)^{-1}, g_{\alpha_\mu}).
    \end{align*}
    Similarly, as Bernstein implies that $\partial^2_xa_{<\alpha_\mu^{-1}}^{\pm}
    \in S( \alpha_\mu^{-\frac{1}{2}} |\xi|, g_{\alpha_\mu})$, we have
    \begin{align*}
      [\partial_x a_{<\alpha_\mu^{-1}}^{\pm}, P_{<\lambda/8} ]
      (\partial_\xi \phi_\theta^{\alpha, \pm} \in S( 
      (\alpha_\mu)^{-\frac{1}{2}}(\alpha
      \lambda)^{-1}, g_{\alpha_\mu}).
    \end{align*}
\end{proof}

  For each $x$, the symbol $\phi_{\theta}^{\alpha, \pm}(t, x, \xi)$ is
supported in a sector $|\hat{\xi} - \widehat{\xi^{\alpha_\mu}_\theta} (t,
  x)| \le c \alpha$, where $\hat{\xi} := \xi / |\xi|$ and \[ (x, \theta)
  \mapsto \bigl( (x^{\alpha_\mu}_\theta(t, x), \theta)
\mapsto  (x, \xi^{\alpha_\mu}_\theta
  (t, x) )
  \]
  parametrizes the graph of the canonical transformation
  $\Phi^{\alpha_\mu, \pm}_t$ (the dependence on $\pm$ is suppressed in
  the notation $\xi_{\theta}^{\alpha_\mu}(t, x)$). The mollified symbol 
  $\phi_{\theta,
    \lambda}^{\alpha, \pm}$ is no longer sharply localized to the sector
  $|\hat{\xi} - \widehat{\xi^{\alpha_\mu}_\theta }(t,x)| \le c \alpha$, but
  we can write
  \begin{align}
    \label{e:input_angular_truncation}
    \phi_{\theta, \lambda}^{\alpha, \pm}(t, x, \xi) = \phi_{\theta,
      \lambda}^{\alpha, \pm}  \chi_{<2c\alpha} (|\hat{\xi} -
    \widehat{\xi^{\alpha_\mu}_\theta} (t, x) |) + r_{\theta,
      \lambda}^{\alpha, \pm},
\end{align}
where the first symbol has the same regularity as $\phi_{\theta,
  \lambda}^{\alpha, \pm}$ and $r_{\theta, \lambda}^{\alpha, \pm} = 
  O(\lambda^{-\infty})$.

For each $\theta$, let
\begin{align*}
  m_\theta(t, x, \xi) := \langle \alpha^{-1} ( |\widehat{\xi} - \widehat{
  \xi_\theta^{\alpha_\mu}} (t, x) ) \rangle^{-1}.
\end{align*}
In the notation of Section~\ref{s:microlocal} one has
\begin{align*}
  \phi_{\theta}^{\alpha, \pm}, \ \phi_{\theta, \lambda}^{\alpha, \pm}  \in  
  S^1_{\alpha} ( m_\theta^\infty, g_{\alpha_\mu}).
\end{align*}

For future reference, we also record a technical lemma regarding the
time-regularity of symbols.
\begin{lemma}
  \label{l:t-derivs}
  There is a decomposition
  \begin{align*}
    \partial_t \phi_{\theta}^{\alpha, \pm} = \psi_1 + \psi_2,
  \end{align*}
  where $\psi_1\in S(1, g_{\alpha_\mu})$ and $\psi_2$ satisfies the
  estimates
  \begin{align*}
    &|\psi_2| \lesssim_N \alpha^{-1} m_\theta^N \text{ for all } N, \\
    &\partial_x \psi_2
    \in \alpha_\mu^{-\frac{1}{2}} \alpha^{-1} S( m_\theta^\infty,
    g_{\alpha_\mu}) + \alpha_\mu^{-1} S(m_\theta^\infty,
      g_{\alpha_\mu}),\\
    &\partial_\xi \psi_2 \in \alpha^{-1}(\alpha_\mu |\xi|)^{-1}
    S(m_\theta^\infty, g_{\alpha_\mu});
  \end{align*}
  in particular $\partial_t \phi_{\theta}^{\alpha, \pm} \in
  \alpha^{-1} S(1, g_{\alpha_\mu})$.
  Also,
  \begin{gather*}
  \partial_t \{ \tau+ a_{<\alpha_\mu^{-1}}^{\pm},
    \phi_{\theta,\lambda}^{\alpha, \pm} \} \in S(
    \alpha_\mu^{-\frac{1}{2}} (\alpha_\mu^2 \lambda)^{-1} m_\theta^\infty, 
    g_{\alpha_\mu}).
  \end{gather*}
\end{lemma}
\begin{proof}
  The definition of the symbol implies that
  \begin{align*}
    \partial_t \phi_{\theta}^{\alpha, \pm} = \langle a_\xi, \partial_x\rangle
    \phi_\theta^{\alpha, \pm}  - \langle a_x, \partial_\xi\rangle
    \phi_{\theta}^{\alpha, \pm}.
  \end{align*}
  The first term evidently belongs to $S(1, g_{\alpha_\mu})$. The second term is
  pointwise bounded by $\alpha^{-1}$ and satisfies
  \begin{gather*}
   \partial_x (\langle a_x, \partial_\xi \rangle \phi_{\theta}^{\pm,
    \alpha}) = \langle a_{xx}, \partial_{\xi} \rangle
    \phi_{\theta}^{\alpha, \pm} + \langle a_x, \partial_\xi \rangle
    \partial_x \phi_{\theta}^{\alpha, \pm} \in
    \alpha_\mu^{-\frac{1}{2}} \alpha^{-1}S(m_\theta^\infty,
    g_{\alpha_\mu}) + \alpha_\mu^{-1} S(m_\theta^\infty,
    g_{\alpha_\mu}),\\
    \partial_\xi( \langle a_x, \partial_\xi \rangle
    \phi_{\theta}^{\alpha, \pm}) = \langle a_{\xi x}, \partial_\xi
    \rangle \phi_{\theta}^{\alpha, \pm} + \langle a_x, \partial_\xi
    \rangle \partial_\xi \phi_{\theta}^{\alpha, \pm} \in \alpha_\mu^{-1} (\alpha
    |\xi|)^{-1} S(m_\theta^\infty, g_{\alpha_\mu}),
  \end{gather*}
  where Bernstein is used for pointwise bounds $a_{xx}$ and further
  derivatives as in~\eqref{e:metric_est1}. These estimates are
  preserved by the mollifier $P_{<\lambda/8}(D_x)$.
  
  The second claim is proved by inspecting the Poisson bracket estimates in the 
  previous lemma:
  \begin{align*}
    \partial_t \{ \tau+ a_{<\alpha_\mu^{-1}}^{\pm},
    \phi_{\theta,\lambda}^{\alpha, \pm} \} &= -(\partial_x \partial_t
                                             a_{<\alpha_\mu^{-1}}
                                             \partial_\xi
                                             \tilde{s}_\lambda(\xi))
                                             \phi^{\pm,
                                             \alpha}_{\theta, \lambda}
                                             - (\partial_x
                                             a_{<\alpha_\mu^{-1}}
                                             \partial_\xi
                                             \tilde{s}_{\lambda}
                                             (\xi)) \partial_t
                                             \phi_{\theta,
                                             \lambda}^{\alpha, \pm}
                                             (t, x, \xi)\\
    &+ s_\lambda(\xi)[ H_{\partial_t a^{\pm}_{<\alpha_\mu^{-1}}},
      P_{<\lambda/8}(D_x)] \phi_\theta^{\alpha, \pm} +s_\lambda(\xi)[ H_{ 
      a^{\pm}_{<\alpha_\mu^{-1}}},
      P_{<\lambda/8}(D_x)] \partial_t \phi_\theta^{\alpha, \pm}\\
    &\in S(  \alpha_\mu^{-\frac{1}{2}} (\alpha_\mu^2 \lambda)^{-1}, 
    g_{\alpha_\mu}),
  \end{align*}
where we have used the Bernstein type estimates $\partial_x \partial_\xi
\partial_ta_{<\alpha_\mu^{-1}}^{\pm} \in S(\alpha_\mu^{-\frac{1}{4}},
g_{\alpha_\mu})$, respectively  $\partial_x^2 \partial_t
a_{<\alpha_\mu^{-1}}^{\pm} \in S( \alpha_\mu^{-\frac{3}{2}}|\xi|, 
g_{\alpha_\mu})$.
\end{proof}

\begin{proposition}
  \label{p:orthogonality}
  Suppose $\alpha \ge \alpha_\mu$. If $u$ is a function at frequency $\lambda > 
\alpha^{-2}$, then
  \begin{align*}
    \sum_{\theta \in \Omega_\alpha} \| \phi_{\theta, \lambda}^{\pm, 
    \alpha}(t,X,D)
    u \|_{X_\pm}^2 \sim \|u \|_{X_\pm}^2,
  \end{align*}
  where
  \begin{align*}
    \| u\|_{X_\pm}^2 := \| u\|_{L^2_{t,x}}^2 + \| (D_t + A^\pm) 
    u\|_{L^2_{t,x}}^2,
  \end{align*}
  and $a^\pm$ are the half wave symbols for the mollified wave operator
  \begin{align*}
    p = \tau^2 - 2 g^{0j}_{<\sqrt{\lambda}} \tau \xi_j -
    g^{ab}_{<\sqrt{\lambda}} \xi_a \xi_b.
  \end{align*}
  More precisely, we have
  \begin{align*}
   & \sum_{\theta} \| \phi^{\pm, \alpha}_{\theta, \lambda} (t, X, D)
     u \|_{L^2}^2 \sim \| u\|_{L^2}^2,\\
    & \sum_{\theta} \| (D_t + A^{\pm} ) \phi^{\pm, \alpha}_{\theta,
      \lambda} (t, X, D) u\|^2_{L^2} \sim \| (D_t + A^\pm)
      u\|_{L^2}^2 + O(\|u\|_{L^2}^2).
  \end{align*}

\end{proposition}

In practice we shall usually need only the ``$\lesssim$''
direction. The following terminology will be convenient for describing
the size of various operators.
\begin{definition}
  We say a collection of operators $\chi_\theta: L^2_x \to L^2_x$ is
  \emph{square-summable} with respect to $\theta$ if
  $\sum_\theta \| \chi_\theta u\|^2_{L^2_x} \lesssim \| u\|^2_{L^2_x}$; that
  is, if $\sum_\theta \chi_\theta^* \chi_\theta$ is bounded on $L^2_x$. 
\end{definition}
If the operators $\chi_\theta$ depend on $t$, it will be clear from the context 
whether the implicit constants are uniform or merely
square-integrable with respect to $t$; for the latter case we use the term 
``$L^2_{t,x}$-square-summable".

From the pseudo-differential calculus in Section~\ref{s:microlocal}
and the Cotlar-Stein almost orthogonality criterion, one immediately
deduces
\begin{lemma}
  \label{l:square-summable}
  If symbols $\phi_\theta\in S ( m_\theta^\infty, g_{\alpha_\mu})$ are
  supported in $|\xi| \sim \lambda \ge \alpha_\mu^{-2}$, then the operators 
  $P_{\lambda}(D)\phi_\theta (X, D)$ are square-summable with respect to 
  $\theta$.
\end{lemma}

\begin{proof}[Proof of Proposition]
  Consider first the $L^2$ component. The previous lemma already shows that
  \begin{align*}
    \sum_\theta \| \phi_{\theta, \lambda}^{\pm, \alpha} (t,
    X,D)u\|_{L^2}^2 \lesssim \|u\|_{L^2}^2.
  \end{align*}
  For the other direction, we note first that if $u$ is localized at
  frequency $\lambda$, then
  \begin{align*}
    \|u\|_{L^2} \sim \Bigl\| \sum_{\theta \in \Omega_\alpha} \phi_{\theta,
    \lambda}^{\pm, \alpha} (t, X, D) u\Bigr\|_{L^2}.
  \end{align*}

As $\phi_{\theta, \lambda}^{\pm, \alpha} \in S(1, g_{\alpha_\mu})$
  and localized to frequency $\lambda$ in both input and output, the direction
  ``$\gtrsim$'' follows directly from Lemma~\ref{l:CV}.  For the
  opposite direction, use the first order calculus~\eqref{e:symbol_expansion} 
  to write
\begin{align*}
u = P_\lambda(D) \Bigl( \sum_{\theta} \phi_{\theta, \lambda}^{\alpha}
  \Bigr)^{-1}(t, X, D) \sum_{\theta} \phi_{\theta,
  \lambda}^{\alpha}(t, X, D) u + P_\lambda(D) r(t,X, D)u,
\end{align*}
where $r \in S( (\alpha \lambda)^{-2}, g_{\alpha_\mu})$, and apply
Lemma~\ref{l:CV} to obtain
\begin{align*}
  \|u\|_{L^2} \lesssim \Bigl\| \sum_{\theta} \phi_{\theta, \lambda}^\alpha
  u\Bigr\|_{L^2} +  (\alpha \lambda)^{-2} \| u\|_{L^2}.
\end{align*}
  (Of course there is nothing to prove if
$\sum_{\theta} \phi_{\theta}^{\pm, \alpha} \equiv 1$ but recall that
according to our construction of the symbols, for a fixed angular
scale $\alpha$ the sum is in general merely bounded above and below.)

Hence we may write
  \begin{align*}
    \|u\|_{L^2}^2 \lesssim \sum_{\theta, \theta'} \langle \chi_\theta u,
    \chi_{\theta'} u\rangle,
  \end{align*}
  The pseudo-differential calculus yields the estimates
  \begin{align*}
    \|\chi_\theta^* \chi_{\theta'}\|_{L^2 \to L^2} \lesssim
    \langle d_{\alpha} (\theta, \theta') \rangle^{-N}, \quad
    d_{\alpha}(\theta, \theta') := | \alpha^{-1}(\theta- \theta')|.
  \end{align*}
  Splitting
  \begin{align*}
    \sum_{\theta, \theta'} \langle \chi_{\theta} u, \chi_{\theta'} u\rangle
    = \sum_{d_{\alpha}(\theta, \theta') \le M} \langle \chi_{\theta} u,
    \chi_{\theta'} u \rangle + \sum_{d_{\alpha}(\theta,
    \theta') > M} \langle \chi_{\theta} u, \chi_{\theta'} u\rangle,
  \end{align*}
  for $M$ large enough the second sum may be absorbed into the left
  side, while the remaining terms are handled by Cauchy-Schwarz.

  Next consider the half-wave component. Without loss of generality we 
  consider just the ``$+$'' case and set 
  $\phi_{\theta, \lambda}^\alpha := \phi_{\theta, \lambda}^+$, $a := a^+$, $A 
  := A^+$. Writing
  \begin{align*}
    (D_t + A) \phi_{\theta, \lambda}^\alpha = \phi_{\theta, \lambda}
    (D_t + A) \phi_{\theta, \lambda}^\alpha + [D_t + A, \phi_{\theta, 
    \lambda}^\alpha],
  \end{align*}
  it suffices by the first part and energy estimates to show that
  \begin{align*}
    \sum_{\theta} \|[D_t + A, \phi_{\theta, \lambda}^\alpha] u\|_{L^2}^2
    \lesssim \|u\|_{L^\infty L^2}^2.
  \end{align*}
  Consider the first the low-frequency portion $A_\mu$, where $A$ is
  the corresponding
  half-wave operator for the low-frequency metric
  $g_{<\sqrt{\mu}}$. By the second-order symbol 
  expansion~\eqref{e:symbol_expansion},
    the commutator $[D_t + A_\mu,
  \phi_{\theta, \lambda}^\alpha]$ has symbol
  \begin{align*}
    \frac{1}{i} \{ \tau + a_\mu, \phi_{\theta, \lambda}^\alpha\} -
    \frac{1}{2} \int_0^1 r_s \, ds,
	\end{align*}
	where
	\begin{align*}
    r_s(t, x, \xi) = &\sum_{j, k} e^{is
    \langle D_y, D_\eta \rangle} \bigl[ \partial_{\eta_j}\partial_{\eta_k}
  a(x, \eta)\partial_{y_j}\partial_{y_k}
  \phi_{\theta,\lambda}^\alpha(y, \xi) \\
  &-
  \partial_{\eta_j}\partial_{\eta_k}\phi_{\theta,\lambda}^\alpha (x, \eta)
   \partial_{y_j} \partial_{y_k} a(y, \xi)\bigr] |_{\stackrel{y=x}{\eta=\xi}}.
  \end{align*}
  The Poisson bracket belongs to
  $S( m_\theta^\infty, g_{\alpha_\mu})$. Hence
  $P_\lambda (D) \{\tau + a_\mu, \phi_{\theta, \lambda}^\alpha\} (t,
  X, D)$ is square-summable by the pseudo-differential calculus.
  The other frequency outputs result from
  $\{P_{>\lambda/16}(D_x) a_\mu, \phi_{\theta, \lambda}^{\alpha}\}(t,
  X, D) \in \lambda^{-N} OPS( m_\theta^\infty, g_{\alpha_\mu})$, and
  are therefore square-summable by a brute force bound
  using Lemma~\ref{l:schur-bound} and the triangle inequality.

  To evaluate the remainder, note that as
  \begin{alignat*}{4}
    &\partial^2_\xi a \in S( |\xi|^{-1}, g_{\alpha_\mu}), &\quad
    &\partial_x^2 \phi_{\theta, \lambda}^\alpha \in \alpha_\mu^{-2} S( 
    m_\theta^\infty, g_{\alpha_\mu}),\\
    &\partial^2_x a \in L^2 S( |\xi|, g_{\alpha_\mu}), &\quad
    &\partial_\xi^2 \phi_{\theta, \lambda}^\alpha \in (\alpha
    \lambda)^{-1} (\alpha_\mu \lambda)^{-1}S(m_\theta^\infty , g_{\alpha_\mu}),
  \end{alignat*}
  by Lemma~\ref{l:gauss_transform} one has
  \begin{align*}
    r_s \in  (\alpha_\mu^2\lambda)^{-1}S(  m_\theta^\infty,
    g_{\alpha_\mu}) +   (\alpha^2 \lambda)^{-\frac{1}{2}}
    (\alpha_\mu^2 \lambda)^{-\frac{1}{2}} L^2 S(m^\infty_\theta, g_{\alpha}).
  \end{align*}
   So $r_s(t, X, D)$ are $L^2_{t,x}$-square-summable by considering
   the operators $P_\lambda(D) r_s(t, X, D)$ and $(1-P_\lambda(D)) r_s(t, X, D)$
   separately as before.

   It remains to show that
   \begin{align*}
     \sum_\theta \| [A -A_\mu, \phi_{\theta,\lambda}^\alpha]
     u\|_{L^2}^2 \lesssim \| u\|_{L^2}^2.
   \end{align*}
   From the computations
   \begin{alignat*}{4}
     \partial_\xi (a-a_\mu) &\in S( \alpha_\mu^2, g_{\alpha_\mu}),
     &\quad &\partial_x \phi_{\theta, \lambda}^\alpha \in S(1,
     g_{\alpha_\mu}),\\
     \partial_x (a-a_\mu) &\in S( \alpha_\mu^2|\xi|, g_{\alpha_\mu}),
     &\quad &\partial_\xi \phi_{\theta, \lambda}^\alpha \in S( (\alpha
     \lambda)^{-1}, g_{\alpha_\mu})
   \end{alignat*}
    the formula
   \begin{align*}
     a \circ b = ab + \frac{1}{i} \int_0^1 e^{is\langle D_y, D_\eta
     \rangle} \langle \partial_\eta a(x, \eta), \partial_y b(y, \xi)
     \rangle|_{\stackrel{y=x}{\eta=\xi}} \, ds,
   \end{align*}
   and Lemma~\ref{l:gauss_transform}, it follows that the symbol of
   $[A-A_\mu, \phi_{\theta, \lambda}^\alpha]$ belongs to
   $S(\alpha_\mu^2 \alpha^{-1}, g_{\alpha_\mu})$. As before this
   implies that $[A-A_\mu, \phi_{\theta, \lambda}^\alpha] =
   \alpha_\mu^2 \alpha^{-1} \chi_\theta$ for some square-summable $\chi_\theta$.
   \end{proof}

\subsection{Characteristic energy estimates} \label{Char:en:subsec}

The purpose of this section is to prove energy estimates for
directionally localized half-waves along certain null surfaces.

We begin by recalling the usual characteristic energy estimate for a general
function $v$. For further details, see for instance 
Alinhac's book~\cite{Alinhac2010}. Suppose $\Omega$ is a spacetime domain
whose boundary $\partial \Omega = \Lambda \cup \Sigma_- \cup \Sigma_+$
decomposes into a null hypersurface $\Lambda$ and the time slices
$\Sigma_\pm = \{t = t_\pm\} \cap \Omega$, where $t_- < t_+$. Let $L$
be a geodesic generator for $\Lambda$ which is extended to a null
frame $\{L, \underline{L}, E\}$ on $\Omega$, so that $L, E$ are
tangent to $\Lambda$.

By contracting the stress energy tensor
\[T = dv \otimes dv - \frac{1}{2} g^{-1}(d v, d v) g\] with
$\partial_t$ and applying the divergence theorem, one controls on
$\Lambda$ the derivatives of $v$ tangential to $\Lambda$:
\begin{equation}
  \label{e:CE1}
  \begin{split}
  \int_{\Lambda} |Lv|^2 + |Ev|^2 \, d\sigma &\lesssim \int_{\Sigma_\pm} | 
  \nabla_{t, x} u|^2 \, dx                                               + 
  \int_{\Omega} |\Box u|
                                              |\partial_t u| \, dx
                                              dt + |\langle T,  
                                              \pi^{(\partial_t)}
                                              \rangle| \, dx dt\\
  &\lesssim (1+|t_+ - t_-|) \| \nabla_{t,x} v \|_{L^\infty L^2}^2 + \|
  \Box v\|_{L^1L^2}^2
  \end{split}
\end{equation}
where $\pi^{(Z)} (X, Y)= \langle \nabla_{X} Z , Y
\rangle + \langle \nabla_Y Z, X \rangle$ is the
deformation tensor of a vector field $Z$. 

Suppose further that $v$ is microlocalized such that $Lv$ and $Ev$ are
smaller than a generic derivative $\nabla v$. Then by
contracting the stress energy tensor instead with $L$, one deduces
\begin{align}
  \label{e:CE2}
  \int_{\Lambda} |L v|^2 \, d\sigma \lesssim \int_{\Sigma_\pm}| Lv(t_\pm,
  x)|^2 + |E (t_{\pm}, x)|^2 \, dx + \int_{\Omega} |\Box v| |Lv| +
  |\langle T, \pi^{(L)} \rangle| \, dxdt.
\end{align}
This yields a tighter estimate for $Lv$ along $\Lambda$ since the
worst component $T^{LL} = T_{\underline{L}\underline{L}} =
|\underline{L}v|^2$ is paired with $\pi^{(L)}_{LL} = 0$.

Now factor the symbol
\[\tau^2 - 2g_{<\sqrt{\lambda}}^{0j}\tau \xi_j -
  g_{\sqrt{\lambda}}^{jk} \xi_a \xi_b = (\tau + a^+)(\tau + a^-).\] In
the sequel we redenote $a := a^+$ and let $A = a(t, X, D)$ denote
the corresponding half-wave operator.

For each direction $\theta$, introduce the associated $+$ null foliation
$\Lambda_{\theta} = \Lambda_{\theta}^\lambda$, defined by the
Hamiltonian flow for the half-wave symbol $\tau + a$, and let
$\{L, \underline{L}, E\}$ denote the associated null frame. 
As before, we parametrize the graph of the flow $(x_0, \xi_0) \mapsto (x_t, 
\xi_t)$ by the variables
$(x_t, \xi_0)$, and write $\xi_{\xi_0}(t, x_t) := \xi_t (x_0,
\xi_0)$. Combining~\eqref{e:flow-deriv} with $\alpha = \sqrt{\lambda}$ and the 
computation~\eqref{e:jacobian_alt_param}, one
deduces
\begin{align*}
  \frac{\partial^k(x_0, \xi_t)}{\partial (x_t, \xi_0)^k } =
  O(\lambda^{\frac{|k|-1}{2}}), \quad |k| \ge 1.
\end{align*}
The operators
$L$, $E$ therefore belong to $OPS^1_{1, \frac{3}{4}}$ when
restricted to input frequencies $\ge \lambda$. To obtain estimates
for $\partial_t \xi_{\theta}$, we use the equation
\begin{align}
  \label{e:xi_eqn}
  \partial_t \xi_{\theta} = -\langle a_{\xi} (t, x, \xi_{\theta}),
  \partial_x \xi_{\theta}\rangle - a_x (t, x, \xi_{\theta}),
\end{align}
obtained by differentiating the definition with respect to $t$, to
deduce
\begin{align}
  \label{e:xi_tderivs}
  |\partial_x^k \partial_t \xi_{\theta}| \lesssim
  \lambda^{\frac{k}{2}} + \lambda^{\frac{k-1}{2}} \lambda^{\frac{1}{4}}, \quad 
  k \ge 1.
\end{align}
where we used the Bernstein-type estimate~\eqref{e:metric_est1} to
bound $a_{xx}$ and higher order $x$ derivatives (one could
alternatively replace the $\lambda^{\frac{1}{4}}$ by
$M(\|\partial^2 g(t)\|_{L^\infty_x})$).

The main result of this section is
\begin{proposition}
  \label{p:char-energy}
  Let $u = P_\lambda(D_x)u$ be supported in $|t| \le 2$ and satisfy
  \[
  \| \nabla_{t, x} u\|_{L^2} + \| \Box_{g_{<\sqrt{\lambda}}} u
  \|_{L^2} < \infty,
  \]
  and let $u = u^+ + u^-$ be the half-wave
  decomposition from Corollary~\ref{c:half-wave}. Suppose
  $\phi_{\theta, \lambda}^\alpha$ is a pseudo-differential
  localization operator of the form~\eqref{e:pd-loc}, defined via
  the $+$ flow of the metric $g_{<\sqrt{\mu}}$ with
  $\mu \le \lambda$. Assume that $\alpha \ge \alpha_\mu := \mu^{-1/2}$,
  and set $\beta := |\theta-\theta'|$.
  \begin{enumerate}
  \item If $L = L_{\theta'}^+$ is the null generator for
    the foliation $\Lambda_{\theta'}$, then
    \begin{align*}
      \sup_h \int_{\Lambda_{h, \theta'}} |L \phi_{\theta, \lambda}^{\alpha}
      u^+|^2 \, d\sigma \lesssim  (\alpha + \beta)^2 \| \chi_\theta^{1} 
      u^+\|_{L^2}^2.
    \end{align*}

  \item If $\beta \ge C \alpha$ for sufficiently large $C>0$, then for any 
  $\varepsilon
    > 0$ one has
    \begin{align*}
      \sup_h \int_{\Lambda_{h, \theta'}^\varepsilon} | \phi_{\theta, 
      \lambda}^{\alpha}
      u^+|^2 \, dx dt \lesssim  (\varepsilon +(\alpha\lambda)^{-2})\beta^{-2}\| 
      \chi_\theta^{0} u^+\|_{L^2}^2,    
    \end{align*}
  \end{enumerate}
  where $\Lambda_{h, \theta'}^\varepsilon$ denotes an
  $\varepsilon$-neighborhood of $\Lambda_{h, \theta'}$, and
  $\chi^j_\theta$ are operators such that
  \begin{align} \label{sq:sum:chi}
    \sum_{\theta \in \Omega_\alpha} \| \chi_\theta^{j} u^+\|_{L^2}^2 \lesssim 
    \lambda^{2(j-1)}\bigl(\| \nabla
    u\|_{L^2}^2 + \| \Box_{g_{<\sqrt{\lambda}}} u\|_{L^2}^2\bigr).
  \end{align}
\end{proposition}

These are variable-coefficient analogues of the estimates on null hyperplanes 
proved in 
\cite[Section 3]{Tat}, and shall play an essential role in 
the proof of the algebra property. The
small angle case $\beta \sim \alpha$ arises when studying low-modulation 
outputs,
while the transversal case $\beta \sim 1$ is used for
high-modulation outputs.

In view of the relation $a^+(t, x, -\xi) = - a^-(t, x, \xi)$, the null 
foliations and generators are related by 
$\Lambda^-_{-\theta} = \Lambda^+_{\theta}$, $L_{-\theta}^- = -L_\theta^+$. 
Indeed the optical functions satisfy 
$\Phi_{-\theta}^- = 
-\Phi_\theta^+$. Consequently one has
\begin{corollary}
  \label{c:char-energy}
Assume the setup of the previous
proposition.
\begin{enumerate}
\item If $L = L_{-\theta'}^-$ is the null generator for
the $-$ foliation $\Lambda^-_{-\theta'}$, then
  \begin{align*}
    \sup_h \int_{\Lambda^-_{h, -\theta'}} |L \phi_{\theta, \lambda}^{\alpha}
    u^+|^2 \, d\sigma \lesssim  (\alpha + \beta)^2 \| \chi_\theta^{1} 
    u^+\|_{L^2}^2.
  \end{align*}

\item If $\beta \ge C \alpha$ for sufficiently large $C>0$, then for any 
$\varepsilon
  > 0$ one has
  \begin{align*}
    \sup_h \int_{(\Lambda_{h, -\theta'}^-)^\varepsilon} | \phi_{\theta, 
    \lambda}^{\alpha}
    u^+|^2 \, dx dt \lesssim  (\varepsilon +(\alpha\lambda)^{-2})\beta^{-2}\| 
    \chi_\theta^{0} u^+\|_{L^2}^2,    
  \end{align*}
\end{enumerate}
\end{corollary}

Later we shall consider bilinear estimates of the form
$ \|(\phi_{\theta,\lambda}^\alpha u) v_T\|_{L^2}$ and
$\| Q(\phi_{\theta, \lambda} u, v_T)\|_{L^2}$, where $v_T$ is a
frequency $\mu$ wave packet concentrating in some tube
$T \in \mathcal{T}_{\mu, \pm \theta'}^\pm$. As each $T$ is contained
in a union
$\bigcup_{|h-h_0| \lesssim \mu^{-1}} \Lambda_{h, \pm \theta'}^\pm$ of
null surfaces, we deduce
\begin{corollary}
  \label{c:char-energy-tubes}
  Assume the setup of the previous corollary, and let
  $T^\pm \in \mathcal{T}_{\pm\theta', \mu}^\pm$ be a
  frequency-$\mu$ ``tube'' with initial direction $\pm \theta'$. 
\begin{enumerate}
\item If $L^\pm = L_{\pm\theta'}^\pm$ is the null generator for
the  foliation $\Lambda^{\pm}_{\pm\theta'}$, then
  \begin{align*}
  \|L^\pm \phi_{\theta,\lambda}^\alpha u^+ \|_{L^2(T^\pm)}
    \lesssim \mu^{-\frac{1}{2}} (\alpha + \beta) \| \chi_{\theta}^1 u^+\|_{L^2}.
  \end{align*}
\item If $\beta \ge C \alpha$ for sufficiently large $C>0$, then 
  \begin{align*}
    \| \phi_{\theta, \lambda}^{\alpha}
    u^+\|_{L^2(T^\pm)} &\lesssim \mu^{-\frac{1}{2}} \beta^{-1} \|
                     \chi_\theta^0 u^+\|_{L^2}.
  \end{align*}
\end{enumerate}
\end{corollary}

\begin{proof}[Proof of Proposition~\ref{p:char-energy}, part (1)]
  We begin by recording the symbol estimates that underpin the gain or
  loss in the angular separation.  By a harmless abuse of
  notation we ignore the prefactor $\sigma$ in $L$ (see
  Lemma~\ref{l:half-fullwave-bichar}) and redenote
  \[L = D_t + \partial_\xi a (t, x, \xi_{\theta'} (t, x)) \cdot
    D_x = D_t + \tilde{A}.\] The difference between $L$ and the
  half-wave operator $D_t + A$ is
  \begin{align*}
    (D_t + A) -L = (A - \tilde{A}),
  \end{align*}
  whose symbol is
  \begin{align}
    \label{e:char-e-symbolest1}
    a - \tilde{a} = a(t,x,\xi) - \langle a_\xi(t, x, \xi_{\theta'}(t, x)), \xi
    \rangle \sim  |\xi| \angle( \wht{\xi},
    \wht{\xi_{\theta'}})^2.
  \end{align}
  Indeed, put $\omega:= \widehat{\xi_{\theta'} (t,x)}$, and let
  $f(\xi) := a(t, x, \xi)$, which is $1$-homogeneous and
  $f_{\xi\xi}(\xi)$ is positive-definite in the orthogonal complement
  of $\xi$. When
  $\angle(\wht{\xi}, -\omega) \ge c \min_{|\omega| = 1} f(\omega)$ for
  a small constant $c >0$, one computes
\begin{align*}
  f(\xi) - f_{\xi}(\omega) \cdot \xi &= |\xi| \langle f_{\xi}
                                       (\widehat{\xi}) - f_{\xi}(\omega), 
                                       \widehat{\xi} \rangle\\
                                     &= \int_0^1 |\xi| \langle f_{\xi\xi}( 
                                     \omega + s(\widehat{\xi} -
                                       \omega)) \, (\widehat{\xi}-\omega), 
                                       \widehat{\xi} - \omega - s(\widehat{\xi} 
                                       -
                                       \omega) \rangle \, ds\\
                                     &= |\xi| \int_0^1 (1-s)\langle f_{\xi\xi}( 
                                     \omega +
                                       s(\widehat{\xi}-\omega)) \, 
                                       (\widehat{\xi} -\omega), 
                                       \widehat{\xi}-\omega \rangle \,
                                       ds\\
                                     &\sim |\xi| \angle (\widehat{\xi}, 
                                     \omega)^2
\end{align*}
In the remaining region, where $\wht{\xi}$ and $\omega$ are nearly
antipodal,
\begin{align*}
  f(\xi) - f_\xi(\omega) \cdot \xi &= f(-\xi) - f_\xi(\omega) \cdot
                                     (-\xi) - 2 f_\xi(\omega) \cdot \xi\\
                                   &= 2|\xi| ( f(\omega) - O( \angle 
                                   (\wht{\xi}, -\omega) ) )\\
  &\gtrsim |\xi|.
\end{align*}

Further, in view of the identity~\eqref{e:E-symbol}, the symbol of $E$ is
  $1$-homogeneous and satisfies
  \begin{align}
    \label{e:char-e-symbolest2}
    e(t, x, \xi) = |\xi|\langle e(t, x), \wht{\xi} \rangle = |\xi|
    \langle e(t, x), \wht{\xi} - \wht{\xi_{\theta'}} \rangle.
  \end{align}

  Recall that the symbol $\phi_{\theta, \lambda}^{\alpha, +}$ is
  essentially supported in the sector
  $\{ \xi : \angle( \xi, \xi^\mu_{\theta}(t, x) ) \lesssim \alpha\}$,
  where $\xi_\theta^{\mu}$ is defined by the metric $g_{<\sqrt{\mu}}$
  mollified at frequency $\mu$.  Precisely, we can decompose
\begin{align*}
  \phi_{\theta, \lambda}^{\alpha, +}  = \chi_{\theta}^{\alpha} + r_\theta,
\end{align*}
where $\chi_{\theta}^{\alpha}$ is essentially
$\phi_\theta^{\alpha, +}(t, x, \xi) s_\lambda(\xi)$ before
mollification in the $x$ variable--see~\eqref{e:pd-loc}--and has the
required support, while
$\|r_\theta\|_{L^2\to L^2} = O(\lambda^{-\infty})$.

  Thus since
  $|\xi_{\theta}(t,x) - \xi_{\theta'}(t,x)| \sim |\theta -\theta'|$
  and $|\wht{\xi_\theta} - \wht{\xi_\theta^\mu}| \lesssim \mu^{-1/2}
  \le \alpha$, one has
  \begin{align}
    \label{e:char-e-symbolest3}
     |(a - \tilde{a}) \phi_{\theta, \lambda}^\alpha| \lesssim
    (\alpha+|\theta -\theta'|)^2\lambda, \quad |e \phi_{\theta,
    \lambda}^\alpha|  \lesssim
    (\alpha+ |\theta-\theta'|)\lambda .
  \end{align}

  On the other hand, if  $|\theta-\theta'| \ge C \alpha$ for some
  large $C$ such that $C\alpha$ dominates the angular width of the cutoff 
  $\chi_{\theta}^\alpha$, the
  symbol $a - \tilde{a}$ is then microelliptic:
\begin{align}
    \label{e:char-e-symbolest4}
    |\wht{\xi} - \wht{\xi_\theta^\mu}| \lesssim \alpha \Rightarrow |(a - 
    \tilde{a})| \sim
    (\alpha+|\theta -\theta'|)^2\lambda, \quad |e|  \lesssim
    (\alpha+ |\theta-\theta'|)\lambda .
  \end{align}

  Without loss of generality we prove the estimate on the surface
  $\Lambda_{0, \theta'}$.   In the sequel we write $\Box :=
  \Box_{g_{<\sqrt{\lambda}}}$.

We apply the estimate~\eqref{e:CE2} to the spacetime region
\begin{align*}
  \Omega =
  \bigcup_{h \le 0} \Lambda_{h, \theta'} \cap \{|t| \le 5\},
\end{align*}
whose boundary is
\begin{align*}
  \partial \Omega = (\Lambda_{0, \theta'} \cap \{ |t| \le 5\}) \cup \{t = \pm 
  5\}.
\end{align*}
In terms of the null frame, we have
\begin{align*}
  \int_{\Lambda_{0, \theta'}} |L v|^2 \, d\sigma \lesssim \sum_{\pm} \int 
  |Lv(\pm 5, x)|^2 +
  |E v(\pm 5, x)|^2 \,dx + \int_{\Omega} |\Box v Lv| \, dx dt +
  \int_{\Omega} | \langle T,  \pi \rangle| \, dx dt,
\end{align*}
where $\pi (X, Y) = \langle \nabla_X L, Y \rangle + \langle \nabla_Y
L, X \rangle$ is the deformation tensor for $L$. 

Put $v = \phi_{\theta,\lambda}^{\alpha}u^+$. The boundary terms vanish
since the half-wave $u^+$ is assumed to be supported in $|t| \le 3$.

For the other terms, write
\begin{align*}
  \langle T, \pi \rangle &= T_{\underline{L}\underline{L}} \pi_{LL} +
  T_{LL} \pi_{\underline{L} \underline{L}} + 2 T_{L\underline{L}}
  \pi_{\underline{L} L} + 2 T_{\underline{L} E} \pi_{LE} + 2 T_{LE}
                           \pi_{\underline{L}E} + T_{EE} \pi_{EE}\\
\end{align*}
The components of the deformation tensor are
  \begin{align*}
    \pi_{LL} &= 2 \langle \nabla_{L} L, L \rangle = 0,\\
    \pi_{LE} &=  \langle \nabla_{L} L, E \rangle + \langle \nabla_E L,
               L \rangle = 0,\\
    \pi_{L \underline{L}} &= \langle \nabla_{L} L, \underline{L}
                            \rangle + \langle \nabla_{\underline{L}} L, L 
                            \rangle = 0,\\
    \pi_{\underline{L} E} &= O(1),\\
    \pi_{\underline{L} \underline{L}} &= O(1),\\
    \pi_{EE} &= O(1);
  \end{align*}
  the last three are a consequence of the derivative 
  estimates~\eqref{l:optical_bounds} for the
  optical function. Thus
\begin{align*}
  |\langle T, \pi\rangle| \lesssim |Lv L
    v|  + |Lv E
    v| + |Ev E
    v |.
\end{align*}
Altogether we obtain
\begin{align*}
  \int_{\Lambda_{\theta, 0}} |L\phi_{\theta, \lambda}^{\alpha}
  u^+|^2 \, d\sigma \lesssim \| \Box \phi_{\theta, \lambda}^{\alpha}
  u^+\|_{L^2}  \| L \phi^\alpha_{\theta, \lambda} u^+\|_{L^2} + \| L 
  \phi^\alpha_{\theta, \lambda} u^+\|_{L^2}^2 + \| E 
  \phi^\alpha_{\theta,\lambda} u^+\|_{L^2}^2.
\end{align*}
The claim now follows from the next lemma.
\end{proof}

\begin{lemma}
  \label{e:char-e-commutator1}
If $\{L, \underline{L}, E\}$ is the null frame for the foliation
$\Lambda_{\theta'}$, then
\begin{align*}
  &L\phi_{\theta, \lambda}^\alpha = (D_t + A) \phi_{\theta,
    \lambda}^\alpha + (\alpha+|\theta-\theta'|)^2 \lambda \chi_\theta,\\
  &E \phi_{\theta, \lambda}^\alpha = (\alpha +|\theta-\theta'|)\lambda  
  \chi_\theta,
\end{align*}
where $\chi_\theta$ are $L^2_x$-square-summable with constants uniform in time. 
Also
\begin{align*}
  \sum_{\theta \in \Omega_\alpha} \| \Box \phi^\alpha_{\theta, \lambda} u^+ 
  \|_{L^2}^2
  \lesssim  \| \nabla_{t,x}u\|_{L^2}^2 + \| \Box u\|_{L^2}^2.
\end{align*}
\end{lemma}

\begin{proof}
  
  The second term on the right side of $L\phi_{\theta,\lambda}^\alpha$
  is
  \begin{align*}
    (\tilde{A} - A) \phi_{\theta,\lambda}^\alpha = P_\lambda
    (\tilde{A} -A) \phi_{\theta, \lambda} + (\tilde{a}_{>\lambda/16}-
    a_{>\lambda/16})(t, X, D) \phi_{\theta, \lambda}.
  \end{align*}

  Since the first term is localized in output frequency, the symbol
  estimates~\eqref{e:char-e-symbolest3} and pseudo-differential
  calculus (Lemmas~\ref{l:gauss_transform} and \ref{l:CV}) imply
  that the first term is $\alpha^2\lambda \chi_\theta$ for some
  square-summable operator $\chi_\theta$.  On the other hand,
  $a, \tilde{a} \in S^1_{1, 3/4}$ in the region $|\xi| \ge \lambda$,
  one has
  $a_{>\lambda/16}, \tilde{a}_{>\lambda/16} \in \bigcap_N \lambda^{-N}
  S^1_{1, 3/4}$, so
  \begin{align*}
    \| (\tilde{a}_{>\lambda/16} - a_{>\lambda/16})(t, X, D)
    P_\lambda(D)\|_{L^2 \to L^2} \lesssim \lambda^{-N} \text{ for any
    } N.
  \end{align*}

  The estimate for $E\phi_{\theta, \lambda}^\alpha$ is similar,
  writing
  $E \phi_{\theta, \lambda}^\alpha = P_\lambda(D) E \phi_{\theta,
    \lambda}^\alpha + R \phi_{\theta, \lambda}^\alpha$, where
  $\|R\|_{L^2 \to L^2} = O(\lambda^{-N})$ for any $N$, and
  using~\eqref{e:char-e-symbolest3} for the main term.

  As $\phi_{\theta,\lambda}^{\alpha}$ are square-summable in
  $\theta$, it suffices to prove that
  \begin{align*}
    \sum_\theta \| [\Box, \phi_{\theta, \lambda}^{\alpha}
    u^+\|_{L^2}^2 \lesssim \| \nabla_{t, x} u\|_{L^2}^2 + \| \Box u\|_{L^2}^2.
  \end{align*}

  To this end we note first of all that
  \begin{align*}
    \sum_\theta \| [\Box - (D_t + A^-)(D_t + A^+) ] \phi_{\theta,
    \lambda}^\alpha u^+\|_{L^2}^2 \lesssim \sum_{\theta} \| \nabla_{t,x}
    \phi_{\theta, \lambda}^\alpha u^+\|_{L^2}^2 \lesssim \|\nabla_{t,x} 
    u\|_{L^2}^2,
  \end{align*}
  and write
  \begin{align*}
    &[(D_t + A^-) (D_t + A^+), \phi_{\theta, \lambda}^\alpha ] =
    (D_t+A^-) [D_t + A^+, \phi_{\theta, \lambda}^\alpha]  + [D_t + A^-, 
    \phi_{\theta,
      \lambda}^\alpha] (D_t+A^+).
  \end{align*}
  For the first term we split $A^+ = A_\mu +
  A^+ - A_\mu$, where $A_\mu$ is the corresponding $+$ operator for the
  low-frequency metric $g_{<\sqrt{\mu}}$. Recall from the proof of 
  Proposition~\ref{p:orthogonality}
  that
  \[[D_t + A_\mu, \phi_{\theta, \lambda}^\alpha] = \frac{1}{i} \{\tau
    + a_\mu, \phi_{\theta, \lambda}^\alpha \}(t, X, D) + r(t, X,
    D), \] where
  $\{\tau+a_\mu, \phi_{\theta, \lambda}^\alpha\} \in S(
  m_\theta^\infty, g_{\alpha_\mu})$,
  $r \in (\alpha_\mu^2 \lambda)^{-1} L^2 S ( m_\theta^\infty,
  g_{\alpha_\mu})$,
  $P_{>\lambda/8}(D_x) r \in \lambda^{-N} L^2 S (m_\theta^\infty)$ for
  any $N$, and
  $\partial_t r \in \mu S( m_\theta^\infty,
  g_{\alpha_\mu})$.  The last claim uses the modified computations
    \begin{alignat*}{4}
    &\partial^2_\xi \partial_t a \in S( |\xi|^{-1}, g_{\alpha_\mu}), &\quad
    &\partial_x^2 \partial_t \phi_{\theta, \lambda}^\alpha \in
    \alpha_\mu^{-4} S( m_\theta^\infty , g_{\alpha_\mu}),\\
    &\partial^2_x \partial_t a \in  S( \mu^{\frac{3}{4}}|\xi|, g_{\alpha_\mu}), 
    &\quad
    &\partial_\xi^2 \partial_t \phi_{\theta, \lambda}^\alpha \in
    \alpha^{-1} (\alpha_\mu \lambda)^{-2} S(m_\theta^\infty , g_{\alpha_\mu}).
  \end{alignat*}

  Then
\begin{align*}
  (D_t + A^-)\{ \tau + a_\mu, \phi_{\theta, \lambda}^\alpha\} (t, X,
  D) &= (D_t \{ \tau + a_\mu, \phi_{\theta, \lambda}^\alpha\}) (t, X,
       D) + \{\tau + a_\mu, \phi_{\theta, \lambda}^\alpha\}(t, X, D)
       D_t \\
  &+ A^- \{ \tau + a_\mu, \phi_{\theta, \lambda}^\alpha \} (t,
    X, D).
\end{align*}
By Lemma~\ref{l:t-derivs} the first term belongs to
$\alpha^{-1} OPS( m_\theta^\infty, g_{\alpha_\mu})$. Modulo a
negligible remainder we may restrict each term to output frequency
$\lambda$, and write
\begin{align*}
  (D_t+A^-) \{\tau + a_\mu, \phi_{\theta, \lambda}^\alpha\} (t, X, D)
  = \lambda T^1_\theta + T^2_\theta D_t
\end{align*}
where $T^1_\theta$ and $T^2_\theta$ are square-summable.

\begin{align*}
  (D_t+A^-) r(t, X, D) &= (D_tr)(t, X, D) + r(t,X,D)D_t + A^- r(t, X,
                         D)\\
  &= \mu T^1_\theta + f(t) T^2_\theta D_t + \lambda T^3_\theta,
\end{align*}
where $f(t) = M(\|\partial^2 g\|_{L^\infty}) \in L^2_t$ and
$T^j_\theta$ are square-summable. This is acceptable in view of the
energy estimate $\| \nabla u\|_{L^\infty L^2} \lesssim \| \nabla
u\|_{L^2} + \| \Box u\|_{L^2}$.

For the term $(D_t+A^-)[A -A_\mu, \phi_{\theta, \lambda}^{\alpha}]$,
we simply recall from the proof of Proposition~\ref{p:orthogonality}
that the commutator belongs to $\alpha_\mu^2\alpha^{-1}OPS(m_\theta^\infty,
g_{\alpha_\mu})$ and outputs essentially at frequency $\lambda$, so
that $A^{-} [A-A_\mu, \phi_{\theta, \lambda}^\alpha] \in \lambda
\chi_\theta$ for some square-summable $\chi_\theta$, and 
\begin{align*}
D_t [A-A_\mu, \phi_{\theta, \lambda}^\alpha ] = [\partial_t A -
  \partial_t A_\mu, \phi_{\theta, \lambda}^\alpha] +[A-A_\mu,
  D_t \phi_{\theta, \lambda}^\alpha] + [A-A_\mu, \phi_{\theta, 
  \lambda}^\alpha]D_t,
\end{align*}
is acceptable as well. This shows that
\begin{align*}
  \sum_{\theta} \| (D_t+A^-) [D_t + A^+, \phi_{\theta,
  \lambda}^\alpha] u^+\|_{L^2}^2 \lesssim \| \nabla u^+\|_{L^2}^2 + \|
  \Box u^+\|_{L^2}^2.
\end{align*}

Next we write
\begin{align*}
  [D_t + A^-, \phi_{\theta, \lambda}^\alpha] = (D_t\phi_{\theta,
  \lambda}^\alpha)(t, X, D) + [A^-, \phi_{\theta, \lambda}^\alpha] =
  \alpha^{-1} \chi_\theta
\end{align*}
for some square summable $\chi_\theta$, so that
\begin{align*}
\sum_\theta   \| [D_t + A^-, \phi_{\theta, \lambda}^\alpha] (D_t+A^+)
  u^+\|_{L^2}^2 \lesssim \alpha^{-2} \| (D_t+A^+) u^+\|_{L^2}^2 \lesssim
  \| \nabla u^+\|_{L^2}^2 + \| \Box u^+\|_{L^2}^2,
\end{align*}
Finally, recall from Proposition~\ref{p:half-wave} that $\|\nabla u^+
\|_{L^2} + \| \Box u^+\|_{L^2} \lesssim \| \nabla u\|_{L^2} + \| \Box 
u\|_{L^2}.$
\end{proof}

For future reference, we collect the key estimates for $\phi_{\theta, 
\lambda}^\alpha u$ in the following
\begin{corollary}
If $\{L, \underline{L}, E\}$ is the null frame for the $+$ foliation
$\Lambda_{\theta'}$, and $|\theta-\theta'| \sim \alpha \ge \alpha_\mu$, then:
	\begin{align}
	\label{O:est:phi}
	&\bigl(\sum_\theta \| \phi_{\theta, \lambda}^\alpha  u \|_{L^\infty L^2}^2 
	\bigr)^{\frac{1}{2}}
	\lesssim \| u\|_{X_+}\\
	\label{Box:est:phi}
	& \bigl(\sum_{\theta} \| \Box_{g_{<\sqrt{\lambda}}} \phi_{\theta, 
	\lambda}^\alpha 
	u\|_{L^2 L^2}^2\bigr)^{\frac{1}{2}} \lesssim \| \nabla u\|_{L^2} + \| 
	\Box_{g_{<\sqrt{\lambda}}} u\|_{L^2} \\
	\label{E:est:phi}
	& \bigl(\sum_\theta \| E \phi_{\theta, \lambda}^\alpha 
	u\|_{L^2_x}^2\bigr)^{\frac{1}{2}} \lesssim 
	\lambda \alpha \| u\|_{X_+}\\
		\label{L:est:phi} 
	&  \bigl(\sum_{\theta} \| L \phi_{\theta, \lambda}^{\alpha} 
	u\|^2_{L^2 L^2} \bigr)^{\frac{1}{2}}
	\lesssim \lambda \alpha^2 \| u\|_{X_+}\\
	\label{barL:est:phi}
	& \bigl(\sum_\theta \| \nabla_{t,x} \phi_{\theta, \lambda}^\alpha 
	u\|_{L^\infty 
	L^2}^2 \bigr)^{\frac{1}{2}}\lesssim \lambda 
	\|  u\|_{X_+}
	\end{align}
	\begin{remark}
		For split metrics $g^{0j} = 0$, the same estimates hold with 
		the replacements $\Lambda_{\theta'} \to \Lambda_{-\theta'}^{-}$, $X_+ 
		\to 
		X_-$, and $\alpha \sim |\theta + \theta'|$.
	\end{remark}
\end{corollary}

\begin{proof}
	\eqref{O:est:phi} follows from the energy estimate on bounded time 
	intervals 
	$\| v\|_{L^\infty L^2} \lesssim \| v\|_{L^2} + \| (D_t+A^\pm) v\|_{L^2}$. 
	
	For the next estimate~\eqref{barL:est:phi} we simply note that by 
	Lemma~\eqref{l:t-derivs}, the commutator $[\nabla_{t, x}, \, \lambda^{-1}
	\phi_{\theta, \lambda}^\alpha ] = \lambda^{-1} (\nabla_{t,x} \phi_{\theta, 
	\lambda}^\alpha)$ is square-summable, and apply~\eqref{O:est:phi}.

	Finally, \eqref{Box:est:phi}, \eqref{E:est:phi}, \eqref{L:est:phi} were 
	proved in the preceding lemma.
\end{proof}

We turn to the $L^2$ estimate in Proposition~\ref{p:char-energy},
which is slightly more involved. Roughly speaking, the angular
separation allows one to microlocally invert the vector field $L$ for
the foliation $\Lambda_{\theta'}$ on the support of the cutoff
$\phi_{\theta, \lambda}^\alpha$.  A similar ellipticity argument was
employed previously in~\cite[Lemma 5.2]{geba2008gradient}.

\begin{proof}[Proof of Prop.~\ref{p:char-energy}, part (2)]
  Recall from ~\eqref{e:char-e-symbolest1},
  \eqref{e:char-e-symbolest4}, that the difference
  $D_t + A - L= A - \tilde{A}$ satisfies
\begin{align}
  \label{e:microelliptic}
q := a - \tilde{a} \approx \beta^2 \lambda
\end{align}
on a neighborhood of the support of $\chi_{\theta}^{\alpha}$. 
We construct a microlocal inverse for $A - \tilde{A}$. Let
$\tilde{\chi}_{\theta}^{\alpha}$ be a slightly wider version of
$\chi_{\theta}^{\alpha}$ supported where~\eqref{e:microelliptic}
holds. Define
\begin{align*}
  \tilde{l}(t, x, \xi) &:= q^{-1} \tilde{\chi}_{\theta}^{\alpha} ,
\end{align*}
and define
\begin{align*}
  \tilde{L} := P_\lambda(D) \tilde{l}(t, x, D).
\end{align*}

\begin{lemma}
  \label{l:microelliptic-regularity}
  On the support of $\tilde{\chi}_{\theta}^\alpha$, one has
  \begin{gather*}
    |(\beta \lambda\partial_\xi)^m \langle \xi, \partial_\xi\rangle^{m'} q| + 
    |(\beta_\lambda
    \partial_x)^k (\beta \partial_x)(\beta \lambda \partial_\xi)^m \langle \xi, 
    \partial_\xi\rangle^{m'} q| \lesssim
    (\beta^2\lambda), \quad
    \alpha_\lambda := \lambda^{-\frac{1}{2}},\\
    |(\alpha_\lambda\partial_x)^k (\beta \lambda \partial_\xi)^m \langle \xi, 
    \partial_\xi\rangle^{m'}
    \partial_t q| \lesssim \beta^{-1} (\beta^2\lambda).
    \end{gather*}
  \end{lemma}
  In conjunction with
  $\tilde{\chi}_{\theta}^\alpha \in S^1_\alpha( m_\theta^\infty,
  g_{\alpha_\mu})$, this quickly leads to
  \begin{corollary}
    \label{c:parametrix_symbol}
    The symbol $\tilde{l}$ satisfies
    \begin{align*}
      \tilde{l}, \ \beta \partial_x \tilde{l}, \ \alpha \lambda
      \partial_\xi \tilde{l} \in  (\beta^2 \lambda)^{-1}S(  m_\theta^\infty,
      g_{\alpha_\lambda}), 
      \end{align*} where as before
$  m_\theta(t, x, \xi) := \langle \alpha^{-1} ( |\widehat{\xi} - \widehat{
  \xi_\theta^{\mu}} (t, x) ) \rangle^{-1}$.
\end{corollary}

\begin{remark}
  The worse regularity of $\tilde{l}$ with respect to $\xi$ is due to the factor
  $\tilde{\chi}_\theta^\alpha$.
\end{remark}
\begin{proof}[Proof of Lemma]
  For simplicity of notation we suppress the $t, x$ variables in the
  arguments to $a$.
  We have
\begin{align*}
  \partial_x^k q (\xi) &= \partial_x^k a(\xi) - \langle \partial_x^k
                        a_{\xi} (\xi_{\theta'}), \xi \rangle - \sum_{j \ge 1} 
                        \langle \partial_x^{k-j}
    a_{\xi\xi}(\xi_{\theta'})(\partial_x^j \xi_{\theta'}), \xi - |\xi| 
    \widehat{\xi_{\theta'}}
    \rangle\\
  &- \sum_{j\ge 1} \bigl[\partial_x^{k-j}  \partial_\xi^2
    a_{\xi}(\xi_{\theta'}) B_2(\partial_x \xi_{\theta'}) + \cdots +
    \partial_x^{k-j}\partial_\xi^j a_{\xi} (\xi_{\theta'}) B_j(\partial_x 
    \xi_{\theta'})\bigr],
\end{align*}
where $B_j(\partial_x \xi_{\theta'})$ denotes a $j$-linear quantity in 
$\partial_x \xi_{\theta'}$ and
its higher $x$ derivatives such that the total order of the
derivatives equals $j$. This yields
\begin{align*}
  |\partial_x q| &\lesssim  \beta^2 \lambda + \beta \lambda,
\end{align*}
and when $k \ge 2$
\begin{align*}
  |\partial_x^k q| &\lesssim \lambda^{\frac{k-2}{2}} f(t) \beta^2\lambda| + 
  \sum_{j =1}^{k-2}
                     \lambda^{\frac{k-j-2}{2}} f(t) (
                     \lambda^{\frac{j-1}{2}} \beta \lambda +
                     \lambda^{\max(\frac{j-2}{2},0)} \lambda)\\
  &+ \lambda^{\frac{k-1}{2}} \beta \lambda + \lambda^{\max(\frac{k-2}{2}, 0)} 
  \lambda,
\end{align*}
where $f := M(\|\partial^2 g\|_{L^\infty_x}) \in L^2_t$. Since the
metric is localized to frequencies $<\sqrt{\lambda}$, we may replace
$f(t)$ by the uniform bound $\lambda^{\frac{1}{4}}$ as 
in~\eqref{e:metric_est1}, so that the
dominant term when
$\lambda \ge \alpha^{-2}$ is
$\alpha_\lambda^{1-k} \beta \lambda = \alpha_\lambda^{1-k} \beta^{-1}
(\beta^2 \lambda)$.

The estimates for the $\xi$ derivatives follow easily from the explicit
form of $q$, and similar considerations handle
\begin{align*}
  \partial_t q = a_t(\xi) - \langle a_{t\xi}(\xi_{\theta'}), \xi
  \rangle - \langle a_{\xi\xi}(\xi_{\theta'}) \partial_t
  \xi_{\theta'}, \xi \rangle.
\end{align*}
\end{proof}

Some basic properties of the parametrix $\tilde{L}$ are recorded in
\begin{lemma}
  \label{l:parametrix_bounds}
  The operator $\tilde{L}$ satisfies
  \begin{gather*}
    \| \tilde{L}\|_{L^2 \to L^2} \lesssim (\beta^2 \lambda)^{-1}\\
    \|P_\lambda(D) [D_t + A, \tilde{L}]P_\lambda(D) \|_{L^2 \to L^2} \lesssim
    (\beta^2 \lambda)^{-1}\\
   (A -
    \tilde{A} ) \tilde{L} -
    \tilde{\chi}_{\theta}^{\alpha, +}(t, X, D) =
    (\alpha \beta \lambda)^{-1} \Phi_\theta,
  \end{gather*}
  where the operators $P_\lambda(D)\Phi_\theta: L^2 \to L^2$ are 
  square-summable in $\theta$.
\end{lemma}

\begin{proof}
  The first follows directly from the estimates for the symbol
  $\tilde{l}$ and Lemma~\ref{l:CV}.

  For the third estimate, use the first-order symbol
  expansion~\eqref{e:symbol_expansion} and
  Lemmas~\ref{l:microelliptic-regularity}, \ref{l:gauss_transform} to
  see that
  \begin{align*}
    (A - \tilde{A}) \tilde{L} = \tilde{\chi}_{\theta}^{\alpha, +} (t,
    X, D) + r(t, X, D), \quad r \in   (\alpha \beta
    \lambda)^{-1}S( m_\theta^\infty, g_{\alpha_\mu}),
  \end{align*}
  and apply Lemma~\ref{l:CV} to the remainder.

  For the commutator estimate we essentially follow the proof of
  \cite[Lemma 5.2]{geba2008gradient}. From the second order symbol
  expansion~\eqref{e:symbol_expansion}, the symbol of the commutator
  is
  \begin{gather}
    \frac{1}{i} \{ \tau + a, \tilde{l} \} -\frac{1}{2} \int_0^1 r_s(t, x, \xi)
    \, ds, \nonumber\\
    r_s(t, x, \xi) = \sum_{j, k} e^{is
    \langle D_y, D_\eta \rangle} [ \partial_{\eta_j}\partial_{\eta_k} a(x, 
    \eta)\partial_{y_j}\partial_{y_k}l(y, \xi) - 
    \partial_{\eta_j}\partial_{\eta_k}l(x, \eta)
   \partial_{y_j} \partial_{y_k} a(y, \xi)]
   |_{\stackrel{y=x}{\eta=\xi}} \label{e:order2_commutator}.
  \end{gather}
We claim that \[r_s \in L^2 S( (\beta^2\lambda)^{-1} (\alpha^2_\mu
\lambda)^{-1}, g_{\alpha_\lambda}) + S( ((\alpha_\mu^2 \lambda)^{-1} +
(\beta^2\lambda)^{-\frac{1}{2}})(\beta^2\lambda)^{-1},
g_{\alpha_\lambda}),\]
and is therefore acceptable. This would follow from
Lemma~\ref{l:gauss_transform} and the symbol estimates
\begin{equation}
  \label{e:parametrix_2nd-derivs}
\begin{aligned}
&\partial_x^2 \tilde{l}  \in  S((\alpha_\mu^{-2} + 
\alpha_\lambda^{-1}\beta^{-1})
                 (\beta^2\lambda)^{-1},
                 g_{\alpha_\lambda}), &\quad &\partial_{\xi}^2 a \in 
                 S(|\xi|^{-1}, g_{\alpha_\lambda})\\
  &\partial_\xi^2 \tilde{l} \in S( (\alpha\lambda)^{-1}(\beta^2 \lambda)^{-1} 
  (\alpha_\mu
    \lambda)^{-1}, g_{\alpha_\lambda}), &\quad   &\partial_{x}^2 a  \in L^2 
    S(|\xi|, g_{\alpha_\lambda}).
  \end{aligned}
  \end{equation}

The bounds for $a$ follow directly from the hypotheses on the
metric, so we consider next \[\partial_x^2 \tilde{l} = \partial_x^2
\tilde{\chi}_\theta^\alpha q^{-1} + 2 \partial_x
\tilde{\chi}_{\theta}^\alpha \partial_x q^{-1} +
\tilde{\chi}_{\theta}^\alpha \partial_x^2 q^{-1}.\]

From Lemma~\ref{l:microelliptic-regularity} one sees that on the
support of $\tilde{\chi}_{\theta}^\alpha$,
\begin{align*}
  \partial_x (q^{-1}) \in \beta^{-1} S((\beta^2\lambda)^{-1},
  g_{\alpha_\lambda}), \quad \partial_x^2 (q^{-1}) =  \beta^{-1}
  \alpha_\lambda^{-1} S( (\beta^2\lambda)^{-1}, g_{\alpha_\lambda}),
\end{align*}
thus
\begin{align*}
  \partial^2_x \tilde{l} \in S( (\alpha_\mu^{-2} +
  \alpha_\lambda^{-1} \beta^{-1}) (\beta^2\lambda)^{-1},
  g_{\alpha_\lambda}) 
\end{align*}
Similarly
\begin{align*}
  \partial_{\xi}^2 \tilde{l} = \partial_{\xi}^2
  \tilde{\chi}_{\theta}^\alpha q^{-1} + 2 \partial_\xi
  \tilde{\chi}_{\theta}^\alpha \partial_\xi q^{-1} +
                                  \tilde{\chi}_{\theta}^\alpha
                               \partial_\xi^2q^{-1}\in S( 
                               (\alpha\lambda)^{-1}(\alpha_\mu \lambda)^{-1}
    (\beta^2\lambda)^{-1}, g_{\alpha_\lambda})
\end{align*}
as claimed.

It remains to estimate the Poisson bracket. Letting $a_\mu$ denote the
half-wave symbol corresponding to the low-frequency metric
$g_{<\sqrt{\mu}}$, we write
\begin{align}
  \label{e:parametrix_halfwave_poisson}
  \{\tau + a, \tilde{l} \} = \tilde{\chi}_{\theta}^\alpha q^{-2} \{ \tau +
  a, q\}  + q^{-1} \{ \tau + a_\mu,
  \tilde{\chi}_\theta^\alpha\} + q^{-1}\{ a - a_\mu, 
  \tilde{\chi}_{\theta}^\alpha\}.
\end{align}

The symbol $q$ vanishes to second order on the submanifold
\begin{align*}
  \{(t, x, \xi_\theta(t, x) ) \} \subset \R_t \times (T^*\R^2_x
  \setminus 0),
\end{align*}
which is invariant under the Hamiltonian flow. Hence its derivative
along the flow $\{\tau + a,
q\}$ vanishes to second order as well, and is smooth and
$1$-homogeneous $\xi$. Therefore we see that \[\{\tau + a, q\} \in S( (\beta^2
|\xi|), g_{\alpha_\lambda}).\]

The argument of Lemma~\ref{l:poisson-bracket} shows that $\{\tau
+ a_\mu, \tilde{\chi}_{\theta}^\alpha\} \in S(1, g_{\alpha_\mu})$.

Finally, by a direct computation $\{a-a_\mu,
\tilde{\chi}_\theta^\alpha\} \in S( (\alpha_\mu^2\lambda) (\alpha
\lambda)^{-1}, g_{\alpha_\mu})$. Altogether we obtain
\begin{align*}
  \{\tau + a, \tilde{l} \} \in S( (\beta^2 \lambda)^{-1}, g_{\alpha_\lambda}).
\end{align*}
\end{proof}

Later it will be useful to have a more computational proof of the Poisson
bracket bound
$\{\tau + a, q\} \in S( \alpha^2\lambda, g_{\alpha_\lambda})$. Using
the equation~\eqref{e:xi_eqn} for $\partial_t \xi_{\theta'}$, we
compute
\begin{align*}
  \{\tau + a, q\} &= \partial_t q + \langle a_\xi, \partial_x q\rangle
                    - \langle a_x, \partial_\xi q \rangle\\
  &= a_t(\xi) - \langle a_{t\xi} (\xi_{\theta'}), \xi\rangle + \bigl
    \langle \langle
    a_{\xi}(\xi), (a_x(\xi) - \langle a_{x \xi}(\xi_{\theta'}), \xi
    \rangle )\bigr\rangle\\
  &+ |\xi| \langle a_{\xi\xi}(\xi_{\theta'}) (\widehat{\xi}-  
  \widehat{\xi_{\theta'}}
    ), \langle a_{\xi}(\xi_{\theta'}) - a_{\xi}(\xi), \partial_x
    \rangle \widehat{\xi_{\theta'}} \rangle\\
  &+ |\xi| \langle a_{\xi}(\xi) - a_{\xi}(\xi_0), a_x
    (\widehat{\xi_{\theta'}}) - a_x (\widehat{\xi}) \rangle \\
  &- |\xi|\langle
    a_{\xi\xi}( a_{\xi}( \widehat{\xi}) - a_{\xi}( \widehat{\xi_{\theta'}}) -
    a_{\xi\xi}(\widehat{\xi_{\theta'}}) (\widehat{\xi} - 
    \widehat{\xi_{\theta'}})
    \rangle,
\end{align*}
and an inductive argument similar to the proof of
Lemma~\ref{l:microelliptic-regularity} yields the estimates on the
support of $\tilde{\chi}_{\theta}^\alpha$
\begin{align}
  \begin{split}
  \label{e:improved_poisson_bracket}
  |(\alpha_\lambda \partial_x)^k (\beta \lambda \partial_\xi)^m \langle \xi, 
  \partial_\xi\rangle^{m'} \{
  \tau + a, q\}| &\lesssim \beta^2 \lambda,\\
|(\alpha_\lambda \partial_x)^k (\beta \lambda \partial_\xi)^m
\langle \xi, \partial_\xi\rangle^{m'}\partial_t \{\tau + a, q\}|
&\lesssim \alpha_\lambda^{-1} (\beta^2\lambda).
\end{split}
\end{align}
Using the expansion~\eqref{e:parametrix_halfwave_poisson} and
Lemma~\ref{l:microelliptic-regularity}, one deduces that
\begin{align}
  \label{e:parametrix_halfwave_poisson_tderiv}
  \partial_t \{\tau + a, \tilde{l} \} \in \alpha_\lambda^{-1} S(
  (\beta^2\lambda)^{-1}, g_{\alpha_\lambda}).
\end{align}

We continue the proof of the proposition. Write
\begin{align*}
  \phi^\alpha_{\theta, \lambda} u^+ &= P_\lambda(D) (D_t+A - L) \tilde{L} 
  \phi^{\alpha}_{\theta,
  \lambda}u^+ + P_\lambda(D)[\tilde{\chi}_{\theta}^\alpha - (A -
  \tilde{A}) \tilde{L}]  \phi^\alpha_{\theta, \lambda}u^+ + P_\lambda(D)(1 -
  \tilde{\chi}^\alpha_{\theta}) \phi_{\theta, \lambda}^\alpha
  u^+\\
  &=P_\lambda(D) (D_t + A - L) \tilde{L} \phi^\alpha_{\theta, \lambda}
    u^+ + R_1 u^+ + R_2 u^+.
\end{align*}
By the previous lemma and the pseudo-differential calculus, the second
and third terms both take the form $ (\lambda^{-1/2} + (\alpha^2
\lambda)^{-1}) \chi_\theta$ where $\chi_\theta$ is
square-summable. Consequently, we estimate
\begin{align*}
  \|(R_j u^+) \|_{L^2( \Lambda_{0, \theta}^\varepsilon)} \lesssim
  (\alpha \beta \lambda)^{-1} \|\chi_\theta
  u^+\|_{L^2}.
\end{align*}

Also,
\begin{align*}
  P_\lambda(D_t+A) \tilde{L} \phi^\alpha_{\theta, \lambda} u^+ =
  \tilde{L}(D_t+A) \phi^\alpha_{\theta, \lambda} u^+ + P_\lambda(D)
  [D_t + A, \tilde{L}] \phi^\alpha_{\theta, \lambda} u^+.
\end{align*}
By the previous lemma, this is bounded in $L^2$ by
\begin{align*}
  &(\beta^2\lambda)^{-1} \| (D_t+A) \phi^\alpha_{\theta, \lambda}
  u^+\|_{L^2} + (\beta^2\lambda)^{-1} \| \phi_{\theta,
    \lambda}^\alpha u^+\|_{L^2}\\
  &\lesssim (\beta^2 \lambda)^{-1} ( \| \phi_{\theta, \lambda}^\alpha
    u^+\|_{L^2} + \| (D_t+A) \phi_{\theta, \lambda}^\alpha u^+\|_{L^2}).
\end{align*}
which is sufficient.

The remaining term is
\begin{align*}
  P_\lambda(D)L \tilde{L} \phi^\alpha_{\theta, \lambda}u^+ = L
  \tilde{L} \phi^\alpha_{\theta, \lambda}u^+ + [P_\lambda, L] \tilde{L}
  \phi^\alpha_{\theta, \lambda}u^+.
\end{align*}
As the commutator $[P_\lambda, L] = [P_\lambda, \tilde{A}]$ is bounded on
$L^2$ at frequency $\lambda$,
\begin{align*}
\| ( [P_\lambda, L]\tilde{L} \phi^\alpha_{\theta, \lambda} u^+)
  \|_{L^2} \lesssim \| \tilde{L}
  \phi^\alpha_{\theta, \lambda} u^+\|_{L^2} \lesssim (\alpha^2
  \lambda)^{-1} \| \phi^\alpha_{\theta, \lambda} u^+\|_{L^2},
\end{align*}
which is acceptable.

The remaining term is
\begin{align*}
  \|L \phi_{\theta, \lambda}^{\alpha}u^+
  \|_{L^2(\Lambda_{0,\theta'}^\varepsilon)}^2 &\lesssim \int_{|h| \le 
  \varepsilon} \| L
  \phi_{\theta, \lambda}^\alpha u^+ \|_{L^2(\Lambda_{h, \theta'})}^2  \, dh.
\end{align*}
For each null surface $\Lambda_{h, \theta'}$, we have
\begin{align*}
  \int_{\Lambda_{h, \theta'}} |L \tilde{L}\phi^{\alpha}_{\theta, \lambda} u^+|^2
  \, d\sigma  \lesssim  \int | \Box \tilde{L}\phi^{\alpha}_{\theta, \lambda} 
  u^+| |L
  \phi^{\alpha}_{\theta, \lambda} u^+| \, dx dt + \int |\langle T, \pi
  \rangle| \, dx dt.
\end{align*}
For the second term, write
\begin{align*}
  \langle T, \pi \rangle &= T_{\underline{L}\underline{L}} \pi_{LL} +
  T_{LL} \pi_{\underline{L} \underline{L}} + 2 T_{L\underline{L}}
  \pi_{\underline{L} L} + 2 T_{\underline{L} E} \pi_{LE} + 2 T_{LE}
                           \pi_{\underline{L}E} + T_{EE} \pi_{EE}\\
  &\lesssim |L\tilde{L}\phi^\alpha_{\theta, \lambda} u^+ L\tilde{L}
    \phi^{\alpha}_{\theta, \lambda} u^+|  + |L\tilde{L}\phi^\alpha_{\theta, 
    \lambda} u^+ E\tilde{L}
    \phi^{\alpha}_{\theta, \lambda} u^+| + |E\tilde{L}\phi^\alpha_{\theta, 
    \lambda} u^+ E\tilde{L}
    \phi^{\alpha}_{\theta, \lambda} u^+ |,
\end{align*}
so
\begin{align*}
 & \| ( L \tilde{L}\phi^\alpha_{\theta,\lambda} u^+ )
   v_T\|_{L^2(\Lambda_{0,\theta'}^\varepsilon)}^2 \\
  &\lesssim
  \varepsilon \Bigl( \|\Box \tilde{L}\phi^\alpha_{\theta, \lambda} u^+\|_{L^2}
  \| L \tilde{L}\phi^\alpha_{\theta, \lambda} u^+\|_{L^2} + \| L \tilde{L}
  \phi^\alpha_{\theta, \lambda} u^+\|_{L^2}^2 + \| E \tilde{L}
                    \phi^\alpha_{\theta,\lambda} u^+\|_{L^2}^2\Bigr)\\
  &\lesssim \varepsilon(\beta^2 \lambda)^{-2}  \Bigl( \| \Box
    \phi_{\theta, \lambda}^\alpha u^+\|_{L^2} \| L \phi_{\theta,
    \lambda}^\alpha u^+\|_{L^2} + \|L \phi_{\theta,\lambda}^\alpha
    u^+\|_{L^2}^2 + \| E\phi_{\theta, \lambda}^\alpha
    u^+\|_{L^2}^2\Bigr)\\
  &+ \varepsilon \Bigl( \| [\Box, \tilde{L}] \phi_{\theta,
    \lambda}^\alpha u^+ \|_{L^2} \bigl( (\beta^2\lambda)^{-1} \|L
    \phi_{\theta, \lambda}^\alpha u^+\|_{L^2} + \| [L, \tilde{L}]
    \phi_{\theta, \lambda}^\alpha u^+\|_{L^2} \bigr)\\
  &+  \bigl( (\beta^2\lambda)^{-1} \| \Box
    \phi_{\theta,\lambda}^\alpha u^+\|_{L^2} + \| [\Box, \tilde{L}
    ]\phi_{\theta,\lambda}^\alpha u^+\|_{L^2}\bigr) \| [L, \tilde{L}]
    \phi_{\theta,\lambda}^\alpha u^+\||_{L^2}\\
  &+  \| [L, \tilde{L}]
    \phi_{\theta,\lambda}^\alpha u^+\||_{L^2}^2 +  \| [E, \tilde{L}]
    \phi_{\theta,\lambda}^\alpha u^+\||_{L^2}^2\Bigr),
\end{align*}
and appeal to Lemmas~\ref{e:char-e-commutator1} and \ref{e:char-e-commutator2}.
\end{proof}

\begin{lemma}
  \label{e:char-e-commutator2}
If $\{L, \underline{L}, E\}$ is the null frame for the foliation
$\Lambda_{\theta'}$, then
\begin{align*}
  &\| [ E, \tilde{L} ] \|_{L^2 \to L^2} \lesssim 
  \alpha^{-1}(\beta^2\lambda)^{-1},\\
  &\| [L, \tilde{L} ]\|_{L^2 \to L^2} \lesssim (\beta^2
    \lambda)^{-1},\\
  &\| [\Box, \tilde{L} ]u \|_{L^2\to L^2} \lesssim (\beta^2\lambda)^{-1}(\| 
  \nabla_{t,x}
    u\|_{L^2} +  \lambda \| (D_t+A) u\|_{L^2}) \text{ if } u =
    P_\lambda(D_x) u.
\end{align*}
\end{lemma}

\begin{proof}
  The estimate for
  $[E, \tilde{L}]$ is simplest, using the first order symbol
  expansion~\eqref{e:symbol_expansion}, the computations
  \begin{alignat*}{4}
    &\partial_x e \in S(|\xi|, g_{\alpha_\mu}), &\quad
  &\partial_{\xi} e \in S(1, g_{\alpha_\mu}), \\
  &\partial_x \tilde{l} \in S(\beta^{-1} (\beta^2 \lambda)^{-1}, 
  g_{\alpha_\mu}) &\quad
  &\partial_\xi \tilde{l} \in S( (\alpha\lambda)^{-1} (\beta^2 \lambda)^{-1},
    g_{\alpha_\mu}),
  \end{alignat*}
  and Lemmas~\ref{l:gauss_transform} and \ref{l:CV}
  (again considering the outputs at frequency $\lambda$ and $\ne
  \lambda$ separately).

  Next, write $[L, \tilde{L}] = [D+A, \tilde{L}] + [\tilde{A} - A,
  \tilde{L}]$. The first term was studied in the previous
  lemma, while the second again follows from first order calculus.

  The commutator estimate for $\Box$ is proven similarly as the lemma
  in the previous section,
  differing mainly in the symbol estimates involved.

It remains to consider $[\Box, \tilde{L}]$. As usual we first replace $\Box$ 
with $(D_t +
  A^-)(D_t + A^+)$ at the cost of an acceptable error, and write
  \begin{align}
    \label{e:commutator_box_parametrix}
    [(D_t+A^-)(D+A^+), \tilde{L}] =(D_t + A^-) [D_t + A^+, \tilde{L}]
    + [D_t + A^-, \tilde{L}] (D_t+A^+).
  \end{align}

  We proceed similarly as in the proof of
  Lemma~\ref{l:parametrix_bounds} but first gather symbol bounds for
  $\partial_t \tilde{l}$. This is a routine computation using
  Lemmas~\ref{l:t-derivs} and \ref{l:microelliptic-regularity} which leads to 
  $\partial_t
  \tilde{l} \in S(\alpha^{-1} (\beta^2\lambda)^{-1},
  g_{\alpha_\lambda})$,
  \begin{equation}
    \label{e:parametrix_2nd-derivs-t}
  \begin{aligned}
    &\partial_x \partial_t \tilde{l} \in S(
    ( \alpha_\mu^{-2} + \alpha^{-1} \alpha_{\lambda}^{-1})
    (\beta^2\lambda)^{-1}, g_{\alpha_\lambda}), &\quad &\partial^2_x
    \partial_t \tilde{l} \in S(  (\alpha_\mu^{-4} + \alpha^{-1}
    \alpha_\lambda^{-2}) (\beta^2\lambda)^{-1},
    g_{\alpha_\lambda}),\\
    &\partial_\xi \partial_t \tilde{l} \in S( (\alpha_\mu \alpha
    \lambda)^{-1} (\beta^2\lambda)^{-1}, g_{\alpha_\lambda}), &\quad
    &\partial_\xi^2 \partial_t \tilde{l} \in S( (\alpha \lambda)^{-1}
    (\alpha_\mu^2\lambda)^{-1} (\beta^2\lambda)^{-1}, g_{\alpha_\lambda});
  \end{aligned}
\end{equation}
the factors of $\alpha$ arise from derivatives that land on the
$\chi_{\theta}^\alpha$ factor. 
  Also note that
  \begin{equation}
    \label{e:halfwave_2nd-derivs-t}
  \begin{aligned}
    &\partial_x \partial_t a \in f(t) S( |\xi|, g_{\alpha_\lambda}),
    &\quad &\partial_x^2 \partial_t a \in \alpha_\lambda^{-1} f(t)
    S(|\xi|, g_{\alpha_\lambda}),\\
    &\partial_\xi \partial_t a \in S(1, g_{\alpha_\lambda}), &\quad
    &\partial_\xi^2 \partial_t a \in S(|\xi|^{-1}, g_{\alpha_\lambda}),
  \end{aligned}
  \end{equation}
  where as usual $f(t) = M(\|\partial^2 g(t)\|_{L^\infty_x})$.

  The second term on the right side
  of~\eqref{e:commutator_box_parametrix} is handled by the first order
  estimates
  \begin{align*}
    \| P_\lambda(D) (\partial_t \tilde{l)}(t, X, D)\|_{L^2 \to L^2}
    \lesssim \alpha^{-1}(\beta^2\lambda)^{-1}, \quad \|[A^-, \tilde{L} 
    ]\|_{L^2\to L^2}
    \lesssim \alpha^{-1}(\beta^2\lambda)^{-1}.
  \end{align*}
  Also recall from the proof of Lemma~\ref{l:parametrix_bounds} that
  the symbol of $[D_t +A, \tilde{L}]$ is
  \begin{align*}
    \frac{1}{i} \{ \tau + a, \tilde{l} \} + r,
  \end{align*}
  where
  \[r \in f(t) S( (\beta^2\lambda)^{-1} (\alpha_\mu^2\lambda)^{-1},
  g_{\alpha_\lambda}) + S(( (\alpha_\mu^2\lambda)^{-1} + (
  \beta^2\lambda)^{-\frac{1}{2}}) (\beta^2\lambda)^{-1},
  g_{\alpha_\lambda}).\]
Combining the estimates
  \eqref{e:parametrix_2nd-derivs}, \eqref{e:parametrix_2nd-derivs-t},
  \eqref{e:halfwave_2nd-derivs-t} with the explicit
  form~\eqref{e:order2_commutator} of the second order commutator
  expansion and Lemma~\ref{l:gauss_transform}, one obtains
  $ \partial_t r = r_1 + r_2 + r_3 + r_4, $ where
  \begin{equation}
    \label{e:order2_commutator_t-deriv}
  \begin{aligned}
    &r_1 \in \alpha^{-1} f(t) S( (\alpha_\mu \alpha_\lambda
    \lambda)^{-1} (\beta^2\lambda)^{-1}, g_{\alpha_\lambda})
    , &\quad    &r_2 \in     \alpha^{-1}f(t) S( (\alpha_\mu^2 \lambda)^{-1} 
    (\beta^2\lambda)^{-1},
    g_{\alpha_\lambda}),\\
    &r_3 \in S((\alpha_\mu^2\lambda)^{-1} +
      (\beta^2\lambda)^{-\frac{1}{2}})(\beta^2\lambda)^{-1},
      g_{\alpha_\lambda}), &\quad &r_4 \in S( (( \alpha_\mu^{-2}
      (\alpha_\mu^2 \lambda)^{-1} + \alpha^{-1} )
      (\beta^2\lambda)^{-1}, g_{\alpha_\lambda}).
    \end{aligned}
  \end{equation}
These correspond to the pairings $\{\partial_\xi^2 \tilde{l},\ 
\partial_x^2 \partial_t a\}$, $\{\partial_\xi^2 \partial_t
\tilde{l}, \  \partial_x^2 a\}$, $\{\partial_x^2 \tilde{l}, \ 
\partial_\xi^2\partial_t a\}$, and $\{\partial_x^2\partial_t \tilde{l}, \
\partial_\xi^2 a\}$.

For the first term on the right side of
of~\eqref{e:commutator_box_parametrix}, we have
\begin{align*}
  \bigl[ (D_t + A^-), [D_t + A^+, \tilde{L}]\bigr] &= \frac{1}{i} (D_t \{ \tau +
                                           a^+, \tilde{l}\})(t, X, D)
                                           + (D_tr) (t, X, D)\\
  &+ \frac{1}{i} [ A^-, \{ \tau + a^+, \tilde{l}\} (t, X, D)] + [A^-,
    r(t, X, D)].
\end{align*}
The bounds~\eqref{e:order2_commutator_t-deriv} imply that
\begin{align*}
  \| (D_tr )(t, X, D) u\|_{L^2} \lesssim \alpha^{-1} (\beta^2
  \lambda)^{-1} \| u\|_{L^\infty L^2} + \mu  (\beta^2\lambda)^{-1} \| u\|_{L^2}
\end{align*}
which is acceptable in view of the energy estimate $\| u\|_{L^\infty
  L^2} \lesssim \| u\|_{L^2} + \|(D_t+A^+)u\|_{L^2}$.

By first order estimates, recalling that $\{\tau + a^+, \tilde{l}\}
\in S( (\beta^2\lambda)^{-1}, g_{\alpha_\lambda})$
the last two terms of the commutator satisfy
\begin{align*}
 &\|  [A^-, \{\tau + a^+, \tilde{l} \}(t, X, D)]\|_{L^2 \to L^2} \lesssim
   \beta^{-2}\\
  &\| [A^-, r(t, X, D)]
  \|_{L^2 \to L^2} \lesssim \beta^{-2} f(t) 
    (\alpha_\mu^2\lambda)^{-1}  + \beta^{-2} \bigl(
  (\alpha_\mu^2 \lambda)^{-1} + (\beta^2\lambda)^{-\frac{1}{2}} \bigr),
\end{align*}
which is also acceptable by the energy estimate.

Finally, the Poisson bracket
estimate~\eqref{e:parametrix_halfwave_poisson_tderiv} shows that
\begin{align*}
  \|(D_t \{\tau + a^+, \tilde{l} \}) (t, X, D)\|_{L^2 \to L^2}
  \lesssim  \beta^{-1} (\beta^2\lambda)^{-\frac{1}{2}}.
\end{align*}
This completes the proof of the lemma.
\end{proof}

\

\section{The algebra property \texorpdfstring{\eqref{alg:est}}{} } \label{Sec:alg}

\

In this section we prove the estimate \eqref{alg:est}. 

\begin{proposition} \label{prop:alg:prop}
Assume that $ \theta>\frac{1}{2} $ and $ s>\tht+\frac{1}{2} $. Then the space $ X^{s,\tht} $ is an algebra. Moreover, for $ \sigma >s $ we have
\be \label{alg:sgm}
\vn{u \cdot v}_{X^{\sgm,\tht}} \ls \vn{u}_{X^{\sgm,\tht}} \vn{v}_{X^{s,\tht}} + \vn{u}_{X^{s,\tht}} \vn{v}_{X^{\sgm,\tht}}
\ee
\end{proposition}

\

The proof of this algebra property will be based on the estimates in the following two propositions. The next Proposition is the crux of our result and is the variable coefficients analogue of Theorem 3 from \cite{Tat} since, due to the low modulations, we may think of $ u_{\lmd,1}, v_{\mu,1} $ as being approximately free waves. 

\

\begin{proposition} \label{PropX} 
Let $ u_{\lmd,1}, v_{\mu,1} $, $ v_{\lmd',1}  $ be functions localized at frequency $ \simeq \lmd $, $ \simeq \mu $, resp. $ \simeq \lmd'  $ and let $ d_0=\min(\mu,  \frac{\lmd}{\mu} ) $. Then

\begin{enumerate}
    \item 

In the high-low case $ \mu \ll \lmd $ we have:
\be \label{bil:main:LH}
\vn{   u_{\lmd,1} \cdot  v_{\mu,1} }_{ X_{\lmd', \leq \mu,\infty}^{0,\frac{1}{4}}}  \ls  \mu^{\frac{3}{4}}
\vn{ u_{\lmd,1}}_{X_{\lmd,1}^{0,\frac{1}{2}}} \vn{ v_{\mu,1} }_{X_{\mu,1}^{0,\frac{1}{2}}}
\ee
\item
In the high-high to low case $ \mu \lesssim \lmd \simeq \lmd' $ we have
\be \label{bil:main:HH}
\vn{ P_{\mu} (   u_{\lmd,1} \cdot  v_{\lmd',1}) }_{ X_{\mu, 
[d_0,\mu],\infty}^{1,\frac{1}{4}} + (  \tilde{X}_{\mu,\mu}^{1,\frac{1}{4}}  
\cap \frac{\lmd}{\mu} X_{\mu,\mu}^{1,\frac{1}{4}}  )}  \ls  
\frac{\mu^{\frac{3}{4}}}{\lmd}
\vn{
u_{\lmd,1}}_{X_{\lmd,1}^{1,\frac{1}{2}}} \vn{ v_{\lmd',1} }_{X_{\lmd',1}^{1,\frac{1}{2}}}
\ee  
\end{enumerate}
\end{proposition}
Here, due to the selection of modulation $1$ inputs, the index $\frac12$ on the right is superfluous. We have kept it in order to make it easier to compare the above result with the next result. 

One can easily go from modulation $1$ to any larger modulations  by a rescaling and time orthogonality argument.  This is accomplished in the next proposition.

\begin{proposition} \label{Lemma:bil:inter}
\begin{enumerate} [leftmargin=*] 
Let $ u_{\lmd,d_1}, v_{\mu,d_2} $, $ v_{\lmd',d_2}  $ be functions localized at frequency $ \simeq \lmd $, $ \simeq \mu $, respectively $ \simeq \lmd'  $ where $ d_1,d_2 \geq 1 $ and denote $ d_{\max}=\max (d_1,d_2) $.

\item (Low modulations)
For $ 1 \leq d_1,d_2 \leq \mu < \lmd \simeq \lmd' $, denoting $ d_0=\min\left( \frac{\lmd}{\mu}, \frac{\mu}{d_{\max} }\right) $
\begin{align}  \label{bil:inter:lowmod:LH}
 \vn{ u_{\lmd,d_1} \cdot v_{\mu,d_2} }_{ X_{\lmd',[d_{\max},\mu]}^{1,\frac{1}{2}}}  & \ls  
\vn{ u_{\lmd,d_1}}_{X_{\lmd,d_1}^{1,\frac{1}{2}}} \vn{ v_{\mu,d_2} }_{X_{\mu,d_2}^{1,\frac{1}{2}}} \\
\label{bil:inter:lowmod:HH}
  \vn{ P_{\mu} ( u_{\lmd,d_1} \cdot v_{\lmd',d_2} ) }_{ X_{\mu,[d_{\max} 
  d_0,\mu]}^{1,\frac{1}{2}} + (  \tilde{X}_{\mu,\mu}^{1,\frac{1}{2}} \cap 
  \frac{\lmd}{\mu} X_{\mu,\mu}^{1,\frac{1}{2}} )}  & \ls  \frac{\mu}{\lmd} 
\vn{ u_{\lmd,d_1}}_{X_{\lmd,d_1}^{1,\frac{1}{2}}} \vn{ v_{\lmd',d_2} }_{X_{\lmd',d_2}^{1,\frac{1}{2}}}
\end{align}
\item (High modulations)
 For $ 1 \leq d_2 \leq \mu \leq d_1 \leq \lmd \simeq \lmd' $ we have
\be \label{bil:inter:highmod:LH}
 \vn{ u_{\lmd,d_1} \cdot v_{\mu,d_2} }_{ X_{\lmd,d_1}^{1,\frac{1}{2}}} \ls 
\vn{ u_{\lmd,d_1}}_{X_{\lmd,d_1}^{1,\frac{1}{2}}} \vn{ v_{\mu,d_2} }_{X_{\mu,d_2}^{1,\frac{1}{2}}} 
\ee 
For $ 1 \leq \mu \leq d_{\max} \leq \lmd \simeq \lmd' $ we have
\be  \label{bil:inter:highmod:HH}
 \vn{ P_{\mu} ( u_{\lmd,d_1} \cdot v_{\lmd',d_2} ) }_{\frac{\lmd}{\mu} 
 X_{\mu,\mu}^{1,\frac{1}{2}} \cap  \tilde{X}_{\mu,\mu}^{1,\frac{1}{2}} }   \ls  
 \frac{\mu}{\lmd} 
 \vn{ u_{\lmd,d_1}}_{X_{\lmd,d_1}^{1,\frac{1}{2}}} \vn{ v_{\lmd',d_2} }_{X_{\lmd',d_2}^{1,\frac{1}{2}}} 
\ee 
\end{enumerate}
\end{proposition}

\subsection{Proof of Proposition \ref{prop:alg:prop}}
The proof consists of a simple summation argument based on Proposition \ref{Lemma:bil:inter}.
Let $ u^1, u^2 \in X^{s,\tht} $ and   write
$$
u^i=\sum_{\lmd \geq 1} P_{\lmd} u^i_{\lmd}, \qquad \qquad 
\vn{u^i}_{X^{s,\tht}}^2 \simeq \sum_{\lmd \geq 1} \vn{ 
u^i_{\lmd}}_{X^{s,\tht}_{\lmd}}^2, \qquad \qquad i\in \overline{1,2}.
$$
By Remark \ref{rmk:tildeX:loc} we may assume that the $ u^i_{\lmd} $ are localized in frequency. By the standard Littlewood-Paley decomposition we write 
$$
u^1 \cdot u^2 =\sum_{\lmd_1,\lmd_2,\lmd_3 \geq 1} P_{\lmd_3} ( P_{\lmd_1} 
u^1_{\lmd_1} \cdot P_{\lmd_2} u^2_{\lmd_2} ). 
$$
By splitting the sum into three terms corresponding to the three cases: $ \lmd_1 \ll \lmd_2 \simeq \lmd_3$, $ \lmd_2 \ll \lmd_1 \simeq \lmd_3$, $ \lmd_3 \ls \lmd_1 \simeq \lmd_2 $ we obtain $  u^1 \cdot u^2 \in X^{s,\tht} $ from the following estimates, stated for any frequency localized functions $ u_{\lmd}, v_{\mu} $: Let $ s=\tht+\frac{1}{2}+\ep $ for $ \ep >0 $.
For $ \mu \ll \lmd $, $ \lmd' \simeq \lmd $ we have
\be \label{alg:loc:HL}
\vn{P_{\lmd}u_{\lmd} \cdot P_{\mu} v_{\mu} }_{X_{\lmd'}^{s,\tht}} \ls 
\frac{1}{\mu^{\ep}} \vn{ u_{\lmd} }_{X_{\lmd}^{s,\tht}} \vn{  v_{\mu} 
}_{X_{\mu}^{s,\tht} }
\ee 

For $ \mu \ls \lmd $, $ \lmd' \simeq \lmd $ we have
\be \label{alg:loc:HH}
\vn{P_{\mu} (P_{\lmd} u_{\lmd} \cdot P_{\lmd'} v_{\lmd'} ) }_{X_{\mu}^{s,\tht}} 
\ls \frac{\mu^{s-1}}{ \lmd^{s-1+\ep}} \vn{ u_{\lmd} }_{X_{\lmd}^{s,\tht}} \vn{ 
v_{\lmd'} }_{X_{\lmd'}^{s,\tht} }
\ee

Here it is essential that we are in the subcritical case $ s>1 $ which allows us to have the the power $ \mu^{-\ep} $ in \eqref{alg:loc:HL}. We write 
$$ u_{\lmd}= \sum_{d=1}^{\lmd} u_{\lmd,d}  \qquad \qquad \vn{u_{\lmd}}_{X^{s,\tht}_{\lmd}}^2 \simeq \sum_{d=1}^{\lmd} \vn{u_{\lmd,d}}_{X^{s,\tht}_{\lmd,d}}^2.
$$
and the similar decompositions for $ v_{\mu}, v_{\lmd'} $. Then, using \eqref{bil:inter:lowmod:LH}, \eqref{bil:inter:highmod:LH}  we have
\begin{align*}
& \vn{P_{\lmd}u_{\lmd} \cdot P_{\mu} v_{\mu} }_{X_{\lmd'}^{s,\tht}} \ls 
\sum_{d_1, d_2 \leq \mu} \vn{P_{\lmd}u_{\lmd,d_1} \cdot P_{\mu} v_{\mu,d_2} 
}_{X_{\lmd',[d_{\max},\mu]}^{s,\tht}}+\sum_{d_2 \leq \mu} 
\vn{P_{\lmd}u_{\lmd,\geq \mu} \cdot P_{\mu} v_{\mu,d_2} }_{X_{\lmd',\geq 
\mu}^{s,\tht}} \\
 & \ls  \lmd^{s-1} \mu^{\tht-\frac{1}{2}} \sum_{d_1, d_2 \leq \mu} \vn{ u_{\lmd,d_1}}_{X_{\lmd,d_1}^{1,\frac{1}{2}}} \vn{ v_{\mu,d_2} }_{X_{\mu,d_2}^{1,\frac{1}{2}}} + \lmd^{s-1} \vn{u_{\lmd}}_{X_{\lmd, [\mu,\lmd]}^{1,\tht}} \sum_{d_2 \leq \mu} \vn{ v_{\mu,d_2} }_{X_{\mu,d_2}^{1,\frac{1}{2}}} \\
&  \ls \frac{ \mu^{\tht-\frac{1}{2}}}{ \mu^{s-1}} \vn{ u_{\lmd} }_{X_{\lmd}^{s,\tht}} \vn{  v_{\mu} }_{X_{\mu}^{s,\tht} } + \frac{1}{ \mu^{s-1}} \vn{ u_{\lmd} }_{X_{\lmd}^{s,\tht}} \vn{  v_{\mu} }_{X_{\mu}^{s,\tht} } \ls \frac{1}{\mu^{\ep}} \vn{ u_{\lmd} }_{X_{\lmd}^{s,\tht}} \vn{  v_{\mu} }_{X_{\mu}^{s,\tht} }
\end{align*}

In this argument we have used the factors $ d_i^{\frac{1}{2}-\tht} $ to get square sums in $ d_i \leq \mu $ while the modulation square summability of $ X_{\lmd',\geq \mu}^{s,\tht} $ is inherited from $ \vn{u_{\lmd}}_{X_{\lmd, [\mu,\lmd]}^{1,\tht}} $ due to  \eqref{bil:inter:highmod:LH}. The proof of \eqref{alg:loc:HH} is similar, using \eqref{bil:inter:lowmod:HH}, \eqref{bil:inter:highmod:HH}: 
\begin{align*}
 \vn{P_{\mu} (P_{\lmd} u_{\lmd} \cdot P_{\lmd'} v_{\lmd'} ) 
 }_{X_{\mu}^{s,\tht}} & \ls  \mu^{s-1} \mu^{\tht-\frac{1}{2}}  \sum_{d_1,d_2} 
 \vn{P_{\mu} (P_{\lmd} u_{\lmd,d_1} \cdot P_{\lmd'} v_{\lmd',d_2} ) 
 }_{X_{\mu}^{1,\frac{1}{2}}}   \\ 
& \ls \mu^{s-1} \mu^{\tht-\frac{1}{2}}  \sum_{d_1,d_2 } \vn{ u_{\lmd,d_1}}_{X_{\lmd,d_1}^{1,\frac{1}{2}}} \vn{ v_{\lmd',d_2} }_{X_{\lmd',d_2}^{1,\frac{1}{2}}}  \\
& \ls  \frac{\mu^{s-1}}{\lmd^{s-1}} \frac{\mu^{\tht-\frac{1}{2}}}{\lmd^{s-1}}  \vn{ u_{\lmd} }_{X_{\lmd}^{s,\tht}} \vn{ v_{\lmd'} }_{X_{\lmd'}^{s,\tht} }  \ls  \frac{\mu^{s-1}}{ \lmd^{s-1+\ep}} \vn{ u_{\lmd} }_{X_{\lmd}^{s,\tht}} \vn{ v_{\lmd'} }_{X_{\lmd'}^{s,\tht} }.
\end{align*}
Finally, \eqref{alg:sgm} also follows from \eqref{alg:loc:HL}, \eqref{alg:loc:HH} by readjusting the weights.

\

\subsection{Proof of Proposition \ref{PropX}}
To be able to use the wave packet decomposition from Proposition \ref{X:wp:dec} and Corollary \ref{Cor:WP:dec} we need to be on a small interval such as $ [k \delta, (k+1) \delta ]$. We fix $ \delta $ and without loss of generality we prove \eqref{bil:main:LH},  \eqref{bil:main:HH} on $ I=[0,\delta] $. We sum these bounds by brute force treating $ 1/\delta = D $ as a universal constant and then $ \vn{ v_{\eta}}_{X_{\eta,D}^{0,\frac{1}{2}}[I]} \simeq  \vn{ v_{\eta}}_{X_{\eta,1}^{0,\frac{1}{2}}[I]} $. Let $ \eta \in \{ \mu, \lmd' \} $ and we refer to $ \eta = \mu $, \eqref{bil:main:LH} as Case 1, and to $  \eta = \lmd' $, \eqref{bil:main:HH} as Case 2.
For the remainder of this proof, we normalize the norms of the inputs as follows:
\be \label{input:normalization}
\vn{ u_{\lmd,1}}_{X_{\lmd,1}^{0,\frac{1}{2}}[I]}=\vn{ v_{\eta,1}}_{X_{\eta,1}^{0,\frac{1}{2}}[I]}=1, \qquad  \qquad \eta \in \{ \mu, \lmd' \}
\ee

\

We first give an overview of the estimates needed to establish \eqref{bil:main:LH}, \eqref{bil:main:HH}.

\

{\bf Step~1.} {\bf(Bilinear angular decomposition) }

We may apply appropriate multipliers such that $ P_{\lmd} u_{\lmd,1}=u_{\lmd,1} $, $ P_{\eta} v_{\eta,1}=v_{\eta,1} $, $ \eta \in \{ \mu, \lmd' \} $.
The terms where $ \mu \simeq 1 $ of Case 1 and $ \lmd \simeq \lmd' \simeq \mu \simeq 1 $ of Case 2 are easily treated by H\"older's inequality and the chain rule. Thus we may assume $ \mu,\lmd' \gg 1 $ are large enough and Corollary \ref{Cor:WP:dec} is applicable, providing a decomposition
$$  v_{\eta,1} = v^{+} + v^{-} + v_{R}, \qquad  \qquad \eta \in \{ \mu, \lmd' 
\}.
$$
The terms $  u_{\lmd,1} \cdot v_R $ are estimated by \eqref{bil:est:remainder1} and \eqref{bil:est:remainder2} in Corollary \ref{Cor:WP:dec:rem}. We collect
\be  \label{v:omega:pm}
v^{\pm}=\sum_{\omega \in  \Omega_{\al_{\eta}}}  v^{\omega,\pm}, \qquad 
v^{\omega,\pm} = P_{\eta} \sum_{ T\in \calT_{\eta}^{\pm}, \omega_T = \omega} 
c_T(t) u_T(t).
\ee

Next, by Proposition~\ref{p:half-wave} and its corollary we decompose
$$ u_{\lmd,1}= u^{+} + u^{-} $$

Recalling \cite[Theorem 3]{Tat}, if $u$ and $v$ are free waves on a Minkowski
background, the modulation of the product $uv$ depends on their
relative positions on the null cone $\tau^2 = |\xi|^2$. Accordingly, we perform 
a bilinear angular decomposition of the products $u^{\pm_1} v^{\pm_2}$  
with angular separation $\simeq \alpha$. The
bilinear decomposition is nonstandard since the two factors are
localized differently: the high-frequency input will be  split
pseudo-differentially, while for the low-frequency input we use a
wave packet decomposition to leverage the characteristic energy estimates from 
the previous section.

Let $\Omega_{\alpha}$ be a partition of the unit circle into angle $\alpha$ arcs and let $ \al_{\eta}=\eta^{-\frac{1}{2}} $. At time $ t=0 $, for any $ \omega \in \Omega_{\al_{\eta}} $ we invoke the partition of unity from the Appendix - Proposition \ref{p:bilinear_pou}: 
\begin{align*}
 1 = \sum_{j} \phi^{\alpha_j, k_j(\omega)}_{\theta_j}(\xi), \qquad t=0, \qquad k_j(\omega) \in \{1, 2, 3, 4\}.
\end{align*}
 For each interval $\omega \in \Omega_{\alpha_\eta}$ and  $(\theta, k) \in  \Omega_\alpha \times [1,4]$, define the relation $\omega
\sim_\alpha (\theta, k)$ if the triple $(\alpha, \theta, k)$ appears
in the above partition of unity. Then $\omega \sim_\alpha (\theta, k)$ only
if $|\theta - \omega| \sim \alpha$ or $ |\theta - \omega| \ls \al_{\eta} $, and for each scale $\alpha$ there
are at most $O(1)$ intervals $\theta \in \Omega_\alpha$ related to
$\omega$.

Let $\Phi^{\alpha_\eta, \pm}_t$ denote the Hamiltonian flows for the
half-wave symbols $\tau + a_{<\alpha_\eta^{-1}}^{\pm}$ in the
factorization $g_{<\sqrt{\eta}}^{\alpha \beta} \xi_\alpha \xi_\beta =
(\tau + a_{<\alpha_\eta^{-1}}^{+})(\tau +
a_{<\alpha_\eta^{-1}}^{-})$. Pulling back both sides by this flow as in \eqref{pullback:flow} and
mollifying in the $x$ variable as in \eqref{e:pd-loc}, we obtain a time-dependent partition of unity for
functions localized at frequency $\lambda$
\begin{align}
\nonumber
  P_{\lambda}(\xi) &= \sum_{j} (P_{<\lambda/8}(D_x)\phi^{\alpha_j,\pm,
  k_j(\omega)}_{\theta_j})(t, x, \xi) P_\lambda(\xi) \\
\label{e:partition}
&= \sum_j
  \phi_{\theta_j, \lambda}^{\alpha_j, \pm, k_j(\omega)} (t,x,\xi)\\
  \label{e:partition-flipped}
  &= \sum_j
  \tilde{\phi}_{\theta_j, \lambda}^{\alpha_j, \pm, k_j(\omega)} (t,x,\xi),
\end{align}
where $\tilde{\phi}_\theta (t, x, \xi) := \phi_\theta(t, x, -\xi)$ (recall that 
the 
multipliers $s_\lambda$ are assumed radial). Note that in general 
$\tilde{\phi}_{\theta}^{\pm} \ne \phi_{-\theta}^{\pm}$.

For any signs $\pm_1, \pm_2 $ let $ \pm=\pm_1 \pm_2$. One has
\begin{align} \nonumber
  u^{\pm_1} v^{\pm_2} &=\sum_{\omega \in  \Omega_{\al_{\eta}}} u^{\pm_1} v^{\omega,\pm_2}    =
        \sum_{\omega} \sum_{j} (\phi_{\theta_j, \lambda}^{\alpha_j, \pm_1,  k_j(\omega)}
                    u^{\pm_1})  v^{\omega, \pm_2} \\
\nonumber
  &=  \sum_{\alpha\in [\alpha_\eta, 1]} \sum_{\theta \in
        \Omega_\alpha} \sum_{k=1}^4 \phi_{\theta, \lambda}^{\alpha, \pm_1, k} u^{\pm_1} \sum_{\omega \sim (\alpha, \pm \theta,
        k)} v^{\omega, \pm_2}\\
\nonumber
  &=  \sum_{\alpha\in [\alpha_\eta,1]} \sum_{\theta \in \Omega_\alpha} \sum_{k=1}^4
    (\phi_{\theta,\lambda}^{\alpha, \pm_1, k} u^{\pm_1} ) v_{\pm 
    \theta}^{\alpha, \pm_2, k}.
\end{align}
Thus for $ \al \not\simeq \alpha_\eta $, $  v_{\pm \theta}^{\alpha, \pm_2, k} $ 
contains packets $ u_T $ corresponding to $ T=(x_T,\omega_T) \in 
\calT_{\eta}^{\pm_2}$ which have angular separation of $ \al $ relative to $ 
\pm \tht $: $ \angle(\omega_T, \pm \tht) \simeq \al $: 
\be  \label{v.theta:sum}
 v^{\pm_2,\al,k}_{\tht}(t)=P_{\eta} \sum_{ T\in \calT_{\eta}^{\pm_2}, \omega_T \sim (\alpha, \pm \theta,
        k) } c_T(t) u_T(t).
\ee
Since the $ \omega_T $'s are separated by $ \simeq \eta^{-\frac{1}{2}} $, for 
every $ \tht$ we have roughly $ \frac{\al}{\eta^{-\frac{1}{2}} }=\al 
\eta^{\frac{1}{2}}$ directions $ \omega $ which obey this condition.

The index $k$ is a technical artifact of our construction and
can be safely ignored. Hence we shall hereafter simply write
\begin{align}
\label{bil:dec:uv}         
u^{\pm_1} v^{\pm_2} = \sum_{\alpha \in [\alpha_\eta, 1]} \sum_{\theta \in 
	\Omega_\alpha} ( \phi_{\theta, \lambda}^{\alpha, \pm_1} u^{\pm_1} )
v_{\pm \theta}^{\alpha, \pm_2}.
\end{align}
A typical term $(\phi_{\theta, \lambda}^{\alpha, +} u) 
v_{\pm\theta}^{\alpha, \pm}$ intuitively involves waves propagating at 
relative angle~$\alpha$.

For studying nonresonant interactions of the form 
$P_\mu \bigl( P_\lambda u^\pm  P_\lambda v^\pm\bigr)$, we need a modified 
decomposition using instead the 
partition~\eqref{e:partition-flipped}:
\begin{align}
\label{bil:dec:uv:neg}         
u^{\pm_1} v^{\pm_2} = \sum_{\alpha \in [\alpha_\eta, 1]} \sum_{\theta \in 
	\Omega_\alpha} ( \tilde{\phi}_{\theta, \lambda}^{\alpha, \pm_1} u^{\pm_1} )
v_{\theta}^{\alpha, \pm_2}.
\end{align}
Note that both terms of the form $(\phi_{\theta, \lambda}^{\alpha,+} u ) 
v_{-\theta}^{\alpha, -}$ and $(\tilde{\phi}_{\theta, \lambda}^{\alpha, +} u) 
v^{+, \alpha}_{\theta}$ only involve interactions between pairs of frequencies 
$(\xi_1, \xi_2)$ with $\angle(\xi_1, -\xi_2) \sim \alpha$.

Finally, we sometimes write $v_{\theta, 
\eta}^{\alpha, 
\pm_2}$ to clarify the 
frequency of the packets constituting $v_\theta^{\alpha, \pm_2}$.

\

{\bf Step~2.} {\bf(Small angles interactions)}

We first consider the minimal angle case consisting of the $ \al \simeq \al_{\eta}=\eta^{-\frac{1}{2}} $ terms in \eqref{bil:dec:uv} , which will mostly follow from H\"older's inequality. 

\begin{proposition} \label{prop:alphamin} Let $ \lmd \gtrsim \eta \in \{ \mu,\lmd' \}  $. Under normalization \eqref{input:normalization}, on $ I $, one has:
\begin{align}
\label{sum:tht:alphamin:L2}
\sum_{\theta \in \Omega_{\al_{\eta}}}  \vn{   \phi_{\theta,\lambda}^{\al_{\eta}, \pm_1} u^{\pm_1} \cdot  v^{\al_{\eta},\pm_2}_{\tht}
}_{ L^2}  \ls  \eta^{\frac{3}{4}} 
\\
\label{sum:tht:alphamin:Q:L2}
\sum_{\theta \in \Omega_{\al_{\eta}}}  \vn{ Q_{g_{<\sqrt{\lmd}}} \big(  
\phi_{\theta,\lambda}^{\al_{\eta}, \pm_1} u^{\pm_1} \cdot  v^{\al_{\eta},\pm_2}_{\tht}
\big) }_{ L^2}  \ls  \lmd \eta^{\frac{3}{4}}   \\
\label{sum:tht:alphamin:Q:L1L2}
\sum_{\theta \in \Omega_{\al_{\eta}}}  \vn{ Q_{g_{<\sqrt{\lmd}}} \big( 
\phi_{\theta,\lambda}^{\al_{\eta}, \pm_1} u^{\pm_1} \cdot  v^{\al_{\eta},\pm_2}_{\tht} 
\big) }_{  L^2 L^1}  \ls  \lmd   
\end{align}
For any $ \al \geq \al_{\eta} $, under \eqref{input:normalization}, one has:
\begin{align}
\label{sum:tht:box:1}
\sum_{\theta \in \Omega_{\al}}  \vn{ 
\Box_{g_{<\sqrt{\lmd}}} \phi_{\theta,\lambda}^{\alpha, \pm_1}  u^{\pm_1} \cdot  v^{\al,\pm_2}_{\tht}
}_{ L^2} +  \vn{  
\phi_{\theta,\lambda}^{\alpha, \pm_1}  u^{\pm_1}  \cdot \Box_{g_{<\sqrt{\lmd}}} v^{\al,\pm_2}_{\tht} 
}_{ L^2}  \ls \lmd \eta \al^{\frac{1}{2}}  \\
\label{sum:tht:box:2}
\sum_{\theta \in \Omega_{\al}}  \vn{ 
\Box_{g_{<\sqrt{\lmd}}} \phi_{\theta,\lambda}^{\alpha, \pm_1}  u^{\pm_1} \cdot  v^{\al,\pm_2}_{\tht}
}_{ L^2 L^1 } +  \vn{ 
\phi_{\theta,\lambda}^{\alpha, \pm_1}  u^{\pm_1}  \cdot \Box_{g_{<\sqrt{\lmd}}} v^{\al,\pm_2}_{\tht}   }_{ L^2 L^1 }  \ls \lmd 
\end{align}
\end{proposition}

As a consequence we will obtain the following estimates in low modulation spaces, which take care of the terms in \eqref{bil:dec:uv} with $ \al \simeq \al_{\eta} $. 

\begin{corollary}  \label{cor:alphamin}
In Case 1, resp. Case 2, under normalization \eqref{input:normalization}, one has:
\begin{align} \label{bil:almin:LH:I}
\sum_{\theta \in \Omega_{\al_{\eta}}} \vn{ 
 \phi_{\theta,\lambda}^{\al_{\mu}, \pm_1} u^{\pm_1} \cdot  v^{\al_{\mu},\pm_2}_{\tht,\mu}
}_{ X_{\lmd', 1}^{0,\frac{1}{4}}[I]} & \ls  \mu^{\frac{3}{4}}
 \\
\label{bil:almin:HH:I}
\sum_{\theta \in \Omega_{\al_{\eta}}}  \vn{ P_{\mu} \big( 
 \phi_{\theta,\lambda}^{\al_{\lmd'}, \pm_1} u^{\pm_1} \cdot  v^{\al_{\lmd'},\pm_2}_{\tht,\lmd'}  
\big) }_{ X_{\mu, d_0}^{0,\frac{1}{4}}[I]} & \ls  \frac{\lmd}{\mu^{\frac{1}{4}}}
\end{align} 
\end{corollary}

\
{\bf Step~3.} {\bf(Non-resonant interactions)}

We continue with the non-resonant parts of Case 2 ($ \eta=\lmd'$), which 
include $  u^{+} \cdot v^{+} $, $ u^{-} \cdot v^{-} $ and the 
terms of $ u^{\pm} \cdot v^{\mp} $ in \eqref{bil:dec:uv} where $ 
\al > \max(\frac{\mu}{\lmd}, \al_{\lmd} ) $. Here we have the following 
estimates, which are responsible for the $ \frac{\lmd}{\mu} $ loss in the high 
modulation bound $ X_{\mu,\mu}^{1,\frac{1}{4}} $ in 
\eqref{bil:main:HH}\footnote{The loss of $ \frac{\lmd}{\mu} $  in 
\eqref{bil:main:HH} caused by \eqref{HH:nonresonant} is essentially due to our 
choice of spaces. Assuming constant coefficients, a $ ++ $ high-high to low $ 
(\lmd,\lmd) \to \mu $ interaction would have output modulation $ \lmd $, while 
we force our modulation weights to be at most equal to the frequency ($ \mu $). 
To compensate for this, we introduced the $ \tilde{X}_{\mu,\mu}^{1,\frac{1}{4}} 
$ norms which retain the expected $ \frac{\mu}{\lmd} $ factor.  }

\begin{proposition}  \label{p:nonresonant}
	Let $ D \leq \mu \ll \lmd \simeq \lmd' $ and let $ \pm $ be a sign. Suppose 
	$u^\pm = P_\lambda u^\pm$, $v = P_{\lambda'} v^\pm$. Then:
	\begin{equation}\label{HH:nonresonant}\vn{P_{\mu} ( u^{\pm}
		\cdot v^{\pm} ) }_{X_{\mu,\mu}^{1,\frac{1}{4}}[I]} \lesssim 
	\mu^{-\frac{1}{4}} \| u^\pm \|_{X^{1,\frac{1}{2}}_{\lambda,1}} \| 
	v^{\pm}\|_{X^{1, \frac{1}{2}}_{\lambda,1}}.
	\end{equation} 
	More precisely, 
	\begin{align*}
	\mu^{ \frac{1}{4}} \| \nabla_{t,x} P_\mu (u^\pm \cdot v^\pm) \|_{L^2} 
	&\lesssim \mu^{-\frac{1}{4}} \frac{\mu}{\lambda}\| u^\pm 
	\|_{X^{1,\frac{1}{2}}_{\lambda,1}} \| 
	v^{\pm}\|_{X^{1, \frac{1}{2}}_{\lambda,1}},\\
	\mu^{\frac{1}{4}-1} \| \Box_{g_{<\sqrt{\mu}}} P_\mu( u^\pm 
	\cdot v^\pm) \|_{L^2} &\lesssim \mu^{-\frac{1}{4}} \| u^\pm 
	\|_{X^{1,\frac{1}{2}}_{\lambda,1}} \| 
	v^{\pm}\|_{X^{1, \frac{1}{2}}_{\lambda,1}}.
	\end{align*}

	For $ \al > \max(\frac{\mu}{\lmd}, \al_{\lmd} ) $ one has:
	\begin{align}\label{HH:ang:nonresonant} 
	\sum_{\theta \in \Omega_{\al}}  \vn{ P_\mu \big( 
		\phi_{ \theta, \lambda}^{\alpha} u^{\pm} \cdot 
		v^{\mp,\alpha}_{-\theta}   \big) }_{ X_{\mu, \mu}^{1,\frac{1}{4}}[I]}
	\ls_N \mu^{-\frac{1}{4}}
	\Bigl(\frac{\mu}{\lambda}\Bigr) \|u^{\pm} \|_{X^{1,
			\frac{1}{2}}_{\lambda,1}} \|v^{\mp} \|_{X^{1,
			\frac{1}{2}}_{\lambda,1}}
	\end{align}
\end{proposition}

The bounds \eqref{HH:nonresonant}, \eqref{HH:ang:nonresonant} are used for the 
$ X_{\mu,\mu}^{1,\frac{1}{4}} $ part of \eqref{bil:main:HH}. To prove the $ 
\tilde{X}_{\mu,\mu}^{1,\frac{1}{4}} $ part we also need $ L^{\infty} L^2 $ 
estimates, but these follow easily by Bernstein $ P_{\mu} L^{\infty} L^1 \to 
\mu L^{\infty} L^{2} $ and H\" older $ L^{\infty} L^2 \times L^{\infty} L^2 \to 
L^{\infty} L^1 $. 

The proof of this proposition uses several technical lemmas, whose proofs are 
deferred to a later section. We begin with a pseudo-differential calculus 
estimate. 
\begin{lemma}
	\label{l:HHLfreqtails}
	Let $u^{\pm} = P_\lambda u^{\pm}, \ 
	v^{\pm} = P_\lambda v^{\pm}$, where $v = v^{\pm} = 
	\sum_{\omega \in \Omega_{\alpha_\lambda}} v^{\omega, \pm}$ is a 
	superposition of 
	frequency $\lambda$ packets. Consider the bilinear decompositions
	\[
	u^{\pm} v^{\mp} = \sum_\alpha \sum_\theta (\phi_{\theta, \lambda}^{\alpha, 
		\pm} u) 
	v_{-\theta}^{\alpha, \mp}, \quad u^{\pm} v^{\pm} = \sum_{\alpha} 
	\sum_\theta 
	(\tilde{\phi}_{\theta, \lambda}^{\alpha, \pm} u) v_{\theta}^{\alpha, 
		\pm}
	\]
	as in~\eqref{bil:dec:uv}, \eqref{bil:dec:uv:neg}.  
	If $\mu \ll \alpha \lambda$ and
	$\alpha \gg \lambda^{-\frac{1}{2}}$, then there is 
	a rapidly converging expansion
	\begin{align*}
	P_\mu ( \phi_{\theta, \lambda}^\alpha u v_{-\theta, \lambda}^\alpha) = 
	\sum_{j=1, 2} \sum_{\vec{k}} 
	(\alpha 
	\lambda)^{-1} P_{\mu, \vec{k}} ( \phi_{\theta, \vec{k}}^{j, \alpha} u 
	\psi_{-\theta, 
		\vec{k}}^{j, \alpha} v) + \tilde{P}_\mu(\phi_{\theta, \lambda}^\alpha u 
	r_{-\theta}^\alpha),
	\end{align*}
	where $\phi^{j, \alpha}_{\theta, \vec{k} }, \ \psi^{j, \alpha}_{-\theta, 
		\vec{k} } \in \langle \vec{k} \rangle^{-N} S(1, 
	g_{\alpha_\lambda})$ are square-summable in $\theta$, $P_{\mu,
		\vec{k}}$ are Fourier multipliers supported in $|\xi| \lesssim \mu$
	with $L^2\to L^2$ norm $O(\langle k \rangle^{-N})$ and 
	$r_{-\theta}^\alpha = 
	\sum_{|\omega \mp \theta| \lesssim \alpha} \sum_{T} c_T r_T$ is a 
	superposition 
	of packets $r_T$  with $\|r_T\|_{WP} = O( (\alpha^2\lambda)^{-\infty})$ 
	uniformly in $T$. Here $WP$ denotes any weighted norm $WP^N_T$ as defined 
	in~\eqref{e:wpnorm}.
\end{lemma}

The next lemma gives a 
variable-coefficient version of some $L^2$ null form estimates considered by 
Foschi and Klainerman~\cite{foschi-klainerman}. Observe 
that 
$\mu^{\frac{1}{2}}$ on the 
right side is consistent with 
scaling.
\begin{lemma}
	\label{l:HHLnullform}
	For $\mu \ll \lambda$ 
	and any pair of signs $\pm_1, \pm_2 \in \{\pm\}$,
	\begin{align}
	\| P_\mu Q_{g_{<\sqrt{\lambda}}}(u^{\pm_1}, v^{\pm_2}) \|_{L^2} &\lesssim 
	\mu^{\frac{1}{2}} \| 
	u^{\pm_1} \|_{X^{1, \frac{1}{2}}_{\lambda, 1}} \| v^{\pm_2} \|_{X^{1,
			\frac{1}{2}}_{\lambda,1}},\\
	\| P_\mu (u^{\pm_1} v^{\pm_2}) \|_{L^2} &\lesssim \mu^{\frac{1}{2}} \| 
	u^{\pm_1} \|_{X^{0, \frac{1}{2}}_{\lambda, 1}} \| v^{\pm_2} \|_{X^{0,
			\frac{1}{2}}_{\lambda,1}}, \label{e:HHL-L2}\\
	\| \pt_t P_\mu (u^{\pm_1} v^{\pm_2}) \|_{L^2} &\lesssim \mu^{\frac{1}{2}} 
	\lmd \| 
	u^{\pm_1}  \|_{X^{0, \frac{1}{2}}_{\lambda, 1}} \| v^{\pm_2} \|_{X^{0,
			\frac{1}{2}}_{\lambda,1}}.  \label{dt:pu:u:v:pm}
	\end{align}
\end{lemma}

Finally, the last part of 
Proposition~\ref{p:nonresonant} utilizes

\begin{lemma}
	\label{l:HHLnullform2}
	For $\alpha > \max(\tfrac{\mu}{\lambda}, \lambda^{-\frac{1}{2}})$ in the 
	bilinear decomposition,
	\begin{gather*}
	\sum_{\theta} \mu^{1 + \frac{1}{2}} \| P_\mu (\phi_{\theta, 
	\lambda}^{\alpha,
		\mp} u v_{-\theta}^{\alpha, \pm}) \|_{L^2} \lesssim_N
	[(\alpha \lambda)^{-1} + (\alpha^2\lambda)^{-N}]
	\Bigl(\frac{\mu}{\lambda}\Bigr)^2 \|u\|_{X^{1,
			\frac{1}{2}}_{\lambda, 1}}  \|v\|_{X^{1,
			\frac{1}{2}}_{\lambda, 1}} ,\\
	\sum_{\theta} \mu^{\frac{1}{2} - 1} \| P_\mu Q ( 
	\phi_{\theta, \lambda}^{\alpha,
		\mp} u, v_{-\theta}^{\alpha, \pm}) \|_{L^2} \lesssim_N [
	\alpha^{\frac{1}{2}} \mu^{-\frac{1}{2}} + (\alpha^2\lambda)^{-N}]
	\Bigl( \frac{\lambda^{\frac{1}{2}}}{\mu} \Bigr) 
	\frac{\mu}{\lambda}\|u\|_{X^{1,
			\frac{1}{2}}_{\lambda, 1}} \| v\|_{X^{1, \frac{1}{2}}_{\lambda, 1}}.
	\end{gather*}
\end{lemma}

Now we show how these claims imply Proposition~\ref{p:nonresonant}.
Consider~\eqref{HH:nonresonant}. The preceding lemmas imply the $ L^2 $ bound 
in \eqref{HH:nonresonant}. 
The $\Box$ part of the norm is bounded by
\begin{align*}
\mu^{\frac{1}{4} - 1} \bigl(\|P_\mu \Box_{g_{<\sqrt{\lambda}}}  
(u_\lambda^+, 
v_\lambda^+) \|_{L^2} + \| 
[P_\mu, \Box_{g_{<\sqrt{\lambda}} } ] (u_\lambda^+ v_{\lambda}^+)\|_{L^2}  + 
\|( 
\Box_{g_{<\sqrt{\mu}}} - \Box_{g_{<\sqrt{\lambda}}}) P_\mu (u_\lambda^+ 
v_\lambda^+)\|_{L^2}\bigr).
\end{align*}
For the last term we use the estimate~\eqref{e:metric_est2}, Bernstein, and the 
energy estimate to infer
\begin{align*}
\|(\Box_{g_{<\sqrt{\mu}}} - \Box_{g_{<\sqrt{\lambda}}})P_\mu
(u_\lambda^+ v_\lambda^+)\|_{L^2} &\lesssim \| P_\mu \partial_t(u_\lambda^+ 
v_\lambda^+)\|_{L^\infty L^2} + \mu \| P_\mu(u_\lambda^+ 
v_{\lambda}^+) \|_{L^\infty L^2} \\
&\lesssim \mu \| \nabla_{t,x} u\|_{L^\infty L^2} \| v\|_{L^\infty L^2} + \mu 
\| 
u\|_{L^\infty L^2} \| \nabla_{t,x} v\|_{L^\infty L^2} \\
&\lesssim \frac{\mu}{\lambda}\|u_{\lambda}^+\|_{X^{1, 
		\frac{1}{2}}_{\lambda, 1}} \| v_\lambda^+\|_{X^{1, 
		\frac{1}{2}}_{\lambda, 1}}.
\end{align*}
Writing the commutator as
\begin{align*}
[P_\mu, \Box_{g_{<\sqrt{\lambda}}} ]  &= \bigl([P_{\mu}, g_{<\mu}]P_{<8\mu} + 
\sum_{\nu \in [\mu,
	\sqrt{\lambda}]} [P_\mu, g_\nu] \tilde{P}_{\nu} 
\bigr)(\partial^2_x + \partial_x \partial_t),
\end{align*}
one deduces as above that
\begin{align*}
\| [P_\mu, \Box_{g_{<\sqrt{\lambda}}} ] (u_\lambda^+ v_\lambda^+)\|_{L^2} 
&\lesssim \frac{\mu}{\lambda} \|
u_\lambda^+\|_{X^{1, \frac{1}{2}}_{\lambda, 1}}  \|
v_\lambda^+\|_{X^{1, \frac{1}{2}}_{\lambda, 1}}. 
\end{align*}
For the first term, write
\begin{align*}
\| P_\mu \Box_{g_{<\sqrt{\lambda}}} (u_\lambda^+ v_\lambda^+)\|_{L^2}
&\lesssim \| P_\mu ( \Box_{g_{<\sqrt{\lambda}}} u_\lambda^+
v_{\lambda}^+ )\|_{L^2} + \| P_\mu (   u_{\lambda}^+ 
\Box_{g_{<\sqrt{\lambda}}} v_\lambda^+
)\|_{L^2} \\
&+ \| P_\mu Q_{<\sqrt{\lambda}} (u_{\lambda}^+, v_{\lambda}^+)\|_{L^2}.
\end{align*}
By Bernstein, H\"{o}lder, and energy estimates, the first two terms are bounded 
by
\begin{align*}
\mu\| \Box_{g_{<\sqrt{\lambda}}} u_{\lambda}^+ v_\lambda^+
\|_{L^2L^1} +  \mu\| u_\lambda^+\Box_{g_{<\sqrt{\lambda}}}
v_{\lambda}^+\|_{L^2L^1} \lesssim \frac{\mu}{\lambda} \|
u_\lambda^+\|_{X^{1, \frac{1}{2}}_{\lambda, 1}}  \|
v_\lambda^+\|_{X^{1, \frac{1}{2}}_{\lambda, 1}}. 
\end{align*}

The main null form term is estimated using Lemma~\ref{l:HHLnullform}.
Collecting all the estimates we obtain the $ \Box_{g_{<\sqrt{\mu}}} $ part of 
\eqref{HH:nonresonant}.

The estimate \eqref{HH:ang:nonresonant} is proved similarly. For the $ L^2 $ 
part 
we appeal to Lemma~\ref{l:HHLnullform2}, while for the $\Box_{g_{<\sqrt{\mu}}}$ 
part we argue as above to first bound 
\begin{align*}
\sum_{\theta} \mu^{\frac{1}{2} - 1}  \|(\Box_{g_{<\sqrt{\mu}}} P_\mu - P_\mu 
\Box_{g_{<\sqrt{\lambda}}})
(\phi_{\theta, \lambda}^\alpha u v_{-\theta, \lambda}^\alpha)\|_{L^2} \lesssim
\mu^{-\frac{1}{2}} [\frac{\mu}{\lambda} +
\lambda^{-\frac{3}{2}}] \|u\|_{X^{1, \frac{1}{2}}_{\lambda, 1}} \|v
\|_{X^{1, \frac{1}{2}}_{\lambda, 1}},
\end{align*}
and estimate $\|P_\mu \Box_{g_{<\sqrt{\lambda}}} (\phi_{\theta, 
	\lambda}^\alpha u
v_{-\theta, \lambda}^\alpha)\|_{L^2}$ as before. Modulo the preceding lemmas, 
the proof 
of Proposition~\ref{p:nonresonant} is complete.

\

{\bf Step~4.} {\bf (Almost resonant interactions)}

We have arrived at the key part of the argument, involving \eqref{bil:dec:uv} for any signs $ \pm_1, \pm_2 $ in Case 1, respectively the signs
$ \pm_1 \pm_2=- $ in Case 2, with the sum restricted to $ \al_{\lmd}=\lmd^{-\frac{1}{2}} < \al \leq \frac{\mu}{\lmd} $. Thus, in Case 2 we may assume $ \sqrt{\lmd} < \mu $. The remaining parts of Case 2 are non-resonant and were treated in Step 3. 

To show that $  u^{\pm_1} \cdot  v^{\pm_2} $ lies in 
$X_{\lmd', \leq \mu,\infty}^{0,\frac{1}{4}}[I] $, respectively $ X_{\mu, [d_0,\mu],\infty}^{0,\frac{1}{4}}[I] $ 
it suffices to use the decomposition \eqref{bil:dec:uv} and for each $ \al $ to define an appropriate modulation $ d $ such that the $ \al $-term is in 
$X_{\lmd', d}^{0,\frac{1}{4}}[I] $, resp. $ X_{\mu, d}^{0,\frac{1}{4}}[I] $, which proceeds as follows:

\

\textbf{Case 1. }($ \eta=\mu $). Define\footnote{\label{note1} The choice of $ d $ is motivated by the Fourier analysis of the constant coefficients case, see \cite{Tat}}
 $ d=\mu \al^2 $. When $ \al $ ranges from $ \al_{\mu} $ to $ 1 $, $ d $ ranges from $ D $ to $ \mu $ and it suffices to prove, under \eqref{input:normalization}:
\be
\label{bil:alpha:X:LH:I}
 \sum_{\theta \in \Omega_{\al}}  \vn{ \phi_{\theta,\lambda}^{\alpha, \pm_1} u^{\pm_1} \cdot  v^{\al,\pm_2}_{\tht,\mu} }_{ X_{\lmd', \mu \al^2}^{0,\frac{1}{4}}[I]}  \ls  \mu^{\frac{3}{4}}
\ee

\

\textbf{Case 2. }($ \eta=\lmd' \simeq \lmd $). Define\footref{note1} $ d=\frac{\lmd^2 \al^2}{\mu} $. When $ \al $ ranges from $ \al_{\lmd} $ to $ \frac{\mu}{\lmd} $, $ d $ ranges from $ d_0=  \frac{\lmd}{\mu} $ to $ \mu $ and it suffices to prove, under \eqref{input:normalization}:
\be
\label{bil:alpha:X:HH:I}
 \sum_{\theta \in \Omega_{\al}}  \vn{ P_{\mu} \big( \phi_{\theta,\lambda}^{\alpha, \pm_1}  u^{\pm_1} \cdot  v^{\al,\pm_2}_{\tht,\lmd'}   \big) }_{ X_{\mu, \frac{\lmd^2 \al^2}{\mu}}^{0,\frac{1}{4}}[I]}  \ls  \frac{\lmd}{\mu^{\frac{1}{4}}}
\ee

Both cases are based on the following proposition, which incorporates the characteristic energy estimate from Section \ref{Char:en:subsec} and the uniform bounds on wave packets. 

\begin{proposition} \label{Prop:Z}
Let $ \al > \al_{\eta} $. In decomposition  \eqref{bil:dec:uv}, under \eqref{input:normalization} we have\footnote{ The estimate \eqref{NF:NE:WP:summed} shows that the effect of the null form $ Q_{g_{<\sqrt{\lmd}}} $ is $ \lmd \eta \al^2 $ which is a familiar factor from the constant coefficients case, given the angular localization.}
\begin{align}
\label{NE:WP:summed}
\sum_{\theta \in \Omega_{\al}} \vn{  \phi_{\theta,\lambda}^{\alpha, \pm_1} u^{\pm_1} \cdot  v^{\al,\pm_2}_{\tht} }_{L^2[I]} & \ls \frac{\eta^{\frac{1}{2}}}{ \al^{\frac{1}{2}}}  \\ 
\label{NF:NE:WP:summed}
\sum_{\theta \in \Omega_{\al}} \vn{Q_{g_{<\sqrt{\lmd}}} \big( \phi_{\theta,\lambda}^{\alpha, \pm_1}  u^{\pm_1} \cdot  v^{\al,\pm_2}_{\tht}  \big) }_{L^2[I]} & \ls (\lmd \eta \al^2) \frac{\eta^{\frac{1}{2}}}{ \al^{\frac{1}{2}}}.
\end{align}
\end{proposition}

 \ 

The $ L^2 $ part of both \eqref{bil:alpha:X:LH:I} and \eqref{bil:alpha:X:HH:I} 
follows immediately from \eqref{NE:WP:summed}, by discarding the multiplier $ 
P_{\mu} $ if needed. For the $ \Box_{g_{<\sqrt{\mu}}} $ part of 
\eqref{bil:alpha:X:HH:I} we write
\begin{equation} \label{Box:S:comm:diff}
 \Box_{g_{<\sqrt{\mu}}} P_{\mu}=  P_{\mu} \Box_{g_{<\sqrt{\lmd}}} +\big( 
 \Box_{g_{<\sqrt{\mu}}} - 
 \Box_{g_{<\sqrt{\lmd}}} \big) P_{\mu} + [\Box_{g_{<\sqrt{\lmd}}}, P_{\mu}].
\end{equation}
For the first term, we discard $P_\mu$ and write 
\begin{align}
\label{Box:dec:Q}
\Box_{g_{<\sqrt{\lmd}}} \big( \phi_{\theta,\lambda}^{\alpha, \pm_1}  u^{\pm_1} 
\cdot  v^{\al,\pm_2}_{\tht}  \big) = & \ 2 \ Q_{g_{<\sqrt{\lmd}}} \big( 
\phi_{\theta,\lambda}^{\alpha, \pm_1}  u^{\pm_1} \cdot  v^{\al,\pm_2}_{\tht}  
\big) + \\
\nonumber
& \Box_{g_{<\sqrt{\lmd}}} \phi_{\theta,\lambda}^{\alpha, \pm_1}  u^{\pm_1} 
\cdot  v^{\al,\pm_2}_{\tht}  +
\phi_{\theta,\lambda}^{\alpha, \pm_1}  u^{\pm_1}  \cdot \Box_{g_{<\sqrt{\lmd}}} 
v^{\al,\pm_2}_{\tht} 
\end{align}
In both \eqref{bil:alpha:X:LH:I} and \eqref{bil:alpha:X:HH:I}, the bound for the $ 
Q_{g_{<\sqrt{\lmd}}} $ term follows from \eqref{NF:NE:WP:summed}, while the bounds for the
other terms follow from \eqref{sum:tht:box:1}.

For the other two terms in~\eqref{Box:S:comm:diff}, we note first of all that as
\begin{align*}
	\mu^{-1}d^{\frac{1}{4}-1} \| (\Box_{g_{<\sqrt{\mu}}} - 
	\Box_{g_{<\sqrt{\lambda}}} ) P_\mu w\|_{L^2} \lesssim \mu^{-1} 
	d^{\frac{1}{4}-1} \|\nabla_{t,x} P_\mu w\|_{L^\infty L^2} \lesssim 
	d^{-\frac{1}{2}} \| P_\mu w\|_{X^{0, 
			\frac{1}{4}}_{\mu, \frac{\lambda^2\alpha^2}{\mu}} [I]},
\end{align*}
the second term may be absorbed in the left side 
of~\eqref{bil:alpha:X:HH:I} so long as $\mu \ll \lambda$, which 
guarantees that $d \gg 
1$.

	Also, since $\mu > \sqrt{\lambda}$ we may estimate
	\begin{align} \label{comm:ngfreq}
	\| [\Box_{g_{<\sqrt{\lambda}}}, P_\mu]  w\|_{L^2} \lesssim 
	\mu \| P_{[\mu/2, 2\mu]} w\|_{L^2} + \| \partial_t P_{[\mu/2, 2\mu]}  w
	\|_{L^2},
	\end{align}
	and again use~\eqref{tildeX:loc} to obtain
	\begin{align} \label{term:rhs:lhs}
	&\sum_{\theta \in \Omega_\alpha} \| P_\mu (\phi_{\theta,\lambda}^{\alpha, 
		\pm_1} u^{\pm_1} \cdot v_{\theta, \lambda'}^{\alpha, \pm_2}) \|_{X^{0, 
			\frac{1}{4}}_{\mu, \frac{\lambda^2\alpha^2}{\mu}} [I]} \\ \nonumber
		&\lesssim 
	\frac{\lambda}{\mu^{\frac{1}{4}}} 
	+ \sum_{\theta \in \Omega_\alpha} 
	\mu^{-1} d^{\frac{1}{4}-1} \| 
	\partial_t P_{[\mu/2, 2\mu]} (\phi_{\theta,\lambda}^{\alpha, \pm_1} 
	u^{\pm_1} \cdot v_{\theta, \lambda'}^{\alpha, \pm_2})\|_{L^2}\\ \nonumber
	&\lesssim \frac{\lambda}{\mu^{\frac{1}{4}}} + \sum_{\theta \in \Omega_\alpha} 
	d^{-1}\| P_{[\mu/2, 
	2\mu]} (\phi_{\theta,\lambda}^{\alpha, \pm_1} 
	u^{\pm_1} \cdot v_{\theta, \lambda'}^{\alpha, \pm_2}) \||_{X^{0, 
			\frac{1}{4}}_{\mu, \frac{\lambda^2\alpha^2}{\mu}} [I]}.
	\end{align}
The summation on the right side is handled by another perturbative argument. 
For 
fixed $\alpha > \lambda^{-\frac{1}{2}}$ and a small constant $c>0$, let $M$ be 
the best constant in~\eqref{bil:alpha:X:HH:I} which is uniform in $\mu \le c 
\lambda$. Invoking Proposition~\ref{p:nonresonant} and \eqref{bil:alpha:X:LH:I} 
for $\mu < \alpha \lambda$ and $\mu > c\lambda$, respectively, we have $M 
\lesssim 1 + d_{c\lambda}^{-1} M$, where $d_\mu := 
\lambda^2\alpha^2/\mu$, and the right term may be absorbed into the left side 
if $c$ is sufficiently small.

\

In what follows we complete the proof of Proposition \ref{PropX} by establishing Proposition \ref{prop:alphamin}, Corollary \ref{cor:alphamin},  Proposition \ref{Prop:Z}, Lemma \ref{l:HHLfreqtails} and Lemmas \ref{l:HHLnullform}, \ref{l:HHLnullform2}.

\subsubsection{\bf Proof of Proposition \ref{Prop:Z}} We successively consider the two estimates in the proposition.

\

{\bf (1)} Using \eqref{sq:sum:chi} and \eqref{reg:coeff:1} the estimate \eqref{NE:WP:summed} reduces to proving
\be \label{NE:WP:summed:red}
\vn{   \phi_{\theta,\lambda}^{\alpha, \pm_1} u^{\pm_1} \cdot  v^{\al,\pm_2}_{\tht}    }_{L^2} \ls \frac{\eta^{\frac{1}{2}}}{ \al^{\frac{1}{2}}}   \vn{ \chi_{\tht}^0 u^{\pm_1} }_{L^2 } \big( \sum_{ T, \omega_T \sim (\al,\pm \tht) } \vm{c_T}_{L^{\infty}_t}^2  \big)^{\frac{1}{2}}
\ee
for any $ \theta \in \Omega_{\al} $. Recall 
\be \label{v:tht:sum} 
 v^{\al,\pm_2}_{\tht} =\sum_{\omega, \omega \sim (\al,\pm \tht) } v^{\omega,\pm_2}
\ee
where $ v^{\omega,\pm_2} $ is given by \eqref{v:omega:pm}.
Since there are $ \simeq  \al \eta^{\frac{1}{2}}$ directions $ \omega $ in the \eqref{v:tht:sum} sum, 
\eqref{NE:WP:summed:red} follows from summing the following with Cauchy-Schwarz in $ \omega $:
\be  \label{NE:WP:summed:red2}
\vn{ \phi_{\theta,\lambda}^{\alpha, \pm_1} u^{\pm_1} \cdot  v^{\omega,\pm_2}  }_{L^2} \ls  
\frac{\eta^{\frac{1}{4}}}{ \al}  
 \vn{ \chi_{\tht}^0 u^{\pm_1} }_{L^2 } 
\big( \sum_{ T, \omega_T=\omega  } \vm{c_T}_{L^{\infty}_t}^2  \big)^{\frac{1}{2}}
\ee
Since for any $ \omega $ we have 
$$ I \times \mb{R}^2 = \bigcup_{T, x_T \in \Xi_{\eta}^{\omega}} T  $$
and the $ T $'s are finitely overlapping, we obtain \eqref{NE:WP:summed:red2} if we have
\be  
\vn{ \phi_{\theta,\lambda}^{\alpha, \pm_1} u^{\pm_1} \cdot P_{\eta} c_T u_T  
}_{L^2(T')} \ls \frac{1}{\lng  d(T,T') \rng^N} \frac{\eta^{\frac{1}{4}}}{ \al}  
 \vn{ \chi_{\tht}^0 u^{\pm_1} }_{L^2 } 
\vm{c_T}_{L^{\infty}_t} 
\ee
for any $ T,T' $ with $ \omega_T=\omega_{T'}=\omega $. This follows from \eqref{wp:Linfty:T} and Corollary \ref{c:char-energy-tubes}.

\

{\bf (2)} The proof of \eqref{NF:NE:WP:summed} proceeds similarly by reducing to 
\be  \label{NF:NE:WP:summed:red2}
\vn{Q_{g_{<\sqrt{\lmd}}} \big( 
\phi_{\theta,\lambda}^{\alpha, \pm_1} u^{\pm_1} \cdot  v^{\omega,\pm_2} 
 \big) }_{L^2} \ls  (\lmd \eta \al^2) 
\frac{\eta^{\frac{1}{4}}}{ \al}  \vn{ \chi_{\tht}^1 u^{\pm_1 }} 
\big( \sum_{ T, \omega_T=\omega  } \vm{c_T}_{L^{\infty}_t}^2 +   \vm{c_T'}_{L^{2}_t}^2 \big)^{\frac{1}{2}}
\ee
for every $ \omega $, based on \eqref{reg:coeff:1}, \eqref{reg:coeff:2}, \eqref{sq:sum:chi} with $ j=1 $. Associated to $ g_{\sqrt{\lmd}} $ and to $ \omega $ we consider vector fields $ L, \underline{L}, E $ which form a null frame as in section \ref{s:nullframe}.  
Then we can express the null form as
$$ 2 Q_{g_{<\sqrt{\lmd}}} \big(u,v \big) = Lu \cdot  \underline{L}v +  \underline{L} u \cdot  Lv - 2 Eu \cdot Ev.
$$ 
For the term $ L \phi_{\theta,\lambda}^{\alpha, \pm_1} u^{\pm_1} \cdot  \underline{L} v^{\omega,\pm_2} $ we proceed as before, reducing to 
\begin{align*}
\vn{L \phi_{\theta,\lambda}^{\alpha, \pm_1} u^{\pm_1} \cdot \underline{L} P_{\eta} 
c_T u_T  }_{L^2(T')} & \ls \vn{ L  \phi_{\theta,\lambda}^{\alpha, \pm_1} 
u^{\pm_1}  }_{L^2(T')} \vn{\underline{L} P_{\eta}\big( c_T u_T 
\big)}_{L^{\infty}(T')}   \\
& \ls 
\lmd  \eta^{\frac{5}{4}} \al \frac{1}{\lng  d(T,T') \rng^N} 
\vn{ \chi_{\tht}^1 u^{\pm_1 }}_{L^2} \vm{c_T}_{L^{\infty}_t} 
\end{align*}
which holds due to  \eqref{wp:barL:cT} and Corollary \ref{c:char-energy-tubes}.

The terms 
$$ \underline{L} \phi_{\theta,\lambda}^{\alpha, \pm_1} u^{\pm_1} \cdot  L  v^{\omega,\pm_2}  \qquad \text{and} \qquad E \phi_{\theta,\lambda}^{\alpha, \pm_1} u^{\pm_1} \cdot E  v^{\omega,\pm_2} $$ 
are easier, as here we obtain the corresponding part of \eqref{NF:NE:WP:summed:red2} by H\" older using \eqref{wp:L:T:sum}, \eqref{wp:E:T:sum}, \eqref{E:est:phi}, \eqref{barL:est:phi} and using the fact that $ \al \geq \eta^{-\frac{1}{2}}  $.

\

\subsubsection{\bf Proof of Proposition \ref{prop:alphamin}} 
For \eqref{sum:tht:alphamin:L2} - \eqref{sum:tht:alphamin:Q:L1L2} we may assume without loss of generality that $ v^{\al_{\eta},\pm_2}_{\tht} = v^{\omega,\pm_2} $. 

\

{\bf (1)} For \eqref{sum:tht:alphamin:L2} we use H\"older's inequality
$$  \vn{  \phi_{\theta,\lambda}^{\al_{\eta}, \pm_1} u^{\pm_1} \cdot  v^{\omega,\pm_2}
 }_{ L^2}  \ls  \vn{  \phi_{\theta,\lambda}^{\al_{\eta}, \pm_1} u^{\pm_1}  }_{L^2} \vn{v^{\omega,\pm_2} }_{ L^{\infty}}
$$
Square summing this with \eqref{wp:Linfty:T:sum} and Prop. \ref{p:orthogonality} we obtain \eqref{sum:tht:alphamin:L2}.

\

{\bf(2)} By square summing and Cauchy-Schwarz we reduce 
\eqref{sum:tht:alphamin:Q:L2} to
$$
  \vn{ Q_{g_{<\sqrt{\lmd}}} \big( 
   \phi_{\theta,\lambda}^{\al_{\eta}, \pm_1} u^{\pm_1} \cdot  v^{\omega,\pm_2}
   \big) }_{ L^2}  \ls   \lmd \eta^{\frac{3}{4}}  
   \|\chi_{\theta}^{\alpha_\eta} 
   u^{\pm_1 }\|_{L^2} 
\big( \sum_{ T, \omega_T=\omega  }  \vm{c_T}_{L^{\infty}_t}^2 +   \vm{c_T'}_{L^{2}_t}^2  \big)^{\frac{1}{2}}
$$
for some square-summable $\chi_\theta^{\alpha_\eta}$. 
Associated to $ g_{\sqrt{\lmd}} $ and to $ \omega $ we consider vector fields $ L, \underline{L}, E $ which form a null frame as in section \ref{s:nullframe}. 
Then we express the null form as
$$ 2 Q_{g_{<\sqrt{\lmd}}} \big(u,v \big) = Lu \cdot  \underline{L}v +  \underline{L} u \cdot  Lv - 2 Eu \cdot Ev
$$ 
We use H\"older's inequality: $L^2 L^2 \times L^\infty L^\infty \to L^2 
L^2$ and \eqref{L:est:phi}, \eqref{wp:barL:T:sum} for $Lu\underline{L}v$; $ L^{\infty} L^2 \times L^2 
L^{\infty} \to 
L^2 L^2 $ with \eqref{barL:est:phi}, \eqref{wp:L:T:sum} for the second term; and $L^2 L^2 \times L^\infty L^\infty \to L^2 
L^2$ with \eqref{E:est:phi}, \eqref{wp:E:T:sum} for the third term.

\

{\bf(3)} For \eqref{sum:tht:alphamin:Q:L1L2} we use H\"older $ L^{\infty} L^2 
\times L^2 L^{2} \to L^2 L^1 $ together with the null frame $ L, \underline{L}, E $ 
associated to $ g_{\sqrt{\lmd}} $ and to $ \tht $, using  
\eqref{E:est:phi}, \eqref{barL:est:phi}, \eqref{L:est:phi} and 
\begin{align*}
& \vn{E  v^{\omega,\pm_2} }_{L^\infty L^2[I]}^2 \ls \sum_{ T, \omega_T=\omega  }  \eta \vm{c_T}_{L^{\infty}_t}^2  \\
& \vn{L v^{\omega,\pm_2}  }_{L^2[I]}^2 \ls  \sum_{ T, \omega_T=\omega  }  \vm{c_T}_{L^{\infty}_t}^2 +  \vm{c_T'}_{L^{2}_t}^2    \\
& \vn{\underline{L} v^{\omega,\pm_2}  }_{L^\infty L^2[I]}^2 \ls  \sum_{ T, \omega_T=\omega  } 
\eta^2  \vm{c_T}_{L^{\infty}_t}^2. 
\end{align*}

\

{\bf(4)} The first part of \eqref{sum:tht:box:1} follows from $ L^2 \times L^{\infty} \to L^2 $ with Prop. \ref{e:char-e-commutator1}, \eqref{reg:coeff:1} and 
\be \label{v:alpha:tht:infty}
\vn{ v^{\al,\pm_2}_{\tht}  }_{L^{\infty}} \ls \eta \al^{\frac{1}{2}} \big( \sum_{ T, \omega_T \sim (\al,\pm \tht)  } \vm{c_T}_{L^{\infty}_t}^2  \big)^{\frac{1}{2}}.
\ee
The second part of  \eqref{sum:tht:box:1} follows from $ L^{\infty} L^2 \times L^2 L^{\infty} \to L^2 L^2 $
based on \eqref{O:est:phi} and
\be \label{Box:v:alpha:tht:infty}
\Big( \sum_{\theta \in \Omega_{\al}} \vn{  \Box_{g_{<\sqrt{\lmd}}} v^{\al,\pm_2}_{\tht}  }_{L^2 L^{\infty}}^2 \Big)^{\frac{1}{2}} \ls \eta^2  \al^{\frac{1}{2}}. 
\ee
This estimate is obtained by decomposing as in \eqref{v:tht:sum}, using 
Cauchy-Schwarz in $ \omega $, using \eqref{reg:coeff:1}-\eqref{reg:coeff:3} and 
$$
\vn{ \Box_{g_{<\sqrt{\lmd}} }  v^{\omega,\pm_2}}_{L^2 L^{\infty}} \ls  \eta^{\frac{7}{4}} \big( \sum_{ T, \omega_T=\omega } \vm{c_T}_{L^{\infty}_t}^2   +  \vm{c_T'}_{L^{2}_t}^2  + \eta^{-2} \vm{c_T''}_{L^{2}_t}^2  \big)^{\frac{1}{2}}
$$

{\bf(5)} Finally, the proof of \eqref{sum:tht:box:2} is similar to the proof of \eqref{sum:tht:box:1} in {\bf(4)}, except that we use $ L^2 L^2 \times L^{\infty} L^2 \to L^2 L^1 $ for the first part and $ L^{\infty} L^2 \times L^2 L^2 \to L^2 L^1 $ for the second part. Here \eqref{v:alpha:tht:infty} is replaced by
$$
\vn{ v^{\al,\pm_2}_{\tht}  }_{L^{\infty} L^2} \ls \big(\sum_{ T, \omega_T \sim (\al,\pm \tht)  } \vm{c_T}_{L^{\infty}_t}^2  \big)^{\frac{1}{2}}
$$
while \eqref{Box:v:alpha:tht:infty} is replaced by 
$$
\Big( \sum_{\theta \in \Omega_{\al}} \vn{  \Box_{g_{<\sqrt{\lmd}}} v^{\al,\pm_2}_{\tht}  }_{L^2 L^{2}}^2 \Big)^{\frac{1}{2}} \ls \eta.
$$

\subsubsection{\bf Proof of Corollary \ref{cor:alphamin} } 
The estimate \eqref{bil:almin:LH:I} follows immediately from \eqref{sum:tht:alphamin:L2}, \eqref{Box:dec:Q}, \eqref{sum:tht:alphamin:Q:L2} and \eqref{sum:tht:box:1} for $ \al=\al_{\mu}$. For \eqref{bil:almin:HH:I} we consider two cases:
\begin{enumerate}
\item  $ \lmd^{\frac{1}{2}} \leq \mu $ and $ d_0= \frac{\lmd}{\mu} $
\item $ \mu \leq \lmd^{\frac{1}{2}}   $ and $ d_0=\mu $. 
\end{enumerate}
The $ L^2 $ part of \eqref{bil:almin:HH:I} also follows from \eqref{sum:tht:alphamin:L2} in both cases (in Case (2) we use $ \mu^{\frac{1}{2}} \leq  \lmd^{\frac{1}{4}}  $).
For the $ \Box_{g_{<\sqrt{\mu}}} P_{\mu} $ part of Case (1), since $  \lmd^{\frac{1}{2}} \leq \mu $ we may freely replace $ \Box_{g_{<\sqrt{\mu}}} P_{\mu} $ by $ \Box_{g_{<\sqrt{\lmd}}} P_{\mu} $.
We estimate the $ P_{\mu} \Box_{g_{<\sqrt{\lmd}}}  $ by  \eqref{Box:dec:Q}, \eqref{sum:tht:alphamin:Q:L2} and \eqref{sum:tht:box:1} for $ \al=\al_{\lmd'}$. Then we treat $ [P_{\mu}, \Box_{g_{<\sqrt{\lmd}}} ] $ by the argument used to prove \eqref{comm:ngfreq}, \eqref{term:rhs:lhs}.

For the $ \Box_{g_{<\sqrt{\mu}}} P_{\mu} $ part of Case (2) we denote $ w=\phi_{\theta,\lambda}^{\al_{\lmd'}, \pm_1} u^{\pm_1} \cdot  v^{\al_{\lmd'},\pm_2}_{\tht,\lmd'} $ and write
\begin{align*}
\Box_{g_{<\sqrt{\mu}}} P_{\mu} w= P_{\mu} \Box_{g_{<\sqrt{\lmd}}} w  - \sum_{\nu \in [\sqrt{\mu},\sqrt{\lmd}]}   P_{\mu} \big( g_{\nu} \cdot \partial_{x} \partial_{t,x} P_{\ls \max (\nu,\mu)} w  \big) + [\Box_{g_{<\sqrt{\mu}}}, P_{\mu}]w
\end{align*}
We will use Bernstein's inequality $ P_{\mu}: L^2 L^1 \to \mu L^2 L^2 $. For the first term we invoke the $ L^2 L^1 $ estimates
\eqref{sum:tht:alphamin:Q:L1L2} and \eqref{sum:tht:box:2} for $ \al=\al_{\lmd'}$.

For the second term we discard the $ P_{\mu} $ and write
\begin{align*} & \vn{g_{\nu} \cdot \partial_{x} \partial_{t,x} P_{\ls \max (\nu,\mu)} w }_{L^2 L^1}  \ls  \frac{1}{\nu^2}  \max (\nu,\mu) \vn{\partial^2 g_{\nu}}_{L^2 L^{\infty}}  \times \\
\times 
& \Big( \vn{  \partial_{t,x} \phi_{\theta,\lambda}^{\al_{\lmd'}, \pm_1} u^{\pm_1}  }_{L^{\infty} L^2} \vn{ v^{\al_{\lmd'},\pm_2}_{\tht,\lmd'}  }_{L^{\infty} L^2}   
 + \vn{ \phi_{\theta,\lambda}^{\al_{\lmd'}, \pm_1} u^{\pm_1}   }_{L^{\infty} L^2}	\vn{  \partial_{t,x} v^{\al_{\lmd'},\pm_2}_{\tht,\lmd'} }_{L^{\infty} L^2} \Big).
\end{align*}
which is more than enough. The third term is estimated similarly since it is of the form $ \approx \mu^{-1} \nabla g \partial_{x} \partial_{t,x} $.

\subsubsection{\bf Proof of Lemma \ref{l:HHLfreqtails}} 

We remark first of all that by standard pseudo-differential calculus arguments 
(see for example~\cite{hormanderv3}), if $A(x, \xi) \in S(1, g_{\alpha})$ 
with $\alpha \ge \alpha_\lambda := 
\lambda^{-1/2}$ and $v_T$ 
is a 
frequency $\lambda$ packet, then so is $A(X, D) v_T$.

From this one infers
\begin{lemma}
\label{e:wp-pdo-tail}
 Let $v_T$ be a frequency $\lambda$ packet with $\omega(T) = \omega$. Suppose 
 $|\theta -\omega| 
\le \alpha/4$ and $\phi_\theta(t, x, \xi) =  \chi \bigl(\frac{ |\wht{\xi} - 
\wht{\xi_\theta}(t, x)|}{ \alpha} \bigr)P_\lambda(\xi)$ where $\chi$ is
smooth and supported in $[-1,1]$. Then
\begin{align*}
 \|(1 -\phi_\theta(t,X,D))  v_T\|_{WP} = O( (\alpha^2\lambda)^{-\infty}).
\end{align*}
\end{lemma}

\begin{proof}
 Put $\psi = 1-\phi$. Without loss of generality assume $v_T$ is centered at $x 
=0$. Let $\chi(x)$ be a bump function supported in ball $|x| \le \alpha/16$. 
Then 
$\chi 
\psi$ only admits input frequencies outside the sector $\angle (\xi, \omega) 
\lesssim \alpha$, so $\|\chi \psi(t,X,D) v_T\|_{WP} = O(  
(\alpha^2\lambda)^{-\infty})$. On the other hand, the 
spatial decay of $\psi 
v_T$ by our initial remark implies that $\|(1-\chi) \psi(t,X,D) v_T\|_{WP} = O( 
(\alpha^2\lambda)^{-\infty})$ as well.
\end{proof}

Using Lemma~\ref{e:wp-pdo-tail} we may write
	\begin{align*}
	v_{-\theta, \lambda}^\alpha = \psi_{-\theta}^\alpha v_{-\theta, 
		\lambda}^\alpha  + 
	r_{\theta}^\alpha,
	\end{align*}
	where
	$\psi_{-\theta}^\alpha = \chi \bigl(\frac{ |\wht{\xi} +
		\wht{\xi_\theta}(t, x)|}{ \alpha} \bigr)s_\lambda(\xi)$ for smooth
	cutoff function $\chi$, and where
	$r_{\theta}^\alpha =\sum_T a_T r_T$ is a sum of packets with
	$\|r_T\|_{WP} = O( (\alpha^2\lambda)^{-\infty})$.
	
	It remains to decompose the psuedo-differential term
	$P_\mu(\phi^\alpha_{\theta, \lambda} u \cdot
	\psi_{-\theta}^{\alpha} v_{-\theta, \lambda}^\alpha)$.  We write
	$v := v_{-\theta, \lambda}^\alpha$ and omit the dependence on $\alpha$, 
	$\lambda$, and 
	$t$ in 
	the other notations. For 
	any $w \in L^2$,
	\begin{align*}
	\langle P_\mu (\phi_{\theta} u \psi_{-\theta} v), w \rangle &= \int e^{i 
		\langle 
		x, \xi + \eta + \zeta\rangle} \phi_{\theta}(x, \xi) \psi_{-\theta} (x, 
	\eta) 
	s_\mu(\zeta) \wht{u}(\xi) \wht{v}(\eta) \wht{w} (\zeta) \, d\xi d\eta 
	d\zeta dx\\
	&= i \int e^{i \langle x, \xi + \eta + \zeta \rangle} \langle \partial_x, 
	F(\xi, 
	\eta, \zeta) \rangle \phi_{\theta}(x, \xi) \psi_{-\theta} (x, \eta) 
	s_\mu(\zeta) \\
	&\times\wht{u}(\xi) \wht{v}(\eta) \wht{w} (\zeta) \, d\xi d\eta d\zeta dx,
	\end{align*}
	where $ F = \tfrac{\xi + \eta + \zeta }{ | \xi +\eta + \zeta|^2 }.  $
	As $\mu \ll \alpha \lambda$, the denominator of $F$ is bounded below
	$|\xi + \eta + \zeta| \gtrsim \alpha \lambda$ on the support of the
	integrand.

	Introducing slightly wider cutoffs $\tilde{\phi}_{\theta}$, 
	$\tilde{\psi}_{-\theta}$ so that $\tilde{\phi}_{\pm \theta} \phi_{\pm 
	\theta} = 
	\phi_{\pm \theta}$, we separate variables, for instance using Fourier 
	series expansion
	\begin{align*}
	\tilde{\phi}_{\theta} \tilde{\psi}_{-\theta} s_{<2\mu} F = (\alpha
	\lambda)^{-1} \sum_{\vec{k}} 
	a_{\vec{k}}(x)
	e_{1, k_1}(\xi) e_{2, k_2} (\eta) e_{3, k_3} (\zeta), 
	\end{align*}
	where $e_{1, k_1}$ and $e_{3, k_3}$ are Fourier characters adapted to 
	$\lambda \times (\alpha \lambda)$ rectangles and $e_{2, k_2}$ are
	adapted to $\mu \times \mu$ rectangles. By the derivative 
	bounds~\eqref{e:flow-deriv} with $\alpha=\sqrt{\lambda}$, the 
	coefficients
	$ a_{\vec{k}} = (\alpha \lambda) \langle \tilde{\phi}_{\theta}
	\tilde{\psi}_{-\theta} s_{<2\mu} F , e_{\vec{k}}  \rangle $
	satisfy
	\begin{align}
	\label{e:ak_deriv_bounds}
	\partial_x^j a_{\vec{k}} = O( \lambda^{\frac{j-1}{2}} |\vec{k}|^{-N}), \ j 
	\ge 
	1.
	\end{align}
	Consequently, the integral takes the form
	\begin{align*}
	&(\alpha \lambda)^{-1}\sum_{\vec{k}} \int e^{i\langle x, \xi + \eta + 
		\zeta\rangle}
	a_{\vec{k}}(x) \bigl[(\partial_x \phi_{\theta})(x, \xi) \psi_{-\theta}(x,
	\eta) + \phi_\theta (x, \xi) (\partial_x \psi ) (x, \eta)]\\
	&\times s_\mu(\zeta) e_{1, k_1}(\xi) e_{2, k_2}(\eta) e_{3,
		k_3}(\zeta)\wht{u}(\xi) \wht{v}(\eta) \wht{w}(\zeta) \, d\xi
	d\eta d\zeta dx\\
	&= \Bigl\langle \sum_{j=1, 2} \sum_{\vec{k}} (\alpha \lambda)^{-1}
	\phi^j_{\theta, \vec{k}} u \psi^j_{-\theta, \vec{k}} v, \, 
	P_{\mu, \vec{k}} (D)w \Bigr\rangle,
	\end{align*}
	where
	\begin{align*}
	\phi^1_{\theta, \vec{k}}(x, \xi) &= |k|^{2N/3} a_{\vec{k}}(x) (\partial_x 
	\phi_\theta) e_{1, k_1}(\xi), \quad 	\phi^2_{\theta, \vec{k}}(x, \xi) = 
	|k|^{2N/3} a_{\vec{k}}(x) 
	\phi_\theta e_{1, k_1}(\xi)\\
	\psi^1_{-\theta, \vec{k}}(x, \eta) &= |k|^{-N/3}\psi_{-\theta} (x, \eta) 
	e_{2, 
	k_2}(\eta), \quad \psi^2_{-\theta, \vec{k}}(x, \eta) = 
	|k|^{-N/3}(\partial_x\psi_{-\theta}) (x, \eta) 
e_{2, 
k_2}(\eta)\\
	P_{\mu, \vec{k}}(\zeta) &= |k|^{-N/3} s_\mu (\zeta) e_{3, k_3}(\zeta)
	\end{align*}
	The $\partial_x$ is harmless since $\phi_{\theta}, \ \psi_{-\theta}$ are 
	Lipschitz in $x$, so we 
	have a rapidly converging series expansion
	\begin{align*}
	P_\mu( \phi_{\theta} u \psi_{-\theta} v) = \sum_{j=1, 2} \sum_{\vec{k}} 
	(\alpha 
	\lambda)^{-1} P_{\mu, \vec{k}}(D) ( \phi^j_{\theta, \vec{k} } u 
	\psi^j_{-\theta, 
		\vec{k}} v)
	\end{align*}
	where
	$\phi^j_{\theta, \vec{k}}, \ \psi^j_{-\theta, \vec{k}}  \in \langle \vec{k} 
	\rangle^{-N} S(1,
	g_{\alpha_\lambda} )$ and retain the angular localization of $\phi_\theta, 
	\ \psi_{-\theta}$. Although these are not quite localized in output
	frequency, the bounds~\eqref{e:ak_deriv_bounds} imply rapid
	decay from frequency $\lambda$ on the $\sqrt{\lambda}$ scale. By 
	Lemma~\ref{l:CV} 
	these  tails are negligible for square summability in $\theta$.

\subsubsection{\bf Proof of Lemmas~\ref{l:HHLnullform} and \ref{l:HHLnullform2}} 

	Both parts can be proved in parallel. Without loss of generality we consider 
	just 
	the $++$ and $+-$ interactions. Begin with the decompositions
	\begin{align*}
	u^+ v^- = \sum_{\lambda^{-1/2} \le \alpha \le 1} 
	\sum_{\theta} 
	(\phi_{\theta, \lambda}^\alpha u) v_{-\theta, \lambda}^\alpha, \ u^{+} 
	v^{+} = \sum_{\alpha} 
	\sum_\theta 
	(\tilde{\phi}_{\theta, \lambda}^{\alpha, \pm} u) v_{\theta, 
	\lambda}^{\alpha, 
		\pm}
	\end{align*}
	where $v^\alpha_{\pm\theta, \lambda} = \sum_{|\omega \mp \theta| \sim 
	\alpha} 
	v^\omega$ is a 
	sum of frequency $\lambda$ packets. The arguments involved depend on the 
	relation between $\alpha$, $\mu/\lambda$, $\lambda^{-1/2}$, and
	for each of the following cases we discuss the null form estimates (Cases 1a, 2a, 3a, 4) and the $L^2$ estimates (Cases 1b, 2b, 3b, 4)
	for both $++$ and $+-$ 
	interactions. Note that Lemma~\ref{l:HHLnullform2} is covered by Cases 
	2 and 4 
	below. Denote $Q := Q_{g_{<\sqrt{\lambda}}}$.
	
	\textbf{Case 1}: $\lambda^{-1/2} \le \alpha \le 
	\mu/\lambda$.
	
	\textbf{Case 1a$++$}:  We can use Cauchy-Schwartz in $\omega$ to obtain
	\begin{align*}
	\|P_\mu Q(\tilde{\phi}_{\theta, \lambda}^\alpha u, v_{\theta, 
	\lambda}^\alpha ) 
	\|_{L^2} 
	\lesssim 
	(\alpha^2 \lambda)^{\frac{1}{4}} \Bigl( \sum_{|\omega - \theta| \sim 
	\alpha} \| 
	P_\mu Q( \tilde{\phi}_{\theta, \lambda}^\alpha u, v^\omega) \|_{L^2}^2 
	\Bigr)^{1/2}
	\end{align*}
	For each $\omega$ we expand $Q$ using the null frame adapted to the
	direction $\omega$ and discard the mollifier $P_\mu$. For the term
	$\underline{L} \tilde{\phi}^\alpha_{\theta, \lambda} u L v^\omega$, 
	H\"{o}lder, 
	energy 
	estimates, and the
	wave packet bounds~\eqref{wp:L:T:sum}  imply
	\begin{align*}
	\| \underline {L} \tilde{\phi}^\alpha_{\theta, \lambda} u L 
	v^\omega\|_{L^2} 
	&\lesssim \| 
	\underline {L} 
	\tilde{\phi}^\alpha_{\theta,
		\lambda}
	u\|_{L^\infty 
		L^2}\| L v^\omega 
	\|_{L^2L^\infty} 
	\lesssim \lambda^{-\frac{1}{4}}\| \tilde{\phi}^\alpha_{\theta, \lambda} 
	u\|_{X^{1, 
			\frac{1}{2}}_{\lambda, 1}} \| 
	v^\omega\|_{X^{1, \frac{1}{2}}_{\lambda, 1}}.
	\end{align*}
	For the other terms we use the characteristic energy 
	estimate \eqref{e:CE1} and the packet bounds~\eqref{wp:Linfty:T} through 
	\eqref{wp:barL:T} to get
	\begin{align*}
	\| E\tilde{\phi}^\alpha_{\theta, \lambda} u E v^\omega \|_{L^2}^2 
	&\lesssim 
	\bigl(\sum_T 
	|c_T|_{L^{\infty}_t}^2 
	\lambda^{\frac{3}{2} + 1}\bigr)  \cdot \sup_T \| E 
	\tilde{\phi}^\alpha_{\theta, 
	\lambda} 
	u\|_{L^2(T)}^2 \\
	&\lesssim 
	\lambda^{-\frac{1}{2}} \| \tilde{\phi}^\alpha_{\theta, \lambda} u\|_{X^{1, 
			\frac{1}{2}}_{\lambda, 1}}^2  
	\sum_{\omega_T 
		= \omega} \lambda^2|c_T|_{L^{\infty}_t}^2,\\
	\| L\tilde{\phi}^\alpha_{\theta, \lambda} u \underline{L} 
	v^\omega\|_{L^2}^2 
	&\lesssim 
	\|  \sum_{T} 	
	c_T L 
	\tilde{\phi}^\alpha_{\theta, \lambda} u \underline{L} u_T 
	\|_{L^2}^2 
	 + 
	\bigl\| 
	L 
	\tilde{\phi}^\alpha_{\theta, \lambda} u \sum_T c_T' v_T \bigr\|_{L^2}^2 \\
	&\lesssim \bigl (\sum_T 
	|c_T|_{L^{\infty}}^2 
	\lambda^{\frac{3}{2} + 2}\bigr) \cdot \sup_T \| L 
	\tilde{\phi}^\alpha_{\theta, 
	\lambda} 
	u\|_{L^2(T)}^2 + 
	\| L 
	\tilde{\phi}^\alpha_{\theta, \lambda} u\|_{L^\infty L^2}^2 \| \sum_T c_T' 
	u_T\|_{L^2 
		L^\infty}^2\\
	&\lesssim 
	\lambda^{\frac{1}{2}} \| \tilde{\phi}^\alpha_{\theta, \lambda} u\|_{X^{1, 
			\frac{1}{2}}_{\lambda, 1}}^2  
	\sum_{\omega_T 
		= \omega } \lambda^2|c_T|_{L^{\infty}_t}^2 + \| \tilde{\phi}_{\theta, \lambda}^\alpha 
		u\|_{X^{1, 
			\frac{1}{2}}_{\lambda, 1}}^2 \sum_{\omega_T = \omega} 
			\lambda^{\frac{3}{2}} 
	|c_T'|_{L^2_t}^2
	\end{align*}
	Overall
	\begin{align*}
	\| P_\mu Q( \tilde{\phi}_{\theta, \lambda}^\alpha u, v^\omega)\|_{L^2} 
	&\lesssim 
	\lambda^{\frac{1}{4}} \|\tilde{\phi}^\alpha_{\theta, \lambda} u\|_{X^{1, 
			\frac{1}{2}}_{\lambda, 
			1}} \Bigl( 
	\sum_{\omega_T = \omega} \lambda^2 |c_T|_{L^\infty_t}^2 + \lambda 
	|c_T'|_{L^2_t}^2
	\Bigr)^{\frac{1}{2}},\\
	\| P_\mu Q(\tilde{\phi}_{\theta, \lambda}^\alpha u, v_{-\theta, 
	\lambda}^\alpha ) 
	\|_{L^2} 
	&\lesssim \mu^{\frac{1}{2}} 
	\Bigl( \frac{ \alpha \lambda}{\mu} \Bigr)^{\frac{1}{2}} 
	\|\tilde{\phi}^\alpha_{\theta, 
		\lambda} 
	u 
	\|_{X^{1, \frac{1}{2}}_{\lambda, 1}} \Bigl( \sum_{|\omega - \theta| \sim 
		\alpha} \sum_{ \omega_T 
		= 
		\omega \sim \alpha} \lambda^2 |c_T|_{L^\infty_t}^2 + \lambda 
		|c_T'|_{L^2_t}^2 
	\Bigr)^{\frac{1}{2}},
	\end{align*}
	which can be summed in $\alpha$ when
	$\alpha \in [\lambda^{-1/2},\, \mu/\lambda]$.
	
	\textbf{Case 1a$+-$}: In this case the waves $\phi^\alpha_{\theta, 
	\lambda} u$ 
	and
	$v^\omega$ propagate at small relative angles, so by 
	Proposition~\ref{p:char-energy} the characteristic energy 
	estimate
	for the term $L\phi^\alpha_{\theta, \lambda} u \underline{L}v^\omega$ 
	improves 
	by a
	factor of $\alpha$. The previous estimates are replaced by
	\begin{align*}
	\| P_\mu Q( \phi^\alpha_{\theta, \lambda} u, v^\omega)\|_{L^2} &\lesssim 
	\alpha \lambda^{\frac{1}{4}} \|\phi^\alpha_{\theta, \lambda} u \|_{X^{1, 
			\frac{1}{2}}_{\lambda, 1}} \Bigl( 
	\sum_{\omega_T = \omega} \lambda^2 
	|c_T|_{L^\infty_t}^2 + \lambda |c_T'|_{L^2_t}^2 \Bigr)^{\frac{1}{2}},\\
	\| P_\mu Q(\phi^\alpha_{\theta, \lambda} u, v_{-\theta, \lambda}^\alpha ) 
	\|_{L^2} 
	&\lesssim \alpha \mu^{\frac{1}{2}} 
	\Bigl( \frac{ \alpha 
		\lambda}{\mu} 
	\Bigr)^{\frac{1}{2}} 
	\|\phi^\alpha_{-\theta, 
		\lambda} u 
	\|_{X^{1, 
			\frac{1}{2}}_{\lambda, 
			1}} \Bigl( \sum_{|\omega 
		+ \theta| \sim \alpha} 
	\sum_{ \omega_T 
		= 
		\omega
		\sim
		\alpha}
	\lambda^2
	|c_T|_{L^\infty_t}^2 + 
	\lambda|c_T'|_{L^2_t}^2
	\Bigr)^{\frac{1}{2}}\\
	&\lesssim \mu^{\frac{1}{2}} \frac{\mu}{\lambda} \Bigl( \frac{\alpha
		\lambda}{\mu} \Bigr)^{\frac{3}{2}} \Bigl( \sum_{|\omega + \theta| \sim 
		\alpha} \sum_{ \omega_T 
		= 
		\omega
		\sim
		\alpha}
	\lambda^2
	|c_T|_{L^\infty_t}^2 + 
	\lambda |c_T'|_{L^2_t}^2
	\Bigr)^{\frac{1}{2}}.
	\end{align*}

	\textbf{Case 1b$++$}: For the $L^2$ estimate we also write
	\begin{align*}
	\|P_\mu (\tilde{\phi}^\alpha_{\theta, \lambda} u v_{\theta, \lambda}^\alpha 
	) 
	\|_{L^2} 
	\lesssim 
	(\alpha^2 \lambda)^{\frac{1}{4}} \Bigl( \sum_{|\omega - \theta| \sim 
	\alpha} \| 
	P_\mu ( \tilde{\phi}^\alpha_{\theta, \lambda} u v^\omega) \|_{L^2}^2 
	\Bigr)^{1/2}
	\end{align*}
	For each $\omega$, discard $P_\mu$ and use the pointwise
	bounds for the packets in $v^\omega$ and the characteristic $L^2$
	estimate for  $\phi^\alpha_{\theta, \lambda} u$ along characteristic 
	surfaces 
	for
	$v^\omega$ (the second part of Proposition~\ref{p:char-energy} with
	$\beta \sim 1$).
	\begin{align*}
	\|\tilde{\phi}^\alpha_{\theta, \lambda} u v^\omega\|_{L^2} &\lesssim 
	\sup_T \|\tilde{\phi}^\alpha_{\theta, \lambda} u\|_{L^2(T)}  
	 \cdot \lambda^{\frac{1}{4}} \Bigl( \sum_{T}
	|c_T|_{L^\infty_t}^2\Bigr)^{\frac{1}{2}} \lesssim \lambda^{\frac{1}{4}} 
	\|\chi_{\theta}^{\alpha} u
	\|_{L^2} \Bigl( \sum_{T} |c_T|_{L^\infty_t}^2  \Bigr)^{\frac{1}{2}},
	\end{align*}
	where $\chi_{\theta}^\alpha$ are square-summable in $\theta$. Thus
	\begin{align*}
	\| \tilde{\phi}^\alpha_{\theta, \lambda} u v_{\theta, 
	\lambda}^\alpha\|_{L^2} 
	&\lesssim 
	(\alpha
	\lambda)^{\frac{1}{2}} \|\chi_{\theta}^\alpha u\|_{L^2}
	\Bigl(\sum_{|\omega - \theta| \sim \alpha} \sum_{ \omega(T) =
		\omega} |c_T|_{L^\infty_t}^2 
	\Bigr)^{\frac{1}{2}} \\
	&\lesssim \mu^{\frac{1}{2}} \Bigl(\frac{\alpha 
	\lambda}{\mu}\Bigr)^{\frac{1}{2}}\|\chi_{\theta}^\alpha u\|_{L^2}
	\Bigl(\sum_{|\omega - \theta| \sim \alpha} \sum_{ \omega(T) =
		\omega} |c_T|_{L^\infty_t}^2 \Bigr)^{\frac{1}{2}}.
	\end{align*}
	
	\textbf{Case 1b$+-$}: The argument of the $++$ case would yield a loss
	of $\alpha^{-1}$ due to small angles ($\beta = \alpha$ in Corollary
	6.9):
	\begin{align*}
	\| \phi^\alpha_{\theta, \lambda} u v_{-\theta, \lambda}^\alpha \|_{L^2} 
	\lesssim
	\mu^{\frac{1}{2}} \Bigl( \frac{\alpha
		\lambda}{\mu}\Bigr)^{\frac{1}{2}} \alpha^{-1} \|\chi_{\theta}^\alpha 
		u\|_{L^2}
	\Bigl(\sum_{|\omega + \theta| \sim \alpha} \sum_{ \omega(T) =
		\omega} |c_T|_{L^\infty_t}^2 \Bigr)^{\frac{1}{2}}.
	\end{align*}
	Instead we use the method of Case 3 below with a null
	foliation \emph{tranverse} to $\phi^\alpha_{\theta, \lambda} u$ to estimate
	\begin{align*}
	\| P_\mu (\phi^\alpha_{\theta, \lambda} u v_{-\theta, 
		\lambda}^\alpha)\|_{L^2} 
	&\lesssim \mu
	\Bigl( \sum_j \| \phi_{\theta, \lambda}^\alpha u \|_{L^2(\Sigma_j)}^2
	\|v_{-\theta, \lambda}^\alpha\|_{L^\infty L^2(\Sigma_j)}
	\Bigr)^{\frac{1}{2}}\\
	&\lesssim \mu^{\frac{1}{2}}
	\|\chi_{\theta}^\alpha u\|_{L^2} \Bigl(\sum_{|\omega + \theta|
		\sim \alpha } \sum_{\omega_T = \omega} |c_T|_{L^\infty_t}^2 
	\Bigr)^{\frac{1}{2}}.
	\end{align*}
	
	\textbf{Case 2}: $\alpha > \mu/\lambda \ge \lambda^{-1/2}$. 
	
	\textbf{Case 2a$++$}: By Lemma~\ref{l:HHLfreqtails}, the estimates of the 
	previous case hold with an additional 
	factor of $(\alpha \lambda)^{-1} + (\alpha^2\lambda)^{-N}$ for any $N$:
	\begin{align*}
	&\| P_\mu Q(\tilde{\phi}_{\theta, \lambda}^\alpha u, v_{\theta, 
	\lambda}^\alpha ) 
	\|_{L^2} \\
	&\lesssim [(\alpha 
	\lambda)^{-1} + (\alpha^2\lambda)^{-N}] \mu^{\frac{1}{2}} 
	\Bigl( \frac{ \alpha \lambda}{\mu} \Bigr)^{\frac{1}{2}} 
	\|\tilde{\phi}_{\theta, 
		\lambda}^\alpha u 
	\|_{X^{1, \frac{1}{2}}_{\lambda, 1}} \Bigl( \sum_{|\omega - \theta| \sim 
		\alpha} \sum_{ \omega_T 
		= 
		\omega} \lambda^2 |c_T|_{L^\infty}^2 + \lambda |c_T'|_{L^2_t}^2 
	\Bigr)^{\frac{1}{2}},
	\end{align*}
	which is summable in $\alpha$.
	
	\textbf{Case 2a$+-$}: 
	\begin{align*}
	&\| P_\mu Q(\phi_{\theta, \lambda}^\alpha u, v_{-\theta, \lambda}^\alpha ) 
	\|_{L^2} \\
	&\lesssim  [(\alpha 
	\lambda)^{-1} + (\alpha^2\lambda)^{-N}] \mu^{\frac{1}{2}} 
	\frac{\mu}{\lambda}
	\Bigl( \frac{ \alpha \lambda}{\mu} \Bigr)^{\frac{3}{2}} \|\phi_{\theta, 
		\lambda}^\alpha u 
	\|_{X^{1, \frac{1}{2}}_{\lambda, 1}} \Bigl( \sum_{|\omega + \theta| \sim 
		\alpha} \sum_{ \omega_T 
		= \omega} \lambda^2 |c_T|_{L^\infty_t}^2 + \lambda |c_T'|_{L^2_t}^2 
	\Bigr)^{\frac{1}{2}}\\
	&\lesssim \mu^{\frac{1}{2}} \frac{\mu}{\lambda} \Bigl(
	\alpha^{\frac{1}{2}} \mu^{-\frac{1}{2}}
	\frac{\lambda^{\frac{1}{2}} }{\mu} + (\alpha^2\lambda)^{-N+\frac{3}{4}}
	\Bigl(\frac{\lambda^{\frac{1}{2}}}{\mu}\Bigr)^{\frac{3}{2}} \Bigr)  
	\|\phi_{\theta, \lambda}^\alpha u 
	\|_{X^{1, \frac{1}{2}}_{\lambda, 1}} \Bigl( \sum_{|\omega + \theta| \sim 
		\alpha} \sum_{ \omega_T 
		= \omega} \lambda^2 |c_T|_{L^\infty_t}^2 + \lambda |c_T'|_{L^2_t}^2 
		\Bigr)^{\frac{1}{2}}
	\end{align*}
	
	\textbf{Case 2b$++$}: The same considerations as for the null form estimate 
	yield
	\begin{align*}
	\| P_\mu (\tilde{\phi}_{\theta, \lambda}^\alpha u v_{\theta, 
	\lambda}^\alpha ) 
	\|_{L^2} &\lesssim [(\alpha 
	\lambda)^{-1} + (\alpha^2\lambda)^{-N}] \mu^{\frac{1}{2}} 
	\Bigl( \frac{ \alpha \lambda}{\mu} \Bigr)^{\frac{1}{2}} 
	\|\chi_{\theta}^\alpha u
	\|_{L^2} \Bigl( \sum_{|\omega - \theta| \sim \alpha} \sum_{ \omega_T 
		= 
		\omega} |c_T|_{L^\infty_t}^2 \Bigr)^{\frac{1}{2}}.
	\end{align*}
	
	\textbf{Case 2b$-$}:
	\begin{align*}
	\| P_\mu (\phi_{\theta, \lambda}^\alpha u v_{-\theta, 
		\lambda}^\alpha)\|_{L^2} 
	&\lesssim [(\alpha\lambda)^{-1} + (\alpha^2\lambda)^{-N}]\mu^{\frac{1}{2}}
	\|\chi_{-\theta}^\alpha u\|_{L^2} \Bigl(\sum_{|\omega - \theta|
		\sim \alpha } \sum_{\omega_T = \omega} |c_T|_{L^\infty_t}^2 
	\Bigr)^{\frac{1}{2}}.
	\end{align*}
	
	\textbf{Case 3}: $\mu \le \lambda^{1/2}$, $\alpha \approx \lambda^{-1/2}$.
	
	\textbf{Case 3a$++$}: 
	A direct application of Bernstein and energy estimates would yield 
	\begin{align*}
	\| P_\mu Q( \tilde{\phi}_{\theta, \lambda}^\alpha u, v_{\theta, 
	\lambda}^\alpha 
	)\|_{L^2} 
	\lesssim 
	\mu \| \tilde{\phi}_{\theta, \lambda}^\alpha u \|_{X^{1, 
	\frac{1}{2}}_{\lambda, 
	1}} \| 
	v_{\theta, 
		\lambda}^\alpha 
	\|_{X^{1, \frac{1}{2}}_{\lambda, 1}}.
	\end{align*}
	The additional $\mu^{-\frac{1}{2}}$ gain required is precisely what would 
	result from a 
	characteristic energy estimate over a $\mu^{-1}$ neighborhood of a 
	null surface. 
	
	Using the null foliation $\Lambda_{h}$ adapted to $v_{-\theta, 
	\lambda}^\alpha$, 
	we partition spacetime into  ``null slabs'' of thickness $\mu^{-1}$
	\begin{align*}
	\R^{1+2} = \bigcup_j \bigcup_{h \in [j \mu^{-1}, (j+1)\mu^{-1}] } 
	\Lambda_{h} =: \bigcup_j \Sigma_j
	\end{align*}
	
	Since the mollifier $P_\mu$ averages functions 
	on the $\mu^{-1}$ spatial scale, we have roughly
	\begin{align}
	\label{e:x-localized-bernstein}
	\|P_\mu f \|_{L^2( \Sigma_j)} ``\lesssim'' \mu\| f\|_{L^2 L^1( \Sigma_j)}.
	\end{align}
	This is not quite accurate since the kernel of $P_\mu$ is not compactly 
	supported. However, by partitioning the kernel one can decompose $P_\mu = 
	\sum_k 
	P_\mu^k$ where for any function $f(x)$ and any set $K \subset \R^2$ one has
	\begin{align*}
	\| P_\mu^k f\|_{L^2( K) } \lesssim_N 2^{-kN} \mu \|f\|_{L^1( K + B(0, 
		2^k\mu^{-1}))} \text{ for any } N.
	\end{align*}
	Hence in the sequel we shall ignore the imprecision in the above 
	Bernstein estimate.
	
	We write
	\begin{align*}
	\| P_\mu Q ( \phi_{\theta, \lambda}^\alpha u, v_{-\theta, \lambda}^\alpha) 
	\|_{L^2}^2 
	&\lesssim 
	\sum_{j}\|   P_\mu Q ( \phi_{\theta, \lambda}^\alpha u, v_{-\theta, 
		\lambda}^\alpha )
	\|_{L^2(\Sigma_j)}^2.
	\end{align*}
	Using the ``space-localized'' Bernstein~\eqref{e:x-localized-bernstein} and 
	the null frame for $\Lambda$, we estimate
	\begin{align*}
	\| P_\mu Q( \tilde{\phi}_{\theta, \lambda}^\alpha u, 
	v_{\theta, \lambda}^\alpha)\|_{L^2(\Sigma_j)} 
	&\lesssim 
	\mu \| Q(\tilde{\phi}_{\theta, \lambda}^\alpha u, 
	v_{\theta, \lambda}^\alpha)\|_{L^2L^1(\Sigma_j)} \\
	&\lesssim  \mu 
	\| L\tilde{\phi}_{\theta, \lambda}^\alpha u \|_{L^2(\Sigma_j)} \| 
	\underline{L} 
	v_{\theta, \lambda}^\alpha 
	\|_{L^\infty L^2 (\Sigma_j)} \\
	&+ \mu \| E \tilde{\phi}_{\theta, \lambda}^\alpha u \|_{L^2( \Sigma_j)} \| 
	E v_{\theta, \lambda}^\alpha \|_{L^\infty L^2(\Sigma_j)} \\
	&+ \mu \| \underline{L} 
	\tilde{\phi}_{\theta}^\alpha u \|_{L^\infty L^2(\Sigma_j)} \| L v_{\theta, 
	\lambda}^\alpha 
	\|_{L^2( 
		\Sigma_j)}.
	\end{align*}
	Apply the characteristic energy estimate~\eqref{e:CE1}
	to $\phi_{\theta, \lambda}^\alpha u$ for the first two terms and to
	$v_{-\theta, \lambda}^\alpha$ for the third term, thus obtaining a factor 
	of 
	$\mu^{-\frac{1}{2}}$. Since each of the resulting
	three products remains localized to $\Sigma_j$ in one factor, we may
	square-sum both sides in $j$ to conclude that
	\begin{align*}
	\| P_\mu Q ( \tilde{\phi}_{\theta, \lambda}^\alpha u, v_{\theta, 
	\lambda}^\alpha) 
	\|_{L^2}
	&\lesssim \mu^{\frac{1}{2}} (\| \nabla
	\tilde{\phi}_{\theta, \lambda}^\alpha u\|_{L^\infty L^2}+ \| \Box u\|_{L^1 
	L^2}) 
	(\|
	\nabla v_{\theta, \lambda}^\alpha \|_{L^\infty L^2} + \| \Box
	v_{-\theta, \lambda}^\alpha \|_{L^1 L^2})\\
	&\lesssim \mu^{\frac{1}{2}} \| \tilde{\phi}_{\theta, \lambda}^\alpha 
	u\|_{X^{1, \frac{1}{2}}_{\lambda, 1}}
	\|v_{-\theta, \lambda}^\alpha \|_{X^{1, \frac{1}{2}}_{\lambda, 1}}.
	\end{align*}
	
	\textbf{Case 3a$+-$}: In this case we omit the spacetime partition and
	directly apply Bernstein, H\"{o}lder, and the following small-angle
	improvements of the above estimates:
	\begin{gather*}
	\lambda \| L \phi_{\theta, \lambda}^\alpha u \|_{L^2} + 
	\lambda^{\frac{1}{2}}   \|E 
	\phi_{\theta, \lambda}^\alpha u\|_{L^2}\lesssim
	\| \phi_{\theta, \lambda}^\alpha u\|_{X^{1, \frac{1}{2}}_{\lambda, 1}}
	\text{ (by Lemma~\ref{e:char-e-commutator1})},\\
	\lambda^{-\frac{1}{2}} \| v_{-\theta, \lambda}^\alpha\|_{X^{1, 
			\frac{1}{2}}_{\lambda, 1}} \text{ 
		(by~\eqref{wp:E:T})},\\
	\Bigl( \sum_{j} \| Lv_{-\theta, \lambda}^\alpha\|_{L^\infty
		L^2(\Sigma_j)}^2\Bigr)^{\frac{1}{2}} \lesssim
	\lambda^{-1} \| v_{-\theta, \lambda}^\alpha\|_{X^{1, \frac{1}{2}}_{\lambda, 
			1}} \text{ 
		(by~\eqref{wp:L:T})}.
	\end{gather*}
	Consequently
	\begin{align*}
	\| P_\mu Q( \phi_{\theta, \lambda}^\alpha u, v_{-\theta, 
		\lambda}^\alpha)\|_{L^2} 
	\lesssim
	\frac{\mu}{\lambda} \| \phi_{\theta, \lambda}^\alpha u \|_{X^{1, 
			\frac{1}{2}}_{\lambda, 1}} \|
	v_{-\theta, \lambda}^\alpha \|_{X^{1, \frac{1}{2}}_{\lambda, 1}}.
	\end{align*}

	\textbf{Case 3b$+\pm $}: For the $L^2$ estimate we also
	use~\eqref{e:x-localized-bernstein}  and invoking the second part of 
	Corollary~\ref{c:char-energy}, with $\beta
	\sim 1$,
	\begin{align*}
	\| P_\mu ( \phi_{\theta, \lambda}^\alpha u 
	v_{-\theta, \lambda}^\alpha)\|_{L^2(\Sigma_j)} 
	&\lesssim 
	\mu \| \phi_{\theta, \lambda}^\alpha u \|_{L^2(\Sigma_j)}
	v_{-\theta, \lambda}^\alpha\|_{L^\infty
		L^2(\Sigma_j)},\\
	&\lesssim   \mu^{\frac{1}{2}} \| \chi_{\theta}^\alpha u\|_{L^2}
	\|v_{-\theta, \lambda}^\alpha\|_{L^\infty L^2(\Sigma_j)},
	\end{align*}
	which can then be square-summed in $j$ and then in $\theta$.

	\textbf{Case 4}: $\mu \le \lambda^{-1/2}, \alpha > \lambda^{-1/2}$. Combine 
	the 
	arguments from Cases 2 and 3. 
	
\


Finally, to deduce~\eqref{dt:pu:u:v:pm} we write $u_\lambda := u^{\pm_1}$, 
$v_\lambda := v^{\pm_2}$, and use the Leibniz rule to 
bound~\eqref{dt:pu:u:v:pm} by
\begin{align*}
	\|P_\mu(\partial_t u_\lambda v_\lambda )\|_{L^2} + \| P_\mu(u \partial_t 
	v_\lambda)\|_{L^2}.
\end{align*}
To estimate the first term, decompose \[u_{\lambda} = P_{\le 
8\lambda}(D_t)u_\lambda + P_{>8\lambda}(D_t) u_\lambda =  
u_{\lambda}^{<\lambda} + 
u_{\lambda}^{>\lambda} \]
The high-frequency piece satisfies the elliptic estimate
\begin{align*}
	\| \nabla_{t,x} u_\lambda^{>\lambda} \|_{L^2} \lesssim \lambda^{-1}( \| 
	\nabla 
	u_\lambda \|_{L^2} + \| \Box u_\lambda \|_{L^2}),
\end{align*}
whose proof is similar to that of Proposition~\ref{p:half-wave}. Then by 
Bernstein and H\"{o}lder, 
\begin{align*}
	\| P_\mu (\partial_t u_\lambda^{>\lambda} v_\lambda)\|_{L^2} &\lesssim \mu 
	\|\partial_t u_\lambda^{>\lambda} \|_{L^2} \| v_\lambda\|_{L^\infty L^2}\\
	&\lesssim \mu \| u_\lambda\|_{X^{0, \frac{1}{2}}_{\lambda, 
	1}} \| v_{\lambda} \|_{X^{0, \frac{1}{2}}_{\lambda, 1}},
\end{align*}
which is more than acceptable. On the other hand, the 
estimate~\eqref{e:HHL-L2}  
\begin{align*}
	\| P_\mu ( \partial_t u_\lambda^{<\lambda} v_\lambda)\|_{L^2} \lesssim 
	\mu^{\frac{1}{2}} \| \partial_t u_\lambda \|_{X_{\lambda, 1}^{0, 
	\frac{1}{2}}} \| v_\lambda\|_{X_{\lambda, 1}^{0, \frac{1}{2}}},
\end{align*}
and by a simple commutation argument  $\|\partial_t 
u_\lambda^{<\lambda} \|_{X^{0, \frac{1}{2}}_{\lambda, 1}} \lesssim \lambda \| 
u_\lambda 
\|_{X^{0, \frac{1}{2}}_{\lambda, 1}}$. 

To treat the term $\|P_\mu(u_\lambda \partial_t v)\|_{L^2}$, we repeat the 
proofs of Case 1b, 2b, 3b, 4 and use \eqref{wp:barL:T} as well as the Bernstein 
estimates~\eqref{reg:coeff:3}, \eqref{reg:coeff:4} for the wave packet 
coefficients.

\subsection{Proof of Proposition \ref{Lemma:bil:inter}} We succesively consider 
the bounds in the proposition:

\

{\bf(1)} \ First we consider the low modulation cases \eqref{bil:inter:lowmod:LH} and \eqref{bil:inter:lowmod:HH}.

Let $ (\chi^j)_j$ be a partition of unity with respect to time intervals $ (I_j)_j $ of length $ \simeq d_{\max}^{-1} $ and let $ (\tilde{\chi}^j)_j, (\tilde{\tilde{\chi}}^j)_j $ be similar families such that $ \tilde{\chi}^j = 1 $ on $ I_j $ and $ \tilde{\tilde{\chi}} =1 $ on the support of $ \tilde{\chi}^j  $.

By rescaling from Prop. \ref{PropX} using \eqref{eq:scaling} we obtain the product estimates for frequency localized functions:
\be
 \label{HLXMap}
X_{\lmd,d_{\max}}^{1,\frac{1}{2}}[I_j] \cdot
X_{\mu,d_{\max}}^{1,\frac{1}{2}}[I_j] \longrightarrow
X_{\lmd',[d_{\max},\mu]}^{1,\frac{1}{2}}[I_j]    
\ee

Note that the regularity of the metric's coefficients improves with the rescaling. 
Suppose, for instance, that $ d_1 \leq d_2 = d_{\max} $. 
Using properties \eqref{time:ort:1}, \eqref{time:ort:2}, \eqref{X:modulations:intervals:eq} we have
\begin{align*}
 \vn{ u_{\lmd,d_1} \cdot v_{\mu,d_2} }_{ X_{\lmd',[d_{\max},\mu]}^{1,\frac{1}{2}}}^2  & \ls  \sum_j  \vn{ \chi^j ( u_{\lmd,d_1} \cdot   v_{\mu,d_2} )}_{ X_{\lmd',[d_{\max},\mu]}^{1,\frac{1}{2}}}^2  \\
&  \ls \sum_j  \vn{ \tilde{\chi}^j  u_{\lmd,d_1} \cdot  \tilde{\chi}^j v_{\mu,d_2} }_{ X_{\lmd',[d_{\max},\mu]}^{1,\frac{1}{2}}[I_j]}^2  \\
& \ls \sum_j \vn{ \tilde{\tilde{\chi}}^j \tilde{\chi}^j  u_{\lmd,d_1}}_{X_{\lmd,d_{\max}}^{1,\frac{1}{2}}[I_j]}^2   \vn{ \tilde{\tilde{\chi}}^j \tilde{\chi}^j v_{\mu,d_2}}_{X_{\mu,d_{\max}}^{1,\frac{1}{2}}[I_j]}^2  \\
& \ls \vn{  u_{\lmd,d_1}}_{X_{\lmd,d_1}^{1,\frac{1}{2}}}^2  \sum_j  \vn{ \tilde{\chi}^j v_{\mu,d_2}}_{X_{\mu,d_2}^{1,\frac{1}{2}}}^2 \ls \vn{  u_{\lmd,d_1}}_{X_{\lmd,d_1}^{1,\frac{1}{2}}}^2   \vn{ v_{\mu,d_2}}_{X_{\mu,d_2}^{1,\frac{1}{2}}}^2. 
\end{align*}

The same argument gives \eqref{bil:inter:lowmod:HH}.


\begin{remark} \label{rk:prod:bfMos}
Note that between the rescaled \eqref{bil:main:LH} and \eqref{HLXMap} we have used the factors $ \big( \frac{d}{\mu} \big)^{\frac{1}{4}} $ to sum over $ d $- which is the modulation of the output. If we choose to keep this factor, the same argument gives decompositions  
$$ P_{\lmd'}(u_{\lmd,d_1} \cdot v_{\mu,d_2})=\sum_{d=d_{\max}}^{\mu} 
w_{\lmd',d}, \qquad P_{\mu} ( u_{\lmd,d_1} \cdot v_{\lmd',d_2} 
)=\sum_{d=d_{\max}d_0/2}^{\mu/2} w_{\mu,d} + w_{\mu,\mu},
$$
for $ 1 \leq d_1,d_2 \leq \mu $, under the assumptions of Prop \ref{Lemma:bil:inter} (1), such that 
\begin{align} \label{opt:bil:inter:lowmod:HL} \vn{w_{\lmd',d}}_{ X_{\lmd',d}^{1,\frac{1}{2}}} & \ls \big( \frac{d}{\mu} \big)^{\frac{1}{4}}   \vn{ u_{\lmd,d_1}}_{X_{\lmd,d_1}^{1,\frac{1}{2}}} \vn{ v_{\mu,d_2} }_{X_{\mu,d_2}^{1,\frac{1}{2}}} \\ 
\label{opt:bil:inter:lowmod:HH}
  \vn{w_{\mu,d} }_{ X_{\mu,d}^{1,\frac{1}{2}}}  & \ls  \frac{\mu}{\lmd} \big( \frac{d}{\mu} \big)^{\frac{1}{4}}
\vn{ u_{\lmd,d_1}}_{X_{\lmd,d_1}^{1,\frac{1}{2}}} \vn{ v_{\lmd',d_2} }_{X_{\lmd',d_2}^{1,\frac{1}{2}}}  \\
 \label{opt:bil:inter:lowmod:HH2}
\vn{w_{\mu,\mu} }_{ \tilde{X}_{\mu,\mu}^{1,\frac{1}{2}}}  &  \ls  \frac{\mu}{\lmd}
\vn{ u_{\lmd,d_1}}_{X_{\lmd,d_1}^{1,\frac{1}{2}}} \vn{ v_{\lmd',d_2} }_{X_{\lmd',d_2}^{1,\frac{1}{2}}}
\end{align}
These will be useful in the proof of the Moser-type estimate.
\end{remark}

{\bf(2)} \ Recall that by Bernstein's inequality and the energy estimate \eqref{energy:X} we have
$$ \vn{v_{\mu}}_{L^{\infty}_{t,x}} \ls \vn{ v_{\mu}}_{ X_{\mu, d_2}^{1,\frac{1}{2}}  }
$$
We now prove \eqref{bil:inter:highmod:LH}. We have
$$  \lmd d_1^{\frac{1}{2}} \vn{ u_{\lmd,d_1} \cdot v_{\mu,d_2} }_{L^2} \ls \big(  \lmd d_1^{\frac{1}{2}}    \vn{ u_{\lmd,d_1}}_{L^2} \big) \vn{v_{\mu,d_2}}_{L^{\infty}} \ls \vn{ u_{\lmd,d_1}}_{X_{\lmd,d_1}^{1,\frac{1}{2}}} \vn{ v_{\mu,d_2} }_{X_{\mu,d_2}^{1,\frac{1}{2}}}. 
$$
For $ \Box_{ g_{<\sqrt{\lmd'}}  } (u_{\lmd,d_1} \cdot v_{\mu,d_2}   ) $ we consider 
\begin{align*}
& d_1^{-\frac{1}{2}} \vn{  \Box_{ g_{<\sqrt{\lmd}} }u_{\lmd,d_1} \cdot v_{\mu,d_2} }_{L^2} \ls d_1^{-\frac{1}{2}} \vn{ \Box_{ g_{<\sqrt{\lmd}} }u_{\lmd,d_1}}_{L^2} \vn{v_{\mu,d_2}}_{L^{\infty}}  \\ 
&  d_1^{-\frac{1}{2}} \vn{ ( \Box_{ g_{<\sqrt{\lmd'}}} - \Box_{ g_{<\sqrt{\lmd}} }  )u_{\lmd,d_1} \cdot v_{\mu,d_2} }_{L^2} \ls \lmd \vn{ u_{\lmd,d_1}}_{L^2} \vn{v_{\mu,d_2}}_{L^{\infty}} \\
& d_1^{-\frac{1}{2}} \vn{ u_{\lmd,d_1} \cdot  \Box_{ g_{<\sqrt{\mu}} } v_{\mu,d_2} }_{L^2}  \ls d_1^{-\frac{1}{2}} \vn{ u_{\lmd,d_1}}_{L^{\infty} L^2} \vn{\Box_{ g_{<\sqrt{\mu}} } v_{\mu,d_2}}_{L^2 L^{\infty}}  \\ 
& \qquad \qquad \qquad \qquad \qquad \quad \ls  \frac{ \mu}{\lmd} \vn{ \nabla u_{\lmd,d_1}}_{L^{\infty} L^2} d_2^{-\frac{1}{2}} \vn{\Box_{ g_{<\sqrt{\mu}} } v_{\mu,d_2}}_{L^2 }     \\
& d_1^{-\frac{1}{2}} \vn{ u_{\lmd,d_1} \cdot ( \Box_{ g_{<\sqrt{\lmd'}}} - \Box_{ g_{<\sqrt{\mu}} }  ) v_{\mu,d_2} }_{L^2} \ls  d_1^{-\frac{1}{2}} \vn{ u_{\lmd,d_1}}_{L^2} \mu \vn{ v_{\mu,d_2}}_{L^{\infty}} \\
& d_1^{-\frac{1}{2}} \vn{ \partial u_{\lmd,d_1} \cdot \partial v_{\mu,d_2} }_{L^2} \ls d_1^{-\frac{1}{2}} \vn{ \partial u_{\lmd,d_1}}_{L^2} \vn{\partial v_{\mu,d_2}  }_{L^{\infty}} \\
& \qquad \qquad \qquad \qquad  \qquad \ls \frac{\mu}{d_1}  \vn{ u_{\lmd,d_1}}_{X_{\lmd,d_1}^{1,\frac{1}{2}}}  \vn{\partial v_{\mu,d_2}  }_{L^{\infty} L^2}
\end{align*}
Each of the five terms above is $ \ls \vn{ u_{\lmd,d_1}}_{X_{\lmd,d_1}^{1,\frac{1}{2}}} \vn{ v_{\mu,d_2} }_{X_{\mu,d_2}^{1,\frac{1}{2}}} $. 

We continue with the proof of \eqref{bil:inter:highmod:HH}. Here we will use Bernstein's inequality 
$$ \vn{P_{\mu} w}_{L^2_{t,x}} \ls \mu \vn{w}_{L^2 L^1}. $$
Assume without loss of generality that $ d_1=d_{\max} $. We have
\begin{align}
\nonumber
\mu^{\frac{5}{2}}  \vn{  u_{\lmd,d_1} \cdot v_{\lmd',d_2}  }_{ L^2 L^1} & \ls  \mu^{\frac{5}{2}} \vn{u_{\lmd,d_1}}_{L^2} \vn{v_{\lmd',d_2 }}_{L^{\infty} L^2} \\ 
\label{bil:easy1}
& \ls \Big( \frac{\mu}{\lmd} \Big)^2 \Big( \frac{\mu}{d_{\max}} \Big)^{\frac{1}{2}} \vn{ u_{\lmd,d_1}}_{X_{\lmd,d_1}^{1,\frac{1}{2}}} \vn{ v_{\lmd',d_2} }_{X_{\lmd',d_2}^{1,\frac{1}{2}}} .
\end{align}
The same argument applies for the $ \nabla_{t,x} $ in the $ L^2$ part of $ \tilde{X}_{\mu,\mu}^{1,\frac{1}{2}} $. For the $ L^{\infty} L^2 $ parts we use Bernstein, the chain rule and $ L^{\infty} L^2 \times L^{\infty} L^2 \to L^{\infty} L^1 $. 

For $  \Box_{ g_{<\sqrt{\mu}} }( u_{\lmd,d_1} \cdot v_{\lmd',d_2}  ) $
we split as before 
\begin{align*}
& \vn{  \Box_{ g_{<\sqrt{\lmd}} } u_{\lmd,d_1} \cdot v_{\lmd',d_2}
                 }_{L^2 L^1} \ls \vn{ \Box_{ g_{<\sqrt{\lmd}}
                 }u_{\lmd,d_1}}_{L^2} \vn{v_{\lmd',d_2}  }_{L^{\infty}
                 L^2}  \\
&  \vn{   u_{\lmd,d_1} \cdot \Box_{ g_{<\sqrt{\lmd'}} } v_{\lmd',d_2}
                                                 }_{L^2 L^1} \ls \vn{
                                                 u_{\lmd,d_1}}_{L^{\infty}
                                                 L^2} \vn{\Box_{
                                                 g_{<\sqrt{\lmd'}} }
                                                 v_{\lmd',d_2}
                                                 }_{L^2} \\
& \vn{  \partial u_{\lmd,d_1} \cdot \partial v_{\lmd',d_2}  }_{L^2
                                                 L^1} \ls  \vn{
                                                 \partial
                                                 u_{\lmd,d_1}}_{L^2}
                                                 \vn{ \partial
                                                 v_{\lmd',d_2}
                                                 }_{L^{\infty} L^2}
                                               \\
& \vn{ ( \Box_{ g_{<\sqrt{\mu}} }- \Box_{ g_{<\sqrt{\lmd}} } )
            u_{\lmd,d_1} \cdot v_{\lmd',d_2}  }_{L^2 L^1} \ls
            \frac{\lmd}{\mu}  \vn{ \partial u_{\lmd,d_1}}_{L^2}
            \vn{v_{\lmd',d_2}  }_{L^{\infty} L^2}  \\
&  \vn{   u_{\lmd,d_1} \cdot ( \Box_{ g_{<\sqrt{\mu}} }- \Box_{
                                                    g_{<\sqrt{\lmd'}}
                                                    } ) v_{\lmd',d_2}
                                                    }_{L^2 L^1} \ls
                                                    \vn{
                                                    u_{\lmd,d_1}}_{L^2}
                                                    \frac{\lmd'}{\mu}
                                                    \vn{ \partial
                                                    v_{\lmd',d_2}
                                                    }_{L^{\infty} L^2}
\end{align*}
Each of the five terms times $ \mu^{\frac{1}{2}} $ is
$
  \lesssim  \|u_{\lambda,
            d_1}\|_{X^{1, \frac{1}{2}}_{\lambda, d_1}} \| v_{\lambda',
  d_2} \|_{X^{1, \frac{1}{2}}_{\lambda', d_2} },
$ which completes the proof.

\

\section{The product estimate \texorpdfstring{\eqref{prod:est}}{}} \label{Sec:Prod:est}

\

We now turn our attention to the proof of \eqref{prod:est}. We recall the notations $ X=X^{s,\tht}$ and  $N=X^{s-1,\tht-1} $ with $ \tht=\frac{1}{2}+ \ep' $ and $ s=1+\ep'+ \ep $. The duality property \eqref{duality} states
$$ N=( X^{-\ep'-\ep,\frac{1}{2}-\ep'}+L^2 H^{\frac{1}{2}-2 \ep'-\ep} )'.
$$
Therefore, by duality, to obtain \eqref{prod:est} it suffices to prove
$$
\vn{u^1 \cdot u^2}_{X^{-\ep'-\ep,\frac{1}{2}-\ep'}+L^2 H^{\frac{1}{2}-2 \ep'-\ep}} \ls \vn{u^1}_{X^{s,\tht}} \vn{u^2}_{X^{-\ep'-\ep,\frac{1}{2}-\ep'}+L^2 H^{\frac{1}{2}-2 \ep'-\ep}}.
$$ 
We reduce this estimate to the following bounds:
\begin{align}
\vn{u^1 \cdot u^2}_{L^2 H^{\frac{1}{2}-2 \ep'-\ep}} \ls \vn{u^1}_{X^{s,\tht}} \vn{u^2}_{L^2 H^{\frac{1}{2}-2 \ep'-\ep}} \label{prod:red1} \\
\vn{u^1 \cdot u^2}_{X^{-\ep'-\ep,\frac{1}{2}-\ep'}+L^2 H^{\frac{1}{2}-2 \ep'-\ep}}  \ls \vn{u^1}_{X^{s,\tht}} \vn{u^2}_{X^{-\ep'-\ep,\frac{1}{2}-\ep'}} \label{prod:red2} 
\end{align}

For both estimates we use the Littlewood-Paley trichotomy and reduce to estimates for terms $ P_{\lmd_3} (u^1_{\lmd_1} \cdot u^2_{\lmd_2} ) $. 

\

Be begin with \eqref{prod:red1}. In the high-high to low case $ \lmd_1 \simeq \lmd_2 $ and in the low-high case $ \lmd_1 \ll \lmd_2 \simeq \lmd_3 $ we use H\" older's inequality $ L^{\infty} L^{\infty}   \times L^{2} L^{2} \to   L^{2} L^{2} $ together with Bernstein's inequality $ \tilde{P}_{\lmd_1} L^{\infty} L^2 \to \lmd_1 L^{\infty} L^{\infty} $. In the high-low case $ \lmd_2 \ll \lmd_1 \simeq \lmd_3 $ we use H\" older's inequality $ L^{\infty} L^{2}   \times L^{2} L^{\infty} \to   L^{2} L^{2} $ together with Bernstein's inequality $ \tilde{P}_{\lmd_2} L^{2} L^2 \to \lmd_2 L^{2} L^{\infty} $.

\

Now we turn to the proof of \eqref{prod:red2}. We write 
$$ u_{\lmd_1}^1= \sum_{d_1=1}^{\lmd_1} u_{\lmd_1,d_1}^1  \qquad \qquad \vn{u_{\lmd_1}^1  }_{X^{s,\tht}_{\lmd_1}}^2 \simeq \sum_{d_1=1}^{\lmd_1} \vn{u_{\lmd_1,d_1}^1 }_{X^{s,\tht}_{\lmd_1,d_1}}^2.
$$
and the similar decomposition of $ u_{\lmd_2}^2 $ relative to the space $ X^{-\ep'-\ep,\frac{1}{2}-\ep'} $.

In the low-high case $ \lmd_1 \ll \lmd_2  $ it suffices to prove 
$$
\vn{ u^1_{\lmd_1} \cdot u^2_{\lmd_2}}_{ X^{0,\frac{1}{2}-\ep'}_{\lmd_2}} \ls \frac{1}{\lmd_1^{\ep}} \vn{ u^1_{\lmd_1} }_{ X^{1+\ep'+\ep, \frac{1}{2}+ \ep'}_{\lmd_1}  } \vn{ u^2_{\lmd_2}}_{ X^{0,\frac{1}{2}-\ep'}_{\lmd_2}}.
$$
We estimate $ u_{\lmd_1,d_1}^1 u_{\lmd_2,d_2}^2 $ in $  X^{0,\frac{1}{2}-\ep'}_{\lmd_2, [\max(d_1,d_2),\lmd_1]   }  $ for $ d_1, d_2 \leq \lmd_1 $ using \eqref{bil:inter:lowmod:LH} and computing the weights we note that there is enough room to sum the modulations $ \leq \lmd_1 $. For $ \lmd_1 \leq d_2 \leq \lmd_2 $ we use \eqref{bil:inter:highmod:LH} instead and we obtain square-summability in $ d_2 $. 

In the high-low case $ \lmd_2 \ll \lmd_1 $ we follow the same argument and here we have a better factor including a power of $ \lmd_2/\lmd_1 $
 
In the high-high to low case $ \lmd_1 \simeq \lmd_2 $ we have
$$
\vn{P_{\lmd_3} (u^1_{\lmd_1} \cdot u^2_{\lmd_2} )}_{X_{\lmd_3}^{-\ep'-\ep,\frac{1}{2}-\ep'}+L^2 H^{\frac{1}{2}-2 \ep'-\ep}}  \ls  \frac{1}{\lmd_3^{\ep}} \vn{ u^1_{\lmd_1} }_{ X^{1, \frac{1}{2}+ \ep'}_{\lmd_1}  } \vn{ u^2_{\lmd_2}}_{ X^{0,\frac{1}{2}-\ep'}_{\lmd_2}}.
$$
For $ d_1, d_2 \leq \lmd_3 $ we estimate $ P_{\lmd_3} ( u_{\lmd_1,d_1}^1 u_{\lmd_2,d_2}^2 ) $ using \eqref{bil:inter:lowmod:HH}: with appropriate weights the $ X^{1.\frac{1}{2}}_{\mu,[d_{\max},\mu]} $ bound transfers to $ X_{\lmd_3}^{-\ep'-\ep,\frac{1}{2}-\ep'}  $, while the $ \tilde{X}_{\mu,\mu}^{1.\frac{1}{2}} $ one transfers to $ L^2 H^{\frac{1}{2}-2 \ep'-\ep} $.  
For $ \max(d_1,d_2) \geq \lmd_3  $ we use Bernstein and \eqref{bil:easy1} to place the output into $ L^2 H^{\frac{1}{2}-2 \ep'-\ep} $.

\

\begin{remark}[Higher regularity] Let $ \sigma >s $. By applying the product estimate with a slightly lower $ s'<s $ (say $ s'=1+\ep'+\ep/2 $) and considering the Littlwood-Paley trichotomy in terms of $ u_{\mu} F_{\lmd} $, $ u_{\lmd} F_{\mu} $ ($ \mu \ll \lmd $) and $ P_{\mu} (u_{\lmd} F_{\lmd'}) $ ($ \mu \ls \lmd \simeq \lmd' $) we obtain inequalities of type
\begin{align*}
& \vn{u_{\mu} F_{\lmd}}_{X^{\sg-1,\tht-1}} \simeq \lmd^{\sg-s'} \vn{u_{\mu} F_{\lmd}}_{X^{s'-1,\tht-1}} \ls \mu^{-\ep/2} \vn{u_{\mu}}_{X^{s,\tht}} \vn{F_{\lmd}}_{X^{\sg-1,\tht-1}} \\
& \vn{P_{\mu} (u_{\lmd} F_{\lmd'})}_{X^{\sg-1,\tht-1}} \ls (\mu/\lmd)^{\sg-s} \vn{u_{\lmd}}_{X^{s,\tht}} \vn{F_{\lmd'}}_{X^{\sg-1,\tht-1}}
\end{align*}
Putting all these estimates together with the similar one for $ u_{\lmd} F_{\mu} $ we obtain
\be \label{prod:est:higher}
\vn{u \cdot F}_{X^{\sg-1,\tht-1}} \ls \vn{u}_{X^{s,\tht}} \vn{F}_{X^{\sg-1,\tht-1}} + \vn{u}_{X^{\sg,\tht}} \vn{F}_{X^{s-1,\tht-1}}
\ee
\end{remark}
\begin{remark}[Higher regularity, Corollaries of \eqref{prod:est:higher}] \label{rk:cor:prodest:higher}
Assuming higher regularity of the metric $ g $, by
Remark \ref{X:high:reg}, which gives $ \vn{\Box_g u}_{X^{\sg-1,\tht-1}} \ls \vn{u}_{X^{\sg,\tht}} $ for $ \sg \leq k+1 $ and the identity 
$$ 2 Q_g(u,v)= 2 g^{\al \beta} \pt_{\al} u \pt_{\beta} v = \Box_g(uv) - u \Box_g v-v \Box_g u,	$$
from \eqref{prod:est:higher} and \eqref{alg:sgm} we obtain the null form bound
\be \label{nf:sgm}
\vn{Q_g(u,v)}_{X^{\sg-1,\tht-1}} \ls \vn{u}_{X^{s,\tht}} \vn{v}_{X^{\sg,\tht}} + \vn{u}_{X^{\sg,\tht}} \vn{v}_{X^{s,\tht}}.
\ee
Furthermore, when $ u,v$ are bounded in $ X^{s,\tht} $ we have the Moser estimates in the form $ \vn{\Gamma(u)}_{X^{\sg,\tht}} \ls \vn{u}_{X^{\sg,\tht}}$ and
\be \label{Mos:higher}
\vn{\Gamma(u)-\Gamma(v) }_{X^{\sg,\tht}} \ls \vn{u-v}_{X^{\sg,\tht}} + \vn{u-v}_{X^{s,\tht}} (\vn{u}_{X^{\sg,\tht}}+ \vn{v}_{X^{\sg,\tht}} )
\ee
Using this together with \eqref{prod:est:higher}, \eqref{nf:sgm}, \eqref{nf:est}, \eqref{moser:est} we obtain
\be \label{nonlin:sgm}
\vn{ \Gamma(u) Q_g(u,u) }_{X^{\sg-1,\tht-1}} \ls \vn{u}_{X^{s,\tht}}^2 \vn{u}_{X^{\sg,\tht}}.
\ee
Similarly one obtains a bound for $ \Gamma(u) Q_g(u,u)- \Gamma(v) Q_g(v,v)$ in $X^{\sg-1,\tht-1} $.
\end{remark}

\

\section{The Moser estimate \texorpdfstring{\eqref{moser:est}}{}} \label{Sec:Moser}

\

In this section we prove the nonlinear estimate \eqref{moser:est}, which resembles the Moser-type estimates in the context of Sobolev spaces (sometimes referred to as Schauder estimates) that can be proved using paradifferential calculus and the chain rule (see \cite[Lemma A.9]{tao2006nonlinear}). Here we will use the iterated paradifferential expansion strategy from \cite{TataruWM} to leverage the product estimates from \emph{Section \ref{Sec:alg}}. 

\begin{proposition} \label{Prop:Moser}
Let $ F $ be a smooth bounded function with uniformly bounded derivatives with $ \partial^{(j)} F(0)=0 $ for $ \vm{j} \leq C $. If $ u \in X^{s,\tht} $ then $ F(u) \in  \tilde{X}^{s,\tht} $ and 
\be \label{Moser:weak}
\vn{F(u)}_{ \tilde{X}^{s,\tht}} \ls \vn{u}_{ X^{s,\tht}}  \big(  1 + \vn{u}_{ X^{s,\tht}}^{15}  \big)
\ee
Moreover, $ F(u) \in  X^{s,\tht} $ and 
\be \label{Moser:strong}
\vn{F(u)}_{ X^{s,\tht}} \ls \vn{u}_{ X^{s,\tht}}  \big(  1 + \vn{u}_{ X^{s,\tht}}^{15}  \big)
\ee
\end{proposition}
 
\

\begin{remark} \label{Rk:tildeX:Moser}
The bound \eqref{Moser:weak} should be seen as a key intermediate step in the 
proof of \eqref{Moser:strong}. The space $ \tilde{X}^{s,\tht} $ is defined in Definition \ref{tld:X} and $ X^{s,\tht} \subset \tilde{X}^{s,\tht} $. The only difference between the two spaces occurs at high modulations, where we discarded the terms $ \lmd^{s+\tht-2} \vn{\Box_{g_{<\sqrt{\lmd}} } u_{\lmd,\lmd}}_{L^2} $, making the norm $  \tilde{X}^{s,\tht} $ smaller. This allows us to have the factor $ (\lmd / \lmd_0)^s $ in Lemma \ref{multilinear}, based on the better factors $ \mu / \lmd $ for $ \tilde{X}_{\mu,\mu}^{1,\frac{1}{2}} $ compared to $ X_{\mu,\mu}^{1,\frac{1}{2}} $ in the estimates \eqref{bil:inter:lowmod:HH}, \eqref{bil:inter:highmod:HH}, thus making \eqref{Moser:weak} easier to prove. 
\end{remark}

\

Notation-wise, we focus the proof on the case of functions $ F $ of a scalar argument and note that it is easy to see that the same argument applies for the case of multivariate arguments $ F(u_1, \dots, u_d) $.

\subsection{Reduction of \eqref{Moser:strong} to \eqref{Moser:weak} } \ 
In light of Lemma \ref{lemma:XXtildeBox} it suffices to show that $ \Box_g F(u) \in L^2 H^{s+\tht-2} $:

\begin{lemma}
If $ u \in X^{s,\tht} $ then 
$$ \vn{ \Box_g F(u)}_{ L^2 H^{s+\tht-2}} \ls \vn{u}_{X^{s,\tht}} \big(1+ \vn{u}_{X^{s,\tht}}^3 \big) 
$$
\end{lemma} 

\begin{proof}
We split $  \Box_g F(u) $ into $ F'(u) \Box_g u $ and $ F''(u) Q_g(u,u) $. Let $ G $ be either one of the terms $ F'(u), F''(u) $ and let $ f $ be $  \Box_g u $ or $ Q_g(u,u) $. 

By Lemma \ref{lin:map:Box}, \eqref{nf:est} and \eqref{tld:X} we have 
\be \label{f:in:H}
\vn{f}_{L^2 H^{s+\tht-2}} \ls \vn{u}_{X^{s,\tht}} \big(1+ \vn{u}_{X^{s,\tht}} \big).  \ee
To place $ G f $ in $ L^2 H^{s+\tht-2} $ we use a Littlewood decomposition and bound $ P_{\lmd} ( G f_{\nu} ) $ using \eqref{f:in:H}. When $ \nu \ls \lmd $ we place $ G \in L^{\infty} $ and get the factor $ (\nu / \lmd)^{2-s-\tht} $. When $ \nu \gg \lmd $, for the term $ P_{\lmd} ( G_{\nu} f_{\nu} ) $ we first use Bernstein $ P_{\lmd} : L^2 L^1 \to \lmd L^2 L^2 $, then we use the fact that $ G \in L^{\infty} H^{s} $ by classical Moser/Schauder estimates (see \cite[Lemma A.9]{tao2006nonlinear}), obtaining the factor $ (\lmd / \nu )^{s+\tht-1} $.
\end{proof}

\

We continue towards the proof of \eqref{Moser:weak} by setting up some preliminaries. 

\subsection{Iterated paradifferential expansions} \label{sec:para:exp}
We write

$$ F(u)-F(v) = (u-v) h(v,u) $$ 
and 
$$ h(v,u)-h(x,y)=(v-x) h_1(x,y,v,u)+ (u-y) h_2 (x,y,v,u) $$ 
and so on, where the $ h $'s are generic smooth functions with uniformly bounded derivatives. For $ \nu < \mu \leq \infty $ we may decompose 
\begin{align*}
F(u_{<\mu}) & =F(u_{<\nu})+ \sum_{\lmd_0 = \nu}^{\mu/2} F(u_{<2 \lmd_0})-F(u_{< \lmd_0}) \\ 
& = F(u_{<\nu})+  \sum_{\lmd_0 = \nu}^{\mu/2} u_{\lmd_0} h(u_{< \lmd_0}, u_{< 2 \lmd_0} ).
\end{align*}
Repeating the same argument for $ h(u_{< \lmd_0}, u_{< 2 \lmd_0} ) $ and denoting 
$$ h(..,u_{< \lmd_1},..)= h_1(u_{< \lmd_1}, u_{<  2 \lmd_1}, u_{< 2 \lmd_1}, u_{< 4 \lmd_1}  )  + h_2 (u_{< \lmd_1/2}, u_{<  \lmd_1}, u_{<  \lmd_1}, u_{< 2 \lmd_1} )
$$ 
we further decompose 
\begin{align*}
F(u_{<\mu}) =  F(u_{<\nu})  +  \sum_{\lmd_0 = \nu}^{\mu/2} u_{\lmd_0} h(u_{1}, u_{\leq 2} )+  
  \sum_{\lmd_0 = \nu}^{\mu/2} \sum_{\lmd_1=2}^{\lmd_0/2} u_{\lmd_0} u_{\lmd_1} h(..,u_{< \lmd_1},..)
\end{align*}
Iterating this argument we can write $ F(u_{<\mu}) $ as a sum of three types of terms:
\begin{enumerate}
\item $ F(u_{<\nu}) $
\item $u_{\lmd_0} h(u_{1}, u_{\leq 2} ), \ u_{\lmd_0} u_{\lmd_1} h(u_{1}, u_{\leq 2}, \dots), \dots, u_{\lmd_0} u_{\lmd_1} \cdots  u_{\lmd_{N-1}} h(u_1, \dots, u_{\leq C}) $ 
\item $u_{\lmd_0} u_{\lmd_1} \cdots  u_{\lmd_{N}} h(u_{< \lmd_N/c}, \dots, u_{< \lmd_N}, \dots, u_{< c \lmd_N} ) $
\end{enumerate}
for $ \nu \leq \lmd_0 < \mu $ and $ \lmd_0 \geq \lmd_1 \geq \cdots \geq \lmd_N \geq 1 $.

\subsection{Bilinear estimates involving $ \tilde{X} $ spaces} 
Here we supplement the estimates from Propositions \ref{PropX} and \ref{Lemma:bil:inter} with some bounds in terms of the $ \tilde{X}_{\mu,\mu}^{s,\tht} $ norms.

\begin{lemma} \label{lemma:bil:suppl}  Let $ \mu \leq d < \lmd $ and  $ u_{\mu,\mu} \in \tilde{X}_{\mu,\mu}^{s,\tht} $. Then there exists a decomposition 
\[
u_{\mu,\mu}= u_{\mu,\mu}^{<\lmd} + u_{\mu,\mu}^{> \lmd}
\]
such that 
\begin{align}
\label{dec:lmd:d:mu1}
\vn{u_{\lmd,d} \cdot u_{\mu,\mu}^{<\lmd}}_{X_{\lmd,d}^{s,\tht}} & \ls  \vn{u_{\lmd,d}}_{X_{\lmd,d}^{s,\tht}} \frac{1}{\mu^{\ep}} \vn{u_{\mu,\mu}}_{\tilde{X}_{\mu,\mu}^{s,\tht}}   \\
\label{dec:lmd:d:mu2}
\vn{ u_{\lmd,d} \cdot u_{\mu,\mu}^{>\lmd}}_{\tilde{X}_{\lmd,\lmd}^{s,\tht}} & \ls \Big( \frac{\mu}{d} \Big)^{\tht} \vn{u_{\lmd,d}}_{X_{\lmd,d}^{s,\tht}} \frac{1}{\mu^{\ep}} \vn{u_{\mu,\mu}}_{\tilde{X}_{\mu,\mu}^{s,\tht}}
\end{align}
Moreover, one has
\begin{align}
\label{dec:lmd:mu:mu1}
\vn{u_{\lmd,\leq \mu} \cdot u_{\mu,\mu}^{<\lmd}}_{X_{\lmd,\mu}^{s,\tht}} & \ls  \vn{u_{\lmd,\leq \mu}}_{X_{\lmd,\leq \mu}^{s,\tht}}  \frac{1}{\mu^{\ep}} \vn{u_{\mu,\mu}}_{\tilde{X}_{\mu,\mu}^{s,\tht}}   \\
\label{dec:lmd:lmd:mu2}
\vn{ u_{\lmd,\leq \mu} \cdot u_{\mu,\mu}^{>\lmd}}_{\tilde{X}_{\lmd,\lmd}^{s,\tht}} & \ls  \vn{u_{\lmd,\leq \mu}}_{X_{\lmd,\leq \mu}^{s,\tht}} \frac{1}{\mu^{\ep}} \vn{u_{\mu,\mu}}_{\tilde{X}_{\mu,\mu}^{s,\tht}}
\end{align}
and
\be \label{dec:lmd:lmd:mu}
\vn{u_{\lmd,\lmd} \cdot u_{\mu}}_{\tilde{X}_{\lmd,\lmd}^{s,\tht}} \ls \vn{u_{\lmd,\lmd} }_{\tilde{X}_{\lmd,\lmd}^{s,\tht}}   \frac{1}{\mu^{\ep}} \vn{u_{\mu}}_{\tilde{X}_{\mu}^{s,\tht}}
\ee
\end{lemma}

\begin{proof} We define $ u_{\mu,\mu}^{<\lmd} $ by averaging $ u_{\mu,\mu} $ in time on the $ \lmd^{-1} $ scale. 

{\bf(1)} We begin with \eqref{dec:lmd:d:mu1} and \eqref{dec:lmd:d:mu2}. Similarly to the proof of \eqref{bil:inter:highmod:LH}, for all terms occurring in each of the products we place the higher frequency term in $ L^2 $ and the lower frequency term in $ L^{\infty} $. For \eqref{dec:lmd:d:mu1} this works because 
$$ \vn{\Box_{g_{<\sqrt{\lmd}}} u_{\mu,\mu}^{<\lmd}}_{L^{\infty}} \ls \lmd \mu  \frac{1}{\mu^{\ep}} \vn{u_{\mu,\mu}}_{\tilde{X}_{\mu,\mu}^{s,\tht}},
$$
while for \eqref{dec:lmd:d:mu2} we use
\be \label{mu:mu:highmod}
 \vn{ u_{\mu,\mu}^{>\lmd}}_{L^{\infty}} \ls \lmd^{-1} \vn{ \partial_t u_{\mu,\mu}}_{L^{\infty}} \ls \frac{\mu}{\lmd} \frac{1}{\mu^{\ep}} \vn{u_{\mu,\mu}}_{\tilde{X}_{\mu,\mu}^{s,\tht}}, 
\ee
and we recall that the $ \tilde{X}_{\lmd,\lmd}^{s,\tht} $ norm does not contain $ \Box_{g_{<\sqrt{\lmd}}} $ terms.

{\bf(2)} In the case of \eqref{dec:lmd:mu:mu1} we use $ L^{\infty} L^2 \times L^2 L^{\infty} \to L^2 L^2 $ and Bernstein for the terms $ u_{\lmd,\leq \mu} \cdot u_{\mu,\mu}^{<\lmd} $, $ \nabla_{t,x} u_{\lmd,\leq \mu} \cdot \nabla_{t,x} u_{\mu,\mu}^{<\lmd} $ and $  u_{\lmd,\leq \mu} \cdot \Box_{g_{<\sqrt{\lmd}}} u_{\mu,\mu}^{<\lmd} $.
We use $ L^2 L^2 \times L^{\infty} L^{\infty} \to L^2 L^2 $ for $ \Box_{g_{<\sqrt{\lmd}}} u_{\lmd,\leq \mu} \cdot u_{\mu,\mu}^{<\lmd} $. The place where we use the $ ^{<\lmd} $ smoothness is 
$$
\lmd^{-1} \mu^{\tht} \vn{\Box_{g_{<\sqrt{\lmd}}} u_{\mu,\mu}^{<\lmd}}_{L^2} \ls \mu^{\tht} \vn{ \nabla_{t,x} u_{\mu,\mu}}_{L^2} \ls  \frac{1}{\mu^{\ep}} \vn{u_{\mu,\mu}}_{\tilde{X}_{\mu,\mu}^{s,\tht}}.
$$

Finally, in the case of \eqref{dec:lmd:lmd:mu2} we always place the $ \lmd $-frequency terms in $ L^{\infty} L^2 $, using either $ L^{\infty} L^2 \times L^2 L^{\infty} \to L^2 L^2 $ or $ L^{\infty} L^2 \times L^{\infty} L^{\infty} \to L^{\infty} L^2 $ and Bernstein for the $ \mu $-frequency terms. This works because of \eqref{mu:mu:highmod} and  
$$ \vn{ u_{\mu,\mu}^{>\lmd}}_{L^2 L^{\infty}} \ls \frac{\mu}{\lmd}  \vn{\partial_t u_{\mu,\mu}^{>\lmd}}_{L^2} \ls \frac{\mu}{\lmd} \frac{1}{\mu^{\tht+\ep}}   \vn{u_{\mu,\mu}}_{\tilde{X}_{\mu,\mu}^{s,\tht}}. $$

{\bf(3)} The proof of \eqref{dec:lmd:lmd:mu} is straightforward by bounding the $ \mu $-frequency terms in $ L^{\infty} $.
\end{proof}

\subsection{Multilinear estimates} 
The next order of business is to obtain effective bounds for the products $u_{\lmd_0} u_{\lmd_1} \cdots  u_{\lmd_{N}} $ in the expansion from section \ref{sec:para:exp}, while the effect of multiplying by $ h(.. u_{< \lmd_N}..) $ is studied in the next subsection.

\begin{lemma} \label{multilinear}
Let $ \lmd_0 \geq \lmd_1\geq \dots \geq \lmd_N $ for $ N \geq 1 $ and let $ u \in \tilde{X}^{s,\tht} $. For any $ \lmd \ls \lmd_0 $ 
\be  \label{multilinear:est:sumd}
\vn{P_{\lmd} \big( u_{\lmd_0}  \dots  u_{\lmd_N} \big)  }_{\tilde{X}^{s,\tht}_{\lmd}} \ls 
\frac{\lmd^s}{\lmd_0^s}  \vn{u_{\lmd_0}}_{\tilde{X}_{\lmd_0}^{s,\tht}} \prod_{i=1}^{N} \frac{1}{\lmd_i^\ep} \vn{u_{\lmd_i} }_{\tilde{X}_{\lmd_i}^{s,\tht}} 
\ee
More precisely, there exists a decomposition 
$$ P_{\lmd} \big( u_{\lmd_0}  \dots  u_{\lmd_N} \big) = \sum_{d=1}^{\lmd}  v_{\lmd,d} $$ 
such that
\be  \label{multilinear:est}
\vn{v_{\lmd,d}}_{\tilde{X}_{\lmd,d}^{s,\tht}} \ls \frac{\lmd^s}{\lmd_0^s} \prod_{i=1}^{\min(3,N)} \min \Big( 1,\frac{d}{ \min(\lmd, \lmd_i)} \Big) ^{\frac{1}{4}} \vn{u_{\lmd_0}}_{\tilde{X}_{\lmd_0}^{s,\tht}} \prod_{i=1}^{N} \frac{1}{\lmd_i^\ep} \vn{u_{\lmd_i} }_{\tilde{X}_{\lmd_i}^{s,\tht}} 
\ee
Moreover, if $ \lmd > d > \lmd_1 $ then we can replace $ \vn{u_{\lmd_0}}_{\tilde{X}_{\lmd_0}^{s,\tht}} $ by $ \vn{u_{\lmd_0,d}}_{X_{\lmd_0, d}^{s,\tht}} $.
\end{lemma}

\begin{remark}
We recall that for $ 1 \leq d < \lmd $ the norms of $ \tilde{X}_{\lmd,d}^{s,\tht} $ and 
$ X_{\lmd,d}^{s,\tht} $ coincide. 
\end{remark}

\

\begin{corollary} \label{multilinear:cor}
Under the assumptions of Lemma \ref{multilinear}, for $ N \geq 3$ and any $ \gamma  \leq \min(\lmd,\lmd_3) $, one has
\be  \label{multilinear:est:cor}
\vn{v_{\lmd,\leq \gamma }}_{\tilde{X}_{\lmd,\gamma}^{s,\tht}} \ls \frac{\lmd^s}{\lmd_0^s}  
\vn{u_{\lmd_0}}_{\tilde{X}_{\lmd_0}^{s,\tht}} \prod_{i=1}^{N} \frac{1}{\lmd_i^\ep} \vn{u_{\lmd_i} }_{\tilde{X}_{\lmd_i}^{s,\tht}}
\ee
\end{corollary}

\

\

\noindent
\emph{Proof of Corollary \ref{multilinear:cor}.}
Using \eqref{multilinear:est} with $ N \geq 3$ to sum the norms of $ \vn{v_{\lmd,d}}_{L^2}$ and $ \vn{\Box_{g_{<\sqrt{\lmd}}} v_{\lmd,d}}_{L^2}$ over $ d \in \overline{1,\gamma} $ we obtain a favorable factor, since $\big(\frac{d}{\lmd_{\min}} \big) ^{\frac{3}{4}} $ can be used to always have a positive power of $ d $ on the RHS. When $ \gamma=\lmd $ we use energy estimates.
\hfill\(\Box\)

\

\noindent
\emph{Proof of Lemma \ref{multilinear}}.
We prove the statement by induction with respect to $ N $. 

{\bf(1)} For $ N=1 $ we use Remark \ref{rk:prod:bfMos}, Prop. \ref{Lemma:bil:inter} (2) and Lemma \ref{lemma:bil:suppl} to obtain the decomposition. We split
$$ P_{\lmd} (u_{\lmd_0} u_{\lmd_1}) = \sum_{1 \leq d_i \leq \lmd_i,i=0,1} P_{\lmd} (u_{\lmd_0,d_0} u_{\lmd_1,d_1} )
$$
where for both $ \lmd_i, \ i=0,1 $  we write, by Definition \ref{tld:X},
$$ u_{\lmd_i}= \sum_{d_i=1}^{\lmd_i/2} u_{\lmd_i,d_i}+u_{\lmd_i,\lmd_i}  \qquad \quad \vn{u_{\lmd_i}}_{X^{s,\tht}_{\lmd_i}}^2 \simeq \sum_{d_i=1}^{\lmd_i/2} \vn{u_{\lmd_i,d_i}}_{X^{s,\tht}_{\lmd_i,d_i}}^2+ \vn{u_{\lmd_i,\lmd_i}}_{\tilde{X}^{s,\tht}_{\lmd_i,\lmd_i}}^2.
$$
We obtain the desired estimate by a summation of $ d_0,d_1 $ argument as in the proof of Prop. \ref{prop:alg:prop}, as follows:

\begin{enumerate}
\item In the high-low case $ \lmd \simeq \lmd_0 \gg \lmd_1 $: for $ d_0,d_1 \leq d < \lmd_1 $ we use 
\eqref{opt:bil:inter:lowmod:HL} to obtain
\be \label{mod:high:low} \vn{v_{\lmd,d}}_{X_{\lmd,d}^{s,\tht}} \ls  \Big( \frac{d}{\lmd_1} \Big) ^{\frac{1}{4}} \frac{1}{\lmd_1^{\ep}} \vn{u_{\lmd_0,\leq d}}_{X_{\lmd_0,\leq d}^{s,\tht}} \vn{u_{\lmd_1,\leq d} }_{X_{\lmd_1,\leq d}^{s,\tht}} 
\ee
while when $ \lmd_1 < d=d_0 < \lmd, \lmd_0   $ we have
\be \label{mod:high:low2} \vn{v_{\lmd,d}}_{X_{\lmd,d}^{s,\tht}} \ls 
\frac{1}{\lmd_1^{\ep}} \vn{u_{\lmd_0, d}}_{X_{\lmd_0, d}^{s,\tht}} \vn{u_{\lmd_1} }_{\tilde{X}_{\lmd_1}^{s,\tht}} 
\ee
based on \eqref{bil:inter:highmod:LH}, \eqref{dec:lmd:d:mu1} and defining $ v_{\lmd,d}=P_{\lmd} ( u_{\lmd_0,d} u_{\lmd_1,<\lmd_1} )+P_{\lmd} ( u_{\lmd_0,d} u_{\lmd_1,\lmd_1}^{<\lmd_0} )  $, where $ u_{\lmd_1,\lmd_1}^{<\lmd_0} $ is defined in Lemma \ref{lemma:bil:suppl}.

When $ d=\lmd_1 $ we define $ v_{\lmd,\lmd_1}=P_{\lmd} ( u_{\lmd_0,\lmd_1} u_{\lmd_1,<\lmd_1} )+P_{\lmd} ( u_{\lmd_0,\leq \lmd_1} u_{\lmd_1,\lmd_1}^{<\lmd_0} )  $ and using \eqref{bil:inter:highmod:LH}, \eqref{dec:lmd:mu:mu1} we have
\be \label{mod:high:low3}
\vn{ v_{\lmd,\lmd_1}}_{X_{\lmd,\lmd_1}^{s,\tht}} \ls \frac{1}{\lmd_1^{\ep}} \vn{u_{\lmd_0, \leq \lmd_1}}_{X_{\lmd_0, \leq \lmd_1}^{s,\tht}} \vn{u_{\lmd_1} }_{\tilde{X}_{\lmd_1}^{s,\tht}} 
\ee
Finally, when $ d=\lmd $ we define $ v_{\lmd,\lmd}= P_{\lmd} ( u_{\lmd_0, < \lmd_0} u_{\lmd_1,\lmd_1}^{>\lmd_0} ) + \sum_{d \simeq \lmd_0} P_{\lmd_0} ( u_{\lmd_0,d} u_{\lmd_1} )
$
and using \eqref{dec:lmd:d:mu2}, \eqref{dec:lmd:lmd:mu2}, \eqref{dec:lmd:lmd:mu} we have
\be \label{mod:high:low4}
  \vn{v_{\lmd,\lmd} }_{\tilde{X}_{\lmd,\lmd}^{s,\tht}} \ls \frac{1}{\lmd_1^{\ep}}  \vn{u_{\lmd_0}}_{\tilde{X}_{\lmd_0}^{s,\tht}} \vn{u_{\lmd_1}}_{\tilde{X}_{\lmd_1}^{s,\tht}} 
\ee

\item In the high-high to low case $ \lmd \ls \lmd_0 \simeq \lmd_1 $: for $ d \in [1,\lmd) $, by \eqref{opt:bil:inter:lowmod:HH} we have
\be \label{mod:high:high1}
\vn{v_{\lmd,d}}_{X_{\lmd,d}^{s,\tht}} \ls \frac{\lmd^s}{\lmd_0^s}  \Big(\frac{d}{\lmd} \Big)^{\frac{1}{4}}  \frac{1}{\lmd_1^{\ep}}   \vn{u_{\lmd_0,\leq d}}_{X_{\lmd_0,\leq d}^{s,\tht}} \vn{u_{\lmd_1,\leq d} }_{X_{\lmd_1,\leq d}^{s,\tht}} 
\ee
while for $ d=\lmd $, $ v_{\lmd,\lmd} $ contains the contributions of $ u_{\lmd_0,d_0} u_{\lmd_1,d_1} $ when $ \max(d_0,d_1) \geq \lmd $ (use \eqref{bil:inter:highmod:HH}) and the contributions when $ d_0,d_1 \leq \lmd $ given by \eqref{opt:bil:inter:lowmod:HH2} obtaining 
\be \label{mod:high:high2}
\vn{v_{\lmd,\lmd}}_{\tilde{X}_{\lmd,\lmd}^{s,\tht}} \ls \frac{\lmd^s}{\lmd_0^s}   \frac{1}{\lmd_1^{\ep}}   \vn{u_{\lmd_0}}_{\tilde{X}_{\lmd_0}^{s,\tht}} \vn{u_{\lmd_1} }_{\tilde{X}_{\lmd_1}^{s,\tht}} 
\ee
\end{enumerate}

{\bf(2)} Now we assume the statement holds for $ N-1 $ and prove it for $ N $. Letting $ v^{(N-1)}=u_{\lmd_0}  \dots  u_{\lmd_{N-1}} $, summing the induction hypothesis for $ v^{(N-1)}_{\nu,d'} $ over $ d' \leq d $ for $ d \leq \min(\nu,\lmd_1) $ we obtain
\be \label{multilinear:est:leqmod}
\vn{ v^{(N-1)}_{\nu,\leq d} }_{\tilde{X}_{\nu,\leq d}^{s,\tht}} \ls \frac{\nu^s}{\lmd_0^s} \prod_{i=1}^{\min(3,N-1)} \min \Big( 1,\frac{d}{ \min(\nu, \lmd_i)} \Big) ^{\frac{1}{4}} \vn{u_{\lmd_0}}_{\tilde{X}_{\lmd_0}^{s,\tht}} \prod_{i=1}^{N-1} \frac{1}{\lmd_i^\ep} \vn{u_{\lmd_i} }_{\tilde{X}_{\lmd_i}^{s,\tht}} 
\ee

We consider three cases:

\begin{enumerate}
\item $ \lmd \ll \lmd_N. $ We have 
$$ P_{\lmd}  \big( u_{\lmd_0}  \dots  u_{\lmd_N} \big) = P_{\lmd}  \big( \tilde{P}_{\lmd_N} v^{(N-1)}  u_{\lmd_N} \big).
$$ 
For $ d<\lmd $ we apply \eqref{mod:high:high1} to $ \tilde{P}_{\lmd_N} v^{(N-1)} $ and $ u_{\lmd_N} $ and then use \eqref{multilinear:est:leqmod} for $ \nu \simeq \lmd_N $. When $ d=\lmd $ we use \eqref{mod:high:high2} instead of \eqref{mod:high:high1}, and then \eqref{multilinear:est:sumd} for $ N-1 $ and $ P_{\nu} $ for $ \nu \simeq \lmd_N $, as no powers of $ d/\lmd $ are necessary. 

\item $ \lmd \simeq \lmd_N. $ We decompose
$$ P_{\lmd}  \big( u_{\lmd_0}  \dots  u_{\lmd_N} \big) = \sum_{\nu \ls \lmd } P_{\lmd}  \big( P_{\nu} v^{(N-1)}  u_{\lmd_N} \big)
$$
For $ d \leq \nu $ we apply \eqref{mod:high:low} and \eqref{mod:high:low3} for $ u_{\lmd_N} $ and $ P_{\nu} v^{(N-1)} $ together with \eqref{multilinear:est:leqmod}, while for $ d > \nu $ we use \eqref{mod:high:low2}, \eqref{mod:high:low4}  together with \eqref{multilinear:est:sumd} for $ P_{\nu} v^{(N-1)} $. Summing the factors that contain $ \nu $ we obtain
$$ \sum_{\nu \ls \lmd } \frac{\nu^s}{\lmd_0^s} \frac{1}{\nu^{\ep}} \min \Big( 1,\frac{d}{ \nu} \Big) ^{\frac{1}{4} \min(3,N) } \simeq \frac{\lmd^s}{\lmd_0^s} \frac{1}{\lmd^{\ep}} \Big( \frac{d}{ \lmd} \Big)^{\frac{1}{4} \min(3,N) } 
$$
which is favorable.

\item $ \lmd \gg \lmd_N. $ We write
$$ P_{\lmd}  \big( u_{\lmd_0}  \dots  u_{\lmd_N} \big) = P_{\lmd}  \big( \tilde{P}_{\lmd} v^{(N-1)}  u_{\lmd_N} \big).
$$ 
For $ d \leq \lmd_N $ we apply \eqref{mod:high:low} and \eqref{mod:high:low3} to $ \tilde{P}_{\lmd} v^{(N-1)}  $ and $ u_{\lmd_N} $ and then use \eqref{multilinear:est:leqmod} for $ \nu \simeq \lmd $. In the case $ \lmd_N < d < \lmd $  \eqref{mod:high:low2} applies with $ \vn{v^{(N-1)}_{\nu,d}}_{X_{\nu,d}^{s,\tht}} $ on the RHS, $ \nu \simeq \lmd $, and this norm is estimated by the induction hypothesis. Finally, when $ d=\lmd $ we use \eqref{mod:high:low4} and then \eqref{multilinear:est:sumd} for $ N-1 $.
 \hfill\(\Box\)
\end{enumerate}

For the sum of the output frequencies and modulations of the products we have:

\begin{corollary} \label{multilinear:cor2}

For $ N \geq 4 $, let $ \lmd_0 \geq \lmd_1\geq \dots \geq \lmd_N $  and let $ u \in X^{s,\tht} $. Denoting $ w=u_{\lmd_0}  \dots  u_{\lmd_N} $ and $ M=\vn{u_{\lmd_0}}_{\tilde{X}_{\lmd_0}^{s,\tht}} \prod_{i=1}^{N} \frac{1}{\lmd_i^\ep} \vn{u_{\lmd_i} }_{\tilde{X}_{\lmd_i}^{s,\tht}} $, one has 
$$
\vn{w}_{L^2}+ \lmd_N^{-\frac{1}{2}} \vn{w}_{L^{\infty} L^2}  + \lmd_N^{-1} \vn{\nabla_{t,x} w}_{L^2} + \lmd_N^{-\frac{3}{2}} \vn{\nabla_{t,x} w}_{L^{\infty} L^2}
 \ls \lmd_N^{-s-\tht} M
$$
\end{corollary}

\begin{proof} 
First consider $ \lmd_N \ls \lmd $. By summing \eqref{multilinear:est} and using also the argument of Cor. \ref{multilinear:cor}, together with \eqref{tildeX:loc}, \eqref{energy:X} we obtain
\be  \label{Y:fixed:freq}
\vn{w_{\lmd}}_{L^2}+ 
\lmd_N^{-\frac{1}{2}}  \vn{w_{\lmd}}_{L^{\infty} L^2} +
\lmd^{-1} \vn{\nabla_{t,x} w_{\lmd}}_{L^2}  + \lmd_N^{-\frac{1}{2}} \lmd^{-1} \vn{\nabla_{t,x} w_{\lmd}}_{L^{\infty} L^2}
\ls
\lmd^{-s} \lmd_N^{-\tht} M 
\ee 
Summing in $ \lmd \geq \lmd_N/C$ one obtains
$$ \vn{w_{\geq \lmd_N}}_{L^2}+ 
\lmd_N^{-\frac{1}{2}}  \vn{w_{\geq \lmd_N}}_{L^{\infty} L^2} +
\lmd_N^{-1} \vn{\nabla_{t,x} w_{\geq \lmd_N}}_{L^2} + \lmd_N^{-\frac{3}{2}} \vn{\nabla_{t,x} w_{\geq \lmd_N}}_{L^{\infty} L^2}
\ls
\lmd_N^{-s-\tht} M
$$
Now we consider frequencies $ \ll \lmd_N $ and write
$$
P_{\ll \lmd_N } w= \tilde{P}_{\lmd_N} \big( u_{\lmd_0}  \dots  u_{\lmd_{N-1}} \big) u_{\lmd_N}.
$$
We place $ u_{\lmd_N}, \nabla_{t,x} u_{\lmd_N} $ in $ L^{\infty}$ while for $ \tilde{P}_{\lmd_N} \big( u_{\lmd_0}  \dots  u_{\lmd_{N-1}} \big)$ we use the argument of the previous step used to obtain \eqref{Y:fixed:freq}. 
\end{proof}

\subsection{Multiplicative properties} 
Now that we have bounds for multilinear terms in $ \tilde{X}_{\lmd,d}^{s,\tht} $, in order to make use of the expansion in terms $u_{\lmd_0} u_{\lmd_1} \cdots  u_{\lmd_{N}} h(.. u_{< \lmd_N}.. ) $ we need to investigate which multiplications by $ h $ leave the $ \tilde{X}_{\lmd,d}^{s,\tht} $ spaces unchanged.

\begin{definition}
For $ d \leq \lmd $ the norm of $ M_{\lmd,d} $ is defined by
$$ \vn{h}_{M_{\lmd,d}} = \vn{h}_{L^{\infty}}+\frac{1}{d} \vn{\nabla_{t,x} h}_{L^{\infty}}+ \frac{1}{d \lmd} \vn{\Box_{g_{< \sqrt{\lmd}}} h}_{  L^{\infty}+ d^{-\frac{1}{2}} L^2 L^{\infty} } 
$$
and for $ \mu \leq \lmd $ the norm of $ N_{\lmd,\mu} $ is defined by 
$$
\vn{h}_{N_{\lmd,\mu}} = \vn{h}_{L^{\infty}}+\frac{1}{\lmd} \vn{\nabla_{t,x} h}_{L^{\infty}}+ \frac{1}{\lmd} \vn{\Box_{g_{< \sqrt{\lmd}}} h}_{  L^{\infty} L^2+ \mu^{-\frac{1}{2}} L^2  }. 
$$
We also define the following version of the $ \tilde{X}_{\lmd,\lmd}^{s,\tht} $ norm adapted for function not assumed to be frequency localized:
$$
\vn{w}_{Y_{\lmd}^{s,\tht}} = \lmd^{s+\tht} \big(
\vn{w}_{L^2}+ \lmd^{-\frac{1}{2}} \vn{w}_{L^{\infty} L^2} +\lmd^{-1} \vn{\nabla_{t,x} w}_{L^2} + \lmd^{-\frac{3}{2}} \vn{\nabla_{t,x} w}_{L^{\infty} L^2}
\big).
$$
\end{definition}

\

Then we have the following multiplication properties.

\begin{lemma} For any function $ u $ localized at frequency $ \simeq \lmd $ one has
\be \label{mult:prop:1}
\vn{u \cdot h}_{X_{\lmd,d}^{s,\tht}} \ls \vn{u}_{X_{\lmd,d}^{s,\tht}} \vn{h}_{M_{\lmd,d}}.
\ee
and similarly with $ X_{\lmd,d}^{s,\tht} $ replaced by  $ \tilde{X}_{\lmd,d}^{s,\tht} $.

For any function $ v $ localized at frequency $ \simeq \mu \leq \lmd $ one has
\begin{align} 
\label{mult:prop:2}
\vn{v \cdot h}_{X_{\lmd,\mu}^{s,\tht}} & \ls \frac{\lmd^s}{\mu^s} \vn{v}_{X_{\mu,\mu}^{s,\tht}} \vn{h}_{N_{\lmd,\mu}} \\
\label{mult:prop:4}
\vn{v \cdot h}_{X_{\lmd,\mu}^{s,\tht}} & \ls \frac{\lmd^s}{\mu^s} \big[ \vn{v}_{\tilde{X}_{\mu,\mu}^{s,\tht}} + \mu^{\tht} \lmd^{-1} \vn{\partial_t^2 v}_{L^2 H^{s-1}} \big] \vn{h}_{N_{\lmd,\mu}} \\
\label{mult:prop:5}
\vn{v \cdot h}_{\tilde{X}_{\lmd,\lmd}^{s,\tht}} & \ls \lmd^{s+\tht} \vn{ (v,\lmd^{-1} \nabla_{t,x}v )}_{L^2 \cap \mu^{\frac{1}{2}} L^{\infty} L^2} \vn{h}_{N_{\lmd,\mu}}.
\end{align}
For functions $ w $ not assumed to be frequency localized one has:
\be \label{mult:prop:3}
\vn{w \cdot h}_{\tilde{X}_{\lmd,\lmd}^{s,\tht}} \ls \vn{w}_{Y_{\lmd}^{s,\tht}} \vn{h}_{M_{\lmd,\lmd}}.
\ee 
\end{lemma}
\begin{proof}
The proof is straightforward, based on Leibniz's rule, H\" older's inequality and the energy estimate \eqref{energy:X}. For \eqref{mult:prop:2} we also use Bernstein's inequality $ L^2_x \to \mu L^{\infty}_x $ for $ v $.
\end{proof}

This property is used in conjunction with the following lemma.
\begin{lemma}
Let $ \mu \leq d \leq \lmd $ and $ c \simeq 1 $. For any smooth, bounded function $ h $ with uniformly bounded derivatives one has:
\begin{align}
\label{M:h:est1}
\vn{ P_{\ls \mu} h(u_{<\mu}) }_{M_{\lmd,d}} & \ls 1+ \vn{u}_{X^{s,\tht}}^2
\\
\label{N:h:est}
\vn{ \tilde{P}_{\lmd} h(u_{<c\lmd}) }_{N_{\lmd,\mu}} & \ls 1+ \vn{u}_{X^{s,\tht}}^2
\\
\label{M:h:est2}
\vn{ \tilde{P}_{\lmd} h(u_{<\mu}) }_{M_{\lmd,\lmd}}  & \ls \big( 1+ \vn{u}_{X^{s,\tht}}^2 \big) \max_{0 \leq k \leq 2} \vn{ \tilde{P}_{\lmd} (\partial^k h) (u_{<\mu}) }_{L^{\infty}}, \ \mu \ll \lmd.
\end{align}
The same statement holds for multivariate functions $ h(u_{<\mu/c}, \cdots, u_{<\mu}, \cdots, u_{<c\mu}) $. 
\end{lemma}
\begin{proof}
The same argument applies for the scalar or multivariate case; however, for brevity we will simply use the notation $ h(u_{<\mu}) $ for both cases. 

We begin with the $ \nabla_{t,x} h (\cdot) $ terms since the $ \vn{h}_{L^{\infty}} $ bound is clear. By the chain rule, for $ \nu \in \{ \mu, c \lmd \} $, we write schematically 
$$ \nabla_{t,x} h(u_{<\nu})= \sum \partial h \cdot \nabla_{t,x} u_{<c_i \nu}.
$$
We have $ \vn{ \partial h}_{L^{\infty}} \ls 1$ and $ \vn{ \nabla_{t,x} u_{<c_i \nu}}_{L^{\infty}} \ls \nu \vn{u}_{X^{s,\tht}} $, which suffices for all three bounds. Next, we write 
$$ \Box_{g_{< \sqrt{\lmd}}} P h(u_{<\nu}) = [\Box_{g_{< \sqrt{\lmd}}}, P] h(u_{<\nu})+ P \sum g_{< \sqrt{\lmd}} \partial^2 h \nabla u_{<c_i \nu}  \nabla u_{<c_j \nu} + P \sum \partial h \Box_{g_{< \sqrt{\lmd}}} u_{<c_i \nu}
$$ 
In the case of \eqref{M:h:est2}, since $ \mu \ll \lmd$ for $ P=\tilde{P}_{\lmd} $, the only parts of $ \partial^2 h, \partial h  $ that contribute to this equality are $ \tilde{P}_{\lmd} \partial^2 h, \tilde{P}_{\lmd} \partial h $ for a frequency $ \lmd $ multiplier $ \tilde{P}_{\lmd} $. 
By Bernstein we have
\begin{align*} 
& \vn{g \partial^2 h \nabla u_{<c_i \mu}  \nabla u_{<c_j \mu}}_{L^{\infty}}  \ls \mu^2 \vn{\partial^2 h}_{L^{\infty}} \vn{\nabla u_{\ls \mu}}_{L^{\infty} L^2}^2 \ls \mu^2 \vn{\partial^2 h}_{L^{\infty}} \vn{u}_{X^{s,\tht} }^2 \\
& \vn{g \partial^2 h \nabla u_{<c_i \lmd}  \nabla u_{<c_j \lmd}}_{L^{\infty} L^2}  \ls \lmd \vn{\nabla u_{\ls \lmd}}_{L^{\infty} L^2}^2 \ls \lmd \vn{u}_{X^{s,\tht} }^2
\end{align*}
Similarly, for $ P \in \{ P_{\ls \mu}, \tilde{P}_{\lmd} \} $ by writing $ [\Box_{g_{< \sqrt{\lmd}}}, P]= [g_{< \sqrt{\lmd}},P] \partial_{t,x} \partial_x $ and using a standard commutator estimate we obtain
\begin{align*} 
 \vn{ [\Box_{g_{< \sqrt{\lmd}}}, P_{\ls \mu} ] h(u_{<\mu})  }_{L^{\infty}} & \ls \mu^2 \vn{\partial_{t,x} u_{\ls \mu}}_{L^{\infty} L^2}^2 + \mu  \vn{\partial_{t,x} \partial_x u_{\ls \mu}}_{L^{\infty} L^2} \ls \mu^2 \big( 1+\vn{u}_{X^{s,\tht} }^2 \big) 
\\
 \vn{ [\Box_{g_{< \sqrt{\lmd}}}, \tilde{P}_{\lmd}] h(u_{<\mu})  }_{L^{\infty}} & \ls \mu^2 \vn{\tilde{P}_{\lmd} \partial^2 h}_{L^{\infty}} \vn{\partial_{t,x} u_{\ls \mu}}_{L^{\infty} L^2}^2 + \mu \vn{\tilde{P}_{\lmd} \partial h}_{L^{\infty}} \vn{\partial_{t,x} \partial_x u_{\ls \mu}}_{L^{\infty} L^2} 
 \\
 & \ls \mu^2 \max \big( \vn{\tilde{P}_{\lmd} \partial h}_{L^{\infty}}, \vn{\tilde{P}_{\lmd} \partial^2 h}_{L^{\infty}} \big) \big( 1+\vn{u}_{X^{s,\tht} }^2 \big) 
 \\
 \vn{ [\Box_{g_{< \sqrt{\lmd}}}, \tilde{P}_{\lmd}] h(u_{<c \lmd})  }_{L^{\infty} L^2}  & \ls \lmd \vn{\partial_{t,x} u_{\ls \lmd}}_{L^{\infty} L^2}^2 + \vn{\partial_{t,x} \partial_x u_{\ls \lmd}}_{L^{\infty} L^2} \ls \lmd \big( 1+\vn{u}_{X^{s,\tht} }^2 \big)
\end{align*}
Using Definitions \ref{def:X:1}, \ref{def:X:2} of the $ X_{\lmd,d}^{s,\tht}, X^{s,\tht} $ spaces, the hypothesis $ \partial^2 g \in L^2 L^{\infty} $, Bernstein and a summation argument in frequencies and modulations one obtains
\begin{align*}
\vn{ \Box_{g_{< \sqrt{\lmd}}} u_{<c \mu}}_{ L^2 L^{\infty}} \ls \mu^{\frac{3}{2}} \vn{u}_{X^{s,\tht} } \\
\vn{ \Box_{g_{< \sqrt{\lmd}}} u_{<c \lmd}}_{ L^2 } \ \ \ls \lmd^{\frac{1}{2}} \vn{u}_{X^{s,\tht} }
\end{align*}
Using these we get
\begin{align*}
\vn{ (\tilde{P}) \partial h \Box_{g_{< \sqrt{\lmd}}} u_{<c \mu} }_{L^2 L^{\infty}} & \ls \mu^{\frac{3}{2}} \vn{(\tilde{P}) \partial h}_{L^{\infty}}  \vn{u}_{X^{s,\tht} }  \\
\vn{  \partial h \Box_{g_{< \sqrt{\lmd}}} u_{<c \lmd} }_{L^2 } & \ls \lmd^{\frac{1}{2}} \vn{u}_{X^{s,\tht} }\end{align*}
Putting all the above inequalities together completes the proof of \eqref{M:h:est1}-\eqref{M:h:est2}.
\end{proof}

The main application of the multiplication properties above is:

\begin{lemma} \label{lemma:mainterm}
Let $ \lmd \simeq \lmd_0 \geq \lmd_1\geq \dots \geq \lmd_N $ for $ N \geq 3 $ and $ v=u_{\lmd_0}  \dots  u_{\lmd_N} $. For any smooth, bounded function $ h $ with uniformly bounded derivatives one has:
\be
\vn{ P_{\lmd} \big( v \cdot P_{\ls \lmd_N} h(u_{<\lmd_N}) \big)  }_{\tilde{X}_{\lmd}^{s,\tht}} \ls \vn{u_{\lmd_0}}_{X_{\lmd_0}^{s,\tht}} \big(1+\vn{u}_{X^{s,\tht}}^2 \big)  \prod_{i=1}^{N} \frac{1}{\lmd_i^{\ep/2}} \vn{u_{\lmd_i} }_{X_{\lmd_i}^{s,\tht}}
\ee
\end{lemma}
\begin{proof}
We decompose
$$
P_{\lmd} \big( v \cdot P_{\ls \lmd_N} h(u_{<\lmd_N}) \big) = \sum_{\lmd' \simeq \lmd} P_{\lmd} \big( v_{\lmd'} \cdot P_{\ls \lmd_N} h(u_{<\lmd_N}) \big)+ \sum_{\mu \ll \lmd} P_{\lmd} \big( v_{\mu} \cdot \tilde{P}_{\lmd} P_{\ls \lmd_N}  h(u_{<\lmd_N})
$$ 
where the second sum is non-zero only when $ \lmd_N \simeq \lmd $, in which case we can redenote $ \tilde{P}_{\lmd} P_{\ls \lmd_N} $ by $ \tilde{P}_{\lmd} $. We separately consider the terms with:

\

{\bf(1)} $ \lmd' \simeq \lmd $. We use Lemma \ref{multilinear} and Corollary \ref{multilinear:cor} to decompose $ v_{\lmd'}$. Applying \eqref{mult:prop:1} together with \eqref{M:h:est1} (with $ \mu=\lmd_N, d=\lmd_N, $ resp. $ d>\lmd_N $) and \eqref{multilinear:est:cor}, resp. \eqref{multilinear:est} we obtain
\begin{align*}
\vn{v_{\lmd',\leq \lmd_N} \cdot P_{\ls \lmd_N} h(u_{<\lmd_N})}_{X_{\lmd,\lmd_N}^{s,\tht}} \ls \vn{u_{\lmd_0}}_{X_{\lmd_0}^{s,\tht}} \big(1+\vn{u}_{X^{s,\tht}}^2 \big)  \prod_{i=1}^{N} \frac{1}{\lmd_i^\ep} \vn{u_{\lmd_i} }_{X_{\lmd_i}^{s,\tht}} \\
\vn{v_{\lmd',d} \cdot P_{\ls \lmd_N} h(u_{<\lmd_N})}_{X_{\lmd,d}^{s,\tht}} \ls \Big( \frac{d}{\lmd_1} \Big)^{\frac{1}{4}} \vn{u_{\lmd_0}}_{X_{\lmd_0}^{s,\tht}} \big(1+\vn{u}_{X^{s,\tht}}^2 \big)  \prod_{i=1}^{N} \frac{1}{\lmd_i^\ep} \vn{u_{\lmd_i} }_{X_{\lmd_i}^{s,\tht}}
\end{align*}
for $ d \in [\lmd_N,\lmd_1] $, while for $ d \in (\lmd_1, \lmd)$ we have
$$
\vn{v_{\lmd',d} \cdot P_{\ls \lmd_N} h(u_{<\lmd_N})}_{X_{\lmd,d}^{s,\tht}} \ls  \vn{u_{\lmd_0,d}}_{X_{\lmd_0,d}^{s,\tht}} \big(1+\vn{u}_{X^{s,\tht}}^2 \big)  \prod_{i=1}^{N} \frac{1}{\lmd_i^\ep} \vn{u_{\lmd_i} }_{X_{\lmd_i}^{s,\tht}}
$$
$$
\vn{v_{\lmd',\lmd'} \cdot P_{\ls \lmd_N} h(u_{<\lmd_N})}_{\tilde{X}_{\lmd,\lmd}^{s,\tht}} \ls  \vn{u_{\lmd_0}}_{X_{\lmd_0}^{s,\tht}} \big(1+\vn{u}_{X^{s,\tht}}^2 \big)  \prod_{i=1}^{N} \frac{1}{\lmd_i^\ep} \vn{u_{\lmd_i} }_{X_{\lmd_i}^{s,\tht}}
$$
We square-sum these bounds to obtain the conclusion for the $ v_{\lmd'} $ terms.

\

{\bf(2)} $ \mu \ll \lmd $, in the case $ \lmd_N \simeq \lmd_{N-1} \simeq \dots \simeq \lmd_0 \simeq \lmd $. 
By \eqref{multilinear:est:cor} from Corollary \ref{multilinear:cor} and \eqref{multilinear:est} we have
$$ 
\vn{v_{\mu,<\mu}}_{X_{\mu,\mu}^{s,\tht}} + \vn{v_{\mu,\mu}}_{\tilde{X}_{\mu,\mu}^{s,\tht}} \ls \frac{\mu^s}{\lmd_0^s}  \vn{u_{\lmd_0}}_{X_{\lmd_0}^{s,\tht}} \prod_{i=1}^{N} \frac{1}{\lmd_i^\ep} \vn{u_{\lmd_i} }_{X_{\lmd_i}^{s,\tht}}.
$$
Using this together with \eqref{mult:prop:2} and \eqref{N:h:est} (with $ c \lmd=\lmd_N $) we obtain :
$$
\vn{ v_{\mu,<\mu} \cdot \tilde{P}_{\lmd} h(u_{<\lmd_N}) }_{{X_{\lmd,\mu}^{s,\tht}}} \ls \vn{u_{\lmd_0}}_{X_{\lmd_0}^{s,\tht}} \big(1+\vn{u}_{X^{s,\tht}}^2 \big)  \prod_{i=1}^{N} \frac{1}{\lmd_i^\ep} \vn{u_{\lmd_i} }_{X_{\lmd_i}^{s,\tht}}.
$$
Summing $ \mu \leq \lmd/C $ gives a factor of $ \log \lmd $ which is overpowered by $ \frac{1}{\lmd^{\ep/2}} \simeq \frac{1}{\lmd_i^{\ep/2}} $.

We regularize $ v_{\mu,\mu} $ in time on the $ \lmd^{-1} $ scale to split $ v_{\mu,\mu}= v_{\mu,\mu}^{<\lmd}+ v_{\mu,\mu}^{>\lmd}  $. We have
$$
\vn{\partial_t^2 v_{\mu,\mu}^{<\lmd}}_{L^2 H^{s-1}} \ls \lmd \mu^{-\tht} \vn{v_{\mu,\mu}}_{\tilde{X}_{\mu,\mu}^{s,\tht}},
$$
so we can use this together with \eqref{mult:prop:4}. Finally, we have
$$
 \vn{ (v^{>\lmd},\lmd^{-1} \nabla_{t,x}v^{>\lmd} )}_{L^2 \cap \mu^{\frac{1}{2}} L^{\infty} L^2} \ls \lmd^{-1}  \vn{  \nabla_{t,x}v^{>\lmd} }_{L^2 \cap \mu^{\frac{1}{2}} L^{\infty} L^2} \ls
 \lmd^{-1}  \mu^{1-s-\tht} \vn{v_{\mu,\mu}}_{\tilde{X}_{\mu,\mu}^{s,\tht}}
$$
and we use this with \eqref{mult:prop:5} to conclude. 
\end{proof}

\subsection{Nonlinear low frequency input $ \to $ high frequency output estimates} \

Low $ \to $ high frequency interactions do not occur in bilinear or multilinear expressions, thus one expects their effect to be under control for nonlinear interactions as well. We begin with $ L^{\infty} $ bounds. 

\begin{lemma}
Let $ \lmd \geq \mu $ and let $ h $ be a smooth, bounded function $ h $ with uniformly bounded derivatives. Then, for any even $ N \geq 2 $
\begin{align}
\label{lowinput:infty}
\vn{P_{\lmd} h(u_{< \mu})}_{L^{\infty}} \ls \Big( \frac{\mu}{\lmd} \Big)^N \vn{u}_{X^{s,\tht}} \big( 1+ \vn{u}_{X^{s,\tht}}^{N-1}  \big) \\
\vn{ \nabla_{t,x} P_{\lmd} h(u_{< \mu})}_{L^{\infty}} \ls \mu \Big( \frac{\mu}{\lmd} \Big)^N \vn{u}_{X^{s,\tht}} \big( 1+ \vn{u}_{X^{s,\tht}}^{N}  \big)
\end{align}
\end{lemma}

\begin{proof}
For any nonzero multi-index $ \beta $ we may use the chain rule to write
$$ \partial_x^{\beta} h(u_{< \mu})= \sum_{j=1}^{\vm{\beta}} \sum_{\beta_1+ \dots + \beta_j=\beta} h^{(j)}(u_{< \mu}) \partial_x^{\beta_1} u_{< \mu} \dots \partial_x^{\beta_j} u_{< \mu}, \qquad \beta_i \neq 0.
$$
Using $ \vn{\partial_x^{\beta}  u_{< \mu} }_{L^{\infty}} \ls \mu^{\vm{\beta}} \vn{u}_{X^{s,\tht}} $ we obtain
$$ \vn{P_{\lmd} h(u_{< \mu})}_{L^{\infty}} \ls \frac{1}{\lmd^N} \vn{\Delta^{N/2} P_{\lmd} h(u_{< \mu}) }_{L^{\infty}} \ls  \Big( \frac{\mu}{\lmd} \Big)^N \vn{u}_{X^{s,\tht}} \big( 1+ \vn{u}_{X^{s,\tht}}^{N-1}  \big).
$$
The same argument is used for $ \nabla_{t,x} P_{\lmd} h(u_{< \mu}) $. 
\end{proof}

For very low frequency terms we have:

\begin{lemma} \label{lem:plmd:u1}
Let $ h $ be a smooth, bounded function with uniformly bounded derivatives such that $ h(0)=0 $. Then, for all $ \lmd \geq 1 $ and even $ N \geq 2 $ one has
\be \label{est:plmd:u1}
\vn{\tilde{P}_{\lmd} h(u_1) }_{X_{\lmd,\lmd}^{s,\tht}} \ls \frac{1}{\lmd^N} \vn{u_1}_{X_1^{s,
\tht}} \big( 1+\vn{u_1}_{X_1^{s,\tht}}^{N+1} \big). 
\ee
The same statement holds for $ h(u_{\leq C}) $ and for multivariate functions $ h(u_{1},  u_{\leq 2}, \cdots, u_{\leq 2^k}) $. 
\end{lemma}

\begin{proof}
When $ \lmd=1 $ we use the Lipschitz property $ \vm{h(u_1)} \ls \vm{u_1} $ to control $ \vn{h(u_1)}_{L^2} $. For $ \vn{\Box_{g_{<1}} \tilde{P}_{1} h(u_1)  }_{L^2} $ we use the chain rule.

Now we assume $ \lmd > 1 $. We have
$$ \vn{\tilde{P}_{\lmd} h(u_1) }_{L^2} \ls \frac{1}{\lmd^{N+2}} \vn{ \Delta^{\frac{N}{2}+1} h(u_1) }_{L^2} 
$$
Using the chain rule we write
$$  \Delta^{\frac{N}{2}+1} h(u_1)=\sum h^{(j)} (u_1) \partial_x^{\beta_1} u_1 \cdots \partial_x^{\beta_j} u_1
$$
We use the $ L^2 $ norm for $  \partial_x^{\beta_1} u_1 $ and $ L^{\infty} $ for the other terms. The same type of argument applies for bounding $ \vn{\Box_{g_{<\sqrt{\lmd}}} \tilde{P}_{\lmd} h(u_1)  }_{L^2} $. 
\end{proof}

With all the preparations above we are ready to treat high frequency outputs in high modulation spaces $ \tilde{X}_{\lmd,\lmd}^{s,\tht} $.

\begin{proposition} \label{prop:low:input:freq}
Let $ F $ be a smooth, bounded function with uniformly bounded derivatives such that $ F^{(j)}(0)=0 $ for $ j \leq 4 $. Then, for any even $ N \geq 2 $ and any $ \mu \ll \lmd  $ one has
\be  \label{low:input:freq}
\vn{P_{\lmd} F(u_{<\mu})}_{\tilde{X}_{\lmd,\lmd}^{s,\tht}} \ls  \Big( \frac{\mu}{\lmd} \Big)^N  \big( 1+\vn{u}_{X^{s,\tht}}^{N+8} \big) \sup_{\nu \leq \mu }  \Big( \frac{\nu}{\mu} \Big)^N \vn{u_{\nu}}_{X_{\nu}^{s,\tht}}.  
\ee
We same result applies for multivariate functions $ P_{\lmd} h(u_{< \mu/c}, \dots, u_{< \mu}, \dots, u_{< c \mu} ) $. 
\end{proposition}

\begin{proof}
We use the iterated paradifferential expansion from Section \ref{sec:para:exp} with $ \nu = 2 $ to express  $ F(u_{<\mu}) $ as a sum of terms of type:

{\bf(1)} $ F(u_{1}), u_{\lmd_0} h(u_{1}, u_{\leq 2} ),  \dots, u_{\lmd_0} u_{\lmd_1} u_{\lmd_2}  u_{\lmd_{3}} h(u_1, \dots, u_{\leq C}) $ for $ \lmd_3 \leq \lmd_2 \leq \lmd_1 \leq \lmd_0 \leq \mu \ll \lmd $.
Using \eqref{est:plmd:u1}, \eqref{dec:lmd:lmd:mu} we get 
\begin{align}
\vn{P_{\lmd}  F(u_{1}) }_{\tilde{X}_{\lmd,\lmd}^{s,\tht}} & \ls  
\frac{1}{\lmd^N}  \big( 1+\vn{u}_{X^{s,\tht}}^{N+1} \big) \vn{u_1}_{X_1^{s,\tht}} \nonumber \\
\vn{ u_{\lmd_0} \tilde{P}_{\lmd}  h(u_{1}, u_{\leq 2} )   }_{\tilde{X}_{\lmd,\lmd}^{s,\tht}} & \ls  
\vn{ \tilde{P}_{\lmd}  h(u_{1}, u_{\leq 2} )  }_{ \tilde{X}_{\lmd,\lmd}^{s,\tht}}  \frac{1}{\lmd_0^{\ep}} \vn{ u_{\lmd_0}}_{X_{\lmd_0}^{s,\tht} }  \nonumber \\ 
& \ls  \frac{1}{\lmd^N} \frac{1}{\lmd_0^{\ep}} \vn{ u_{\lmd_0}}_{X_{\lmd_0}^{s,\tht} } \big( 1+\vn{u}_{X^{s,\tht}}^{N+1} \big) \vn{u_{\leq 2}}_{X^{s,\tht}} \nonumber \\
& \dots \label{dots:eq}  \\ 
\vn{u_{\lmd_0} u_{\lmd_1} u_{\lmd_2}  u_{\lmd_{3}} \tilde{P}_{\lmd} h(u_1, \dots, u_{\leq C})}_{\tilde{X}_{\lmd,\lmd}^{s,\tht}} & \ls \frac{1}{\lmd^N} \big( 1+\vn{u}_{X^{s,\tht}}^{N+1} \big) \prod_{i=0}^{3}  \frac{1}{\lmd_i^{\ep}} \vn{ u_{\lmd_i}}_{X_{\lmd_i}^{s,\tht} }  \vn{u_{\leq C}}_{X^{s,\tht}} \nonumber
\end{align}
Then we may sum $ \lmd_0, \dots, \lmd_3 $.

{\bf(2)} $ u_{\lmd_0} u_{\lmd_1} u_{\lmd_2}  u_{\lmd_{3}}  u_{\lmd_{4}} h(u_{< \lmd_4/c}, \dots, u_{< \lmd_4}, \dots, u_{< c \lmd_4} ) $ for $ \lmd_4 \leq \dots \leq \lmd_0 \leq \mu \ll \lmd $.

By abuse of notation we denote 
$$ h(u_{< \lmd_4}) = h(u_{< \lmd_4/c}, \dots, u_{< \lmd_4}, \dots, u_{< c \lmd_4} ), \qquad  w= u_{\lmd_0} u_{\lmd_1} u_{\lmd_2}  u_{\lmd_{3}}  u_{\lmd_{4}} $$
Then 
$$ P_{\lmd} \big( w \cdot h(u_{< \lmd_4}) \big) = P_{\lmd} \big( w \cdot \tilde{P}_{\lmd} h(u_{< \lmd_4}) \big) $$
We discard the $ P_{\lmd}  $ and use \eqref{mult:prop:3}:
$$ \vn{P_{\lmd} \big( w \cdot h(u_{< \lmd_4}) \big) }_{\tilde{X}_{\lmd,\lmd}^{s,\tht}} \ls \vn{w}_{Y_{\lmd}^{s,\tht}} \vn{ \tilde{P}_{\lmd} h(u_{< \lmd_4}) }_{M_{\lmd,\lmd}} 
$$
By Corollary \ref{multilinear:cor2} we obtain
$$   \vn{w}_{Y_{\lmd}^{s,\tht}} \ls \Big( \frac{\lmd}{\lmd_4} \Big)^{s+\tht} 
 \vn{u_{\lmd_0}}_{X_{\lmd_0}^{s,\tht}} \prod_{i=1}^{4} \frac{1}{\lmd_i^\ep} \vn{u_{\lmd_i} }_{X_{\lmd_i}^{s,\tht}}
$$ 
Using \eqref{M:h:est2} together with \eqref{lowinput:infty} (for $ h, \partial h $ and $ \partial^2 h $) we obtain
$$ \vn{ \tilde{P}_{\lmd} h(u_{< \lmd_4}) }_{M_{\lmd,\lmd}} \ls \Big( \frac{\lmd_4}{\lmd} \Big)^{N+2} \vn{u}_{X^{s,\tht}} \big( 1+ \vn{u}_{X^{s,\tht}}^{N+3}  \big)
$$ 
We put these together and sum over all but $ \lmd_4 $. It remains to bound
$$ \sum_{\lmd_4 \leq \mu} \log \big( \frac{\mu}{\lmd_4} \big)  \Big( \frac{\lmd_4}{\lmd} \Big)^{N+2-s-\tht}   \vn{u_{\lmd_4} }   \big( 1+ \vn{u}_{X^{s,\tht}}^{N+8}  \big) \ls
$$
$$
\ls \Big( \frac{\mu}{\lmd} \Big)^N  \Big[ \sup_{\lmd_4 \leq \mu }  \Big( \frac{\lmd_4}{\mu} \Big)^N \vn{u_{\lmd_4}}_{X_{\lmd_4}^{s,\tht}} \Big]  \sum_{\lmd_4 \leq \mu}  \Big( \frac{\lmd_4}{\lmd} \Big)^{2-s-\tht-\ep}   \big( 1+ \vn{u}_{X^{s,\tht}}^{N+8}  \big)
$$
which completes the proof.
\end{proof}

\subsection{Conclusion - Proof of Proposition \ref{Prop:Moser}} \
It remains to prove \eqref{Moser:weak}. We use the iterated paradifferential expansion from section \ref{sec:para:exp} with $ \mu=\infty $ and $ \nu = \lmd/C $ to express  $ F(u) $ as a sum of terms of type:
\begin{enumerate}
\item $ F(u_{\ll \lmd}) $
\item $u_{\lmd_0} h(u_{1}, u_{\leq 2} ), \ u_{\lmd_0} u_{\lmd_1} h(u_{1}, u_{\leq 2}, \dots), \dots, u_{\lmd_0} u_{\lmd_1} u_{\lmd_2}  u_{\lmd_{3}} h(u_1, \dots, u_{\leq C}) $ 
\item $ u_{\lmd_0} u_{\lmd_1} u_{\lmd_2}  u_{\lmd_{3}}  u_{\lmd_{4}} h(u_{< \lmd_4/c}, \dots, u_{< \lmd_4}, \dots, u_{< c \lmd_4} ) $
\end{enumerate}
for $ \lmd \ls \lmd_0 $ and $ \lmd_0 \geq \lmd_1 \geq \cdots \geq \lmd_4 \geq 1 $.

The term $ \vn{P_{\lmd} F(u_{\ll \lmd})}_{\tilde{X}^{s,\tht}_{\lmd}} $ is estimated by \eqref{low:input:freq}. Note that the RHS is square-summable in $ \lmd $. 

For the terms in (2) we use \eqref{multilinear:est:sumd}, for $ \eta \ls \lmd_0 $ 
$$
\vn{P_{\lmd} \big[ \Pi_i u_{\lmd_i} P_{\eta} h(u_1,..) \big]}_{\tilde{X}_{\lmd}^{s,\tht}}  \ls 
\frac{\lmd^s}{\lmd_0^s} \vn{ u_{\lmd_0}}_{\tilde{X}_{\lmd_0}^{s,\tht} }  \vn{P_{\eta} h(u_1,..) \big}_{\tilde{X}_{\eta}^{s,\tht}} 
\prod_{i\neq0}  \frac{1}{\lmd_i^{\ep}} \vn{ u_{\lmd_i}}_{\tilde{X}_{\lmd_i}^{s,\tht} }  
$$
Then we apply Lemma \ref{lem:plmd:u1} and sum in $ \eta, \lmd_i, i \geq 1 $. We have square-summability in $ \lmd $ due to the factors $ \lmd^s \lmd_0^{-s} \vn{ u_{\lmd_0}}_{\tilde{X}_{\lmd_0}^{s,\tht} } $.

We continue with the term (3). For $ \eta \gg \lmd_4 $ we use \eqref{multilinear:est:sumd} and Prop. \ref{prop:low:input:freq} for
$$ P_{\lmd} \big[ u_{\lmd_0} u_{\lmd_1} u_{\lmd_2}  u_{\lmd_{3}}  u_{\lmd_{4}} P_{\eta} h(u_{< \lmd_4/c}, \dots, u_{< \lmd_4}, \dots, u_{< c \lmd_4} ) \big], $$ 
and then sum $ \eta \geq C \lmd_4 $ to obtain
$$
\vn{P_{\lmd} \big[ \Pi_i u_{\lmd_i} P_{\gg \lmd_4} h(..u_{< \lmd_4}.. ) \big]}_{\tilde{X}_{\lmd}^{s,\tht}}  \ls 
\frac{\lmd^s}{\lmd_0^s} \vn{ u_{\lmd_0}}_{X_{\lmd_0}^{s,\tht} } 
 \big( 1+ \vn{u}_{X^{s,\tht}}^{11}  \big)
\prod_{i=1}^{4}  \frac{1}{\lmd_i^{\ep}} \vn{ u_{\lmd_i}}_{X_{\lmd_i}^{s,\tht} }.
$$
It remains to consider
$$  
P_{\lmd} \big[ u_{\lmd_0} u_{\lmd_1} u_{\lmd_2}  u_{\lmd_{3}}  u_{\lmd_{4}} P_{\ls \lmd_4 } h(u_{< \lmd_4/c}, \dots, u_{< \lmd_4}, \dots, u_{< c \lmd_4} ) \big] 
$$
For $ \lmd_0 \simeq \lmd $ we use Lemma \ref{lemma:mainterm}. Now we assume $ \lmd_0 \gg \lmd $, in which case $ \lmd_1 \simeq \lmd_0 $ and
$$
P_{\lmd} \big[ \Pi_i u_{\lmd_i} P_{\ls \lmd_4} h(..u_{< \lmd_4}.. ) \big] = P_{\lmd}  \big[ u_{\lmd_0} \tilde{P}_{\lmd_0} \big(  u_{\lmd_1} u_{\lmd_2}  u_{\lmd_{3}}  u_{\lmd_{4}} P_{\ls \lmd_4} h(..u_{< \lmd_4}.. )   \big) \big] 
$$
We bound this by first using the estimate \eqref{multilinear:est:sumd} for the two terms and then Lemma \ref{lemma:mainterm} applied to $ \tilde{P}_{\lmd_0} \big(  u_{\lmd_1} u_{\lmd_2}  u_{\lmd_{3}}  u_{\lmd_{4}} P_{\ls \lmd_4} h(..u_{< \lmd_4}.. )   \big) $. This concludes the proof of Proposition \ref{Prop:Moser}.  \hfill\(\Box\)

\appendix

\

\section{Smith's wave packets} \
\label{s:smithpackets}
\

We briefly review Smith's wave packet parametrix construction for wave 
equations 
\begin{align*}	(\partial_t^2 - g^{ab}(t, x)\partial_a \partial_b) u = 0, \ 
u[0] = (u_0, 
	u_1)
\end{align*}
with $C^{1,1}$ coefficients~\cite{smith1998parametrix}. This construction works as well for coefficients $g^{ij}$ satisfying $ \partial^2 g \in L^2 L^{\infty} $. Begin with a 
partition of unity on $ \mb{R}^n_{\xi}: $
$$ 1= \vm{h_1(\xi)}^2 + \sum_{\lmd \geq 2} \sum_{\omega} \vm{h_{\lmd}^{\omega} 
(\xi) }^2 $$  
where, for each frequency $ \lmd \geq 2 $, the $ \omega $'s are summed over $ 
\simeq \lmd^{\frac{n-1}{2}} $ directions that are uniformly separated on the 
unit sphere. The smooth functions $ h_{\lmd}^{\omega} (\xi) $  vanish outside 
the annular sectors
$$ \{ \ \vm{\xi} \simeq \lmd, \ \vm{ \xi / \vm{\xi}-\omega } \ls 
\lmd^{-\frac{1}{2}} \ \}
$$ 
and satisfy the natural bounds 
$$  \vm{  (\omega \cdot \nabla_{\xi})^j \partial_{\xi}^{\al} \ 
h_{\lmd}^{\omega} (\xi) } \ls \lmd^{-j-\frac{\vm{\al}}{2}}.
$$
Thus each $ h_{\lmd}^{\omega} (\xi) $ is supported in a rectangle of size $ 
\simeq (\lmd^{\frac{1}{2}})^{n-1} \times \lmd $ oriented in the $ \omega $ 
direction. To each of them we associate a lattice $ \Xi_{\lmd}^{\omega} $ in 
the physical space $ \mb{R}_x^{n} $ on the dual scale, i.e. spaced $ \lmd^{-1} 
$ in the $ \omega $ direction and spaced $ \lmd^{-\frac{1}{2}} $ in directions 
in $ \omega^{\perp} $. For each $ \lmd $ we denote \[ \calT_{\lmd}=\{ 
T=(x,\omega) \ | \ x \in \Xi_{\lmd}^{\omega} \} \] and to each $ T $ we will 
associate a "tube" and a function $ \varphi_T $ defined by
$$  \widehat{\varphi}_T(\xi)= (2 \pi)^{-\frac{n}{2}} \lmd^{-\frac{n+1}{4}} 
e^{-i x \cdot \xi}  h_{\lmd}^{\omega} (\xi).
$$
This function is concentrated in phase space around $ (x, \lmd \omega) $. 
Thus, we have
$$  \vm{  (\omega \cdot \nabla_{y})^j (\omega^{\perp} \cdot \nabla_{y})^{\al} 
\ \varphi_T(y) } \ls \lmd^{j+\frac{\vm{\al}}{2}+\frac{n+1}{4}} 
\frac{1}{(1+\lmd 
\vm{ \omega \cdot (y-x)} + \lmd \vm{y-x}^2  )^N}
$$

\

These properties imply that for any sum we have
$$ \int_{\mb{R}^n } \big| \sum_{\lmd, T\in \calT_{\lmd}}  d_T \varphi_T \big|^2 
\dd 
y 
\ls \sum_{\lmd, T\in \calT_{\lmd}} \vm{d_T}^2.
$$
The family $ (\varphi_T)_{T,\lmd} $ is used to decompose arbitrary functions 
into highly localized components. Indeed, for any $ f \in L^2(\mb{R}^n) $ we 
have
$$ f=\sum_{\lmd, T\in \calT_{\lmd}} c_T \varphi_T , \qquad  c_T \defeq 
\int_{\mb{R}^n } \overline{\varphi_T} f \dd y $$
In addition, we have 
$$  \int_{\mb{R}^n } \vm{f}^2 \dd y = \sum_{\lmd, T\in \calT_{\lmd}} 
\vm{c_T}^2.
$$
Note that if $ f $ is localized at frequencies $ \simeq \lmd $, then it's 
decomposition only contains terms $  \varphi_T $ with $ T \in \calT_{\lmd'} $ 
where $ \lmd' \simeq \lmd $. By abuse of notation we will sometimes write 
simply the sum over $ T \in \calT_{\lmd} $. The decomposition is, in general, 
not unique since the $ \varphi_T $'s are not linearly independent. 

\

Let $ T=(x_T,\omega_T) \in \calT_{\lmd} $ and fix a sign $ \pm $.  Let $ 
(x_T(t),\omega_T(t)) $ be the projection to $S^*(\R^n)$ of the bicharacteristic 
initialized at $ (x_T,\omega_T) $. In other words, we 
set $ \omega_T(t)=\frac{\xi(t)}{\vm{\xi(t)}} $ and then $ (x_T(t),\omega_T(t) 
)$ solves 
\begin{equation}
  \label{e:smith_hamilton_flow}
 \begin{cases}
             \frac{\dd x}{\dd t}=\pm \partial_{\xi} a(t,x,\omega) \\
            \frac{\dd \omega}{\dd t}=\mp \partial_x a(t,x,\omega) \pm \lng 
            \omega, \partial_x a(t,x,\omega) \rng \omega,
       \end{cases} \quad a(t, x, \xi) := (g^{ab}_{<\lambda^{1/2}} \xi_a 
       \xi_b)^{1/2}.
\end{equation}
Define the
orthogonal matrix $\Theta^\pm(t)$ by the ODE
\begin{align*}
  \dot{\Theta^\pm} = \mp\Theta^\pm [ \omega \otimes a_x (t, x,
  \omega) - (a_x) (t, x, \omega ) \otimes
  \omega ],
\end{align*}
where $v \otimes w$ is the linear map $x \mapsto v \langle w, x
\rangle$; by construction, $ \Theta^\pm(t)
 \omega_T(t)=\omega_T $.
Set 
\begin{align*}
u_T^\pm (t,y)&=\varphi_T (\Theta^\pm(t)(y-x_T(t))+x_T),\\
v_T^\pm (t,y)&= \frac{1}{a(0, x_T, \omega_T)} \psi_T (\Theta^\pm(t)(y-x^\pm 
_T(t))+x_T),
\end{align*} 
where
\begin{align*}
\widehat{\psi}_T(\xi) := -\lambda \langle \omega_T, \xi \rangle^{-1} 
\widehat{\varphi}_T(\xi).
\end{align*}

A parametrix for initial data $u[0] := (u_0, u_1)$ is then given by
\begin{align*}
w(t) &:= \frac{1}{2}\sum_{\lambda} \sum_{T \in \mathcal{T}_\lambda} 
(u_T^+(t) + u_T^{-}(t)) \langle 
\varphi_T, u_0\rangle_{L^2_x} \\
&+ \frac{1}{2} \Bigl[\sum_{\lambda \le \lambda_0} \sum_{T \in 
	\mathcal{T}_\lambda} \lambda^{-1} (v_T^+(t) - v_T^{-}(t)) 
\langle \varphi_T, (I + E)u_1 \rangle_{L^2_x} \Bigr] + tP_{<\lambda_0} 
(I+E) 
u_1\\
&=: \mathbf{c}(t, 0) u_0 + \mathbf{s}(t, 0) u_1,
\end{align*}
where $\lambda_0 \gg 1$ is an absolute constant, and the operator $E$ satisfies 
$\| 
E\|_{L^2 \to L^2} \lesssim 
\lambda_0^{-\frac{1}{2}}$ and ensures that $\partial_t 
\mathbf{s}|_{t = 0} = I$. The operators $\mathbf{c}(t, 0)$, $\mathbf{s}(t, 0)$ 
are approximations of the usual wave propagators $\cos t\sqrt{-\Delta}$, 
$(\sqrt{-\Delta})^{-1}\sin t\sqrt{-\Delta}$. 

The functions $u_T^{\pm}$, $\lambda^{-1} v_T^{\pm}$ then satisfy the 
definitions of the 
packets in Section~\ref{s:packets}, and shall hereafter be denoted generically 
by $u_T$. If $(u_0, 
u_1)$ is localized at 
frequency $\lambda \gg \lambda_0$, then the frequency-$\lambda$ part of the 
solution $u$ to $\Box_g u= 0$ is approximated in the sense 
of~\eqref{wp:dec1}, \eqref{wp:dec2}, \eqref{param:small}, 
\eqref{homogeneous:sol}, by retaining just the terms in 
$w$ with 
frequencies comparable to $\lambda$; see \cite[Theorem 
4.3]{smith1998parametrix}. 

As
$$ \lng \Theta^\pm(t)(y-x_T(t)), \omega_T \rng=\lng y-x_T(t), \omega_T(t) \rng, $$
for any $ N \geq 0 $ we have
\begin{align}
&\vm{  (\omega_T(t)  \nabla_{y})^j (\omega_T(t)^{\perp}  \nabla_{y})^{\al} \ 
	u_T(t,y) } \nonumber \\
&\ls \lmd^{j+\frac{\vm{\al}}{2}+\frac{n+1}{4}} \frac{1}{ (1+\lmd \vm{ 
		\omega_T(t) \cdot (y-x_T(t))} + \lmd \vm{y-x_T(t)}^2 )^N}. 
		\label{e:smith-decay}
\end{align}

Thus, $ u_T(t,\cdot) $ are concentrated on spatial rectangles of size $ 
\lmd^{-1} \times (\lmd^{\frac{1}{2}})^{n-1} $ that get rotated according to $ 
\omega_T(t) $ as time evolves and are centered around $ x_T(t) $. By slight 
abuse of notation, we will also denote by $ T $ this space-time region, called 
a tube, where $ u_T $ is concentrated. For fixed $ \omega $ and a sign $ \pm $, 
corresponding to the lattice $ 
\Xi_{\lmd}^{\omega} $  we obtain a family $ 
\calT_{\lmd,\omega}^{\pm} $ of spacetime tubes which are finitely overlapping. 


We introduce the null foliation $\Lambda_\theta$ with
direction $\theta$ associated to the metric $g_{<\lambda^{1/2}}$, and
construct a null frame $\{L, \underline{L}, E\}$ as before.

The following computation is a variation of the proof of
\cite[Lemma 3.4]{smith1998parametrix}.
\begin{lemma}
  $Lu_T$ satisfies the decay estimate~\eqref{e:packet_decay}.
\end{lemma}

\begin{proof}
  By Lemma~\ref{l:half-fullwave-bichar}, we can replace $L$ with the
  operator $\partial_t + \langle a_\xi (t, y, \widehat{\xi}(t, y) ), \partial_y
  \rangle$.
  
  Assume WLOG that $x_T = 0$. Then we have

\begin{align*}
  (\partial_t + \langle a_{\xi} (t, x(t), \widehat{\xi} (t) ),
  \partial_y \rangle) u_T (t, y)&= \langle
                                  \widehat{\xi} (t), y - x(t) \rangle
                                  \langle a_x (t, x(t), \widehat{\xi}(t)), \partial_y
                                  \rangle u (t, y)\\
                                 &-\langle \widehat{\xi}(t)), y - x(t) \rangle
                                  \langle \widehat{\xi}(t), \partial_y
                                  \rangle u(t, y).
\end{align*}

Therefore
\begin{align*}
  (\partial_t + \langle a_\xi (t, y, \widehat{\xi}(t, y) ), \partial_y
  \rangle) u_T &= \langle a_\xi (t, y, \widehat{\xi} (t, y)) - a_\xi (t,
  x(t), \widehat{\xi} (t, x(t))), \partial_y \rangle u_T\\
  &- \langle a_x ( t, x(t), \widehat{\xi} (t, x(t))), y - x(t) \rangle
    \langle \widehat{\xi} (t, x(t)), \partial_y \rangle u_T \\
  &+ \langle                         \widehat{\xi} (t), y - x(t) \rangle
    \langle a_x (t, x(t), \widehat{\xi}(t)), \partial_y
    \rangle u_T.
\end{align*}
The third term is $\mu |\langle y - x(t), \widehat{\xi}(t)\rangle |$ times an
order $0$ packet, therefore is also order $0$.

To estimate the first two terms, for simplicity of notation we drop
the dependence in $t$ and write $x := x(t)$, $\xi(y) := \xi(t, y)$, 
$a(y, \widehat{\xi}(y)) := a(t, y, \widehat{\xi}(t, y))$. The first
two terms can then be written as
\begin{align*}
  &\langle a_\xi (y, \widehat{\xi}(y)) - a_{\xi} (x, \widehat{\xi} (x))
  - \langle a_x (x, \widehat{\xi}(x)), y-x \rangle \widehat{\xi}(x),
  \partial_y \rangle u_T\\
  &=\langle a_{\xi}(y, \widehat{\xi}(y)) - a_{\xi}(y,
    \widehat{\xi}(x)), \partial_y \rangle u_T \\
  &+ \langle a_{\xi}(y, \widehat{\xi}(x)) - a_{\xi} (x,
    \widehat{\xi}(x)) - \langle a_x (x, \widehat{\xi} (x)), y-x
    \rangle \widehat{\xi}(x), \partial_y \rangle u_T.
\end{align*}
For the first term, 
write $a_{\xi}(y, \widehat{\xi}(y)) - a_{\xi} (y, \widehat{\xi}(x)) =
a_{\xi \xi} (y, \widehat{\xi}(x))[ \widehat{\xi}(y)-\widehat{\xi}(x) ] +
O(|\widehat{\xi}(y)-\widehat{\xi}(x)|^2)$. The remainder is $O(|y-x|^2)$,
while the linear term has size $O(|y-x|)$ and is orthogonal to
$\widehat{\xi}(x)$ by homogeneity:
\begin{align*}
  \langle a_{\xi \xi} (y, \widehat{\xi}(x)) [ \widehat{\xi} (y) -
  \widehat{\xi}(x)], \widehat{\xi}(x) \rangle =\langle \widehat{\xi} (y) -
  \widehat{\xi}(x), a_{\xi \xi}(y, \widehat{\xi}(x)) \widehat{\xi}(x)
  \rangle = 0.
\end{align*}
Therefore
\begin{align*}
  \langle a_{\xi}(y, \widehat{\xi}(y)) - a_{\xi}(y,
    \widehat{\xi}(x)), \partial_y \rangle u_T  = \mu^{1/2}|y-x| v_T +
  \mu|y-x|^2 w_T,
\end{align*}
where $v_T$ and $w_T$ satisfy order $0$ decay, hence is also order $0$.
Finally, the remaining term is acceptable since
\begin{align*}
&\langle  \widehat{\xi}(x), a_{\xi} (y, \widehat{\xi}(x)) - a_\xi (x, \widehat{\xi}(x)) -
  \langle a_{x} (x, \widehat{\xi}(x)), y-x \rangle \widehat{\xi}(x) \rangle \\
  &=  a(y, \widehat{\xi}(x)) - a(x,
    \widehat{\xi}(x)) - \langle a_x (x,
    \widehat{\xi}(x)), y-x\rangle\\
  &= O(\| \partial^2 g(t)\|_{L^\infty_x}|y-x|^2),
\end{align*}
while the component orthogonal to $\widehat{\xi}(x)$ has size $O(|y-x|)$.
\end{proof}

\subsection{General metrics}

  Smith's construction adapts, with minor modifications, to more general
equations of the form
\begin{align*}
	g^{\alpha \beta} \partial_\alpha \partial_\beta u = 0, \ u[0] = (u_0, u_1).
\end{align*}
Here we will harmlessly assume that $g^{00} = 1$. 

For each frequency $\lambda \gg 1$, factor the principal symbol of the 
mollified operator $g^{\alpha\beta}_{<\lambda^{1/2}} \partial_\alpha 
\partial_\beta$ as $(\tau - a^+)(\tau - a^-)$, where
\[
  a^\pm(t, x, \xi) := g_{<\lambda^{1/2}}^{0b} \xi_b 
  \pm\bigl[(g^{0b}_{<\lambda^{1/2}}\xi_b)^2 + 
  g^{ab}_{<\lambda^{1/2}} \xi_a 
  \xi_b\bigr]^{1/2}.
\]
are smooth convex/concave symbols in view of the hyperbolicity condition.

Then the analogue of~\eqref{e:smith_hamilton_flow} is
\begin{equation}
  \label{e:smith_hamilton_flow2}
 \begin{cases}
             \dfrac{\dd x}{\dd t}= \partial_{\xi} a^\pm(t,x,\omega) \\ \\
            \dfrac{\dd \omega}{\dd t}= - \partial_x a^\pm(t,x,\omega) \pm \lng 
            \omega, \partial_x a^\pm(t,x,\omega) \rng \omega,\\ \\
              \dfrac{\dd \Theta^\pm}{dt} = -\Theta^\pm [ \omega \otimes a^\pm_x (t, x,
  \omega) - (a^\pm_x) (t, x, \omega ) \otimes
  \omega ],
       \end{cases}
\end{equation}
and the analogues of $u_T^\pm$ and $v_T^\pm$ are furnished by the following 
construction.
\begin{lemma}
	Fix $(x_T, \omega_T)$, and let $\varphi_T$ be as before. Then for all 
	sufficiently large frequencies $\lambda \gg 1$, there exist 
	functions $\varphi_T^\pm, \psi_T^\pm$ with similar properties as 
	$\varphi_T$ such that
	\begin{align*}
		u_T^\pm(t,y) := \varphi_T^\pm(\Theta^\pm(t)(y - x_T(t)) ), \quad 
		v_T^\pm(t,y) := \psi_T^\pm(\Theta^\pm(t)(y - x_T(t)) ),
	\end{align*}
	satisfy the same estimates as before, and
	\begin{align*}
	\begin{cases}
		u_T^+(0) + u_T^-(0) = \varphi_T + \lambda^{-1/2}\widetilde{\varphi_T}\\
		\partial_t u_T^+(0) + \partial_t u_T^-(0) = 
		\lambda^{-1/2}\widetilde{\varphi_T}
		\end{cases}
		\quad 	\begin{cases}
		v_T^+(0) - v_T^-(0) = \lambda^{-1/2}\widetilde{\varphi_T}\\
		\partial_t v_T^+(0) - \partial_t v_T^-(0) =  \varphi_T + 
		\lambda^{-1/2}\widetilde{\varphi_T},
		\end{cases},
	\end{align*}
	where $\tilde{\varphi_T}$ denotes generic function with similar smoothness 
	and decay as $\varphi_T$.
\end{lemma}

\begin{proof}
	Without loss of generality consider the first system. Using
	equations~\eqref{e:smith_hamilton_flow2}, and writing $\Phi_T = 
	(\varphi_T^+, \varphi_T^-)^*$, $F_T = (\varphi_T, 0)^*$, the system takes 
	the form
	\begin{align*}
	[M_T(D) + R_T(X, D)] \Phi_T = F_T + O(\lambda^{-1/2}).
	\end{align*}
	where 
	\begin{align*}
		M_T(D) &= \left(\begin{array}{cc} 1 & 1 \\-\langle a_\xi^+(0, z_T), 
		\partial_x \rangle & -\langle a_\xi^-(0, z_T), \partial_x 
		\rangle \end{array}\right), \quad z_T = (x_T,  \omega_T)\\
		R_T(X, D) &= \left(\begin{array}{cc}0 & 0 \\ \langle \dot{\Theta}^+(0) 
		(x-x_T), 
		\partial_x \rangle &  \langle \dot{\Theta}^- (0)(x-x_T) , \partial_x 
		\rangle 
		\end{array}\right).
	\end{align*}
When $\xi$ is restricted to a small sector centered at $\omega_T$ the main term 
is elliptic:
\begin{align*}
	|\det M_T(\xi)| = |\langle (a_\xi^+(0, z_T) - a_\xi^-(0, z_T)), \xi 
	\rangle| 
	\sim |\xi|(a^+ - a^-)(0, z_T) \sim |\xi|.
\end{align*}
Further, in view of spatial localization the
operator $R_T(X, D)$ 
has order $\tfrac{1}{2}$ when acting on functions of the form 
$\widetilde{\varphi_T}$. Thus it suffices to set $\Phi_T := M_T(D)^{-1} F_T$.
\end{proof}

For each frequency $\lambda >1 $ and time $t$, let $\Phi_\lambda(t): L^2 \times 
\lambda^{-1}L^2 
\to L^2 \times \lambda^{-1} L^2$ denote the operator
\begin{gather*}
	\Phi_\lambda (t)\binom{f}{g} := \sum_{T \in \mathcal{T}_\lambda} \Phi_T (t)
	\binom{f_T}{g_T},\\
	\Phi_T(t) := \left(\begin{array}{cc}
	u_T^+ + u_T^{-} &  v_T^+ - v_T^{-}\\
	\partial_t u_T^+  + \partial_t u_T^{-}  & \partial_t v_T^+ - 
	\partial_t v_T^{-}
	\end{array}\right)(t), \quad \binom{f_T}{g_T} = \binom{\langle f, \varphi_T \rangle}{\langle g, \varphi_T \rangle},
\end{gather*}
and set
\begin{align*}
	\Phi(t) := \sum_{\lambda > \lambda_0} \Phi_\lambda(t) P_\lambda + 
	\left(\begin{array}{cc} 1 & t\\0 &  
	1\end{array}\right) P_{\le\lambda_0}.
\end{align*}
Then as $\|\Phi(0) - I\|_{L^2 \times H^{-1} \to L^2 \times H^{-1}} = O(\lambda_0^{-1/2})$, by choosing $\lambda_0$ sufficiently large the operator $\Phi(0)$ is invertible on $L^2 \times H^{-1}$. Hence
\begin{align*}
    \widetilde{\Phi}(t) := \Phi(t) \Phi(0)^{-1}
\end{align*}
is a parametrix in the sense of Property~\ref{param:property}.

\

\section{Microlocal analysis tools}

\label{s:microlocal}

\

\subsection{Symbols and phase space metrics}

We begin by briefly reviewing the framework of Hormander metrics;
for further details consult~\cite[Section 18.4]{hormanderv3}.

Let $g$ be a slowly varying metric on phase space
$W :=T^* \R^n = \R^n_x \times (\R^n)^*_\xi$. This induces norms on
tensors in the usual manner; in particular, if
$l \in T_zW \times \cdots T_z W \to \R$ is $k$-linear for some
$z \in (x, \xi)$, set
\begin{align*}
  |l_z|_k^g := \sup_{0 \ne t_j \in T_zW} \frac{ |l_z (t_1, \dots, t_k)
  | }{ \prod_{j=1}^k g_z(t_j)^{1/2}}.
\end{align*}

\begin{definition} \label{def:symbol:class}
    If $m$ is a slowly varying function on $W$, write $S(m, g)$ for the
  space of functions $u$ on $W$ such that
  \begin{align*}
   \sup_z |\nabla^k u|^g_k (z) / m(z) \le C_k \quad \text{for all } z\in W.
  \end{align*}
\end{definition}

\begin{definition}
  A map $\chi : W \to W$ is $g$-smooth if the pullback $\chi^* S(1,g)
  \subset S(1,g)$. 
\end{definition}
By the chain rule and induction, this
definition is equivalent to requiring that
\begin{align*}
  |D^k \chi (z; t_1, \dots, t_k)|_{g_{\chi(z)}} \le C_k \prod_{j=1}^k
g_z(t_j)^{\frac{1}{2}} \text{ for all } t_j \in T_zW, \text{ uniformly in } z.
\end{align*}

\subsection{Pseudo-differential calculus}
Suppose $A$ be the quadratic form on $W$ given by
$A(y, \eta) = \langle y, \eta \rangle$, and for $\alpha \le 1$ let $g_\alpha$ 
denote the phase 
space metric~\eqref{e:g}. Let  $U_\alpha$ denote the phase space 
region
\[
 U_\alpha := \{ (x, \xi): |\xi| \ge \alpha^{-2}\}
\]In the language of Hormander 18.4,
the phase space metric $g$ is $A$-\emph{temperate} in $U_\alpha$ in the sense 
that there exist
constants $C, N$
\begin{align*}
  g_w(t) \le C g_z (t) (1 + g^A_w (z-w) )^N, \text{ for all } z, w \in U_\alpha,
\end{align*}
where for a general quadratic form $A$ one has
\begin{align*}
  g^A _w(A\zeta) := \sup_{\eta \in \operatorname{Im}(A)} \frac{ |\langle \eta , \zeta \rangle|^2}{g_w(\eta)}
\end{align*}
and $g^A_w (\beta):= \infty$ for $\beta \notin \operatorname{Im}(A)$. In the
present case,
\begin{align*}
    g^{A}_{(x, \xi)} (y, \eta) &= g^{-1}_{(x, \xi)} (\eta, y) = \alpha^4 |\langle \eta, \wht{\xi}
  \rangle|^2 + \alpha^2 | \eta \wedge \wht{\xi}|^2 + |\xi|^2 | \langle
  y, \wht{\xi} \rangle|^2 + \alpha^2 |\xi|^2 | y \wedge \wht{\xi} |^2,
\end{align*}
where we write $\wht{\xi} := \xi/|\xi|$.

Similarly, a slowly varying function $m$ is $A$-\emph{temperate} with respect 
to $z \in U_\alpha$ if
\begin{align*}
  m(w) \le C m(z) (1 + g_w^A (z-w) )^N \text{ for all } w \in U_\alpha.
\end{align*}

For a symbol $a$, we define the corresponding pseudo-differential
operator using right-quantization
\begin{align*}
  a(X, D) = (2\pi)^{-n} \int_{\R^n} e^{i\langle x -y, \xi \rangle}
  a(x, \xi) \, d\xi,
\end{align*}
and write $OPS(m, g)$ for the quantizations of symbols in $S(m, g)$.

Recall that if $a$ and $b$ are symbols, then formally
\begin{align*}
  a(X,D) b(X,D) = (a \circ b)(X, D),
\end{align*}
where
\begin{align*}
  a \circ b (x, \xi) := e^{i\langle D_\eta, D_y \rangle} [a(x, \eta)
  b(y, \xi)]|_{\stackrel{y=x}{\eta=\xi}}.
\end{align*}
In particular, we have the first order and second order symbol
expansions:
\begin{align}
  \label{e:symbol_expansion}
  \begin{split}
  a \circ b &= ab + \frac{1}{i} \int_0^1 r_{1,s} \, ds\\
  &= ab + \frac{1}{i} a_\xi b_x -\frac{1}{2} \int_0^1 r_{2, s} \, ds,
  \end{split}
\end{align}
where
\begin{align*}
  r_{j, s} (x, \xi) =  e^{is \langle D_y, D_\eta \rangle}
 \langle \partial_y, \partial_\eta\rangle^j  [a(x, \eta) b(y, \xi)]_{\stackrel{y=x}{\eta=\xi}}.
\end{align*}

The remainder can be estimated as in~\cite[Section 18.4]{hormanderv3}. For a real
parameter $t \ge 0$, define
\begin{align*}
  a \circ_t b (x, \xi) := e^{it\langle D_\eta, D_y \rangle} [a(x, \eta)
  b(y, \xi)]|_{\stackrel{y=x}{\eta=\xi}}.
\end{align*}

\begin{lemma}
  \label{l:gauss_transform}
  Suppose $m_1, m_2$ are slowly varying and $A$-temperate in the
  region $U_\alpha$. If $a \in S(m_1, 
  g_{\alpha})$,
  $b \in S(m_2, g_{\alpha} )$ are symbols supported in $U_\alpha$, then
  \begin{align*}
    a \circ_t b \in S( m_1 m_2, g_{\alpha})
  \end{align*}
 with constants uniform in $0
  \le t \le 1$.
\end{lemma}
\begin{proof}[Proof sketch]
  Follow the proof of Theorem 18.4.10 and Proposition 18.5.2 in 
  Hormander~\cite{hormanderv3}. The main
  point is that estimates for the Gauss transform $e^{it A(D)}$ only
  improve when $t \le 1$ since $g^{tA} = t^{-2} g^{A} \ge g^{A}$.

  Let $B_t:$ denote the quadratic form on $W \oplus W$ defined by
  $B_t(x, \eta), (y, \xi)) := t \langle y, \eta \rangle$, then one
  needs to verify that
  \begin{itemize}
  \item The metric $G := g_\alpha \oplus g_\alpha$ is slowly varying and
    $B_t$-temperate with respect to the diagonal $(x, \xi, x, \xi)$
    in $U_\alpha \times U_\alpha$, and $G \le G^{B_t}$.
  \item The weight function $M := m_1 \otimes m_2$ is slowly varying and temperate with respect
    to $B_t$ at the diagonal where $|\xi| \ge \alpha^{-2}$.
  \end{itemize}
\end{proof}

We hereafter denote by $S(m, g_\alpha)$ those symbols
supported in the region $|\xi| \ge \alpha^{-2}/8$. The corresponding
operators $a(X, D) \in OPS(m, g_\alpha)$ accept input frequencies
$\gtrsim \alpha^{-2}$.
The previous lemma shows that
$OPS(m_1, g_\alpha) \cdot OPS(m_2, g_\alpha) \subset OPS(m_1 m_2,
g_\alpha)$.

In our applications we encounter a slightly more complicated class of
symbols associated to two angular scales. For $\beta \ge \alpha$, let
\begin{align*}
S^1_\beta (m, g_\alpha) := \{   \phi \in S(m, g_\alpha): \partial_x \phi \in S(m, g_\alpha), \
  \partial_\xi \phi \in S( (\beta |\xi|)^{-1} m, g_\alpha)\}.
\end{align*}
In view of Lemma~\ref{l:gauss_transform} and the identities
\begin{align*}
  \partial_x (a \circ b) &= (\partial_x a) \circ b + a \circ
                           (i \partial_x b),\\
  \partial_\xi (a \circ b) &= (i\partial_\xi a) \circ b + a \circ
                             (\partial_\xi b),
\end{align*}
one sees that
\begin{align*}
  OPS^1_\beta (m_1, g_\alpha) \cdot OPS^1_\beta (m_2, g_\alpha)
  \subset OPS^1_\beta ( m_1 m_2, g_\alpha).
\end{align*}

However, we do not quite have the usual
pseudo-differential calculus since $OPS(m, g_\alpha)$ is not closed
under taking adjoints. One remedy is to consider the subclass
of operators in $OPS(m, g_\alpha)$ which output at frequencies $\ge \alpha^{-2}$.

\begin{lemma}
  \label{l:adjoint}
  Suppose $m$ is slowly varying and temperate with respect to
  $g_{\alpha}$. If $\phi \in S(m,
  g_{\alpha})$ satisfies $\phi(X,D) = S_{\ge\alpha^{-2}}(D) \phi(X,D)
  S_{\ge\alpha^{-2}}(D)$
  then 
  \begin{align*}
    \phi^* (x, \xi) = [e^{i\langle D_y, D_\eta \rangle}
    \phi(y, \eta) ]_{\stackrel{y=x}{\eta=\xi}} \in S(m, g_{\alpha})
  \end{align*}
  and is supported in $|\xi| \ge \alpha^{-2}/8$.
\end{lemma}

This leads to the main $L^2$ estimate:
\begin{lemma}{\cite[Theorem 4.8]{geba2007phase}}
  \label{l:CV}
  If  $a \in S(1, g_\alpha)$ in addition satisfies
  \begin{align*}
    |S_{<\lambda}(D_x + \xi) a (x, \xi)| \le C_N \Bigl(
    \frac{\lambda}{\langle \xi \rangle} \Bigr)^N, 1
    \le \lambda \le \langle \xi \rangle,
  \end{align*}
  then $a(X, D)$ is continuous on $L^2$. In particular, the conclusion
  holds for operators of the form
  \[a(X, D) = S_{>\lambda/8}(D)a(X, D) S_{\lambda}(D), \quad \lambda
    \ge \alpha^{-2}.\]
\end{lemma}
\begin{proof}
  For the last statement, note that since
  \begin{align*}
    \wht{a(X,D) u}(\eta) = \int \wht{a}(\eta - \xi, \xi)
    \wht{u}(\xi) \, d\xi,
  \end{align*}
  the hypothesis implies that  $S_{<\lambda/8} (D_x + \xi) a(x, \xi) =
  0$.
\end{proof}

We also need bounds for certain pseudo-differential operators
which are strongly localized in input frequency but do not quite
satisfy the hypotheses of the previous lemma. For a direction field
$x \mapsto \Theta(x) \in S^{n-1}$ on $\R^n$, let
\begin{align*}
  m = m_\Theta (x, \xi) := \langle \alpha^{-1} (\wht{\xi} - \Theta(x) )
  \rangle^{-1}, \quad \wht{\xi} := \xi / |\xi|.
\end{align*}
To express the angular localization of various symbols it will be
convenient to introduce the following

\textbf{Notation}. For a weight function $m$, write $S(m^\infty,
g_{\alpha}) := \bigcap_N S( m^N, g_\alpha)$.

\begin{lemma}
  \label{l:schur-bound}
  Suppose $x \mapsto \Theta(x) \in S^{n-1} \subset \R^n$ is Lipschitz, and let $\phi
  \in S(m^\infty, g_\alpha)$ be supported in an annulus $|\xi| \sim \lambda \ge
  \alpha^{-2}$. Then $\phi(X, D)$ is bounded on $L^p$ for all $1 \le p
  \le \infty$.
\end{lemma}

\begin{proof}
  We use Schur's test on the kernel $K(x, y) = (2\pi)^{-n}\int e^{i\langle x-y,
    \xi\rangle} \phi(x, \xi) \, d\xi = \mathcal{F}_2^{-1} \phi (x,
  x-y)$ of $\phi(X, D)$. For fixed $x$, $\xi \mapsto \phi(x, \xi)$ is
  a Schwartz function with height $1$ and adapted to the sector
  $|\wht{\xi} -\Theta(x)| \lesssim \alpha$, $|\xi| \sim
  \lambda$. Therefore 
  \begin{align*}
    |K(x, y)| \lesssim \lambda (\alpha \lambda)^{n-1} \langle \lambda
    |\langle x-y, \Theta(x) \rangle| + \alpha \lambda |x-y \wedge
    \Theta(x)| \rangle^{-N} \text{ for any }N,
  \end{align*}
  so
  \begin{align*}
    \sup_{x} \int |K(x, y)| \, dy < \infty.
  \end{align*}
  To evaluate
  \begin{align*}
    \sup_y \int |K(x,y)| \, dx,
  \end{align*}
  decompose $\phi = \sum \phi_{\nu}^\omega$, where $\omega$ ranges
  over a partition of the annulus $|\xi| \sim \lambda$ into $\lambda \times
  (\alpha \lambda)^{n-1}$ sectors, and for each $\omega$, $\nu$ ranges over a partition
  of space into parallel $\alpha^2 \times (\alpha)^{n-1}$ parallelpipeds
  $R^{\omega}_\nu$ with orientation
  $\omega$. The kernel decomposes as
  \begin{align*}
    K(x,y) = \sum_{\omega} \sum_\nu \mathcal{F}_2^{-1} \phi_\nu^\omega (x,
    x-y) = \sum_{\omega} \sum_\nu K_\nu^\omega (x, y).
  \end{align*}
Then
  \begin{align*}
    |K_\nu^\omega(x,y)| \lesssim \mathrm{1}_{R_\nu^\omega}(x) \langle
    \alpha^{-1} (\omega - \Theta(x_\nu^\omega) ) \rangle^{-N} \lambda
    (\alpha \lambda)^{n-1} \langle  \lambda |\langle x-y, \omega
    \rangle| + \alpha \lambda |x-y \wedge \omega | \rangle^{-N}.
  \end{align*}
  For a constant $c$ depending on the Lipschitz constant of $\Theta$, 
    \begin{align*}
    &\int |K(x, y)| \, dx \lesssim \sum_{ |x_\nu^\omega-y|
    \le c(|\omega - \Theta(y)| + \alpha)} \langle \alpha^{-1} (
      \omega - \Theta(y) ) \rangle^{-N} \\
      &\times \int_{R_\nu^\omega} \lambda
    (\alpha \lambda)^{n-1} \langle  \lambda |\langle x-y, \omega
      \rangle| + \alpha \lambda |x-y \wedge \omega | \rangle^{-2N} \, dx\\
    &+ \sum_{ |x_\nu^\omega-y|
    > c(|\omega - \Theta(y)| + \alpha)} \langle \alpha^{-1} (
                           \omega - \Theta(x_\nu^\omega) )
      \rangle^{-N}\langle \alpha \lambda |\omega - \Theta(y)| \rangle^{-N} \\
      &\times \int_{R_\nu^\omega} \lambda
    (\alpha \lambda)^{n-1} \langle  \lambda |\langle x-y, \omega
      \rangle| + \alpha \lambda |x -y \wedge \omega | \rangle^{-N} \,
        dx\\
      &\lesssim \sum_{\omega} \langle \alpha^{-1}(\omega - \Theta(y))
        \rangle^{-N} < \infty.
    \end{align*}
    
\end{proof}

\section{An angular partition of unity} 

\

Throughout this discussion we fix a 
minimum angular scale $\alpha_\mu
\ll 1$.

Choose $0 \le \eta \in C^\infty_0\bigl( (-1, 1) \bigr)$ 
with $\int \eta = 1$, and let $\eta_h = h^{-1} \eta( h^{-1} \cdot )$. For each $\theta 
\in \Omega_\alpha$, with  $\theta = [a. b)$, define the scale functions
\begin{align*}
 &h_1(\xi) = \left\{\begin{array}{ll} \frac{1}{8}\alpha, & \xi \le 
\frac{a+b}{2},\\
            \frac{1}{16} \alpha, & \xi > \frac{a+b}{2}
                   \end{array}
\right. \quad
 &h_2(\xi) = \left\{\begin{array}{ll} \frac{1}{16}\alpha, & \xi \le 
\frac{a+b}{2},\\
            \frac{1}{8} \alpha, & \xi > \frac{a+b}{2}
                   \end{array}
\right.\\
 &h_3(\xi) = \left\{\begin{array}{ll} \frac{1}{16}\alpha, & \xi \le 
\frac{a+b}{2},\\
            \frac{1}{16} \alpha, & \xi > \frac{a+b}{2}
                   \end{array}
\right. \quad
&h_4(\xi) = \left\{\begin{array}{ll} \frac{1}{8}\alpha, & \xi \le 
\frac{a+b}{2},\\
            \frac{1}{8} \alpha, & \xi > \frac{a+b}{2}
                   \end{array}
\right.
\end{align*}
and for $k = 1, 2, 3, 4$, define $\phi_{\theta}^{\alpha, k}$ by mollifying the
characteristic functions of $\theta$ on a position-dependent scale:
\begin{align*}
 \phi^{\alpha, k}_\theta (\xi) := \mathrm{1}_{\theta} * \eta_{h_k(\xi)} (\xi).
\end{align*}

Then each $\phi^{\alpha, k}_\theta$ is smooth on the $\alpha$ scale and supported in 
a $\tfrac{1}{8}|\theta|$ neighborhood of the interval $\theta$, and
there exist $c_1, c_2 > 0$ such that
\begin{align}
\label{e:approx_partition}
 c_1 \le \sum_{\theta \in \Omega_\alpha} \phi^{\alpha, k(\theta)}_\theta \le c_2
\end{align}
for any sequence $k(\theta) \in \{1, 2, 3, 4\}$.

\begin{proposition}
  \label{p:bilinear_pou}
  Fix a small dyadic number $\alpha_\mu < 1/4$.
 For each $\omega \in \Omega_{\alpha_\mu}$, there exist dyadic intervals $\theta_0 
= \omega, \theta_1, \dots $ on the circle from the families $ \Omega_{\alpha} $ with, $|\theta_j - \omega| \sim \alpha_j$ when $j 
\ge 1$, such 
that we have the (fixed time) partition of unity 
\begin{align*}
 1 = \sum_{j} \phi^{\alpha_j, k(\omega)}_{\theta_j},
\end{align*}
where $k(\omega) \in \{1, 2, 3, 4\}$, and the scales $\alpha_j$
satisfy
\begin{itemize}
 \item $\tfrac{1}{2}\alpha_{j-1} \le \alpha_{j} \le 2 \alpha_{j-1}$, and
 \item at most $O(1)$ consecutive $\alpha_j$'s are equal.
\end{itemize}
The families $ \Omega_{\alpha} $ are independent of $ \omega $. 
\end{proposition}

Before the proof, we consider 

\begin{lemma}
 Given $\omega \in \Omega_{\alpha_0}$, there exists a sequence $\omega = 
\theta_0, \theta_1, \theta_2, \dots$ of dyadic intervals in $\mathbb{R}$ with the following 
properties:
\begin{itemize}
 \item $\theta_j$ is right-adjacent to $\theta_{j-1}$.
 \item $2|\theta_{j-1}| \le |\theta_j| \le | \theta_{j+1}|$ for all $j$.
 \item If $|\theta_{j-1}| = 
|\theta_j|$, then $|\theta_{j+1}| = 2 |\theta_j|$.
\item $\sum_{j=1}^J |\theta_j| \le 4 |\theta_J|$ for all $J$.
\end{itemize}
The analogous statement with ``right'' and ``left'' interchanged 
also holds.
\end{lemma}

\begin{proof}
The idea of the construction is to ``double whenever possible while moving right.''
 
 Using the usual tree terminology, define the sequence $\theta_j$ inductively as follows:
 \begin{itemize}
  \item If $\theta_{j}$ is the left-child of its parent, let $\theta_{j+1}$ be its 
sibling.
\item Else let $\theta_{j+1}$ be the right neighbor of $\theta_j$'s parent.
 \end{itemize}
Then $|\theta_{j+1}| = |\theta_j|$ in the first case and $|\theta_{j+1}| = 
2|\theta_j|$ in the second. Since each interval has only two children, there cannot be 
more than two consecutive intervals of the same width.

The inequality $\sum_{j=1}^J |\theta_j| \le 4|\theta_J|$ is verified inductively. 
Assume it holds for all smaller $J$. If $|\theta_{J+1}| = 2|\theta_J|$, then 
\begin{align*}
 \sum_{j=1}^{J+1} |\theta_j| \le 4 |\theta_J| + 2|\theta_J| \le 4|\theta_{J+1}|,
\end{align*}
while if $|\theta_{J+1}| = |\theta_J|$, then $|\theta_J| = 2|\theta_{J-1}|$ and we have
\begin{align*}
 \sum_{j=1}^{J+1} |\theta_j| \le 4|\theta_{J-1}| + 2|\theta_{J-1}| + 2|\theta_{J-1}| \le 
4 |\theta_{J+1}|.
\end{align*}
\end{proof}

\begin{proof}[Proof of Prop.~\ref{p:bilinear_pou}]
  
Starting with $\omega$, we construct dyadic intervals $\theta_1^r, \theta_2^r, \dots$ 
and $\theta_1^l, \theta_2^{l}, \dots$ to the right and left of $\omega$, respectively, 
according to the lemma. Choose $J^r$ and $J^l$ maximal such that 
\begin{align*}
 \sum_{j=1}^{J^r} |\theta^r_j| \le \frac{1}{2} - |\omega|, \quad \sum_{j=1}^{J^l} 
|\theta^l_j| \le \frac{1}{2} - |\omega|.
\end{align*}
By the lemma one must have
\begin{align*}
 \frac{1}{8} - \frac{|\omega|}{4} \le |\theta^r_{J^r}|, |\theta^l_{J^l} | \le \frac{1}{2} - 
|\omega|;
\end{align*}
that is,
\begin{align*}
 |\theta^r_{J^r}|, |\theta^l_{J^l}| \in \Bigl\{ \frac{1}{16}, \frac{1}{8}, \frac{1}{4} \Bigr\}.
\end{align*}
Assume with loss of generality that $|\theta^l_{J^l}| \le |\theta^r_{J^r}|$. 

\begin{itemize}
 \item If $|\theta^l_{J^l}| = \tfrac{1}{16}$, $|\theta^r_{J^r}| = \tfrac{1}{4}$, then 
$\theta^r_{J^r}$ and $\theta^l_{J^l}$ are separated modulo 1 by less than $10$ dyadic 
intervals of width $\tfrac{1}{16}$ (the worst case being when $|\theta^l_{J^l + 1}| = 
\tfrac{1}{8}$ and $|\theta^r_{J^r + 1}| = \tfrac{1}{2}$). We replace $\theta^r_{J^r}$ by 
its two children $\theta^{r, 1}_{J^r}, \ \theta^{r, 2}_{J^r}$ of width $\tfrac{1}{8}$, and 
reindex the intervals by defining $\theta^r_{J^r} := \theta^{r, 1}_{J^r}$, $\theta^r_{J^r 
+ 1} := \theta^{r, 2}_{J^r}$, and replacing $J^r$ by $J^r + 1$.

\item If $|\theta^l_{J^l}|$ and $|\theta^r_{J^r}|$ differ by at most one dyadic scale, 
then they are separated modulo 1 by at less than $6$ dyadic intervals of width 
$|\theta^l_{J^l}|$.
\end{itemize}
Let $\theta'_{1}, \dots, \theta'_{n} \in 
\Omega_{|\theta^l_{J^l}|}, \ n < 10$ be the intervening intervals. 
Projecting the intervals to $\mathbb{R} / \mathbb{Z}$, we relabel the sequence
\[\theta_0, \theta^r_{1}, \dots, \theta^r_{J^r}, \theta'_1, \dots, \theta'_n, 
\theta^l_{J^l}, \dots, \theta^{l}_1\]
as
\[
\theta_0, \theta_1, \dots, \theta_{M-1},
\]
and write \[\alpha_j := |\theta_j|.\] Then, interpreting the indices modulo $M$, we have
\begin{itemize}
 \item $\tfrac{1}{2}\alpha_{j-1} \le \alpha_{j} \le 2 \alpha_{j-1}$, and
 \item at most $O(1)$ consecutive $\alpha_j$'s are equal.
\end{itemize}
For each $\theta_j = [a_j, b_j)$ (mod 1), define the scales
\begin{align*}
h_{j}(\xi) = \left\{\begin{array}{ll} \frac{\alpha_j}{8}, & \alpha_j \le \alpha_{j+1}\\
                     \frac{\alpha_j}{16}, &  \alpha_j > \alpha_{j+1}
                    \end{array}
\right.
\end{align*}
and define $\phi^{\alpha_j}_{\theta_j}$ by mollifying the characteristic functions of 
$\theta_j$ near the junction points.
\begin{align*}
 \phi^{\alpha_j}_{\theta_j} (\xi) = \left\{\begin{array}{ll} \mathrm{1}_{\theta_j} * 
\eta_{h_{j-1}} (\xi), &  a_j - \frac{\alpha_j}{8}  \le \xi \le \frac{a_j + b_j}{2}, \\
\mathrm{1}_{\theta_j} * \eta_{h_j} (\xi), & \frac{a_j+b_j}{2} < \xi \le b_j + 
\frac{\alpha_j}{8},\\
0, & \text{otherwise}
\end{array}\right.
\end{align*}
This function has one of the four forms asserted in the proposition, and
\begin{align*}
 1 = \sum_{j=0}^{M-1} \phi^{\alpha_j}_{\theta_j}
\end{align*}
since the same is true for the sum of characteristic functions $\mathrm{1}_{\theta_j}$ 
and the mollification scale is the same near the boundary between $\theta_j$ and 
$\theta_{j+1}$.
\end{proof}

\

\nocite{*}
\bibliographystyle{abbrv}
\bibliography{bibliography.bib}
\

\end{document}